\documentclass{sfuthesis}

\title{Rooted Graph Minors and Reducibility of Graph Polynomials}
\thesistype{Dissertation}
\author{Benjamin Richard Moore}
\previousdegrees{%
B.A., Thompson Rivers University, 2015}
\degree{Masters of Science}
\discipline{Mathematics}
\department{Department of Mathematics and Statistics}
\faculty{Faculty of Science}
\copyrightyear{2017}
\semester{Spring 2017}
\date{April 6, 2017}


\committee{%
	\chair{Dr.\ Ladislav Stacho}{Associate Professor}
	\member{Dr.\ Karen Yeats}{Senior Supervisor\\Associate Professor}
	\member{Dr.\ Luis Goddyn}{Supervisor\\Professor}
	\member{Dr.\ Bojan Mohar}{Internal Examiner\\ Professor}
}
\usepackage{amsmath,amsthm}
\usepackage{color}
\usepackage{tikz}
\usetikzlibrary{calc}
\usepackage{amsfonts}
\usepackage{graphicx}
\usepackage{mathtools}
\usepackage{pdfpages}

\newtheorem{theorem}{Theorem}
\newtheorem{corollary}[theorem]{Corollary}
\newtheorem{proposition}[theorem]{Proposition}
\newtheorem{definition}[theorem]{Definition}
\newtheorem{remark}[theorem]{Remark}
\newtheorem{Conjecture}[theorem]{Conjecture}
\newtheorem{observation}[theorem]{Observation}

\newtheorem{example}[theorem]{Example}
\newtheorem{lemma}[theorem]{Lemma}
\newtheorem{question}[theorem]{Question}
\newtheoremstyle{named}{}{}{\itshape}{}{\bfseries}{.}{.5em}{\thmnote{#3's }#1}
\theoremstyle{named}
\newtheorem*{namedtheorem}{Theorem}
\usepackage{color}
\newtheorem{jordantheorem}{Jordan Curve Theorem}

\usepackage{afterpage}
\usepackage{amsmath,amssymb,amsthm}
\usepackage[pdfborder={0 0 0}]{hyperref}
\usepackage{graphicx}
\usepackage{caption}

\begin{document}

\frontmatter
\maketitle{}
\makecommittee{}

\begin{abstract}
	 In 2009, Brown gave a set of conditions which when satisfied imply that a Feynman integral evaluates to a multiple zeta value. One of these conditions is called reducibility, which loosely says there is an order of integration for the Feynman integral for which Brown's techniques will succeed. Reducibility can be abstracted away from the Feynman integral to just being a condition on two polynomials, the first and second Symanzik polynomials. The first Symanzik polynomial is defined from the spanning trees of a graph, and the second Symanzik polynomial is defined from both spanning forests of a graph and some edge and vertex weights, called external momenta and masses. Thus reducibility is a property of graphs augmented with certain weights. We prove that for a fixed number of external momenta and no masses, reducibility is graph minor closed, correcting the previously claimed proofs of this fact. A computational study of reducibility was undertaken by Bogner and L\"{u}ders who found that for graphs with $4$-on-shell momenta and no masses, $K_{4}$ with momenta on each vertex is a forbidden minor. We add to this and find that when we restrict to graphs with four on-shell external momenta the following graphs are forbidden minors: $K_{4}$ with momenta on each vertex, $W_{4}$ with external momenta on the rim vertices, $K_{2,4}$ with external momenta on the large side of the bipartition, and one other graph. We do not expect that these minors characterize reducibility, so instead we give structural characterizations of the graphs not containing subsets of these minors. We characterize graphs not containing a rooted $K_{4}$ or rooted $W_{4}$ minor, graphs not containing rooted $K_{4}$ or rooted $W_{4}$ or rooted $K_{2,4}$ minors, and also a characterization of graphs not containing all of the known forbidden minors. Some comments are made on graphs not containing $K_{3,4}$, $K_{6}$ or a graph related to Wagner's graph as a minor.
\end{abstract}

%\begin{dedication} % optional
%\end{dedication}

\begin{acknowledgements} % optional
I'd like to thank my supervisor Karen Yeats for suggesting this problem to me and being a great supervisor. I'd like to thank Rick Brewster and Sean McGuinness for introducing me to research in graph theory. I'd like to thank Bruce Crofoot for initially helping me decide to pursue mathematics. I'd like to thank Erik Panzer for helping me use his code to determine if graphs are reducible. Also, I'd like to thank any friends and family who were supportive. Lastly, I'd like to thank NSERC for financial support. 

\end{acknowledgements}

\addtoToC{Table of Contents}\tableofcontents\clearpage
%\addtoToC{List of Tables}\listoftables\clearpage
\addtoToC{List of Figures}\listoffigures

%   MAIN MATTER  %%%%%%%%%%%%%%%%%%%%%%%%%%%%%%%%%%%%%%%%%%%%%%%%%%%%%%%%%%%%%%

%   Start writing your thesis --- or start \include ing chapters --- here.
%

\mainmatter%

\chapter{Introduction}

\section{Motivation}

This thesis looks at when graphs have a property called reducibility. Reducibility is closely related to when Feynman integrals evaluate to multiple zeta values. Feynman integrals arise naturally in perturbative quantum field theory and have been the focus of a large amount of research.

One advance in calculating Feynman integrals was obtained by Brown in \cite{FrancisBig}, where he gave sufficient conditions for when a Feynman integral evaluates to a multiple zeta value. He also showed that these conditions hold for an infinite family of Feynman integrals, leading to one of the largest known families of graphs for which it is known how to calculate the Feynman integral. One of Brown's conditions is reducibility, which essentially tells you a suitable order of integration for Brown's techniques. However, reducibility can be abstracted away from Feynman integrals into simply being a property of polynomials. As the polynomials relevant for Feynman integrals come from graphs, reducibility can be viewed as a property of graphs. Furthermore, it can be shown that reducibility is graph minor closed, and thus one can use graph minor theory to try and understand reducibility. This thesis looks at reducibility from a graph minor point of view. The outline of this thesis is as follows. First, in chapter one, we  give an overview of the Symanzik polynomials. In chapter $2$, we give a proof that reduciblity is graph minor closed (Corollary \ref{MinorClosed}), correcting the proofs given in \cite{Martin} and \cite{Bognerpaper}. We also give a new forbidden minor, which we denote $L$, for reducibility for graphs with four on-shell external momenta, and rewrite the non-reducible graphs in  \cite{Martin} in terms of forbidden minors. In chapter $3$, we give two different excluded minor theorems for rooted $K_{4}$ and rooted $W_{4}$ minors (Theorem \ref{w4cuts} and Theorem \ref{3connectivityW4K4}). We also give the complete class of graphs not containing any of a rooted $K_{4}$, rooted $W_{4}$, or rooted $K_{2,4}$-minor (Theorem \ref{k4w4k24characterization}). Additionally, we give the complete class of graphs not containing any of a rooted $K_{4}$, rooted $W_{4}$, rooted $K_{2,4}$ or rooted $L$-minors (although here we suspect a nicer characterization should be able to be found). We also make some observations about $K_{3,4}$-minors, $K_{6}$-minors and minors of a graph related to Wagner's graph.  

\section{Graph Theory Basics}
Graphs will be the fundamental objects throughout this thesis. This section will give basic definitions of graph theory terms which will be used throughout. A good textbook for basic graph theory is \cite{BondyAndMurty}.

A graph $G$ is a pair $(V(G), E(G))$ where $E(G)$ is a 2-element (multi)set of $V(G)$. The elements of $V(G)$ are called the \textit{vertices} of the $G$, and the elements of $E(G)$ are the edges of $G$. 
Two vertices $u$ and $v$ in a graph are said to be adjacent if $\{u,v\} \in E(G)$, which we abbreviate to $uv \in E(G)$.  If $e =uv \in E(G)$, we say that $u$ and $v$ are \textit{neighbours} and that $e$ has \textit{endpoints} $u$ and $v$.  The set of all neighbors of a vertex $v$ is called the \textit{neighborhood} of $v$, denoted $N(v)$. Given a set of vertices $X$, we similarly define $N(X)$ to be the set of neighbours for the vertices in $X$, not including vertices in $X$. Given a graph $G = (V(G),E(G))$, and sets $V' \subseteq V(G)$ and $E' \subseteq E(G)$ such that $e=uv \in E'$ implies $u,v \in V'$, then the graph $G' = (V',E')$ is a \textit{subgraph} of $G$. If we have a graph $G = (V(G),E(G))$ and $X \subseteq V(G)$, we let $G[X]$ be the graph where $V(G[X]) = X$ and $E(G[X]) = \{uv | \ u \in X \ \text{and} \ v \in X\}$. We say that $G[X]$ is the \textit{graph induced by $X$}. If $e_{1} = uv \in E(G)$ and $e_{2} = uv \in E(G)$, but $e_{1} \neq e_{2}$ then we say that $e_{1}$ and $e_{2}$ are \textit{parallel edges}. If for some edge $e = uv \in E(G)$, we have $u=v$, then we say $e$ is a \textit{loop edge}. A graph is \textit{simple} if there are no parallel or loop edges.

Given a graph $G$, a vertex $v \in V(G)$ is a \textit{dominating vertex} if the neighbourhood of $v$ is $V(G) \setminus \{v\}$. Some authors refer to this as an \textit{apex vertex}.

Given a graph $G$, its \textit{line graph} $L(G)$ is a graph such that for each edge $e$ in $G$, we create a vertex $v$ in $L(G)$ and given two vertices $x,y \in V(L(G))$, we have $xy \in E(L(G))$ if and only if the corresponding edges in $G$ share an endpoint.

Given an edge $e= xy \in E(G)$, the graph obtained by \textit{contracting} $e$, denoted $G / e$, is the graph where we delete the vertices $x$ and $y$, and replace them with a new vertex $z$ where for each vertex $u$ adjacent to $x$, we have an edge $ux$, and for each edge adjacent to $y$ we have an edge $uy$. Given a graph $G$, to \textit{delete an edge}, for some edge $e \in E(G)$, is to create a graph $G \setminus e$ where $V(G \setminus e) = V(G)$ and $E(G \setminus e) = E(G) \setminus \{e\}$. It is an easy exercise to show that deletion and contraction commute.

A graph $G$ has a \textit{$H$-minor} for some graph $H$, if a graph isomorphic to $H$ can be obtained from $G$ through the contraction and deletion of some edges, and possibly the removal of some isolated vertices. It is easy to see that if $G$ is connected, then we never need to delete isolated vertices. If for some $H$, a graph $G$ does not contain an $H$-minor, then we say that $G$ is \textit{$H$-minor-free}. If $J$ is a minor of a graph $G$, then $J$ is \textit{proper} if $J$ is not isomorphic to $G$. A set of graphs, $\mathcal{G}$, is \textit{graph minor closed} if for every graph $G \in \mathcal{G}$, all proper minors of $G$ are in $\mathcal{G}$.
Let $\mathcal{G}$ be a set of graphs which is graph minor closed. A \textit{forbidden minor} of $\mathcal{G}$ is a graph $H$ such that $H \not \in \mathcal{G}$ but all proper minors of $H$ are in $\mathcal{G}$.

Given a graph $G$ and an edge $e = xy \in E(G)$, to \textit{subdivide an edge $e$} is to create a new graph $G'$ where $V(G') = V(G) \cup \{z\}$ and $E(G') = E(G) \setminus \{xy\} \cup \{zx,zy\}$. A graph $G$ has a graph $H$ as a \textit{topological minor} if there is a subgraph of $G$ which is isomorphic to a iterated subdivision of $H$ in $G$.

Given a graph $G$, a \textit{$k$-separation} of $G$ is a pair $(A,B)$ such that $A \subseteq V(G)$, $B \subseteq V(G)$, $A \cup B = V(G)$, $|A \cap B| \leq k$, and if $v \in B \setminus A$, and $u \in A \setminus B$, then $uv \not \in E(G)$. The vertices in $A \cap B$ are called the \textit{vertex boundary} of the separation. We say a $k$-separation is \textit{proper} if $A \setminus (A \cap B) \neq \emptyset$ and $B \setminus (A \cap B) \neq \emptyset$.  A proper $k$-separation is \textit{tight} if for all subsets $X \subsetneq A \cap B$, the set $X$ is not the vertex boundary of a separation. In a proper $1$-separation $(A,B)$, the vertex in $A \cap B$ is called a \textit{cut-vertex}. In general, a set of vertices, $\{v_{1},\ldots,v_{k}\}$, is a \textit{$k$-vertex-cut} if there is a proper $k$-separation $(A,B)$ such that $A \cap B = \{v_{1},\ldots,v_{k}\}$. 

We will run into a few common graph families throughout the thesis. Let $n \in \mathbb{N}$.  The \textit{complete graph on $n$ vertices}, denoted $K_{n}$, is the graph where we can label the vertices $\{v_{0},\ldots,v_{n-1}\}$ such that $E(G) = \{v_{i}v_{j} |\  \forall \ i,j \in \{0,\ldots,n-1\} , i\neq j\}$. If a graph $G$ has a complete graph on $k$ vertices as a subgraph, the subgraph will be called  a \textit{$k$-clique}.

Let $G_{1}$ and $G_{2}$ be graphs a $k$-cliques as subgraphs. A \textit{$k$-clique-sum} or just \textit{$k$-sum} of $G_{1}$ and $G_{2}$ is a bijective identification of pairs of vertices in the two $k$-cliques with, if desired, removal of some edges from the new $k$-clique. We note sometimes authors enforce that all edges are removed in a $k$-sum. In practice, under the assumption we can have parallel edges, this is equivalent to the above definition, as one simply adds parallel edges as desired. 

A graph $G = (V,E)$ is \textit{bipartite} if there is a bipartition $(A,B)$ of $V$ such that if $e = uv \in E$, then exactly one of $u$ and $v$ is in $A$ and exactly one of $u$ or $v$ is in $B$. 
Let $n,m \in \mathbb{N}$. A \textit{complete bipartite graph} with partition sizes $n$ and $m$, denoted $K_{n,m}$, is a bipartite graph with a bipartition $(A,B)$ such that $|A| = n$, $|B|=m$ and every vertex in $A$ is adjacent to every vertex in $B$. 

A graph $P$ is a \textit{path on $n$ vertices} if we can label the vertices of $P$,  $v_{0},\ldots,v_{n-1}$ such that $E(P) = \{v_{i}v_{i+1} | \ \forall i \in \{0,\ldots,n-2\}\}$. Given a path $P$, the vertices $v_{0}$ and $v_{n-1}$ are the \textit{endpoints} of $P$. An \textit{$(a,b)$-path} is a path with endpoints $a$ and $b$. Two paths are \textit{disjoint} if their vertex sets are disjoint. Two $(a,b)$-paths are \textit{internally disjoint} if the intersection of their vertex sets is $\{a,b\}$. Given two sets of vertices $X$ and $Y$, an $(X,Y)$-path is a path with one endpoint in $X$ and one endpoint in $Y$.
  A graph is \textit{connected} if for every pair of vertices, $x,y$, it contains an $(x,y)$-path. A graph $G$ is \textit{$k$-connected} if $|V(G)| \geq k+1$ and for every pair of vertices $x,y$, there are $k$ internally disjoint $(x,y)$-paths. Menger's Theorem is a very important and well known result on $k$-connectivity (\cite{Menger1927}).
  
 \begin{namedtheorem}[Menger] 
Let $G$ be a graph and $x,y$ be non-adjacent vertices. Then the maximum number of pairwise internally disjoint $(x,y)$-paths is equal to the minimum order of a separation $(A,B)$ where $x \in A$ and $y \in B$.
\end{namedtheorem}

A useful immediate implication of Menger's Theorem is:

\begin{corollary}
\label{XYpaths}
Suppose $G$ is $k$-connected. Let $X,Y \subseteq V(G)$ and $|X|, |Y| \geq k$. Then there is a family of $k$ disjoint $(X,Y)$-paths.
\end{corollary}

  A \textit{cycle on $n$ vertices}, denoted $C_{n}$,  is a graph where $V(C_{n}) = \{v_{0},\ldots,v_{n-1}\}$ and $E(C_{n}) = \{v_{i}v_{i+1}\} \cup \{v_{0}v_{n-1}\}$ for $i \in \{0,\ldots,n-2\}$. A \textit{tree} is a connected graph with no cycles. Given a graph $G$, a \textit{spanning tree of $G$}, $T$, is a subgraph of $G$ where $T$ is a tree and $V(T) = V(G)$. 
 
  A \textit{wheel on $n$ spokes}, denoted $W_{n}$, is a graph where $V(W_{n}) = \{v_{0},\ldots,v_{n}\}$ such that $\{v_{0},\ldots,v_{n-1}\}$ induces a cycle on $n$ vertices, and $v_{n}$ is adjacent to every other vertex. The vertices $\{v_{0},\ldots,v_{n-1}\}$ are called \textit{rim vertices} and $v_{n}$ is called the \textit{hub}.

\begin{figure}
\begin{center}
\includegraphics[scale =0.5]{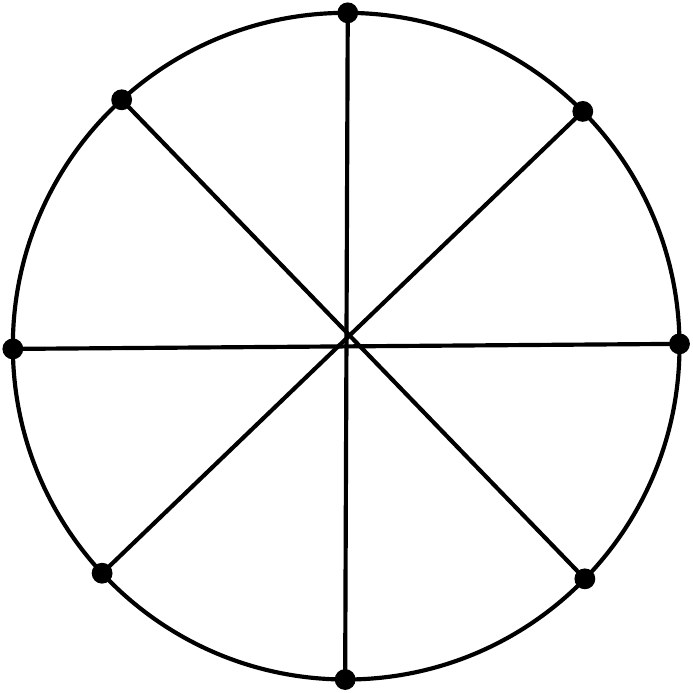}
\caption{The Wagner graph, $V_{8}$.}
\end{center}
\end{figure}

The \textit{Wagner graph}, denoted  $V_{8}$, is defined to have vertex set $\{0,1,2,3,4,5,6,7\}$ and edge set $\{01,12,23,34,45,56,67,70,04,15,26,37\}$. 

A graph $G$ is \textit{planar} if $G$ can be embedded in the Euclidean plane, where embedded means the graph can be drawn in the Euclidean plane without edges crossing. For $2$-connected graphs which are embedded in the plane, a \textit{face} is a region bounded by a cycle $C$ such that all vertices and edges not in $C$ are drawn outside of $C$. A cycle which bounds a face is a \textit{facial cycle}. A \textit{facial path} between two vertices on a facial cycle $C$ is a path $P$ between the two vertices such that $V(P) \subseteq V(C)$. We will use the following well known result.

\begin{jordantheorem}
Any continuous simple closed curve in the plane separates the plane into two disjoint regions, the interior and the exterior.
\end{jordantheorem}

Any other undefined graph theory terminology can be found in Bondy and Murty's graph theory textbook (\cite{BondyAndMurty}).

\section{The Symanzik Polynomials}

This section will give an overview of the \textit{Symanzik polynomials}, which are graph polynomials relevant to Feynman integrals.

To define the Symanzik polynomials, we first need to augment our graph with additional labels. Let $G$ be a graph. First, for each edge $e \in E(G)$ we create a variable $\alpha_{e}$, called a \textit{Schwinger parameter}. Additionally, for each edge $e \in E(G)$, we associate with $e$ an edge weight, $m_{e} \in \mathbb{R}$, called the \textit{mass} of edge $e$. We note that some authors' (see \cite{Bognerpaper}) allow for complex masses to allow for particles such as tachyons, however we restrict to real masses. An edge $e$ is \textit{massive} if $m_{e} \neq 0$. Also, for each vertex $v\in V(G)$, we will associate a vector $\rho_{v} \in \mathbb{R}^{4}$, called the \textit{external momentum at $v$}. We will enforce $\rho_{v} = 0$ for most vertices throughout this thesis. For the vertices where we allow $\rho_{v} \neq 0$, we will say there is an \textit{external edge} at $v$. In figures, we will draw external edges as a edge with only one endpoint (see Figure \ref{examplegraph}).

 We will enforce that the external momenta of our graph will satisfy conservation of momentum. Namely, for a given graph $G$, 
\[ \sum_{\mathclap{v \in V(G)}} \rho_{v} =0. \]

 \begin{definition}
Let $G$ be a graph with associated external momenta, masses, and Schwinger parameters. Let $\mathcal{T}$ be the set of spanning trees of $G$. Then the \emph{first Symanzik polynomial} of $G$ is 
 
 \[ \psi_{G} = \sum_{T \in \mathcal{T}} \prod_{e \not \in E(T)} \alpha_{e}. \]
  \end{definition}

 In the special case where $G$ is disconnected, then $\psi_{G} = 0$. Although we do not do so here, it is common to derive the above definition from the Laplacian matrix by using the Matrix-tree Theorem (\cite{Feynmangraphpolynomials}). We note that some authors refer to the first Symanzik polynomial as the \textit{Kirchoff polynomial} (\cite{Crumppaper}), or occasionally as the \textit{dual Kirchoff polynomial} (\cite{Feynmangraphpolynomials}).
 
 \begin{example}
 \label{cycleSymanzikpoly}
 Consider any cycle, $C_{n}$. Then $\psi_{C_{n}} = \sum_{e 
 \in E(C_{n})} \alpha_{e}$.
 \end{example}
 
 To see this, notice that deleting any edge of $C_{n}$ results in a path and every spanning tree of $C_{n}$ is such a path.
 
 Before defining the second Symanzik polynomial we need some definitions. Given a graph $G$, a \textit{spanning $2$-forest} of $G$ is a pair $(T_{1},T_{2})$ where $T_{1}$ and $T_{2}$ are trees such that $V(T_{1}) \cup V(T_{2}) = V(G)$, and $V(T_{1}) \cap V(T_{2}) = \emptyset$.

 Let $G$ be a graph and $H$ a subgraph of $G$. We will say the momentum flowing into $H$ is 
 \[ \sum_{\mathclap{v \in V(H)}} \rho_{v}.\]
 
 We will denote the momentum flowing into $H$ as  $\rho^{H}$. Now we can define the second Symanzik polynomial.

 \begin{definition}
 Let $G$ be a graph with associated external momenta, masses, and Schwinger parameters.  Let $\mathcal{T}$ be the set of spanning $2$-forests of $G$. Then the \emph{second Symanzik polynomial} is 
 
 \[\phi_{G} = \sum_{\mathclap{(T_{1},T_{2}) \in \mathcal{T}}} (\rho^{T_{1}})^{2}  \prod_{\mathclap{e \not \in T_{1} \cup T_{2}}} \alpha_{e} + \psi_{G}\sum_{i =1}^{|E(G)|} \alpha_{i}m_{i}^{2}.\]
 \end{definition}
 
Here we note that the two forests are unordered, as in $(T_{1},T_{2}) = (T_{2},T_{1})$. Also, $(\rho^{T_{1}})^{2}$ is taken to mean take a norm squared, under an appropriate norm (we define the Minkowski metric later). In the literature, it is common to take just the terms involving the momenta to be the second Symanzik polynomial, and include the masses as a separate polynomial (see \cite{Brownsurveypaper}, \cite{Feynmangraphpolynomials}). Additionally, as with the first Symanzik polynomial, it is possible to derive the second Symanzik polynomial from a matrix. See \cite{Feynmangraphpolynomials} for an exposition. As a special case, we  note that $\phi_{G}=0$ when $G$ has more than $2$ connected components.
 
  As an example, we calculate the second Symanzik polynomial of the graph in Figure \ref{examplegraph}. Notice every spanning $2$-forest of $K_{3}$ is an edge combined with an isolated vertex. By example \ref{cycleSymanzikpoly}, $\psi_{K_{3}} = \alpha_{1} + \alpha_{2} + \alpha_{3}$. Therefore,
 \[\phi_{K_{3}} = (\rho_{1} + \rho_{2})^{2}\alpha_{2}\alpha_{3} + (\rho_{2} + \rho_{3})^{2}\alpha_{1}\alpha_{2} + (\rho_{3} + \rho_{1})^{2}\alpha_{3}\alpha_{1} + \psi_{K_{3}}(\alpha_{3}m_{3}^{2}).\]
 
Additionally, by conservation of momenta, we have $\rho_{1} + \rho_{2} + \rho_{3} =0$, so one can rewrite the above equation as,

\[\phi_{K_{3}} = \rho_{3}^{2}\alpha_{2}\alpha_{3} + \rho_{1}^{2}\alpha_{1}\alpha_{2} + \rho_{2}^{2}\alpha_{3}\alpha_{1} + \psi_{K_{3}}\alpha_{3}m_{3}^{2}.\] 

\begin{figure}
\begin{center}
\includegraphics[scale =0.5]{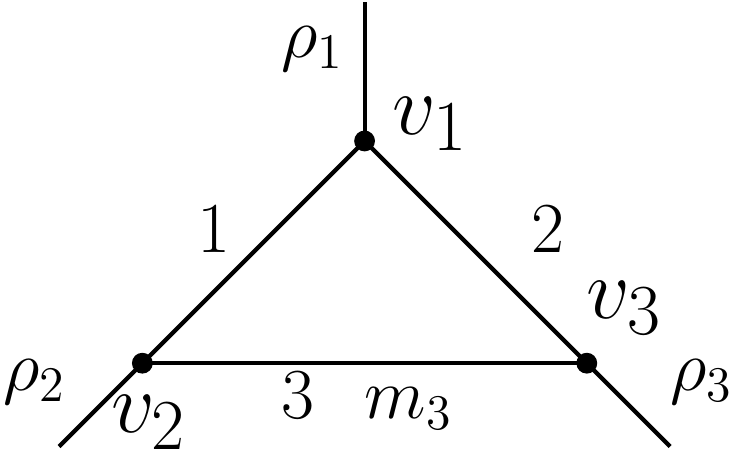}
\caption{A small graph. Edges are labelled $1$, $2$, $3$. Edge $3$ has mass $m_{3}$ and each vertex has external momenta incoming at the vertex.}
\label{examplegraph}
\end{center}
\end{figure}

 With these, we define the Feynman integral of a graph $G$ in parametric space to be:
 \[\frac{\Gamma(\nu - LD/2)}{\prod_{j=1}^{m} \Gamma(\nu_{j})} \int_{\alpha_{j} \geq 0} \delta(1 - \sum_{i =1}^{m} \alpha_{i})(\prod_{j=1}^{m}(d\alpha_{j})\alpha_{j}^{\nu_{j}-1})\frac{\psi^{\nu-(L+1)D/2}}{\phi^{\nu-LD/2}}.
\] 
 
 Here, $m$ is the number of edges in $G$, $D$ is the dimension of the physical theory, $\Gamma$ is the the gamma function, $\delta$ is the Dirac delta function, $L$ is the rank of the cycle space of $G$ (where the rank of the cycle space of $G$ is $|E(G)| - |V(G)| +1$), and $\alpha_{i}$ is the Schwinger parameter for edge $i$. The variable $\nu_j$ is a parameter which can be used to regularize the integral (converting a divergent integral into an expansion with poles
in the regularization parameters) or it can be used for keeping track of how subgraph insertions into the original graph affect the integral. It is noted in \cite{Martin} that for most applications $\nu_{j} =1$, for all $j \in \{1,\ldots,m\}$, however in general that might not be the case. Lastly, $\nu = \sum_{j=1}^{m} \nu_{j}$. Other techniques have been created to deal with divergent Feynman integrals, such as renormalization, however this thesis will not look at these tools (see, for example, \cite{HopfAlgebraRenormalization}).

 We will now go through some basic properties of the Symanzik polynmials.  An easy observation is that $\psi_{G}$ is a homogeneous polynomial linear in all variables, and if there are no massive edges, then $\phi_{G}$ is a homogeneous polynomial linear in all variables (\cite{Feynmangraphpolynomials}). A second observation which is well known (see \cite{Bognerpaper}, \cite{Feynmangraphpolynomials}, for example) is that the Symanzik polynomials behave nicely under deletion and contraction. 
  
\begin{lemma}
\label{deletioncontraction}
Let $G$ be a graph. Let $e \in E(G)$ such that $e$ is not massive and $e$ is not a loop. Then the following identities hold:

\[\psi_{G} = \psi_{G\setminus e}\alpha_{e} + \psi_{G / e},\]

\[\phi_{G} = \phi_{G\setminus e}\alpha_{e} + \phi_{G / e}.\]
\end{lemma}

\begin{lemma} 
\label{derivativeandevaluationsetrelationship}
Let $G$ be a graph. Let $e \in E(G)$ such that $e$ is not massive and $e$ is not a loop. Then the following identities hold:

\[\frac{\partial}{\partial \alpha_{e}}\phi_{G} = \phi_{G\setminus e}, \ \phi_{G}|_{\alpha_{e} = 0} = \phi_{G / e},\]

\[\frac{\partial}{\partial\alpha_{e}}\psi_{G} = \psi_{G\setminus e}, \  \psi_{G}|_{\alpha_{e} = 0} = \psi_{G / e}.\]

\end{lemma}

\begin{lemma}
\label{loopedgesSymanzikPolys}
Let $G$ be a graph and consider any edge $e$ which is a loop. Then $\psi_{G} = \alpha_{e}\psi_{G \setminus e}$, and $\phi_{G} = \phi_{G \setminus e}\alpha_{e}$.
\end{lemma}

Of course, the fact that this thesis deals with graph minors and the Symanzik polynomials exhibit nice deletion contradiction identities is no coincidence. Throughout this thesis, we will impose a natural condition on the external momenta which will simplify the second Symanzik polynomial. Before defining the condition, we will abuse notation slightly and introduce the Minkowski metric. Let $\rho \in \mathbb{R}^{4}$. Then under the Minkowski metric we define $\rho^{2} = \rho_{1}^{2} + \rho_{2}^{2} + \rho_{3}^{2} - \rho_{4}^{2}$. Here physically we have $\rho_{1}, \rho_{2}$ and $\rho_{3}$ corresponding to the three spacial dimensions, and $\rho_{4}$ corresponding to time.  This does not actually define a metric since there are non-zero vectors such that $\rho_{1}^{2} + \rho_{2}^{2} + \rho_{3}^{2} - \rho_{4}^{2} =0$, but since quantum field theory is formulated in Minkowski space, it is used by physicists and thus we abuse notation and call it a metric.  Sometimes these types of functions are called pseudometrics. Now we introduce a natural physics condition which we will impose on our momenta.

\begin{definition}
  Let $G$ be a graph let $v \in V(G)$. The momentum $\rho_{v}$ is \emph{on-shell} if $\rho_{v}^{2} =m^2$, where $m$ is the mass of the external edge corresponding to $\rho_{v}$. For the purposes of this thesis, we will always assume $m^2=0$, so a momentum vector $\rho_{v}$ is on-shell if $\rho_{v}^{2}=0$.
\end{definition} 

 Observe that under the Euclidean norm, if a momenta $\rho$ is on-shell, then $\rho =\vec{0}$. Thus if for some graph $G$, all of the momenta were on-shell, the terms involving the external momenta in the second Symanzik polynomial would be zero. However, under the Minkowski metric we can have on-shell momenta and still have a non-trivial dependence on external momenta in the second Symanzik polynomial. Notice under the Minkowski metric, we have $\rho_{1}^{2} + \rho_{2}^{2} + \rho_{3}^{2} = \rho_{4}^{2}$, which is a cone, so being on-shell means that the momentum vector is lying on the shell of the cone.  From a physics perspective, being on-shell means that the momenta is satisfying Einstein's energy and momentum relationship, and in our restricted notion of on-shell, this is saying that the external momenta are behaving in a light-like fashion. Therefore it is reasonable to study graphs whose momenta are on-shell.

As we are interested in studying reducibility, which we define in chapter two, we will mainly be interested in graphs where there are exactly four vertices with external edges and all momenta are on-shell or graphs with no external momenta. To justify this seemingly arbitrary choice of external momenta, we can show that the cases with $1$, $2$ or $3$ on-shell external momenta reduce to the case with zero external momenta. If a graph $G$ had one external momentum, $\rho$, then by conservation of momenta $\rho$ is the zero vector, so we actually have no dependence on the external momenta in the second Symanzik polynomial. If $G$ has two external momenta and they are both on-shell, then again by conservation of momenta and the on-shell condition, there is no dependence on external momenta in the second Symanzik polynomial.  With two external momenta, even if they are not on-shell, in \cite{FrancisSmall} it was shown that for a wide class of graphs, we can reduce the problem to a different graph with no external momenta. In essence what is found in \cite{FrancisSmall} is a relationship between the second Symanzik polynomial of graphs with two external momenta, and the first Symanzik polynomial of a specifically constructed graph. That result encourages looking at the case where we have zero external momenta instead of two.  

Now suppose we have three on-shell external momenta for some graph $G$. If for some spanning $2$-forest, all of the momenta are in one of the trees, then conservation of momenta implies that the $2$-forest contributes nothing to the second Symanzik polynomial. If we have exactly one momentum in one of the trees, then the on-shell condition implies that two forest contributes nothing to the second Symanzik polynomial. Therefore, at least for graphs with no massive edges, having three on-shell momenta is the same as having zero external momenta. Therefore, graphs with four on-shell external momenta are the smallest case where there is a non-trivial dependence on the external momenta. However, we can still simplify the second Symanzik polynomial in this case, though not to the point of triviality.

\begin{lemma}
\label{MomentumLemma}
Let $G$ be a graph and let $X = \{x_{1},x_{2},x_{3},x_{4}\} \subseteq V(G)$ be the set of vertices with external momenta in $G$ such that for all $i \in \{1,2,3,4\}$, $\rho_{x_{i}}$ is on-shell. Let  $(T_{1}, T_{2})$ be a spanning $2$-forest of $G$. If $|X \cap V(T_{1})| \neq 2$, then \[(\rho_{1}^{T_{1}})^{2} \prod_{\mathclap{e \not \in T_{1} \cup T_{2}}} \alpha_{e} =0.\]
\end{lemma}

\begin{proof}
We will analyze the four cases individually.

\textbf{Case 1:} Suppose $|X \cap V(T_{1})| = 0$. Then $\rho_{1}^{T_{1}} =  \sum_{v \in V(T_{1})} \rho_{v} = \sum_{v \in V(T_{1})} 0 =0$ and thus $(\rho_{1}^{T_{1}})^{2} \prod_{e \not \in T_{1} \cup T_{2}} \alpha_{e} =0$. 

\textbf{Case 2:} Suppose $|X \cap V(T_{1})| =4$. As $V(T_{1}) \cup V(T_{2}) = V(G)$ and $V(T_{1}) \cap V(T_{2}) = \emptyset$ by conservation of momentum we have that $\rho^{T_{1}} = -\rho^{T_{2}}$. By the argument from case $1$, we have that $-\rho^{T_{2}} =0$ and thus $(\rho_{1}^{T_{1}})^{2} \prod_{e \not \in T_{1} \cup T_{2}} \alpha_{e} =0$.

\textbf{Case 3:} Suppose that $|X \cap V(T_{1})|=1$. Without loss of generality, suppose $X \cap V(T_{1}) = \{x_{1}\}$. Then $(\rho^{T_{1}})^{2} = \rho_{x_{1}}^{2} = 0$ by the on-shell condition. Thus $(\rho_{1}^{T_{1}})^{2} \prod_{e \not \in T_{1} \cup T_{2}} \alpha_{e} =0$.

\textbf{Case 4:} Suppose that $|X \cap V(T_{1})| = 3$. Then $|X \cap V(T_{2})| =1$. Thus by the argument in case $3$, we have that $(\rho^{T_{2}})^{2} =0$, and thus by conservation of momentum we have $(\rho^{T_{1}})^{2} =0$. Thus $(\rho_{1}^{T_{1}})^{2} \prod_{e \not \in T_{1} \cup T_{2}} \alpha_{e} =0$.
\end{proof}

\begin{corollary}
 Let $G$ be a graph and $X = \{x_{1},x_{2},x_{3},x_{4}\} \subseteq V(G)$ such that for $i \in \{1,2,3,4\}$, $\rho_{x_{i}}$ is on-shell. Suppose that for all vertices $v \in V(G) \setminus X$, we have $\rho_{v} = \vec{0}$. Let $\mathcal{T}$ be the set of spanning $2$-forests $(T_{1},T_{2})$ such that $|X \cap T_{1}| =2$. Then 
  \[\phi_{G} = \sum_{\mathclap{(T_{1},T_{2}) \in \mathcal{T}}}(\rho^{T_{1}})^{2}  \prod_{\mathclap{e \not \in T_{1} \cup T_{2}}} \alpha_{e} + \psi_{G}\sum_{i =1}^{|E(G)|} \alpha_{i}m_{i}^{2}.\]
\end{corollary}

\begin{proof}
Follows immediately from Lemma \ref{MomentumLemma}.
\end{proof}

This simplified version of the second Symanzik polynomial will prove useful. We end this section by proving some identities for the first and second Symanzik polynomials in graphs with cut vertices. These lemmas are going to allow us to do some connectivity reductions when looking at reducibility. The first lemma is well-known (see, for example \cite{FrancisBig}).
 
\begin{lemma}
\label{firstsymanzikoneconnected}
Let $G$ be a graph such that $G$ has a $1$-separation $(A,B)$. Then $\psi_{G} = \psi_{G[A]}\psi_{G[B]}$. 
\end{lemma}

\begin{proof}

Let $T_{A}$, $T_{B}$ be spanning trees of $G[A]$, $G[B]$ respectively. Notice that $T_{A} \cup T_{B}$ is a spanning tree of $G$. Similarly, for any spanning tree $T$ of $G$, the graphs $G[V(T) \cap A]$ and $G[V(T) \cap B]$ are spanning trees of $G[A]$ and $G[B]$ respectively. The result follows.
\end{proof}

Unlike the first Symanzik polynomial, we have to be careful where the external edges are in the graph relative to the $1$-separation and restrict the masses to obtain similar factorization results.

\begin{lemma}
\label{fourzerocaseSymanzikPolynomials}
Let $G$ be a graph with no massive edges and let $X =\{x_{1},x_{2},x_{3},x_{4}\} \subseteq V(G)$ be the set of vertices with external momenta. Suppose $G$ has a $1$-separation $(A,B)$ such that $A \cap B = \{v\}$. Suppose $X \subseteq A$ . Then $\phi_{G} = \phi_{G[A]}\psi_{G[B]}$.
\end{lemma}

\begin{proof}
By Lemma \ref{MomentumLemma} we restrict our attention to spanning $2$-forests where each of the trees contain exactly two vertices of $X$. Let $\mathcal{T}$ denote the set of spanning two forests for $G$, $\mathcal{T}_{1}$ the set of spanning $2$-forests in $G[A]$, and $\mathcal{T}_{2}$ the set of spanning trees in $G[B]$. Before doing the calculation, we make the easy observation that if $(T_{1},T_{2})$ is a spanning $2$-forest of $G$, and $v \in V(T_{2})$, then $G[B \cap V(T_{2})]$ is a spanning tree of $G[B]$. Up to relabelling, we will assume that $v \in T_{2}$. Then,

\begin{align*}
 \phi_{G} &=
  \quad \sum_{\mathclap{(T_{1},T_{2}) \in \mathcal{T}}} (\rho^{T_{1}})^{2}  \prod_{\mathclap{e \not \in T_{1} \cup T_{2}}} \alpha_{e} \\[0.6ex]
  &= \quad \sum_{\mathclap{(T_{1},T_{2}) \in \mathcal{T}}} (\rho^{T_{1}})^{2}  (\prod_{\mathclap{\substack{e \not \in T_{1} \cup T_{2} \\ e \in G[A]}}}\alpha_{e} \cdot \prod_{\mathclap{\substack{e \not \in T_{2} \\ e \in G[B]}}} \alpha_{e}) 
  \\[0.6ex]
   &= \quad (\sum_{\mathclap{(T_{1},T_{2}) \in \mathcal{T}_{1}}} (\rho^{T_{1}})^{2}  \prod_{\mathclap{\substack{e \not \in T_{1} \cup T_{2} \\ e \in G[A]}}}\alpha_{e}) \cdot (\sum_{\substack{T \in \mathcal{T}_{2}}}\prod_{\mathclap{\substack{e \not \in T \\[0.6ex] e \in G[B]}}} \alpha_{e}) \\[0.6ex] 
  &= \phi_{G[A]}\psi_{G[B]},
\end{align*}

which completes the claim.
\end{proof}

\begin{lemma}
\label{momentaoneconnectivityreduction}
Let $G$ be a graph with no massive edges and let $X =\{x_{1},x_{2},x_{3},x_{4}\} \subseteq V(G)$  be the set of vertices with external momenta, and furthermore suppose all external momenta are on-shell. Suppose $G$ has a $1$-separation $(A,B)$ such that $A \cap B = \{v\}$. Suppose $x_{1},x_{2},x_{3} \in A$ and  $x_{4} \in B \setminus \{v\}$. Let $G[A]$ have the same external momenta, masses and Schwinger parameters as $A$ in $G$ with the exception that $\rho_{x_{4}} = \rho_{v}$. Then $\phi_{G} = \phi_{G[A]}\psi_{G[B]}$.
\end{lemma}

\begin{proof}

First note that $G[A]$ satisfies conservation of momentum as we have

\begin{align*}
\sum_{\mathclap{x \in V(G[A])}} \rho_{x} = \rho_{x_{1}} + \rho_{x_{2}} + \rho_{x_{3}} + \rho_{v} = \rho_{x_{1}} + \rho_{x_{2}} + \rho_{x_{3}} + \rho_{x_{4}} = \sum_{\mathclap{x \in V(G)}} \rho_{x} =  0.
\end{align*}

By Lemma \ref{MomentumLemma} we restrict our attention to spanning $2$-forests which contain exactly two vertices of $X$. Let $(T_{1},T_{2})$ be a spanning $2$-forest of $G$ and without loss of generality suppose that $|T_{1} \cap X \cap A| =2$. Then $|T_{2} \cap A \cap X| =1$. Thus $|T_{2} \cap B \cap X|=1$ which implies $v \in T_{2}$. As $(T_{1},T_{2})$ is a spanning $2$-forest of $G$, we have that $T_{2} \cap B$ is a spanning tree of $B$. Therefore,

\begin{align*} 
(p^{T_{1}})^{2}  \prod_{\mathclap{e \not \in T_{1} \cup T_{2}}} \alpha_{e} &= (p^{T_{1}})^{2} (\prod_{\mathclap{\substack{e \not \in T_{1} \cup T_{2} \\ e \in G[A]}}} \alpha_{e} \cdot \prod_{\mathclap{\substack{e \not \in T_{2} \\ e \in G[B]}}} \alpha_{e}).
\end{align*}

Now let $(T_{1},T_{2})$ be a spanning $2$-forest of $G[A]$ and without loss of generality, suppose that $v \in T_{2}$. We can extend $(T_{1},T_{2})$ to a spanning $2$-forest of $G$ by considering any spanning tree of $B$, say $T_{3}$ and noticing $(T_{1},T_{2} \cup T_{3})$ is a spanning $2$-forest of $G$. Also notice that for $(T_{1},T_{2} \cup T_{3})$, we have $|T_{1} \cap X| =2$.  Also, as $\rho_{v} = \rho_{w}$ in $G[A]$, we have $(\rho^{T_{1}})^{2}$ is the same in $G[A]$ as in $G$.

 Let $\mathcal{T}$ denote the set of spanning $2$-forests in $G$, $\mathcal{T}_{1}$ denote the set of spanning $2$-forests in $G[A]$, and $\mathcal{T}_{2}$ denote the set of spanning trees in $B$. Then we have 

\begin{align*}
 \phi_{G} &=
  \quad \sum_{\mathclap{(T_{1},T_{2}) \in \mathcal{T}}} (\rho^{T_{1}})^{2}  \prod_{\mathclap{e \not \in T_{1} \cup T_{2}}} \alpha_{e} \\[0.6ex]
  &= \quad \sum_{\mathclap{(T_{1},T_{2}) \in \mathcal{T}}} (\rho^{T_{1}})^{2}  (\prod_{\mathclap{\substack{e \not \in T_{1} \cup T_{2} \\ e \in G[A]}}}\alpha_{e} \cdot \prod_{\mathclap{\substack{e \not \in T_{2} \\ e \in G[B]}}} \alpha_{e}) 
  \\[0.6ex] &= \quad (\sum_{\mathclap{(T_{1},T_{2}) \in \mathcal{T}_{1}}} (\rho^{T_{1}})^{2}  \prod_{\mathclap{\substack{e \not \in T_{1} \cup T_{2} \\ e \in G[A]}}}\alpha_{e}) \cdot (\sum_{\substack{T \in \mathcal{T}_{2}}}\prod_{\mathclap{\substack{e \not \in T \\[0.6ex] e \in G[B]}}} \alpha_{e}) \\[0.6ex] 
  &= \phi_{G[A]}\psi_{G[B]},
\end{align*}

which completes the claim.

\end{proof}

\begin{lemma}
\label{symanzikpolynomialtwomomentaoneachside}
Let $G$ be a graph with no massive edges and let $X = \{x_{1},x_{2},x_{3},x_{4}\} \subseteq V(G)$ be the set of vertices with external momenta, and furthermore suppose all external momenta is on-shell. Suppose $G$ has a $1$-separation $(A,B)$ such that $A \cap B = \{v\}$. Suppose and that $X \cap (A \setminus \{v\}) = \{x_{1},x_{2}\}$ and that $X \cap (B \setminus \{v\}) = \{x_{3},x_{4}\}$. Let $A'= G[A]$ have the same external momenta and Schwinger parameters as in $G$ with the exception that we let $\rho_{v} = \rho_{x_{1}} + \rho_{x_{2}}$. Similarly, Let $B' = G[B]$ have the same external momenta and Schwinger parameters as in $G$ with the exception that we let $\rho_{v} = \rho_{x_{3}} + \rho_{x_{4}}$. Then

\[\phi_{G} = \phi_{A'}\psi_{B'} + \phi_{B'}\psi_{A'}.  \]
\end{lemma}

\begin{proof}
First we claim that both $A'$ and $B'$ satisfy conservation of momentum. For $A'$ we have

\[\sum_{\mathclap{x \in V(A')}} \rho_{x} = \rho_{x_{1}} + \rho_{x_{2}} + \rho_{v} = \rho_{x_{1}} + \rho_{x_{2}} + \rho_{x_{3}} + \rho_{x_{4}} = \sum_{\mathclap{x \in V(G)}} \rho_{x} =  0. \]

Applying a similar argument gives the result for $B'$. 
Now consider a spanning $2$-forest $(T_{1},T_{2})$ of $G$. By Lemma \ref{MomentumLemma} we restrict our attention to spanning two forests where $|X \cap T_{1}| =2$. If $v \in V(T_{2})$, then $T_{2} \cap B'$ is a spanning tree of $B'$. Otherwise, $v \not \in V(T_{2})$ and $T_{1} \cap B'$ is a spanning tree of $B'$. 

Now let $(T_{1},T_{2})$ be a spanning $2$-forest of $A'$. Note that $\rho_{v}$ may not be on-shell. However, as $\rho_{v} = \rho_{x_{3}} + \rho_{x_{4}}$, and both $\rho_{x_{1}}$ and $\rho_{x_{2}}$ are on-shell, applying the arguments from Lemma \ref{MomentumLemma}, we may assume that either $x_{1},x_{2} \in V(T_{1})$ and $v \not \in V(T_{1})$ or that $v \in V(T_{1})$ and $x_{1},x_{2} \not \in V(T_{1})$. Let $T_{3}$ be a spanning tree of $B$. If $v \in T_{1}$, then $(T_{1} \cup T_{3}, T_{2})$ is a spanning $2$-forest of $G$ such that $|T_{2} \cap X| =2$. If $v \in V(T_{2})$, then we have that $(T_{1},T_{2} \cup T_{3})$ is a spanning $2$-forest of $G$ such that $|T_{1} \cap X| =2$. Similar statements hold for $B'$. 

Let $\mathcal{T}, \mathcal{T}_{1},\mathcal{T}_{2}$  be the set of spanning $2$-forests for $G$, $A'$ and $B'$ respectively. Let $\mathcal{T}_{3},\mathcal{T}_{4}$ be the set of spanning trees of $A'$ and $B'$ respectively. Throughout the upcoming equations, we will assume when $(T_{1},T_{2}) \in \mathcal{T}$, that $v \in V(T_{1})$. Then we have:

\begin{align*}
\phi_{G} &=  \quad \sum_{\mathclap{(T_{1},T_{2}) \in \mathcal{T}}} (\rho^{T_{1}})^{2}  \prod_{\mathclap{e \not \in T_{1} \cup T_{2}}} \alpha_{e} \\[0.6ex]
 &= \quad \sum_{\mathclap{\substack{(T_{1},T_{2}) \in \mathcal{T} \\ (T_{1} \cap A') \in \mathcal{T}_{3}}}} (\rho^{T_{1}})^{2} \prod_{\mathclap{e \not \in T_{1} \cup T_{2}}} \alpha_{e} + \sum_{\mathclap{\substack{(T_{1},T_{2}) \in \mathcal{T} \\ (T_{1} \cap B') \in \mathcal{T}_{4}}}}(\rho^{T_{1}})^{2} \prod_{\mathclap{e \not \in T_{1} \cup T_{2}}} \alpha_{e}
\\[0.6ex] 
&= \quad \sum_{\mathclap{\substack{(T_{1},T_{2}) \in \mathcal{T} \\ (T_{1} \cap A') \in \mathcal{T}_{3}}}} (\rho^{T_{1}})^{2} \prod_{\mathclap{\substack{e \not \in T_{1} \cup T_{2} \\ e \in E(A')}}} \alpha_{e} \cdot \prod_{\mathclap{\substack{e \not \in T_{1} \\ e \in E(B')}}} \alpha_{e} + \sum_{\mathclap{\substack{(T_{1},T_{2}) \in \mathcal{T} \\ (T_{1} \cap B') \in \mathcal{T}_{4}}}}(\rho^{T_{1}})^{2}\prod_{\mathclap{\substack{e \not \in T_{1} \cup T_{2} \\ e \in E(B')}}}\alpha_{e} \cdot \prod_{\mathclap{\substack{e \not \in T_{2} \\ e \in E(A')}}} \alpha_{e}
\\[0.6ex] &= \quad (\sum_{\mathclap{(T_{1},T_{2}) \in \mathcal{T}_{1}}} (\rho^{T_{1}})^{2} \prod_{\mathclap{e \not \in T_{1} \cup T_{2}}} \alpha_{e}) \cdot (\sum_{T \in \mathcal{T}_{4}} \prod_{\mathclap{\substack{e \not \in T }}} \alpha_{e}) + (\sum_{\mathclap{(T_{1},T_{2}) \in \mathcal{T}_{2}}} (\rho^{T_{1}})^{2} \prod_{\mathclap{e \not \in T_{1} \cup T_{2}}} \alpha_{e}) \cdot (\sum_{T \in \mathcal{T}_{3}}\prod_{\mathclap{e \not \in T}} \alpha_{e})
\\[0.6ex] &=\quad \phi_{A'}\psi_{B'} + \phi_{B'}\psi_{A'}.
\end{align*}

\end{proof}

Since in Lemma \ref{symanzikpolynomialtwomomentaoneachside} the momenta at $v$ is not necessarily on-shell, the equation we get is not quite as nice as one might hope for. For reducibility, this is going to mean that we are not going to be able to reduce the problem to $2$-connected graphs in all of the cases. Another well known equation for the Symanzik polynomials is:

\begin{proposition}
\label{subdivided}
Let $H$ be a graph and consider any edge $e$. Let $G$ be the graph obtained by subdividing edge $e$ into two edges $e_{1}$ and $e_{2}$ with no external momenta on the vertex obtained from the subdivision. Furthermore, suppose $G$ has no massive edges and four on-shell momenta. Then $\phi_{G} = \phi_{H \setminus e}(\alpha_{e_{1}} + \alpha_{e_{2}}) + \phi_{H / e}$, and $\psi_{G} = \psi_{H \setminus e}(\alpha_{e_{1}} + \alpha_{e_{2}}) + \psi_{H /e}$.
\end{proposition}

\begin{proof}
Let $T$ be a spanning tree of $H$. Then either $e \in E(T)$ or $e \not \in E(T)$. If $e \not \in E(T)$, then $T$ can be extended to a spanning tree of $G$ by using exactly one of edge $e_{1}$ or $e_{2}$. Otherwise, $T$ can be extended to a spanning tree of $G$ by using both edges $e_{1}$ and $e_{2}$. It is easy to see that these are all the spanning trees of $G$. Thus $\psi_{G} = \psi_{H \setminus e}(\alpha_{e_{1}} + \alpha_{e_{2}}) + \psi_{H / e}$.

Now consider a spanning two forest $(T_{1},T_{2})$ of $H$. If $e \not \in E(T_{1}) \cup E(T_{2})$,  we can extend this to a spanning $2$-forest of $G$ by either adding edge $e_{1}$ or $e_{2}$ to one of the trees. Otherwise without loss of generality we have that $e \in T_{1}$. Then we can extend $(T_{1},T_{2})$ to a spanning $2$-forest of $G$ by adding $e_{1}$ and $e_{2}$ to $T_{1}$. Now notice that this covers all spanning $2$-forests of $G$ except for the ones that have the subdivided vertex by itself as a tree. But by Lemma \ref{MomentumLemma}, that spanning $2$-forest does not have exactly two momenta in each tree, and thus the term goes to zero. The result follows. 
\end{proof}

Our final Lemma is an observation of how the second Symanzik polynomial behaves in graphs with exactly two connected components.

\begin{lemma}
Let $G$ be a graph with no massive edges and let $X = \{x_{1},x_{2},x_{3},x_{4}\} \subseteq V(G)$ be the set of vertices with external momenta, and furthermore suppose all external momenta is on-shell. Let $G$ have exactly two connected components, $G_{1}$ and $G_{2}$. If $|X \cap V(G_{1})| \neq 2$, then $\psi_{G} =0$. Otherwise, let $x_{1}$ and $x_{2}$ be in $V(G_{1})$. Then $\psi_{G} = (\rho_{x_{1}} + \rho_{x_{2}})^{2}(\psi_{G_{1}}\psi_{G_{2}})$. 
\end{lemma}

\begin{proof}
Let $(T_{1},T_{2})$ be any spanning $2$-forest of $T$. Then up to relabelling, $T_{1}$ is a spanning tree of $G_{1}$ and $T_{2}$ is a spanning tree of $G_{2}$. Then it is immediate from Lemma \ref{MomentumLemma} that if $|X \cap V(G_{1})| \neq 2$, then $\psi_{G} =0$. Therefore we assume that $x_{1}$ and $x_{2}$ are in $V(G_{1})$. Notice that since $G$ is disconnected that $\psi_{G} =0$, and so the terms involving masses in the second Symanzik polynomial go to zero. Let $\mathcal{T}$ denote the spanning two forests of $G$, $\mathcal{T}_{1}$ denote the spanning trees of $G_{1}$, and $\mathcal{T}_{2}$ denote the spanning trees of $G_{2}$. Thus we have

\begin{align*}
\psi_{G} &= \sum_{\mathclap{(T_{1},T_{2}) \in \mathcal{T}}} (\rho^{T_{1}})^{2}  \prod_{\mathclap{e \not \in T_{1} \cup T_{2}}} \alpha_{e} + \psi_{G}\sum_{i =1}^{|E(G)|} \alpha_{i}m_{i}^{2} \\
 &= \sum_{\mathclap{(T_{1},T_{2}) \in \mathcal{T}}} (\rho^{T_{1}})^{2}  \prod_{\mathclap{e \not \in T_{1} \cup T_{2}}} \alpha_{e} \\
 &=  (\rho_{x_{1}} + \rho_{x_{2}})^{2}(\sum_{T \in \mathcal{T}_{1}} \prod_{e \not \in T} \alpha_{e})(\sum_{T \in \mathcal{T}_{2}} \prod_{e \not \in T} \alpha_{e}.) \\
 &= (\rho_{x_{1}} + \rho_{x_{2}})^{2} \psi_{G_{1}}\psi_{G_{2}}
\end{align*}
\end{proof}

Lastly, Erik Panzer has implemented functions in his program HyperInt (\cite{HyperintArticle}) which can compute the Symanzik polynomials for a given graph and given external momenta. All Symanzik polynomial calculations done throughout this thesis are done with HyperInt.

\chapter{Reducibility}

This chapter introduces the notion of reducibility, and in particular will look at reducibility of graphs with respect to the Symanzik polynomials. 

Francis Brown introduced the notion of reducibility which tells you if there is a suitable order of integration for his integration algorithm (\cite{FrancisBig}). The notion of reducibility has gone through numerous iterations, with each refinement extending the number of graphs which are reducible. As the definitions build on each other, it is instructive to see how the definitions are modified over time. We note that the actual integration algorithm always remains the same, but the different notions of reducibility get closer to correctly representing when the integration algorithm can be used. 

Polynomial reduction algorithms are at the heart of reducibility. If we focus on the Symanzik polynomials, what a polynomial reduction algorithm does is decide if there is an admissible order of integration for Brown's integration technique. A nice property of reduction algorithms is they do not require looking at the integral at all, they just act on the polynomials. Since the Symanzik polynomials come from graphs, what we will see is that we can decide when the reduction algorithm will succeed solely by looking at the structure of the graph. The most simplistic of the reduction algorithms was introduced in \cite{FrancisSmall}, and is aptly named the simple reduction algorithm.

\section{Reduction Algorithms}
\subsection{The Simple Reduction Algorithm}

Let $S$ be a set of polynomials in the polynomial ring $\mathbb{Q}[\alpha_{1},\alpha_{2},\ldots, \alpha_r]$.  Let $\sigma$ be a permutation of $\{\alpha_{1},\ldots,\alpha_{r}\}$. As an abuse of notation, for this thesis we will let $\sigma(i) = \sigma(\alpha_{i})$.  The idea of the algorithm will be to create a sequence of new sets of polynomials with rational coefficients $S,S_{(\sigma(1))},S_{(\sigma(1),\sigma(2))},\ldots, S_{(\sigma(1),\ldots,\sigma(r))}$ and check that each polynomial in the set $S_{(\sigma(1),\ldots,\sigma(i))}$ is linear in $\sigma(i+1)$ for each $i \in \{1,\ldots,r-1\}$.  Now for the definition; suppose we are at the $k^{th}$ iteration of the algorithm. If $k>1$, we have a set of polynomials with rational coefficients $S_{(\sigma(1),\ldots,\sigma(k-1))} = \{f_{1},\ldots,f_{n}\}$. Otherwise $k=1$ and we use $S$.  We then do the following:

\begin{enumerate}
\item{If there is a polynomial $f \in S_{(\sigma(1),\ldots,\sigma(k-1))}$ such that $f$ is not linear in $\sigma(k)$, we end the algorithm. Otherwise continue.}

\item{Then for all $i \in \{1,\ldots,n\}$, we write $f_{i} = g_{i}\sigma(k) + h_{i}$, where $g_{i} =  \frac{\partial f_{i} }{\partial \sigma(k)}$ and $h_{i} = f_{i}|_{\sigma(k)=0}$.}

\item{Let $S^{1} = \{g_{i} | i \in \{1,\ldots,n\}$. Let $S^{2} = \{h_{i} | i \in \{1,\ldots,n\}$. Let $S^{3} = \{g_{i}h_{j}-h_{i}g_{j}| i,j \in \{1,\ldots,n\}, i \neq j\}$. Let $S^{4} = S^{1} \cup S^{2} \cup S^{3}$}. 

\item{Let $f \in S^{4}$ and let $f_{1},f_{2},f_{3},\ldots,f_{m}$ be the polynomials such that $\prod_{i =1}^{m} f_{i} = f$, where each $f_{i}$ has coefficients in $\mathbb{Q}$ and each $f_{i}$ is irreducible over $\mathbb{Q}$. Let $\tilde{S}$ be the set of polynomials $f_{1},\ldots,f_{m}$ for each $f \in S^{4}$.} 

\item{We define $S_{(\sigma(1),\ldots,\sigma(k))}$ to be $\tilde{S}$. Now repeat the above steps with $S_{(\sigma(1),\ldots,\sigma(k))}$ in place of $S_{(\sigma(1),\ldots,\sigma(k-1))}$ for $r-1$ iterations or the until first step ends the algorithm.}
\end{enumerate}

\begin{definition}
Let $S$ be a set of polynomials in the polynomial ring $\mathbb{Q}[\alpha_{1},\ldots \alpha_{r}]$ and let $\sigma$ be a permutation of $\{\alpha_{1},\ldots, \alpha_{r}\}$. Then $S$ is \textit{simply reducible with respect to $\sigma$} if for all $i \in \{1,\ldots,r-1\}$, all polynomials in $S_{(\sigma(1),\ldots,\sigma(i))}$ are linear in $\sigma(i+1)$. If there exists a permutation $\sigma$ of $\{\alpha_{1},\ldots,\alpha_{r}\}$ such that $S$ is simply reducible with respect to $\sigma$, then we say that $S$ is simply reducible. If for some graph $G$, the set $S \subseteq \{\psi_{G}, \phi_{G}\}$ is simply reducible, then we will say $G$ is \emph{simply reducible with respect to $S$}.
\end{definition} 

Throughout this section and the other reduction algorithms, if we say a graph is reducible without mentioning what set of polynomials, reducibility is assumed to be with respect to both Symanzik polynomials. 
\begin{remark}
\label{monomialsandconstants}
If $S_{(\sigma(1),\ldots,\sigma(k))}$ contains any monomial of the form $\alpha_{i}$ or any constant, we may remove them as they do not affect whether the set is simply reducible.  As in, if $S = [S',c]$ for some constant $c$, and some set of polynomials $S'$, then $S$ is simply reducible if and only if $S'$ is simply reducible. The same statement holds if $S = [S',\alpha]$ for some variable $\alpha$.
\end{remark}

\begin{proof}
Suppose $S'$ is a set of polynomials in the polynomial ring $\mathbb{Q}[\alpha_{1},\ldots,\alpha_{r}]$ which is simply reducible for some permutation $\sigma$ of $\alpha_{1},\ldots,\alpha_{r}$. Now let $S = [S', c]$, where $c$ is some constant. For $\sigma(1)$, the polynomial $c$ forces $c \in S^{2}$ and $0 \in S^{1}$. Thus the polynomial $g_{i}h_{j} - h_{i}g_{j} = cg_{j}$ when $f_{i} = c$. Then after factoring over $\mathbb{Q}$, we get $\tilde{S} = \{c, S'_{\sigma(1)}\}$. Notice that this happens at every step of the algorithm. Then since $c$ is linear in every variable, by induction $S$ is simply reducible with respect to $\sigma$.

Similarly, suppose we have $S = [S', \alpha]$ for some variable $\alpha$. Let $S'$ and $\sigma$ be as above. First suppose $\sigma(1) = \alpha$. Then from $\alpha$, we have $1 \in S^{1}$ and $0 \in S^{2}$. Thus the polynomial $g_{i}h_{j} - h_{i}g_{j} = h_{j}$, when  $f_{i} = \alpha$. Thus $\tilde{S} = \tilde{S'} = S'_{\sigma(1)}$. Then since $S'$ is reducible, $S$ is reducible. 

Now consider when $\sigma(1) \neq \alpha$. Then from $\alpha$, we have $0 \in S^{1}$ and $\alpha \in S^{2}$.  Thus the polynomial $g_{i}h_{j} - h_{i}g_{j}= g_{j}\alpha$, when $f_{i} = \alpha$. Thus after factoring we get $S_{\sigma(1)} = \{\alpha,S'_{\sigma(1)}\}$. Noticing that the above argument holds for every step of the algorithm, by induction we have that $S$ is simply reducible with respect to $\alpha$. See Lemma \ref{subsetlemma} for the proof of the converse.
\end{proof}

Then if we are given a polynomial which is just a monomial, we can ignore it in the reduction algorithm, since a monomial factors into single monomials of the form $\alpha_{i}$, and then we can apply Remark \ref{monomialsandconstants}. Before continuing we give some examples.
\begin{example}
Let $n \in \mathbb{N}$. Let $S =\{\Sigma_{i=1}^{n} x_{i}\}$. Then $S$ is simply reducible.
\end{example}
\begin{proof}
We proceed by induction on $n$. If $n =1$ this is immediate. Now assume $n \geq 2$. Let $\sigma$ be the identity permutation on $\{x_{1},\ldots,x_{n}\}$. Then applying the simple reduction algorithm to $\sigma(1) = x_{1}$, we get $S^{4} = \{1, \Sigma_{i=2}^{n} x_{i}\}$. As these polynomials are already irreducible over $\mathbb{Q}$, we have $\tilde{S} = S^{4}$. Thus $S_{\sigma(1)} = \{1, \Sigma_{i=2}^{n} x_{i}\}$. By Remark \ref{monomialsandconstants}, we may remove the constant and continue the reduction with $S_{\sigma(1)} = \{\Sigma_{i =2}^{n} x_{i}\}$. By induction, $S_{\sigma(1)}$ is simply reducible with the identity permutation, therefore $S$ is reducible with the identity permutation.

\end{proof}

From example \ref{cycleSymanzikpoly}, $\psi_{C_{n}} = \Sigma_{i = 1}^{n} \alpha_{i}$. So the above example says that all cycles are simply reducible with respect to the first Symanzik polynomial. Also, notice that the identity permutation was not critical in the above example, with some slight modifications, any permutation would have sufficed. So cycles are simply reducible with respect to the first Symanzik polynomial for any permutation.

As a remark, notice that the second Symanzik polynomial may not have coefficients in $\mathbb{Q}$ since it depends on masses and external momenta. As we want to test reducibility on the second Symanzik polynomial, we will allow coefficients of polynomials to be algebraic functions of the masses and external momenta or rational numbers, following what is done in \cite{Bognerpaper}. While we will generally assume all our graphs have no massive edges, for the next example we will consider massive edges. 

\begin{example}
Let $G$ be a $2$-connected graph. Suppose every edge is massive.  Then $G$ is not simply reducible with respect to $\psi$. 
\end{example}

\begin{proof}

As $G$ is $2$-connected, for every edge $e \in E(G)$ there exists a spanning tree $T$ such that $e \not \in E(T)$. Thus for every $e \in E(G)$, there is a term of $\phi_{G}$ which has $\alpha_{e}$ in it. Then, as each edge is massive, 
\[\psi_{G}\sum_{i =1}^{|E(G)|} \alpha_{i}m_{i}^{2}\]

is quadratic in every variable. Therefore $\phi_{G}$ is quadratic in every variable. Thus for every permutation of the variables of $\phi_{G}$, the simple reduction algorithm will end at the first step.
\end{proof}

The next example was first calculated in Martin L{\"u}ders masters thesis (\cite{Martin}), although we give a slightly different argument, but one with the same spirit. It will become a forbidden minor for reducibility, so it is worthwhile to work out all the details.

\begin{figure}
\centering
\includegraphics[clip,trim=1cm 7.5cm 1cm 2cm, scale =0.4]{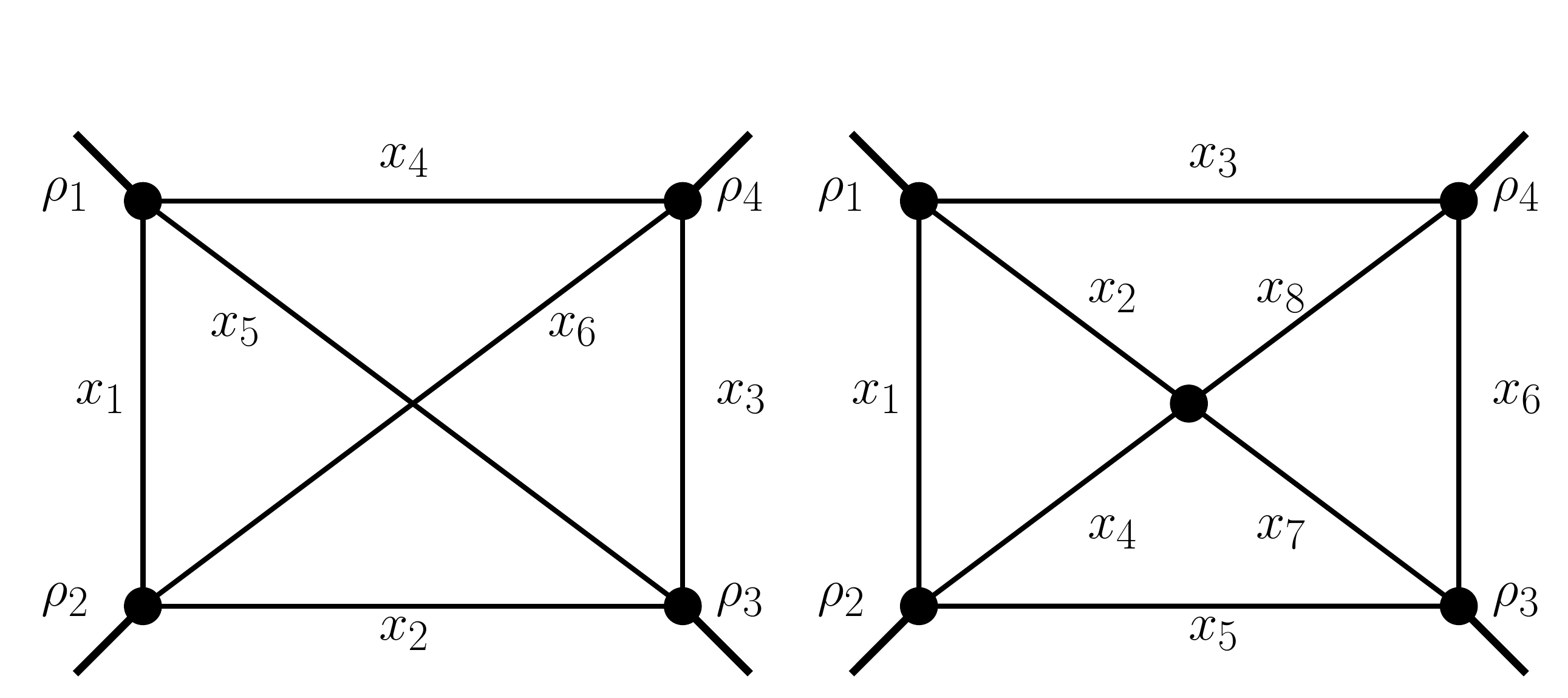}
\caption{A labelling of $K_4$ and $W_4$.}
\label{W4K4labelling}
\end{figure}

\begin{example}[\cite{Martin}]
\label{K4notreducible}
The graph $K_{4}$ is not simply reducible with respect to both $\{\psi,\phi\}$ when each vertex has on-shell external momenta and no edges are massive.
\end{example}
\begin{proof}

  Let $\sigma$ be a permutation of $\{x_{1},\ldots,x_{6}\}$. Using the labels as in Figure \ref{W4K4labelling}, one can calculate that,
\begin{equation*}
\begin{aligned}
\psi_{K_{4}} = {} & x_{{1}}x_{{3}}x_{{2}}+x_{{1}}x_{{4}}x_{{2}}+x_{{1}}x_{{5}}x_{{2}}+x_{{
1}}x_{{3}}x_{{4}}+x_{{3}}x_{{1}}x_{{5}}+x_{{1}}x_{{3}}x_{{6}}+x_{{1}}x
_{{4}}x_{{6}}+x_{{1}}x_{{5}}x_{{6}} \\ & +x_{{2}}x_{{3}}x_{{4}} + x_{{2}}x_{{3
}}x_{{6}}+x_{{2}}x_{{4}}x_{{5}}+x_{{2}}x_{{4}}x_{{6}}+x_{{2}}x_{{5}}x_
{{6}}+x_{{3}}x_{{4}}x_{{5}}+x_{{3}}x_{{5}}x_{{6}}+x_{{4}}x_{{5}}x_{{6}
}.
\end{aligned}
\end{equation*}

When calculating $\phi_{K_{4}}$, since the four momenta are on-shell,  by Lemma \ref{MomentumLemma} there are only three relevant spanning $2$-forests. Therefore we have:
\[\phi_{K_{4}} =  (p_1 + p_3)^{2}(x_{1}x_{2}x_{3}x_{4}) + (p_1 + p_2)^{2}(x_{4}x_{2}x_{5}x_{6}) + (p_1 +p_4)^{2}(x_{1}x_{3}x_{5}x_{6}).\]

For notational ease, we will let $(p_{1} +p_{2})^{2} = s$, $(p_1 + p_3)^{2} = t$, $(p_1 + p_2)^{2} = u$. Due to the symmetry of $K_{4}$, without loss of generality we may assume that $\sigma(1) = x_{1}$. 

Then we calculate:

$S^{1} = \{x_{2}x_{3} + x_{2}x_{4} + x_{2}x_{5} + x_{3}x_{4} + x_{3}x_{5} + x_{3}x_{6} + x_{4}x_{6} + x_{5}x_{6},sx_{2}x_{3}x_{4} + ux_{3}x_{5}x_{6} \}, $

$S^{2} = \{x_{2}x_{3}x_{4} + x_{2}x_{3}x_{6} + x_{2}x_{4}x_{5} + x_{2}x_{4}x_{6} + x_{2}x_{5}x_{6} + x_{3}x_{4}x_{5} + x_{3}x_{5}x_{6} + x_{4}x_{5}x_{6},tx_{2}x_{4}x_{5}x_{6} \},$

$S^{3} = \{(x_{2}x_{3} + x_{2}x_{4} + x_{2}x_{5} + x_{3}x_{4} + x_{3}x_{5} + x_{3}x_{6} + x_{4}x_{6} + x_{5}x_{6})(tx_{2}x_{4}x_{5}x_{6}) - (sx_{2}x_{3}x_{4} + ux_{3}x_{5}x_{6})(x_{2}x_{3} + x_{2}x_{4} + x_{2}x_{5} + x_{3}x_{4} + x_{3}x_{5} + x_{3}x_{6} + x_{4}x_{5} + x_{5}x_{6})\}. $

We claim that the polynomial in $S^{3}$ does not factor into linear polynomials in any variable. Let $f$ denote the polynomial in $S^{3}$. By the symmetry of $K_{4}$, it suffices to show that $f$ does not factor into polynomials linear in $x_{2}$ or $x_{3}$. If we collect the coefficients of $f$ in terms of $x_2$, we get $f = Ax_{2}^{2} + Bx_{2} + C$ where 

$A = -sx_{{3}}x_{{4}}x_{{5}}x_{{6}}-s{x_{{4}}}^{2}x_{{5}}x_{{6}}-sx_{{4}}{x
_{{5}}}^{2}x_{{6}}+t{x_{{3}}}^{2}{x_{{4}}}^{2}+t{x_{{3}}}^{2}x_{{4}}x_
{{6}}+tx_{{3}}{x_{{4}}}^{2}x_{{5}}+tx_{{3}}{x_{{4}}}^{2}x_{{6}}+tx_{{3
}}x_{{4}}x_{{5}}x_{{6}},
$

$B = -sx_{{3}}{x_{{4}}}^{2}x_{{5}}x_{{6}}-sx_{{3}}x_{{4}}{x_{{5}}}^{2}x_{{6
}}-sx_{{3}}x_{{4}}x_{{5}}{x_{{6}}}^{2}-s{x_{{4}}}^{2}x_{{5}}{x_{{6}}}^
{2}-sx_{{4}}{x_{{5}}}^{2}{x_{{6}}}^{2}+t{x_{{3}}}^{2}{x_{{4}}}^{2}x_{{
5}}+t{x_{{3}}}^{2}x_{{4}}x_{{5}}x_{{6}}+tx_{{3}}{x_{{4}}}^{2}x_{{5}}x_
{{6}}+u{x_{{3}}}^{2}x_{{4}}x_{{5}}x_{{6}}+u{x_{{3}}}^{2}x_{{5}}{x_{{6}
}}^{2}+ux_{{3}}x_{{4}}{x_{{5}}}^{2}x_{{6}}+ux_{{3}}x_{{4}}x_{{5}}{x_{{
6}}}^{2}+ux_{{3}}{x_{{5}}}^{2}{x_{{6}}}^{2},
$

$C = u{x_{{3}}}^{2}x_{{4}}{x_{{5}}}^{2}x_{{6}}+u{x_{{3}}}^{2}{x_{{5}}}^{2}{
x_{{6}}}^{2}+ux_{{3}}x_{{4}}{x_{{5}}}^{2}{x_{{6}}}^{2}
$.

Suppose $f = f_{1}f_{2}$ for some polynomials $f_{1}$ and $f_{2}$ where $\deg(f_{i},x_{2}) = 1$, $i \in \{1,2\}$. Then by the quadratic formula, $B^{2} - 4AC$ is a perfect square. If we treat $s,t$ and $u$ as integers, then setting $s=t=u =1$ and $x_{i} = 1$, for all $i \in \{3,4,5,6\}$,  we get $B^2 -4AC =-15$ which is not a square. Therefore $B^2 -4AC$ is not a perfect square and thus $f$ does not factor into two polynomials linear in $x_2$. 

Applying the same strategy to $x_{3}$, we collect the coefficients of $f$ in terms of $x_3$, and we get $f = Ax_{3}^{2} + Bx_{3} +C$ where 

$A= t{x_{{2}}}^{2}{x_{{4}}}^{2}+t{x_{{2}}}^{2}x_{{4}}x_{{6}}+tx_{{2}}{x_{{
4}}}^{2}x_{{5}}+tx_{{2}}x_{{4}}x_{{5}}x_{{6}}+ux_{{2}}x_{{4}}x_{{5}}x_
{{6}}+ux_{{2}}x_{{5}}{x_{{6}}}^{2}+ux_{{4}}{x_{{5}}}^{2}x_{{6}}+u{x_{{
5}}}^{2}{x_{{6}}}^{2}
$,

$B = -s{x_{{2}}}^{2}x_{{4}}x_{{5}}x_{{6}}-sx_{{2}}{x_{{4}}}^{2}x_{{5}}x_{{6
}}-sx_{{2}}x_{{4}}{x_{{5}}}^{2}x_{{6}}-sx_{{2}}x_{{4}}x_{{5}}{x_{{6}}}
^{2}+t{x_{{2}}}^{2}{x_{{4}}}^{2}x_{{5}}+t{x_{{2}}}^{2}{x_{{4}}}^{2}x_{
{6}}+t{x_{{2}}}^{2}x_{{4}}x_{{5}}x_{{6}}+tx_{{2}}{x_{{4}}}^{2}x_{{5}}x
_{{6}}+ux_{{2}}x_{{4}}{x_{{5}}}^{2}x_{{6}}+ux_{{2}}x_{{4}}x_{{5}}{x_{{
6}}}^{2}+ux_{{2}}{x_{{5}}}^{2}{x_{{6}}}^{2}+ux_{{4}}{x_{{5}}}^{2}{x_{{
6}}}^{2}
$,

$C = -s{x_{{2}}}^{2}{x_{{4}}}^{2}x_{{5}}x_{{6}}-s{x_{{2}}}^{2}x_{{4}}{x_{{5
}}}^{2}x_{{6}}-sx_{{2}}{x_{{4}}}^{2}x_{{5}}{x_{{6}}}^{2}-sx_{{2}}x_{{4
}}{x_{{5}}}^{2}{x_{{6}}}^{2}
$.

Now setting $x_{i} = 1$ for all $i \in \{2,3,4,6\}$, $x_{5} = 2$, and $s=t=u =1$ we get $B^2 -4AC = 985$, which is not a perfect square. Thus $B^2 -4AC$ is not a perfect square so $f$ does not factor into two polynomials linear in $x_{3}$. With this we have shown that for all $x_{i}$, $i \in \{2,\ldots,6\}$, $f$ does not factor into two polynomials linear in $x_{i}$. Thus $f \in \tilde{S}$ and thus $S_{(\sigma(1))}$ has a polynomial not linear in $\sigma(2)$, completing the claim. 

\end{proof}

The same technique used above can be used to prove that $W_{4}$ is not simply reducible when we add external momenta to the rim vertices. As with $K_{4}$, this example will end up becoming a forbidden minor for reducibility, and thus it is worthwhile to work out the details.
\begin{example}
\label{W4notreducible}
The graph $W_{4}$ is not simply reducible with respect to both $\{\psi,\phi\}$ when the rim vertices have on-shell external momenta and no edges are massive. 
\end{example}

\begin{proof}
Let $\sigma$ be any permutation of $\{x_{1},x_{2},\ldots,x_{8}\}$. Using the labels in Figure \ref{W4K4labelling}, one can calculate that:

$\psi_{W_{4}} = x_{{1}}x_{{2}}x_{{4}}x_{{6}}+x_{{1}}x_{{2}}x_{{4}}x_{{7}}+x_{{1}}x_{{2
}}x_{{4}}x_{{8}}+x_{{1}}x_{{2}}x_{{5}}x_{{6}}+x_{{1}}x_{{2}}x_{{5}}x_{
{7}}+x_{{1}}x_{{2}}x_{{5}}x_{{8}}+x_{{1}}x_{{2}}x_{{6}}x_{{7}}+x_{{1}}
x_{{2}}x_{{7}}x_{{8}}+x_{{1}}x_{{3}}x_{{4}}x_{{6}}+x_{{1}}x_{{3}}x_{{4
}}x_{{7}}+x_{{1}}x_{{3}}x_{{4}}x_{{8}}+x_{{1}}x_{{3}}x_{{5}}x_{{6}}+x_
{{1}}x_{{3}}x_{{5}}x_{{7}}+x_{{1}}x_{{3}}x_{{5}}x_{{8}}+x_{{1}}x_{{3}}
x_{{6}}x_{{7}}+x_{{1}}x_{{3}}x_{{7}}x_{{8}}+x_{{1}}x_{{4}}x_{{6}}x_{{8
}}+x_{{1}}x_{{4}}x_{{7}}x_{{8}}+x_{{1}}x_{{5}}x_{{6}}x_{{8}}+x_{{1}}x_
{{5}}x_{{7}}x_{{8}}+x_{{1}}x_{{6}}x_{{7}}x_{{8}}+x_{{2}}x_{{3}}x_{{4}}
x_{{6}}+x_{{2}}x_{{3}}x_{{4}}x_{{7}}+x_{{2}}x_{{3}}x_{{4}}x_{{8}}+x_{{
2}}x_{{3}}x_{{5}}x_{{6}}+x_{{2}}x_{{3}}x_{{5}}x_{{7}}+x_{{2}}x_{{3}}x_
{{5}}x_{{8}}+x_{{2}}x_{{3}}x_{{6}}x_{{7}}+x_{{2}}x_{{3}}x_{{7}}x_{{8}}
+x_{{2}}x_{{4}}x_{{5}}x_{{6}}+x_{{2}}x_{{4}}x_{{5}}x_{{7}}+x_{{2}}x_{{
4}}x_{{5}}x_{{8}}+x_{{2}}x_{{4}}x_{{6}}x_{{7}}+x_{{2}}x_{{4}}x_{{6}}x_
{{8}}+x_{{2}}x_{{5}}x_{{6}}x_{{8}}+x_{{2}}x_{{5}}x_{{7}}x_{{8}}+x_{{2}
}x_{{6}}x_{{7}}x_{{8}}+x_{{3}}x_{{4}}x_{{5}}x_{{6}}+x_{{3}}x_{{4}}x_{{
5}}x_{{7}}+x_{{3}}x_{{4}}x_{{5}}x_{{8}}+x_{{3}}x_{{4}}x_{{6}}x_{{7}}+x
_{{3}}x_{{4}}x_{{7}}x_{{8}}+x_{{4}}x_{{5}}x_{{6}}x_{{8}}+x_{{4}}x_{{5}
}x_{{7}}x_{{8}}+x_{{4}}x_{{6}}x_{{7}}x_{{8}}$.

As the momenta are on-shell, appealing to Lemma \ref{MomentumLemma}, we calculate that:

$\phi_{W_{4}} = sx_{{1}}x_{{3}}x_{{5}}x_{{7}}x_{{8}}+sx_{{2}}x_{{3}}x_{{4}}x_{{5}}x_{{
6}}+sx_{{2}}x_{{3}}x_{{4}}x_{{5}}x_{{7}}+sx_{{2}}x_{{3}}x_{{4}}x_{{5}}
x_{{8}}+sx_{{2}}x_{{3}}x_{{5}}x_{{7}}x_{{8}}+sx_{{3}}x_{{4}}x_{{5}}x_{
{7}}x_{{8}}+ux_{{1}}x_{{2}}x_{{4}}x_{{6}}x_{{7}}+ux_{{1}}x_{{2}}x_{{4}
}x_{{6}}x_{{8}}+ux_{{1}}x_{{2}}x_{{5}}x_{{6}}x_{{8}}+ux_{{1}}x_{{2}}x_
{{6}}x_{{7}}x_{{8}}+ux_{{1}}x_{{3}}x_{{4}}x_{{6}}x_{{7}}+ux_{{1}}x_{{4
}}x_{{6}}x_{{7}}x_{{8}}$, 

where $u = (p_{1} + p_{4})^{2}$ and $s = (p_{1} + p_{3})^{2}$. 

By symmetry of $W_{4}$, without loss of generality we may assume that $\sigma(1) = x_1$ or $\sigma(1) = x_{2}$. We first consider the case where $\sigma(1) = x_{1}$. Then we have

$S^{1} = \{ x_{{2}}x_{{4}}x_{{6}}+x_{{2}}x_{{4}}x_{{7}}+x_{{2}}x_{{4}}x_{{8}}+x_{{
2}}x_{{5}}x_{{6}}+x_{{2}}x_{{5}}x_{{7}}+x_{{2}}x_{{5}}x_{{8}}+x_{{2}}x
_{{6}}x_{{7}}+x_{{2}}x_{{7}}x_{{8}}+x_{{3}}x_{{4}}x_{{6}}+x_{{3}}x_{{4
}}x_{{7}}+x_{{3}}x_{{4}}x_{{8}}+x_{{3}}x_{{5}}x_{{6}}+x_{{3}}x_{{5}}x_
{{7}}+x_{{3}}x_{{5}}x_{{8}}+x_{{3}}x_{{6}}x_{{7}}+x_{{3}}x_{{7}}x_{{8}
}+x_{{4}}x_{{6}}x_{{8}}+x_{{4}}x_{{7}}x_{{8}}+x_{{5}}x_{{6}}x_{{8}}+x_
{{5}}x_{{7}}x_{{8}}+x_{{6}}x_{{7}}x_{{8}},sx_{{3}}x_{{5}}x_{{7}}x_{{8}}+ux_{{2}}x_{{4}}x_{{6}}x_{{7}}+ux_{{2}}x_
{{4}}x_{{6}}x_{{8}}+ux_{{2}}x_{{5}}x_{{6}}x_{{8}}+ux_{{2}}x_{{6}}x_{{7
}}x_{{8}}+ux_{{3}}x_{{4}}x_{{6}}x_{{7}}+ux_{{4}}x_{{6}}x_{{7}}x_{{8}} \}, $ 

$S^{2} = \{x_{{2}}x_{{3}}x_{{4}}x_{{6}}+x_{{2}}x_{{3}}x_{{4}}x_{{7}}+x_{{2}}x_{{3
}}x_{{4}}x_{{8}}+x_{{2}}x_{{3}}x_{{5}}x_{{6}}+x_{{2}}x_{{3}}x_{{5}}x_{
{7}}+x_{{2}}x_{{3}}x_{{5}}x_{{8}}+x_{{2}}x_{{3}}x_{{6}}x_{{7}}+x_{{2}}
x_{{3}}x_{{7}}x_{{8}}+x_{{2}}x_{{4}}x_{{5}}x_{{6}}+x_{{2}}x_{{4}}x_{{5
}}x_{{7}}+x_{{2}}x_{{4}}x_{{5}}x_{{8}}+x_{{2}}x_{{4}}x_{{6}}x_{{7}}+x_
{{2}}x_{{4}}x_{{6}}x_{{8}}+x_{{2}}x_{{5}}x_{{6}}x_{{8}}+x_{{2}}x_{{5}}
x_{{7}}x_{{8}}+x_{{2}}x_{{6}}x_{{7}}x_{{8}}+x_{{3}}x_{{4}}x_{{5}}x_{{6
}}+x_{{3}}x_{{4}}x_{{5}}x_{{7}}+x_{{3}}x_{{4}}x_{{5}}x_{{8}}+x_{{3}}x_
{{4}}x_{{6}}x_{{7}}+x_{{3}}x_{{4}}x_{{7}}x_{{8}}+x_{{4}}x_{{5}}x_{{6}}
x_{{8}}+x_{{4}}x_{{5}}x_{{7}}x_{{8}}+x_{{4}}x_{{6}}x_{{7}}x_{{8}}, sx_{{2}}x_{{3}}x_{{4}}x_{{5}}x_{{6}}+sx_{{2}}x_{{3}}x_{{4}}x_{{5}}x_{{
7}}+sx_{{2}}x_{{3}}x_{{4}}x_{{5}}x_{{8}}+sx_{{2}}x_{{3}}x_{{5}}x_{{7}}
x_{{8}}+sx_{{3}}x_{{4}}x_{{5}}x_{{7}}x_{{8}} \}, $

$S^{3} = \{( x_{{2}}x_{{4}}x_{{6}}+x_{{2}}x_{{4}}x_{{7}}+x_{{2}}x_{{4}}x_{{
8}}+x_{{2}}x_{{5}}x_{{6}}+x_{{2}}x_{{5}}x_{{7}}+x_{{2}}x_{{5}}x_{{8}}+
x_{{2}}x_{{6}}x_{{7}}+x_{{2}}x_{{7}}x_{{8}}+x_{{3}}x_{{4}}x_{{6}}+x_{{
3}}x_{{4}}x_{{7}}+x_{{3}}x_{{4}}x_{{8}}+x_{{3}}x_{{5}}x_{{6}}+x_{{3}}x
_{{5}}x_{{7}}+x_{{3}}x_{{5}}x_{{8}}+x_{{3}}x_{{6}}x_{{7}}+x_{{3}}x_{{7
}}x_{{8}}+x_{{4}}x_{{6}}x_{{8}}+x_{{4}}x_{{7}}x_{{8}}+x_{{5}}x_{{6}}x_
{{8}}+x_{{5}}x_{{7}}x_{{8}}+x_{{6}}x_{{7}}x_{{8}}) ( sx_{
{2}}x_{{3}}x_{{4}}x_{{5}}x_{{6}}+sx_{{2}}x_{{3}}x_{{4}}x_{{5}}x_{{7}}+
sx_{{2}}x_{{3}}x_{{4}}x_{{5}}x_{{8}}+sx_{{2}}x_{{3}}x_{{5}}x_{{7}}x_{{
8}}+sx_{{3}}x_{{4}}x_{{5}}x_{{7}}x_{{8}}) - ( sx_{{3}}x_{{
5}}x_{{7}}x_{{8}}+ux_{{2}}x_{{4}}x_{{6}}x_{{7}}+ux_{{2}}x_{{4}}x_{{6}}
x_{{8}}+ux_{{2}}x_{{5}}x_{{6}}x_{{8}}+ux_{{2}}x_{{6}}x_{{7}}x_{{8}}+ux
_{{3}}x_{{4}}x_{{6}}x_{{7}}+ux_{{4}}x_{{6}}x_{{7}}x_{{8}}) 
 ( x_{{2}}x_{{3}}x_{{4}}x_{{6}}+x_{{2}}x_{{3}}x_{{4}}x_{{7}}+x_{{
2}}x_{{3}}x_{{4}}x_{{8}}+x_{{2}}x_{{3}}x_{{5}}x_{{6}}+x_{{2}}x_{{3}}x_
{{5}}x_{{7}}+x_{{2}}x_{{3}}x_{{5}}x_{{8}}+x_{{2}}x_{{3}}x_{{6}}x_{{7}}
+x_{{2}}x_{{3}}x_{{7}}x_{{8}}+x_{{2}}x_{{4}}x_{{5}}x_{{6}}+x_{{2}}x_{{
4}}x_{{5}}x_{{7}}+x_{{2}}x_{{4}}x_{{5}}x_{{8}}+x_{{2}}x_{{4}}x_{{6}}x_
{{7}}+x_{{2}}x_{{4}}x_{{6}}x_{{8}}+x_{{2}}x_{{5}}x_{{6}}x_{{8}}+x_{{2}
}x_{{5}}x_{{7}}x_{{8}}+x_{{2}}x_{{6}}x_{{7}}x_{{8}}+x_{{3}}x_{{4}}x_{{
5}}x_{{6}}+x_{{3}}x_{{4}}x_{{5}}x_{{7}}+x_{{3}}x_{{4}}x_{{5}}x_{{8}}+x
_{{3}}x_{{4}}x_{{6}}x_{{7}}+x_{{3}}x_{{4}}x_{{7}}x_{{8}}+x_{{4}}x_{{5}
}x_{{6}}x_{{8}}+x_{{4}}x_{{5}}x_{{7}}x_{{8}}+x_{{4}}x_{{6}}x_{{7}}x_{{
8}} )\}. $

Let $f$ be the polynomial in $S^{3}$. We claim that $f$ does not factor into polynomials linear in any of $x_{4},x_{5},x_{6}$ or $x_{8}$.  By the symmetry of $W_{4}$, this suffices to show that $S_{\sigma(1)}$ is not linear in $\sigma(2)$. For $x_{5}$, we assume that $f = f_{1}f_{2}$ where $f_{1}$ and $f_{2}$ are both linear in $x_{5}$ and we write $f = Ax_{5}^{2} + Bx_{5} + C$. Then from the quadratic formula, $\sqrt{B^{2} -4AC}$ is a polynomial. However by evaluating $B^{2} -4AC$ at $1$ for all $s,t,u$ and all $x_{i}, i \in \{2,3,4,6,7,8\}$, we get $B^{2} -4AC = 3564$ which is not a perfect square, a contradiction. 

Applying the same strategy to $x_{4}$, evaluating $B^{2} - 4AC$ at one for $s,t,u$ and $x_{i} =1$ for $i \in \{2,3,5,6,7,8\}$ we get $-1008$ which is not perfect square, a contradiction.

For $x_{8}$, evaluating $B^{2}-4AC$ at one for $s,u$ and $x_{i}$, $i \in \{2,\ldots,7\}$ we get $585$ which is not a perfect square, a contradiction. 

For $x_{6}$, evaluating $B^{2} - 4AC$ at one for $s,u$ and $x_{i}$, $i \in \{2,3,4,5,7,8\}$ we get $10368$ which is not a perfect square, a contradiction. 

With that, we may assume that $\sigma(1) \neq x_{1}$. Therefore  $\sigma(1) =x_{2}$. Then we calculate:

$S^{1} = \{x_{{1}}x_{{4}}x_{{6}}+x_{{1}}x_{{4}}x_{{7}}+x_{{1}}x_{{4}}x_{{8}}+x_{{
1}}x_{{5}}x_{{6}}+x_{{1}}x_{{5}}x_{{7}}+x_{{1}}x_{{5}}x_{{8}}+x_{{1}}x
_{{6}}x_{{7}}+x_{{1}}x_{{7}}x_{{8}}+x_{{3}}x_{{4}}x_{{6}}+x_{{3}}x_{{4
}}x_{{7}}+x_{{3}}x_{{4}}x_{{8}}+x_{{3}}x_{{5}}x_{{6}}+x_{{3}}x_{{5}}x_
{{7}}+x_{{3}}x_{{5}}x_{{8}}+x_{{3}}x_{{6}}x_{{7}}+x_{{3}}x_{{7}}x_{{8}
}+x_{{4}}x_{{5}}x_{{6}}+x_{{4}}x_{{5}}x_{{7}}+x_{{4}}x_{{5}}x_{{8}}+x_
{{4}}x_{{6}}x_{{7}}+x_{{4}}x_{{6}}x_{{8}}+x_{{5}}x_{{6}}x_{{8}}+x_{{5}
}x_{{7}}x_{{8}}+x_{{6}}x_{{7}}x_{{8}}, sx_{{3}}x_{{4}}x_{{5}}x_{{6}}+sx_{{3}}x_{{4}}x_{{5}}x_{{7}}+sx_{{3}}x_
{{4}}x_{{5}}x_{{8}}+sx_{{3}}x_{{5}}x_{{7}}x_{{8}}+ux_{{1}}x_{{4}}x_{{6
}}x_{{7}}+ux_{{1}}x_{{4}}x_{{6}}x_{{8}}+ux_{{1}}x_{{5}}x_{{6}}x_{{8}}+
ux_{{1}}x_{{6}}x_{{7}}x_{{8}} \},$

$S^{2} = \{x_{{1}}x_{{3}}x_{{4}}x_{{6}}+x_{{1}}x_{{3}}x_{{4}}x_{{7}}+x_{{1}}x_{{3
}}x_{{4}}x_{{8}}+x_{{1}}x_{{3}}x_{{5}}x_{{6}}+x_{{1}}x_{{3}}x_{{5}}x_{
{7}}+x_{{1}}x_{{3}}x_{{5}}x_{{8}}+x_{{1}}x_{{3}}x_{{6}}x_{{7}}+x_{{1}}
x_{{3}}x_{{7}}x_{{8}}+x_{{1}}x_{{4}}x_{{6}}x_{{8}}+x_{{1}}x_{{4}}x_{{7
}}x_{{8}}+x_{{1}}x_{{5}}x_{{6}}x_{{8}}+x_{{1}}x_{{5}}x_{{7}}x_{{8}}+x_
{{1}}x_{{6}}x_{{7}}x_{{8}}+x_{{3}}x_{{4}}x_{{5}}x_{{6}}+x_{{3}}x_{{4}}
x_{{5}}x_{{7}}+x_{{3}}x_{{4}}x_{{5}}x_{{8}}+x_{{3}}x_{{4}}x_{{6}}x_{{7
}}+x_{{3}}x_{{4}}x_{{7}}x_{{8}}+x_{{4}}x_{{5}}x_{{6}}x_{{8}}+x_{{4}}x_
{{5}}x_{{7}}x_{{8}}+x_{{4}}x_{{6}}x_{{7}}x_{{8}},sx_{{1}}x_{{3}}x_{{5}}x_{{7}}x_{{8}}+sx_{{3}}x_{{4}}x_{{5}}x_{{7}}x_{{
8}}+ux_{{1}}x_{{3}}x_{{4}}x_{{6}}x_{{7}}+ux_{{1}}x_{{4}}x_{{6}}x_{{7}}
x_{{8}} \},$

$S^{3} = \{ ( x_{{1}}x_{{4}}x_{{6}}+x_{{1}}x_{{4}}x_{{7}}+x_{{1}}x_{{4}}x_{{
8}}+x_{{1}}x_{{5}}x_{{6}}+x_{{1}}x_{{5}}x_{{7}}+x_{{1}}x_{{5}}x_{{8}}+
x_{{1}}x_{{6}}x_{{7}}+x_{{1}}x_{{7}}x_{{8}}+x_{{3}}x_{{4}}x_{{6}}+x_{{
3}}x_{{4}}x_{{7}}+x_{{3}}x_{{4}}x_{{8}}+x_{{3}}x_{{5}}x_{{6}}+x_{{3}}x
_{{5}}x_{{7}}+x_{{3}}x_{{5}}x_{{8}}+x_{{3}}x_{{6}}x_{{7}}+x_{{3}}x_{{7
}}x_{{8}}+x_{{4}}x_{{5}}x_{{6}}+x_{{4}}x_{{5}}x_{{7}}+x_{{4}}x_{{5}}x_
{{8}}+x_{{4}}x_{{6}}x_{{7}}+x_{{4}}x_{{6}}x_{{8}}+x_{{5}}x_{{6}}x_{{8}
}+x_{{5}}x_{{7}}x_{{8}}+x_{{6}}x_{{7}}x_{{8}} )  ( sx_{{1}}
x_{{3}}x_{{5}}x_{{7}}x_{{8}}+sx_{{3}}x_{{4}}x_{{5}}x_{{7}}x_{{8}}+ux_{
{1}}x_{{3}}x_{{4}}x_{{6}}x_{{7}}+ux_{{1}}x_{{4}}x_{{6}}x_{{7}}x_{{8}}
 ) - ( sx_{{3}}x_{{4}}x_{{5}}x_{{6}}+sx_{{3}}x_{{4}}x_{{5}}
x_{{7}}+sx_{{3}}x_{{4}}x_{{5}}x_{{8}}+sx_{{3}}x_{{5}}x_{{7}}x_{{8}}+ux
_{{1}}x_{{4}}x_{{6}}x_{{7}}+ux_{{1}}x_{{4}}x_{{6}}x_{{8}}+ux_{{1}}x_{{
5}}x_{{6}}x_{{8}}+ux_{{1}}x_{{6}}x_{{7}}x_{{8}})( x_{{1}
}x_{{3}}x_{{4}}x_{{6}}+x_{{1}}x_{{3}}x_{{4}}x_{{7}}+x_{{1}}x_{{3}}x_{{
4}}x_{{8}}+x_{{1}}x_{{3}}x_{{5}}x_{{6}}+x_{{1}}x_{{3}}x_{{5}}x_{{7}}+x
_{{1}}x_{{3}}x_{{5}}x_{{8}}+x_{{1}}x_{{3}}x_{{6}}x_{{7}}+x_{{1}}x_{{3}
}x_{{7}}x_{{8}}+x_{{1}}x_{{4}}x_{{6}}x_{{8}}+x_{{1}}x_{{4}}x_{{7}}x_{{
8}}+x_{{1}}x_{{5}}x_{{6}}x_{{8}}+x_{{1}}x_{{5}}x_{{7}}x_{{8}}+x_{{1}}x
_{{6}}x_{{7}}x_{{8}}+x_{{3}}x_{{4}}x_{{5}}x_{{6}}+x_{{3}}x_{{4}}x_{{5}
}x_{{7}}+x_{{3}}x_{{4}}x_{{5}}x_{{8}}+x_{{3}}x_{{4}}x_{{6}}x_{{7}}+x_{
{3}}x_{{4}}x_{{7}}x_{{8}}+x_{{4}}x_{{5}}x_{{6}}x_{{8}}+x_{{4}}x_{{5}}x
_{{7}}x_{{8}}+x_{{4}}x_{{6}}x_{{7}}x_{{8}})\}. 
$

As above, let $f$ be the polynomial in $S^{3}$. We will show that $f$ does not factor into polynomials linear in any of $x_{1},x_{4}$ or $x_{5}$. By the symmetry of $W_{4}$, this suffices to show that $f$ does not factor into polynomials linear in any of the variables. 

As before, we assume that $f = f_{1}f_{2}$ where $f_{1}$ and $f_{2}$ are linear in $x_{1}$. Then we write $f$ as $f = Ax_{1}^{2} + Bx_{1} +C$, and evaluate $B^2 -4AC$ at one for all $s,t,u$ and $x_{i}$, $i \in \{3,\ldots,8\}$ gives $-1008$ which is not a perfect square, a contradiction.

For $x_{4}$, writing $f$ as $f= Ax_{4}^2 + Bx_{1} +C$, and evaluating $B^2-4AC$ at one for $s,u$ and $x_{i}$, $i \in \{1,3,5,6,7,8\}$ we get $-567$ which is not a square, a contradiction. 

For $x_{5}$, writing $f$ as $f= Ax_{5}^2 + Bx_{5} +C$, and evaluating $B^2-4AC$ at one for $s,u$ and $x_{i}$, $i \in \{1,3,4,6,7,8\}$ we get $585$ which is not a perfect square, a contradiction. 

Therefore, $S_{(\sigma(1))}$ is not linear in $\sigma(2)$ for any permutation $\sigma$, and thus $W_{4}$ with on-shell external momenta on the rim vertices is not simply reducible. 
\end{proof}

The last example in this section first appeared in \cite{FrancisSmall} and is an example of a non-trivial graph which is simply reducible with respect to the first Symanzik polynomial. 

\begin{example}
The graph $K_{4}$ is simply reducible with respect to the first Symanzik polynomial. 
\end{example}

\begin{proof}
Using the labels from Figure \ref{W4K4labelling} recall that
\begin{equation*}
\begin{aligned}
\psi_{K_{4}} = {} & x_{{1}}x_{{3}}x_{{2}}+x_{{1}}x_{{4}}x_{{2}}+x_{{1}}x_{{5}}x_{{2}}+x_{{
1}}x_{{3}}x_{{4}}+x_{{3}}x_{{1}}x_{{5}}+x_{{1}}x_{{3}}x_{{6}}+x_{{1}}x
_{{4}}x_{{6}}+x_{{1}}x_{{5}}x_{{6}} \\ & +x_{{2}}x_{{3}}x_{{4}} + x_{{2}}x_{{3
}}x_{{6}}+x_{{2}}x_{{4}}x_{{5}}+x_{{2}}x_{{4}}x_{{6}}+x_{{2}}x_{{5}}x_
{{6}}+x_{{3}}x_{{4}}x_{{5}}+x_{{3}}x_{{5}}x_{{6}}+x_{{4}}x_{{5}}x_{{6}
}.
\end{aligned}
\end{equation*}

Consider the permutation $\sigma$ where $\sigma(1) = x_{1}, \sigma(2) =x_{2}, \sigma(3) = x_{5}, \sigma(4) = x_{6}, \sigma(5) = x_{3}$. One can calculate that:
\begin{equation*}
\begin{aligned}
S_{\sigma(1)} = {} & \{x_{{2}}x_{{3}}+x_{{2}}x_{{4}}+x_{{2}}x_{{5}}+x_{{3}}x_{{4}}+x_{{5}}x_{{3}}+x_{{6}}x_{{3}}+x_{{4}}x_{{6}}+x_{{5}}x_{{6}},x_{{2}}x_{{3}}x_{{4}}+x_{{2}}x_{{3}}x_{{6}} \\ &+x_{{2}}x_{{4}}x_{{5}}+x_{{2}}x_{{4}}x_{{6}}+x
_{{2}}x_{{5}}x_{{6}}+x_{{3}}x_{{4}}x_{{5}}+x_{{3}}x_{{5}}x_{{6}}+x_{{4
}}x_{{5}}x_{{6}}
\}. 
\end{aligned}
\end{equation*}

Now for $\sigma(2)$ we calculate that:

$S^{1} = \{x_{{3}}+x_{{4}}+x_{{5}},x_{{3}}x_{{4}}+x_{{6}}x_{{3}}+x_{{5}}x_{{4}}+x_{{4}}x_{{6}}+x_{{5}}x_{{6}} \},$

$S^{2} = \{ x_{{3}}x_{{4}}+x_{{5}}x_{{3}}+x_{{6}}x_{{3}}+x_{{4}}x_{{6}}+x_{{5}}x_
{{6}},x_{{3}}x_{{4}}x_{{5}}+x_{{3}}x_{{5}}x_{{6}}+x_{{4}}x_{{5}}x_{{6}
} \}, $

$S^{3} = \{-{x_{{3}}}^{2}{x_{{4}}}^{2}-2\,{x_{{3}}}^{2}x_{{4}}x_{{6}}-{x_{{3}}}^{
2}{x_{{6}}}^{2}-2\,x_{{3}}{x_{{4}}}^{2}x_{{6}}-2\,x_{{3}}x_{{4}}x_{{5}
}x_{{6}}-2\,x_{{3}}x_{{4}}{x_{{6}}}^{2}-2\, x_{{3}}x_{{5}}{x_{{6}}}^{2}
-{x_{{4}}}^{2}{x_{{6}}}^{2}-2\,x_{{4}}x_{{5}}{x_{{6}}}^{2}-{x_{{5}}}^{
2}{x_{{6}}}^{2}\}.  $

Now one can check that the polynomial in $S^{3}$ can be expressed as $- \left( x_{{3}}x_{{4}}+x_{{3}}x_{{6}}+x_{{4}}x_{{6}}+x_{{5}}x_{{6}}
 \right) ^{2}$, and that $x_{{3}}x_{{4}}x_{{5}}+x_{{3}}x_{{5}}x_{{6}}+x_{{4}}x_{{5}}x_{{6}} = x_{5}\left( x_{{3}}x_{{4}}+x_{{3}}x_{{6}}+x_{{4}}x_{{6}} \right)$. Ignoring the monomial, we have
 
 $S_{(\sigma(1),\sigma(2))} = \{x_{{3}}+x_{{4}}+x_{{5}},x_{{3}}x_{{4}}+x_{{3}}x_{{6}}+x_{{4}}x_{{5}}+
x_{{4}}x_{{6}}+x_{{5}}x_{{6}},x_{{3}}x_{{4}}+x_{{3}}x_{{5}}+x_{{3}}x_{
{6}}+x_{{4}}x_{{6}}+x_{{5}}x_{{6}},x_{{3}}x_{{4}}+x_{{3}}x_{{6}}+x_{{4
}}x_{{6}},x_{{3}}x_{{4}}+x_{{3}}x_{{6}}+x_{{4}}x_{{6}}+x_{{5}}x_{{6}} \}.$

Then for $\sigma(3)$, ignoring monomials and constants  we calculate that:

$S_{(\sigma(1),\sigma(2),\sigma(3))} = \{x_{{3}}+x_{{4}},x_{{4}}+x_{{6}},x_{{3}}x_{{4}}+x_{{3}}x_{{6}}+x_{{4}}
x_{{6}},x_{{3}}+x_{{6}},x_{{3}}-x_{{4}}
 \}.$ 
 
 Then for $\sigma(4)$ we calculate: 
 $S_{(\sigma(1),\sigma(2),\sigma(3), \sigma(4))} = \{ x_{3} + x_{4}, x_{3} -x_{4}\}$. 
 
 Then finally we see that $S_{(\sigma(1),\ldots,\sigma(5))}$ after removing monomials is the empty set and thus $K_{4}$ is simply reducible with respect to the first Symanzik polynomial. 
\end{proof}

We can relate the simple reduction algorithm back to Francis Brown's integration algorithm for Feynman integrals. The set $S$ at the start is the Symanzik polynomials appearing the Feynman integral and the permutation $\sigma$ is corresponding to a choice of order of integration. The polynomials we create at each iteration are an overestimate of the polynomials which appear at that step of the integration under that ordering. So a graph being simply reducible says that there is an order such that at each step the polynomials appearing will not cause a problem for the integration algorithm.

The reason I say we are only approximating the polynomials appearing at each step during the intergration is because with the above algorithm, we possibly get spurious polynomials. By this I mean, we possibly get polynomials in the simple reduction which do not appear in the integration. Clearly, the more polynomials we have in the reduction sets the more likely we are to say the graph is not simply reducibile. Thus a set of Symanzik polynomials may be amenable to Brown's integration algorithm but not simply reducible. Given that the point of the reduction algorithm is to decide if we can use Brown's integration algorithm, it would be advantageous to reduce the number of spurious polynomials. An easy improvement comes from noting that the simple reduction algorithm fundamentally relied on an ordering of the variables, but by Fubini's Theorem, the order of integration does not change the outcome of the integration. The next reduction algorithm reflects this, and was introduced by Brown in \cite{FrancisSmall}. 

\subsection{The Fubini Reduction Algorithm}

As in the simple reduction algorithm, let $S$ be a set of polynomials in the polynomial ring $Q[\alpha_1,\alpha_2,\ldots, \alpha_r]$. Let $\sigma$ be a permutation of $\{\alpha_1,\ldots,\alpha_r\}$. Recall that $S_{(\sigma(1),\ldots,\sigma(k))}$ is the set obtained by running the simple reduction algorithm on the variables $\sigma(1),\ldots,\sigma(k)$ in that order.

We define $S_{[\sigma(1)]} = S_{(\sigma(1))}$ and $S_{[\sigma(1),\sigma(2)]} = S_{(\sigma(1),\sigma(2))} \cap S_{(\sigma(2),\sigma(1))}$. In general, we have:

\begin{center}
\[S_{[\sigma(1),\ldots,\sigma(k)]} = \bigcap_{1 \leq i \leq k} S_{[\sigma(1),\ldots,\hat{\sigma(i)},\ldots,\sigma(k)](\sigma(i))}  \]
\end{center}

Where $\hat{\sigma(i)}$ means we are omitting $\sigma(i)$. As stated this definition is not well defined. It is entirely possible that in the intersection one of the sets $S_{[\sigma(1),\ldots,\hat{\sigma(i)},\ldots,\sigma(k)](\sigma(i))}$ is undefined.  In this case, we simply ignore that set and continue on. In the event $S_{[\sigma(1),\ldots,\hat{\sigma(i)},\ldots,\sigma(k)](\sigma(i))}$ is undefined for all $i$, the algorithm stops. Additionally, here when we take intersections, we do so up to constants. Thus if two polynomials differ by a constant factor, we do not remove them in the intersection.

\begin{definition}
Let $S$ be a set of polynomials in the polynomial ring $\mathbb{Q}[\alpha_{1},\ldots,\alpha_{r}]$.  Let $\sigma$ be a permutation of $\{\alpha_{1},\ldots,\alpha_{r}\}$. We say $S$ is \emph{Fubini reducible with respect to $\sigma$} if for all $ 1 \leq i \leq r-1$, all polynomials in the set $S_{[\sigma(1),\ldots,\sigma(i)]}$ are linear in $\sigma(i+1)$. If there exists a permutation $\sigma$ such that $S$ is Fubini reducible with respect to $\sigma$, we say that $S$ is \emph{Fubini reducible}. Given a graph $G$, and $S \subseteq \{\psi_{G},\phi_{G}\}$, if $S$ is Fubini reducible, then we say $G$ is reducible with respect to $S$.
\end{definition}

\begin{remark}
\label{Simpleimpliesfubini}
If $S$ is simply reducible with respect to some permutation $\sigma$, then $S$ is Fubini reducible with respect to $\sigma$.
\end{remark}

\begin{proof}
It follows immediately from the definitions that for any integer $k$,  $S_{[\sigma(1),\ldots,\sigma(k)]} \subseteq S_{(\sigma(1),\dots,\sigma(k))}$. As $S$ is simply reducible, $S_{(\sigma(1),\dots,\sigma(k))}$ is linear in $\sigma(k+1)$ for all $k$. Therefore, $S_{[\sigma(1),\ldots,\sigma(k)]}$ is linear in $\sigma(k+1)$ for all $k$, completing the proof.
\end{proof}

\begin{example}
The graphs $K_{4}$ and $W_{4}$ with external momenta as in example \ref{K4notreducible} and example \ref{W4notreducible} are not Fubini reducible with respect to the first and second Symanzik polynomials.
\end{example}

\begin{proof}
Notice that given any set of polynomials $S$, and any permutation $\sigma$ of the Schwinger parameters, $S_{[\sigma(1)]} = S_{(\sigma(1))}$. From example \ref{K4notreducible} and example \ref{W4notreducible} both $W_{4}$ and $K_{4}$ have $S_{(\sigma(1),\sigma(2))}$ undefined for all permutations $\sigma$. Thus we can conclude that $K_{4}$ and $W_{4}$ are not Fubini reducible with respect to the first and second Symanzik polynomials. 
\end{proof}

\begin{example}[\cite{FrancisSmall}]
The graph $K_{4}$ is Fubini reducible with respect to the first Symanzik polynomial, and the Fubini reduction produces less polynomials than the simple reduction. 
\end{example}

\begin{proof}
The fact that $K_{4}$ is Fubini reducible is immediate by Remark \ref{Simpleimpliesfubini} as $K_{4}$ is simply reducible with respect to the first Symanzik polynomial. We partially go through the calculation so as to show that the Fubini reduction algorithm does indeed remove some spurious polynomials. As before, we use the labelling in Figure \ref{W4K4labelling}. Recall that 

\begin{equation*}
\begin{aligned}
\psi_{K_{4}} = {} & x_{{1}}x_{{3}}x_{{2}}+x_{{1}}x_{{4}}x_{{2}}+x_{{1}}x_{{5}}x_{{2}}+x_{{
1}}x_{{3}}x_{{4}}+x_{{3}}x_{{1}}x_{{5}}+x_{{1}}x_{{3}}x_{{6}}+x_{{1}}x
_{{4}}x_{{6}}+x_{{1}}x_{{5}}x_{{6}} \\ & +x_{{2}}x_{{3}}x_{{4}} + x_{{2}}x_{{3
}}x_{{6}}+x_{{2}}x_{{4}}x_{{5}}+x_{{2}}x_{{4}}x_{{6}}+x_{{2}}x_{{5}}x_
{{6}}+x_{{3}}x_{{4}}x_{{5}}+x_{{3}}x_{{5}}x_{{6}}+x_{{4}}x_{{5}}x_{{6}
}.
\end{aligned}
\end{equation*}

Let $\sigma$ be a permutation of $\{x_{1},x_{2},x_{3},x_{4},x_{5}\}$ such that $\sigma(1) = x_{1}, \sigma(2) = x_{2}, \sigma(3) = x_{5}, \sigma(4) = x_{6},$ and $ \sigma(5) = x_{3}$.  We have already calculated that:

$S_{[\sigma(1)]} = \{ x_{{2}}x_{{3}}+x_{{2}}x_{{4}}+x_{{2}}x_{{5}}+x_{{3}}x_{{4}}+x_{{5}}x_{{3}}+x_{{6}}x_{{3}}+x_{{4}}x_{{6}}+x_{{5}}x_{{6}},x_{{2}}x_{{3}}x_{{4
}}+x_{{2}}x_{{3}}x_{{6}}+x_{{2}}x_{{4}}x_{{5}}+x_{{2}}x_{{4}}x_{{6}}+x_{{2}}x_{{5}}x_{{6}}+x_{{3}}x_{{4}}x_{{5}}+x_{{3}}x_{{5}}x_{{6}}+x_{{4}}x_{{5}}x_{{6}}\}.$

Similarly, we have that,

$S_{[\sigma(2)]} = \{x_{{1}}x_{{3}}+x_{{1}}x_{{4}}+x_{{1}}x_{{5}}+x_{{3}}x_{{4}}+x_{{3}}x_
{{6}}+x_{{4}}x_{{5}}+x_{{4}}x_{{6}}+x_{{5}}x_{{6}},x_{{1}}x_{{3}}x_{{4
}}+x_{{1}}x_{{3}}x_{{5}}+x_{{1}}x_{{3}}x_{{6}}+x_{{1}}x_{{4}}x_{{6}}+x
_{{1}}x_{{5}}x_{{6}}+x_{{3}}x_{{4}}x_{{5}}+x_{{3}}x_{{5}}x_{{6}}+x_{{4
}}x_{{5}}x_{{6}}\}.$

Then we calculate that:

$S_{[\sigma(1)](\sigma(2))} = \{x_{{3}}+x_{{4}}+x_{{5}},x_{{3}}x_{{4}}+x_{{3}}x_{{5}}+x_{{3}}x_{{6}}+
x_{{4}}x_{{6}}+x_{{5}}x_{{6}},x_{{3}}x_{{4}}+x_{{3}}x_{{6}}+x_{{4}}x_{
{5}}+x_{{4}}x_{{6}}+x_{{5}}x_{{6}},x_{{3}}x_{{4}}+x_{{3}}x_{{6}}+x_{{4
}}x_{{6}}+x_{{5}}x_{{6}},x_{{3}}x_{{4}}+x_{{3}}x_{{6}}+x_{{4}}x_{{6}}\}.$

Similarly we can calculate that $S_{[\sigma(2)](\sigma(1))} = S_{[\sigma(1)](\sigma(2))}$. Therefore $S_{[\sigma(1),\sigma(2)]} = S_{[\sigma(1)](\sigma(2))}$.

Since $S_{[\sigma(1),\sigma(2)]} = S_{[\sigma(1)](\sigma(2))}$, we have that $S_{[\sigma(1),\sigma(2)](\sigma(3))} = S_{(\sigma(1),\sigma(2),\sigma(3))}$ so we have that:

$S_{[\sigma(1),\sigma(2)](\sigma(3))} = \{x_{{3}}+x_{{4}},x_{{4}}+x_{{6}},x_{{3}}x_{{4}}+x_{{3}}x_{{6}}+x_{{4}}
x_{{6}},x_{{3}}+x_{{6}},x_{{3}}-x_{{4}}
 \}.$ 
 
 Now we calculate $S_{[\sigma(1),\sigma(3)](\sigma(2))}$ and we get that:

$S_{[\sigma(1),\sigma(3)]} = \{x_{{2}}+x_{{4}}+x_{{5}}+x_{{6}},x_{{4}}+x_{{6}},x_{{2}}+x_{{5}},x_{{4
}}+x_{{5}},x_{{2}}+x_{{6}},x_{{4}}x_{{2}}-x_{{5}}x_{{6}}\}$, and that

$S_{[\sigma(1),\sigma(3)](\sigma(2))} = \{x_{{4}}+x_{{6}},x_{{4}}+x_{{5}},x_{{4}}+x_{{5}}+x_{{6}}\}$.

Now notice that $S_{[\sigma(1), \sigma(3)](\sigma(2))} \cap S_{[\sigma(1), \sigma(2)](\sigma(3))} = \{x_{4} + x_{6}\}$.  This set has significantly fewer polynomials than $S_{(\sigma(1), \sigma(2),\sigma(3))}$. Therefore $S_{[\sigma(1),\sigma(2),\sigma(3)]}$ has fewer polynomials than $S_{(\sigma(1),\sigma(2),\sigma(3))}$ so the Fubini reduction algorithm does reduce the number of polynomials in the reduction. We do not finish the calculation as we already know that $W_{4}$ is Fubini reducible with respect to $\psi$.
\end{proof}

We define the \textit{loop order} of a graph to be $E(G) - E(T)$ where $T$ is a spanning tree of $G$. The loop order is also known as the dimension of the cycle space or the \textit{Betti number} of $G$. It turns out, if $G$ has loop order less than $6$, then $G$ is Fubini reducible with respect to the first Symanzik polynomial (\cite{panzerphdthesis}). Furthermore, most massless graphs with four on-shell momenta and loop order less than $3$ are Fubini reducible \cite{panzerphdthesis}. A computational study of Fubini reducibility for graphs with no massive edges and four on-shell momenta was undertaken by  L{\"u}ders in \cite{Martin}. For $3$-loop graphs, he tested $109$ graphs and found that $39$ of them were not Fubini reducible. For $4$-loop graphs, $600$ graphs were tested and around $200$ were found to not be Fubini reducible. Largely, the graphs that he found to not be reducible can be explained by having $W_{4}$ as a minor with the external momenta as in example \ref{W4notreducible}. The other non-Fubini reducible graphs all will turn out to be reducible by the next polynomial reduction algorithm we introduce. For those that would like to actually run the Fubini algorithm to do tests as in \cite{Martin}, Christian Bogner has implemented the Fubini algorithm in Maple and his program MPL can be freely downloaded (\cite{MPLarticle}).

 However despite the improvement the Fubini algorithm offers, this algorithm still produces spurious polynomials. Analysing Francis Brown's integration algorithm, it turns out that the polynomials $h_{i}g_{j}-h_{j}g_{i}$ only appear in the integration sometimes. Therefore one can create a bookkeeping method which keeps track of which keeps track of when we want to consider $h_{i}g_{j}-h_{j}g_{i}$ in our reduction. Additionally, we note that under computational tests, the Fubini reduction algorithm fares much better than the simple reduction algorithm for the number of graphs which are determined to be reducible. However, it is hard to prove infinite families of graphs are Fubini reducible as one needs to have a strong understanding of how the polynomials factor throughout the reduction. The compatibility graph algorithm deals with both concerns.

\subsection{Compatibility Graph Reduction Algorithm}

Compatibility graphs are essentially a book-keeping method for keeping track of when we want to consider the polynomial $h_{i}g_{j}-h_{j}g_{i}$ in our reduction algorithms. The notion first arose in \cite{FrancisBig}, and was essential in proving an infinite family of graphs have Feynman integrals  which evaluate to multiple zeta values. Let $S$ be a set of polynomials. A \textit{compatibility graph} for $S$, $C_{S}$, is a graph where $V(C_{S}) = \{f | f \in S\}$. If $fg \in E(C_{S})$, then we will say the polynomials $f$ and $g$ are \textit{compatible}. First, we describe the compatibility graph reduction modification for the simple reduction algorithm.

Let $S$ be a set of polynomials in the polynomial ring $Q[\alpha_1,\alpha_2,\ldots, \alpha_r]$.  Let $\sigma$ be a permutation of $\{\alpha_1,\ldots,\alpha_r\}$. Let $C_{S}$ be the initial compatibility graph for $S$.  The idea will be to create a sequence of new sets of polynomials with rational coefficients and compatibility graphs $(S_{(\sigma(1))}, C_{(\sigma(1)}))$,
$(S_{(\sigma(1),\sigma(2))},C_{(\sigma(1),\sigma(2))}),\ldots, (S_{(\sigma(1),\ldots,\sigma(r))},C_{(\sigma(1),\ldots,\sigma(r))})$  and check that each polynomial in the set $S_{(\sigma(1),\ldots,\sigma(i))}$ is linear in $\sigma(i+1)$ for all $i \in \{1,\ldots,r-1\}$. Suppose we are at the $k^{th}$ iteration of the algorithm. If $k \geq 2$, then  we have a set of polynomials with rational coefficients $S_{(\sigma(1),\ldots,\sigma(k-1))} =  \{f_{1},\ldots,f_{n}\}$ and compatibility graph $C_{(\sigma(1),\ldots,\sigma(k-1))}$. Otherwise $k=1$ and we use $S$ and $C_{S}$. We then do the following:

\begin{enumerate}
\item{If there is a polynomial $f \in S_{(\sigma(1),\ldots,\sigma(k))}$ such that $f$ is not linear in $\sigma(k)$, we end the algorithm, otherwise continue.}

\item{Then for all $i \in \{1,\ldots,n\}$, given a polynomial $f_i \in S_{(\sigma(1),\ldots,\sigma(k))}$, we write $f_{i} = g_{i}\sigma(k) + h_{i}$, where $g_{i} =  \frac{\partial f_{i} }{\partial \sigma(k)}$ and $h_{i} = f_{i}|_{\sigma(k)=0}$.}

\item{Let $S^{1} = \{g_{i} | i \in \{1,\ldots,n\}$. Let $S^{2} = \{h_{i} | i \in \{1,\ldots,n\} \}$. Let $S^{3} = \{g_{i}h_{j}-h_{i}g_{j}| i,j \in \{1,\ldots,n\}, i \neq j  \text{ where }  f_{i} \text{ and } f_{j} \text{ are compatible}\}$. Let $S^{4} = S^{1} \cup S^{2} \cup S^{3}$. }

\item{Let $f \in S^{4}$ and let $f_{1},f_{2},f_{3},\ldots,f_{N}$ be the polynomials such that $\prod_{i =1}^{N} f_{i} = f$, and $f_{i}$ is an irreducible polynomial over $\mathbb{Q}$. Let $\tilde{S}$ be the set of polynomials $f_{1},\ldots,f_{N}$ for each $f \in S^4$.}

\item{Let $S_{(\sigma(1),\ldots,\sigma(k))} = \tilde{S}$ and construct a new compatibility graph $C_{(\sigma(1),\ldots,\sigma(k))}$ with some rule set (see below).}

\item{ Repeat the above steps with $S_{(\sigma(1),\ldots,\sigma(k))}$ in place of $S_{(\sigma(1),\ldots,\sigma(k-1))}$ and $C_{(\sigma(1),\ldots,\sigma(k))}$ in place of $C_{(\sigma(1),\ldots,\sigma(k-1))}$.}
\end{enumerate}

Observe that this algorithm is equivalent to the simple reduction algorithm if we set the compatibility graphs to be the complete graph at each step.  Brown noticed that there is a scheme for constructing compatability graphs which results in removing non-trivial polynomials, while still getting an output which gives a suitable ordering for his integration algorithm. However, while his method for constructing compatability graphs works for the Symanzik polynomials, if one abstracts Feynman integrals to a more general class of integrals, it is unclear if Brown's compatability method still works to give a suitable order of integration for the integrals (\cite{panzerphdthesis}). In \cite{panzerphdthesis}, Panzer describes an alternative construction which holds for more general integrals. We outline both constructions below.

\subsubsection{First Compatibility Graph Construction}

Before describing Francis Brown's method for constructing compatibility graphs, we give some intuition for the construction. Essentially, at some iteration of the reduction, we look back at the genealogy of the polynomials in our current set. If far back enough, the polynomials do not share a common ancestor then we will say those two polynomials are not compatible. 
 
Now for the definition. Suppose we are at the $k^{th}$ iteration of the reduction algorithm under a given permutation $\sigma$ of $\{\alpha_{1},\ldots,\alpha_{r}\}$. Then we have the set  $S_{(\sigma(1),\ldots \sigma(k-1))}$ and at the end of the iteration we obtain the set  $S_{(\sigma(1),\ldots,\sigma(k))}$. 

We now define $C_{(\sigma(1),\ldots,\sigma(k))}$. Recall, $V(C_{(\sigma(1),\ldots,\sigma(k))}) = \{f \ | \ f \in S_{(\sigma(1),\ldots \sigma(k-1))}\}$. To determine the edges of $C_{(\sigma(1),\ldots,\sigma(k))}$ we associate a set of $2$-tuples to each vertex of $C_{(\sigma(1),\ldots,\sigma(k))}$ to keep track of how the polynomials were created. Let $m \in V(C_{(\sigma(1),\ldots,\sigma(k))})$. 
\begin{center}
\textbf{$2$-tuple assignments}
\end{center}
\begin{itemize}
\item{If $m$ is an irreducible factor of a polynomial $g_{i} \in S^{1}$ we associate the set $\{0,i\}$ with $m$.}

\item{ If $m$ is an irreducible factor of some polynomial $h_{i} \in S^{2}$ then we associate the set $\{i,\infty\}$ with $m$. Additionally, if $h_{i} = f_{i}$ where $f_{i} \in S_{(\sigma(1),\ldots \sigma(k-1))}$ then we associate the set $\{0,i\}$ to $m$ as well as $\{i,\infty\}$.}

\item{If $m$ is the irreducible polynomial of some polynomial $g_{i}h_{j}-h_{i}g_{j} \in S^{3}$, then we associate the set $\{i,j\}$ to $m$.}
\end{itemize}

Now let $m,n \in V(C_{(\sigma(1),\ldots,\sigma(k)})$. The edge $mn \in E(C_{(\sigma(1),\ldots,\sigma(k)})$ if and only if there exists an set associated with $m$ and an set associated with $n$ such that their intersection is non-empty. 

 As an example, suppose a polynomial $f$ is associated with the sets $\{1,\infty\}$ and $\{3,4\}$, and polynomial $g$ is associated with the sets $\{2,\infty\}$ and $\{1,3\}$. Then since $\{3,4\} \cap \{1,3\} = \{3\}$, the polynomials $f$ and $g$ are compatible. However if $g$ instead was associated with the sets $\{5,6\}$ and $\{0,2\}$, then $f$ and $g$ would not be compatible. 

Note that this construction does indeed remove sone non-trivial polynomials. For the first Symanzik polynomial one must go at least three iterations before you can see the effects. We will see later in the chapter that when looking at both Symanzik polynomials we see the effect of compatibility graphs much sooner. The general principle is, under the above construction, if two polynomials are compatible, then they share an ancestor two  iterations back (\cite{FrancisBig}). With this, we can now state one version of the compatibility graph reduction.

\subsection{Brown's Compatibility Graph Reduction}

Throughout this algorithm, all compatibility graphs will be constructed using Brown's technique outlined above.
 Let $S$ be a set of polynomials in the polynomial ring $Q[\alpha_1,\alpha_2,\ldots, \alpha_r]$. Initialize $C_{S}$ to be the complete graph on $|S|$ vertices and let $\sigma$ be a permutation of $\{\alpha_1,\ldots,\alpha_r\}$.

We define $S_{[\sigma(1)]} = S_{(\sigma(1))}$ where we calculate $S_{(\sigma(1))}$ using the simple compatibility graph algorithm outlined at the start of the section. We define $C_{[\sigma(1)]} = C_{(\sigma(1))}$. We define $S_{[\sigma(1),\sigma(2)]} = S_{(\sigma(1),\sigma(2))} \cap S_{(\sigma(2),\sigma(1))}$ and $C_{[\sigma(1),\sigma(2)]}$ to be a graph where $fg \in E(C_{[\sigma(1),\sigma(2)]})$ if and only if $fg \in E(C_{(\sigma(1),\sigma(2))}) \cap E(C_{(\sigma(2),\sigma(1))})$. We define

\begin{center}
\[S_{[\sigma(1),\ldots,\sigma(k)]} = \bigcap_{1 \leq i \leq k} S_{[\sigma(1),\ldots,\hat{\sigma(i)},\ldots,\sigma(k)](\sigma(i))},  \]
\end{center}

and we define $C_{[\sigma(1),\ldots,\sigma(k)]}$ to be the compatibility graph for $S_{[\sigma(1),\ldots,\sigma(k)]}$ such that $fg \in C_{[\sigma(1),\ldots,\sigma(k)]}$ if and only if

 \[fg \in  \bigcap_{1 \leq i \leq k} E(C_{[\sigma(1),\ldots,\hat{\sigma(i)},\ldots,\sigma(k)](\sigma(i))}).\]

\begin{definition}
Let $S$ be a set of polynomials in the polynomial ring $\mathbb{Q}[\alpha_{1},\ldots,\alpha_{r}]$.  Let $\sigma$ be a permutation of $\{\alpha_{1},\ldots,\alpha_{r}\}$. We say $S$ is \textit{compatibility graph reducible with respect to $\sigma$} if for all $ 1 \leq i \leq r-1$, all polynomials in the set $S_{[\sigma(1),\ldots,\sigma(i)]}$ are linear in $\sigma(i+1)$. If there exists a permutation $\sigma$ such that $S$ is compatibility graph reducible with respect to $\sigma$, we say that $S$ is \emph{compatibility graph reducible}. Given a graph $G$, and $S \subseteq \{\psi_{G},\phi_{G}\}$, if $S$ is compatibility graph reducible, then we say $G$ is compatibility graph reducible with respect to $S$.
\end{definition}

As this will be the notion of reducibility which will be used for the rest of the thesis, we will abbreviate  compatibility graph reducibility to just reducibility.

\begin{remark}
If $S$ is Fubini reducible with respect to $\sigma$, then $S$ is reducible with respect to $\sigma$. 
\end{remark}

\begin{proof}
Observe that the Fubini reduction algorithm is the compatibility graph reduction algorithm described in this section where at each step the compatibility graph is a complete graph. Let $S'_{[\sigma(1),\ldots,\sigma(k)]}$ denote the set obtained by the Fubini reduction algorithm at iteration $k$ and $S_{[\sigma(1),\ldots,\sigma(k)]}$ the set obtained from the compatibility graph reduction algorithm at iteration $k$. Then we have $S_{[\sigma(1),\ldots,\sigma(k)]} \subseteq S'_{[\sigma(1),\ldots,\sigma(k)]}$ for every $k$. Since $S$ is Fubini reducible, $S'_{[\sigma(1),\ldots,\sigma(k)]}$ is linear in $\sigma(k+1)$, and thus $S_{[\sigma(1),\ldots,\sigma(k)]}$ is linear in $\sigma(k+1)$, completing the proof.
\end{proof}

\begin{figure}
\begin{center}
\includegraphics[scale =0.4]{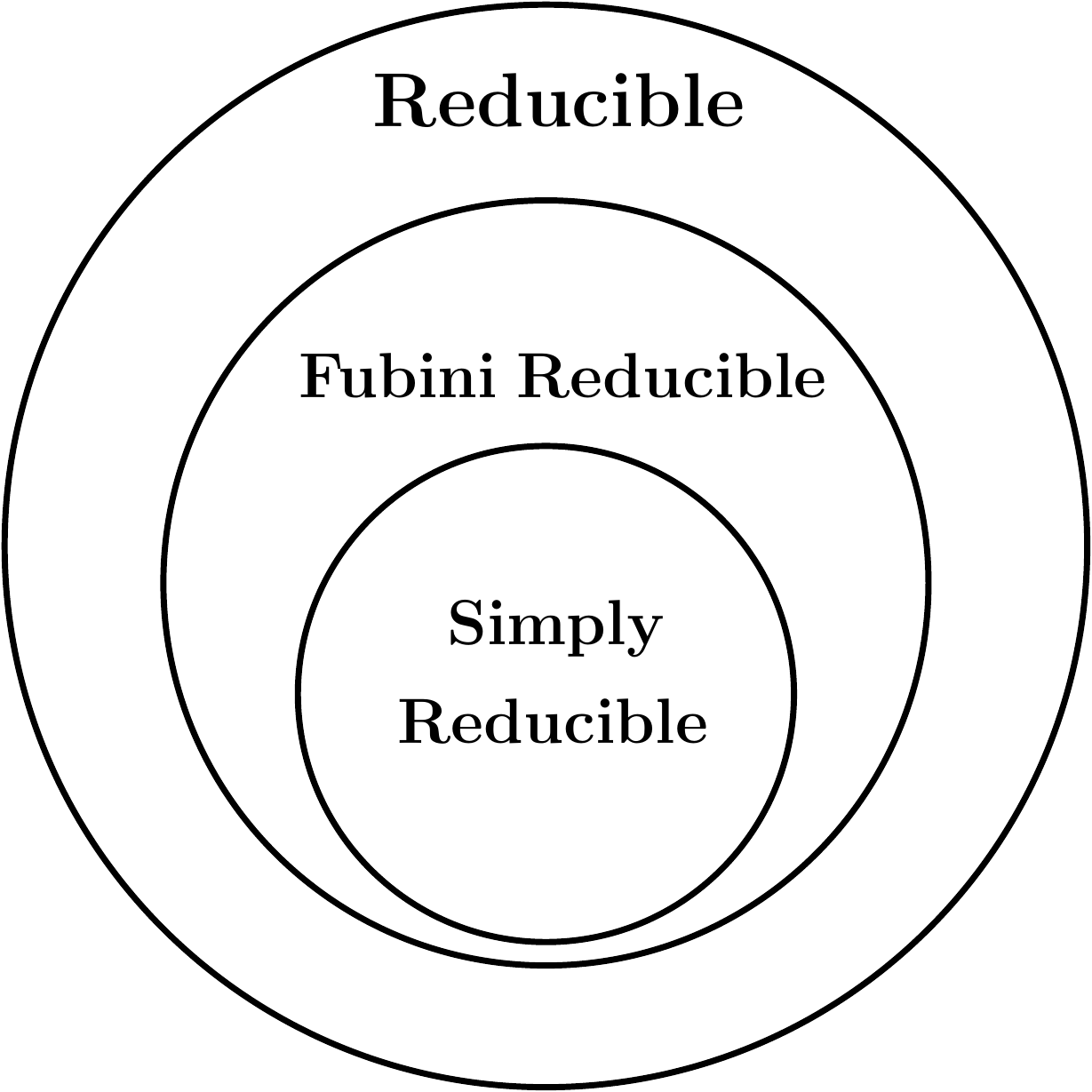}
\end{center}
\caption{The relationship between the different versions of reducibility.}
\label{reducibilityinclusion}
\end{figure} 

\begin{example}
The graphs $W_{4}$ and $K_{4}$ with external momenta as in example \ref{K4notreducible} and example \ref{W4notreducible} are not reducible with respect to the first and second Symanzik polynomials. 
\end{example}

\begin{proof}
Let $\sigma$ be a permutation. In the compatibility graph reduction algorithm, the initial compatibility graph is the complete graph. Therefore, $S_{[\sigma(1)]}$ is the same under the Fubini algorithm as the compatibility graph algorithm. As seen in example \ref{K4notreducible} and example \ref{W4notreducible}, $S_{[\sigma(1)]}$ is not linear in $\sigma(2)$, and thus $W_{4}$ and $K_{4}$ are not reducible.
\end{proof}

There is an nice infinite family of graphs which are known to be reducible with respect to the first Symanzik polynomial. 

\begin{definition}
Let $G$ be a graph with $m$ edges. The \emph{width} of an ordering $e_1, e_2,\ldots, e_m$ of $E(G)$ is the maximum order of a separation of the form $(\{e_1,\ldots,e_l\},\{e_{l+1},...,e_m\})$ for $l \in \{1,\ldots,m\}$. To be consistent with the given definition of separation, here we are viewing $\{e_1, \ldots,e_l\}$ as a set of vertices induced by the edges $e_{1},\ldots, e_{l}$. The \emph{vertex width} of $G$ is the minimum width among all edge orders of $G$.
\end{definition}

We note that the notion of vertex width is similar to the more well known width variant pathwidth. 

\begin{theorem}[Brown, \cite{FrancisBig}]
\label{vertexwidth3}
All graphs with vertex width $ \leq 3$ are reducible with respect to the first Symanzik polynomial. 
\end{theorem}

At this point, all infinite families of reducible graphs for the first Symanzik polynomial reside in the family of graphs with vertex width less than $3$. Also it is important to note that in the proof of Theorem \ref{vertexwidth3}, compatability graphs are essential.  A nice result of Black, Crump, Deos, and Yeats (\cite{Crumppaper}) is that a $3$-connected graph $G$ has vertex width $ \leq 3$ if and only if $G$ does not contain one of the cube, octahedron, $K_{5}$, $K_{3,3}$ and one other graph as a minor (see Figure \ref{vertexwidth3pic}). This result also appears in Crump's masters thesis (\cite{CrumpMastersthesis}).

Note there are graphs which are do not satisfy vertex width $\leq 3$ but are reducible with respect to the first Symanzik polynomial. For example, in \cite{FrancisBig}, it is shown that both $K_{3,3}$ and $K_{5}$ are reducible with respect to the first Symanzik polynomial, but it is easy to see that both $K_{3,3}$ and $K_{5}$ have vertex width greater than $3$, and they are in fact both forbidden minors for graphs having vertex width $3$.

We note that in \cite{Martin}, a computational study of reducibility for graphs with four on-shell momenta and no massive edges was undertaken. From L{\"u}der's tests, he found $34$ graphs which are not reducible with respect to both Symanzik polynomials where the graph is massless and there are four on-shell momenta. Compared to the $39$ graphs he found for Fubini reducibility, this is a slight improvement.  This tells us there are graphs which are reducible but not Fubini reducible. See Figure \ref{Martinexample} for a graph which is reducible but not Fubini reducible. We note that every single graph that he found was not reducible has either a rooted $K_{4}$-minor or a rooted $W_{4}$-minor (defined in chapter $3$) which will explain all the graphs he found to be non-reducible.  We end this section with an example of a compatibility graph reduction.

\begin{figure}
\begin{center}
\includegraphics[scale =0.5]{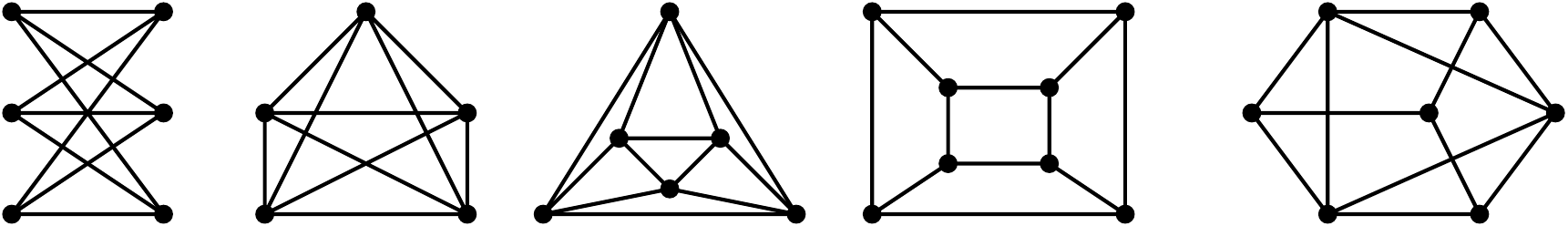}
\caption{Forbidden minors for $3$-connected graphs with vertex width $3$.}
\label{vertexwidth3pic}
\end{center}
\end{figure}

While it takes at least three iterations to start to see the usefulness of compatibility graphs for the first Symanzik polynomial, their usefulness is easy to see for reductions involving both Symanzik polynomials. Take any  $3$-connected graph $G$ with some external momenta and no massive edges such that $\phi_{G} \neq 0$. Consider $\{\psi_{G},\phi_{G}\}$. Let $\sigma$ be any permutation of the variables in $\psi_{G}$ and $\phi_{G}$. Then in $S_{\sigma(1)}$, we have the polynomials $\psi_{G / e},\psi_{G \setminus e}, \phi_{G / e},$ and $ \phi_{G \setminus e}$, for some edge $e$ corresponding to the edge associated with the Schwinger variable $\sigma(1)$. We note that those polynomials do not factor  as $G \setminus e$ and $G /e$ are both at least $2$-connected since $G$ is $3$-connected.
Assume that $\psi_{G / e} \phi_{G \setminus e} - \psi_{G \setminus e}\phi_{G / e}$ does not factor into any of $\psi_{G / e}, \psi_{G \setminus e}, \phi_{G / e},$ and $ \phi_{G \setminus e}$. Let $\psi_{G} = f_{1}$ and $\phi_{G} = f_{2}$. Then the polynomial $\phi_{G /e}$ only is associated with the set $\{1,\infty\}$ and the polynomial $\phi_{G \setminus e}$ is only associated with the set $\{0,2\}$. Thus $\phi_{G /e}$ and $\phi_{G \setminus e}$ are not compatible. So as quickly as the second step of the reduction algorithm we likely see fewer polynomials in the compatibility graph reduction versus the Fubini reduction.

For those who would like to test if sets are reducible,  Panzer in \cite{panzerphdthesis} implemented this algorithm in his program HyperInt (\cite{HyperintArticle}) and it is freely available for download. Additionally, Christian Bogner's MPL can also be used to test if sets are reducible (\cite{MPLarticle}).

\subsubsection{Second Compatibility Graph Construction}

Recall that the point of the reduction algorithms is to find a suitable order of integration for Brown's integration technique. The technique was created initially for Feynman integrals, but in principle it can be applied to a huge class of integrals far more general than just Feynman integrals.  
Up until the previous reduction algorithm, the reduction algorithms could be used to find an order for an arbitrary integral satisfying the correct properties and not necessarily just the Feynman integrals. However, in the last reduction algorithm the idea of intersecting compatibility graphs from the given construction to reduce the number of spurious polynomials may not guarantee that Brown's integration technique will still succeed for integrals not defined from the Symanzik polynomials (\cite{panzerphdthesis}). In \cite{FrancisBig}, the compatability graph reduction given above was proven to maintain the necessary properties for Feynman integrals, but no proof was given for arbitrary integrals. In \cite{panzerphdthesis}, there is some evidence that intersecting compatibility graphs as in the previous algorithm may still satsify the properties we want to use the integration algorithm, but no proof has been given. In \cite{panzerphdthesis}, Panzer gave a different construction which gets around these problems.  Following the notation of \cite{panzerphdthesis} we have:

\begin{definition}
Let $f_{1}$ and $f_{2}$ be polynomials linear in $\alpha$. Then $f_{1} = g_{1}\alpha + h_{1}$ and $f_{2} = g_{2}\alpha + h_{2}$.  Then we  define $[f_{1},0]_{\alpha}=  g_{1}$, $[f_{1},f_{2}]_{\alpha} = g_{1}h_{2} - g_{2}h_{1}$ and  $[f_{1}, \infty]_{\alpha} = h_{1}$ if $h_{1} \neq 0$, otherwise $[f_{1},\infty]_{\alpha} = g_{1}$.
\end{definition}

\begin{figure}
\begin{center}
\includegraphics[scale =0.5]{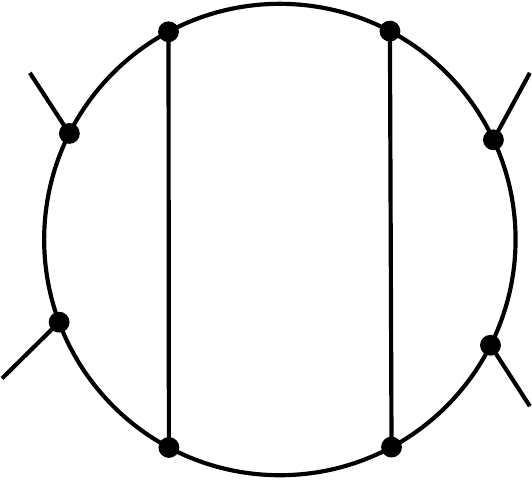}
\caption{A graph which is not Fubini reducible, but is reducible with respect to both Symanzik polynomials. Here all edges are massless, and the external momenta are on-shell. Calculated in \cite{Martin}.}
\label{Martinexample}
\end{center}
\end{figure}

Let $(S,C)$ be a set of polynomials in $\mathbb{Q}[\alpha_{1},\ldots,\alpha_{n}]$ and its corresponding compatibility graph. Suppose all polynomials in $S$ are linear in $\alpha_{i}$. We define $(S_{\alpha_{i}},C_{\alpha_{i}})$ as follows. We let $S_{\alpha_{i}} = \text{irreducible factors of} \;  \{[f,0]_{\alpha_{i}} \; [f,g]_{\alpha_{i}} \; [f,\infty]_{\alpha_{i}} \} $ for $f \in S$ and $f,g \in S$ such that $fg \in E(C)$.

Then we say two polynomials $p,q$ are compatible, or $pq \in E(C_{\alpha_{i}})$,  if there exists $f,g,h \in S \cup \{0,\infty\}$ such that $p \ |\  [f,g]_{\alpha_{i}}$, $q \  | \  [g,h]_{\alpha_{i}}$ and each pair $fg,gh,hf \in E(C)$. Here we consider that $0$ and $\infty$ are compatible with each other and compatible with all other polynomials. 

Applying the above construction for compatibility graphs to the simple reduction algorithm gives an alternative reduction algorithm which in \cite{panzerphdthesis} can be used to find admissible orderings for a large class of integrals. Notice that this construction generates fewer polynomials per iteration, but by not intersecting off compatibility graphs, we possibly end up with extra polynomials relative to the other compatibility graph construction.

\begin{definition}
Let $S$ be a set of polynomials in $\mathbb{Q}[\alpha_{1},\ldots,\alpha_{n}]$. Let $C$ be the complete graph on $|S|$ vertices. If there exists a permutation $\sigma$ of $\{\alpha_{1},\ldots,\alpha_{n}\}$ such that $(S_{(\sigma(1),\ldots,\sigma(i))},C_{(\sigma(1),\ldots,\sigma(i))})$ exist and are linear in $\sigma(i+1)$ for all $i \in \{1,\ldots,n-1\}$, then we will say \emph{$S$ is reducible with respect to Panzer's method}.  
\end{definition}

Aside from proving that this reduction holds for more general integrals, this reduction algorithm was also used to proof a generalization of Theorem \ref{vertexwidth3} to graphs with three not necessarily on-shell external momenta. 

\begin{definition}
Let $G$ be a graph and $e_{1},\ldots,e_{|E(G)|}$ an ordering of the edges of $G$. This ordering induces a sequence of graphs of $G$, $\emptyset \subsetneq G_{1} \subsetneq G_{2} \subsetneq \ldots \subsetneq G_{|E(G)|} = G$ such that $G_{i}$ differs from $G_{i+1}$ by edge $e_{i+1}$. Similarly, the ordering induces a sequence of graphs of $G$, $G = G^{1} \supsetneq G^{2} \supsetneq \ldots \supsetneq \emptyset$ where $G^{i}$ differs from $G^{i+1}$ by exactly one edge for all $i \in \{1,\ldots,|E(G)|-1\}$. Define $A_{i} = V(G_{i}) \cap V(G^{i+1})$. The graph $G$ is \emph{$k$-constructable} if for all $A_{i}$, we have that $|A_{i}| = k$. 
\end{definition}

We note that in the above definition, additional vertices are added to the graph if the edge being added has that vertex as an endpoint. Thus the notion of \textit{adding a vertex} in the construction is well defined.  Observe that if we restrict ourselves to $3$-constructable graphs, they are precisely the graphs with vertex width less than $3$.
 
\begin{theorem}[\cite{panzerphdthesis}]
Let $G$ be $3$-constructable and let $v_{1},v_{2},v_{3}$ be the last three vertices added in the construction. Suppose $v_{1},v_{2},v_{3}$ are the only vertices with external momenta. Suppose there are no massive edges. Then $G$ is reducible with respect to Panzer's method. 
\end{theorem}

While this reduction algorithm is seemingly weaker than the technique given by Brown, in \cite{panzerphdthesis} it is shown to prove most of the known results for reducible graphs. Additionally, a large number of computational results on reduciblity using this variant are covered in \cite{panzerphdthesis} which we do not include here. 

Erik Panzer has implemented this reduction alogrithm as well as the reduction algorithm for Brown's construction in his software package HyperInt, which is freely available (\cite{HyperintArticle}).

\section{Some facts about Reducibility}

Now that we have gone through the reduction algorithms, we are ready to prove some facts about reducibility. First we will prove some reducibility facts for arbitrary sets of polynomials, and then move into reducibility specifically for the Symanzik polynomials, with the main result giving a correct proof that reducibility is graph minor closed. We note that throughout this section, most of these observations have been noted in other papers or theses but not always explicitly written down. 

Additionally, to avoid specifying which type of reducibility all the time, the word reducibility will refer to Brown's compatibility graph algorithm exclusively. That being said, most of the proofs apply to all the algorithms. We start this section by noting that reducibility is well behaved under subsets.

\begin{lemma}
\label{subsetlemma}
Let $S = \{f_1,\ldots,f_n\}$ be a set of polynomials in the polynomial ring $\mathbb{Q}[\alpha_{1},\ldots,\alpha_{r}]$ and let $\sigma$ be a permutation of $\{\alpha_{1},\ldots,\alpha_{r}\}$. Suppose $S$ is reducible with respect to $\sigma$. Then any subset $L \subseteq S$ is reducible with respect to $\sigma$.  
\end{lemma}

\begin{proof}

Consider the sequence of sets and compatibility graphs \[(S,C_{S}), (S_{[\sigma(1)]},C_{[\sigma(1)]}),\ldots,(S_{[\sigma(1),\ldots,\sigma(r-1)]},C_{[\sigma(1),\ldots,\sigma(r-1)]}).\]

 We claim that that for all $i \in \{0,1,\ldots,r-1\}$, the set $L_{[\sigma(1),\ldots,\sigma(i)]} \subseteq S_{[\sigma(1),\ldots,\sigma(i)]}$ and that $C_{L_{[\sigma(1),\ldots,\sigma(i)]}}$ is a subgraph of $C_{S_{[\sigma(1),\ldots,\sigma(i)]}}[L_{[\sigma(1),\ldots,\sigma(i)]}]$. When $i=0$ we consider the sets $(S,C_{S})$ and $(L,C_{L})$. We proceed by induction on $i$.
 
The base case follows trivially. Now consider the set $L_{[\sigma(1),\ldots,\sigma(i-1)]}$ and the compatibility graph $L_{C_{[\sigma(1),\ldots,\sigma(i-1)]}}$. By induction we have $L_{[\sigma(1),\ldots,\sigma(i-1)]} \subseteq S_{[\sigma(1),\ldots,\sigma(i-1)]}$ and that $C_{L_{[\sigma(1),\ldots,\sigma(i-1)]}}$ is a subgraph of $C_{S_{[\sigma(1),\ldots,\sigma(i-1)]}}[L_{[\sigma(1),\ldots,\sigma(i-1)]}]$. We will now apply one step of the reduction algorithm to these sets. 

We claim that for all $1 \leq j \leq i$, we have $ L_{[\sigma(1),\ldots,\hat{\sigma(j)},\ldots,\sigma(i)](\sigma(j))} \subseteq S_{[\sigma(1),\ldots,\hat{\sigma(j)},\ldots,\sigma(i)](\sigma(j))}$ when the set  $S_{[\sigma(1),\ldots,\hat{\sigma(j)},\ldots,\sigma(i)](\sigma(j))}$ exists. We know that $S$ is reducible with respect to $\sigma$, so there exists a $j \in \{1,\ldots,i\}$ such that $S_{[\sigma(1),\ldots,\hat{\sigma(j)},\ldots,\sigma(i)](\sigma(j))}$ exists. Fix any such $j \in \{1,\ldots,i\}$. Then by induction we have 

\[L_{[\sigma(1),\ldots,\hat{\sigma(j)},\ldots,\sigma(i)]} \subseteq S_{[\sigma(1),\ldots,\hat{\sigma(j)},\ldots,\sigma(i)]},\] and $C_{L_{[\sigma(1),\ldots,\hat{\sigma(j)},\ldots,\sigma(i)]}}$ is a subgraph of $C_{S_{[\sigma(1),\ldots,\hat{\sigma(j)},\ldots,\sigma(i)]}}[L_{[\sigma(1),\ldots,\hat{\sigma(j)},\ldots,\sigma(i)]}].$

Note as $S_{[\sigma(1),\ldots,\hat{\sigma(j)},\ldots,\sigma(i)](\sigma(j))}$ exists, all polynomials in $L_{[\sigma(1),\ldots,\hat{\sigma(j)},\ldots,\sigma(i)]}$ are linear in $\sigma(j)$. We will denote the set $L^{4}$ to be the set obtained in the second step of the reduction algorithm when starting with $L_{[\sigma(1),\ldots,\hat{\sigma(j)},\ldots,\sigma(i)]}$, and $S^{4}$ will be the set obtained in the second step of the reduction algorithm applied to $S_{[\sigma(1),\ldots,\hat{\sigma(j)},\ldots,\sigma(i)]}$.  Then $L^{4} \subseteq S^{4}$ as $L_{[\sigma(1),\ldots,\hat{\sigma(j)},\ldots,\sigma(i)]} \subseteq S_{[\sigma(1),\ldots,\hat{\sigma(j)},\ldots,\sigma(i)]}$ and since if two polynomials in $L_{[\sigma(1),\ldots,\hat{\sigma(j)},\ldots,\sigma(i)]}$ are compatible then they are compatible in $ S_{[\sigma(1),\ldots,\hat{\sigma(j)},\ldots,\sigma(i)]}$.

 Let $\tilde{S}$ and $\tilde{L}$ be the sets obtained from step four of the reduction algorithm in for $S_{[\sigma(1),\ldots,\hat{\sigma(j)},\ldots,\sigma(i)]}$ and $L_{[\sigma(1),\ldots,\hat{\sigma(j)},\ldots,\sigma(i)]}$ respectively. Then as $L^{4} \subseteq S^{4}$, we have that  $\tilde{L} \subseteq \tilde{S}$.
 Now consider the compatibility graph constructed $C_{L_{[\sigma(1),\ldots,\hat{\sigma(j)},\ldots,\sigma(i)](\sigma(j))}}$. If $mn \in E(C_{L_{[\sigma(1),\ldots,\hat{\sigma(j)},\ldots,\sigma(i)](\sigma(j))}})$, then there is an $2$-tuple associated to $m$ and a $2$-tuple associated to  $n$ such that their intersection is non-empty. By construction, these $2$-tuples are 
 derived from some polynomials in $L_{[\sigma(1),\ldots,\hat{\sigma(j)},\ldots,\sigma(i)](\sigma(j))}$. But we know that $L_{[\sigma(1),\ldots,\hat{\sigma(j)},\ldots,\sigma(i)](\sigma(j))} \subseteq S_{[\sigma(1),\ldots,\hat{\sigma(j)},\ldots,\sigma(i)](\sigma(j))}$ 
 and thus the polynomials which gave $m$ and $n$ the associated $2$-tuples in $L_{[\sigma(1),\ldots,\hat{\sigma(j)},\ldots,\sigma(i)](\sigma(j))}$ exist in $S_{[\sigma(1),\ldots,\hat{\sigma(j)},\ldots,\sigma(i)](\sigma(j))}$. Thus $mn \in C_{S_{[\sigma(1),\ldots,\hat{\sigma(j)},\ldots,\sigma(i)](\sigma(j))}}$. Therefore we have 
 $C_{L_{[\sigma(1),\ldots,\hat{\sigma(j)},\ldots,\sigma(i)](\sigma(j))}}$ is a subgraph of  
\\ $C_{S_{[\sigma(1),\ldots,\hat{\sigma(j)},\ldots,\sigma(i)](\sigma(j))}}[L_{[\sigma(1),\ldots,\hat{\sigma(j)},\ldots,\sigma(i)](\sigma(j))}].$ Therefore,
 
 \[L_{[\sigma(1),\ldots,\sigma(i)]} = \bigcap_{1 \leq j \leq k} L_{[\sigma(1),\ldots,\hat{\sigma(j)},\ldots,\sigma(i)](\sigma(j))} \subseteq \bigcap_{1 \leq j \leq k} S_{[\sigma(1),\ldots,\hat{\sigma(j)},\ldots,\sigma(i)](\sigma(j))} = S_{[\sigma(1),\ldots,\sigma(i)]},  \]
 
 and $C_{L_{[\sigma(1),\ldots,\sigma(i)]}}$ is a subgraph of $C_{S_{[\sigma(1),\ldots\sigma(i)]}}[L_{[\sigma(1),\ldots,\sigma(i)]}]$, which completes the claim.
\end{proof}

\begin{definition}
Let $S = \{f_1,\ldots,f_n\}$ be a set of polynomials in the polynomial ring  \\ $\mathbb{Q}[\alpha_{1},\ldots,\alpha_{r}]$. We will denote the \emph{set of reducible orders}, $O_{S}$, to be the permutations $\sigma$ of $\{\alpha_{1},\dots,\alpha_{r}\}$ such that $S$ is reducible with respect to $\sigma$. 
\end{definition}

A possibly more useful restatement of Lemma \ref{subsetlemma} is:

\begin{corollary}
Let $S = \{f_1,\ldots,f_n\}$ be a set of polynomials in the polynomial ring \\ $\mathbb{Q}[\alpha_{1},\ldots,\alpha_{r}]$. Let $L_{1},\ldots, L_{n}$ be subsets of $S$. If 
\[ \bigcap_{i =1}^{n} O_{L_{i}}= \emptyset,\] then $S$ is not reducible. 
\end{corollary}

We would like to assume that in our initial set of polynomials, all of the polynomials are all irreducible with respect to $\mathbb{Q}$. Essentially we can, as long as we are able to get past the first step of the reduction algorithm.

\begin{lemma}
\label{irreduciblepolynomials}
Let $S = \{f_{1},\ldots,f_{n}\}$ be a set of polynomials in the polynomial ring $\mathbb{Q}[\alpha_{1},\ldots,\alpha_{r}]$. Let $\sigma$ be a permutation of $\{\alpha_{1},\ldots,\alpha_{r}\}$. Let $S^{I}$ be the set of irreducible factors of $S$ with rational coefficients.  If $S$ is reducible with respect to $\sigma$ then $S^{I}$ is reducible with respect to $\sigma$. Furthermore, if $S^{I}$ is reducible with respect to $\sigma$,  and all polynomials in $S$ are linear in $\sigma(1)$, then $S$ is reducible with respect to $\sigma$. 
\end{lemma}

\begin{proof}

Suppose $S$ is reducible with respect to $\sigma$. Let 
$p \in S$ and $p = p_{1}p_{2}$ where $p_{1},p_{2}$ are polynomials in $\mathbb{Q}[\alpha_{1},\ldots,\alpha_{r}]$. Let $L = S \setminus \{p\} \cup \{p_{1},p_{2}\}$. We claim $L$ is reducible with respect to $\sigma$. Note this completes the first claim as  one can repeatedly apply this fact. We will show that $S_{[\sigma(1)]} = L_{[\sigma(1)]}$ and that their compatibility graphs are isomorphic. We consider two cases.

\textbf{Case 1:} Suppose $\deg(p,\sigma(1)) =0$. Then  $\deg(p_{i},\sigma(1)) =0$ for $i \in \{1,2\}$. We perform the first iteration of the reduction algorithm on $S$ and $L$. Let $L^{4}$ and $S^{4}$ be the sets obtained in step three of the reduction algorithm for $\sigma(1)$ for $L$ and $S$ respectively. As $\deg(p,\sigma(1)) = 0$, we have that $p_{1},p_{2} \in L^{4}$ and $p \in S^{4}$. Notice that the polynomial $g_{i}h_{j} - g_{j}h_{i} = pg_{1}$ when one of the polynomials is $p$. Similarly, when one of the polynomials is $p_{1}$ or $p_{2}$, we get $p_{1}g_{i}$ or $p_{2}g_{i}$. Then after factoring, we have  $\tilde{L} = \tilde{S}$, and so $L_{[\sigma(1)]} = S_{[\sigma(1)]}$. Thus it suffices to show $C_{L_{\sigma(1)}}$ is isomorphic to $C_{S_{\sigma(1)}}$.

As $L = S \setminus \{p\} \cup \{p_{1},p_{2}\}$, the graphs $C_{L_{[\sigma(1)]}}$ and $C_{S_{[\sigma(1)]}}$ have the same vertex set, and we may restrict our attention to compatibilities caused by $p$, $p_{1}$ and $p_{2}$. Notice the irreducible factors of $p$, $p_{1}$ and $p_{2}$ all contain an associated $2$-tuple containing $\infty$ and $0$, thus we may restrict our attention to $2$-tuples not containing $\infty$ or $0$.

Suppose an irreducible factor of $p$, say $p'$, is adjacent to a vertex $g'$ in $S_{C_{\sigma(1)}}$, where $g'$ is an irreducible factor of some $g_{i}$. Since we are assuming that the compatibility did not arise from a $2$-tuple containing $\{0\}$ or $\{\infty\}$ we may assume that compatibilities comes from the polynomial $pg_{j}$. Then without loss of generality, $p'$ is an irreducible factor of  $p_{1}$ and so $p_{1}g_{j}$ would provide the desired index. The same argument works in the other direction as well, so we have that $C_{L_{[\sigma(1)]}}$ is isomorphic to $C_{S_{[\sigma(1)]}}$. As $S_{[\sigma(1)]} = L_{[\sigma(1)]}$ and their compatibility graphs are the same, since $S$ is reducible, $L$ is reducible.

\textbf{Case 2:} Suppose $\deg(p,\sigma(1)) =1$. Notice that as $p$ is linear in $\sigma(1)$, exactly one of $p_{1}$ or $p_{2}$ is linear in $\sigma(1)$. Without loss of generality, we assume $\deg(p_{1},\sigma(1)) =1$.
 
 As before let $S^{4}, \tilde{S}$ and $L_{4},\tilde{L}$ denote the sets obtained from the third and fourth step of the reduction for $S$ 
 and $L$ respectively. Let $p_{1} = g_{1}\sigma(1) +h_{1}$. Then $p = p_2(g_{1}\sigma(1) + h_{1})$. Then  $p_{2},h_{1},g_{1} \in L^{4}$ , and $p_{2}g_{1}, p_{2}h_{1} \in S^{4}$.  Notice that the polynomial $g_{i}h_{j} - h_{i}g_{j} = p_2(g_{1}h_{j}-h_{1}g_{j})$ when one of the polynomials is $p$. Thus  $S_{[\sigma(1)]}$ will contain all the irreducible factors obtained from $p_{1}$ and $p_{2}$. Similarly, $L_{[\sigma(1)]}$ will contain all the irreducible factors obtained from $p$, so $S_{[\sigma(1)]} = L_{[\sigma(1)]}$. 
 
For the compatibility graphs, as before all irreducible factors of $p$ are adjacent in $C_{S_{[\sigma(1)]}}$ as they share an $2$-tuple from $p$. In $C_{L_{[\sigma(1)]}}$, all irreducible factors of either $p_{1}$ or $p_{2}$ are adjacent. This follows since $\deg(p_{2},\sigma(1))=0$, all irreducible factors of $p_{2}$ have a $2$-tuple with $0$ and a $2$-tuple with $\infty$, and every irreducible factor of $p_{1}$ has a $2$-tuple containing either $0$ or $\infty$. If $f$ and $g$ are adjacent in $C_{S_{[\sigma(1)]}}$ and the $2$-tuples which make them adjacent came from a polynomial $p_{2}(g_{1}h_{j}-h_{1}g_{j})$, then the desired $2$-tuples exist for $C_{L_{[\sigma(1)]}}$ as the polynomials $p_{2}g_{1} \in L^{4}$ and $g_{1}h_{j}-h_{1}g_{j} \in L^{4}$. A similar statement holds for two adjacent polynomials in $C_{L_{[\sigma(1)]}}$. Therefore one can see that $C_{L_{[\sigma(1)]}}$ is isomorphic to $C_{S_{[\sigma(1)]}}$, and thus as $S$ is reducible with respect to $\sigma$, $L$ is reducible with respect to $\sigma$. Notice for the partial converse, the same argument works, as the only obstruction is if the reduction algorithm stops immediately.
\end{proof}

As a remark, in the above proof we did not rely on the fact that the initial compatibility graph is complete. As we will see later in the chapter, it can be convenient to apply the lemma in the middle of the a reduction.

\begin{lemma}
\label{independantpolynomials}
Let $S= \{f_{1},\ldots,f_{n}\}$ be a set of irreducible polynomials in $\mathbb{Q}[\alpha_{1},\ldots,\alpha_{r}]$. Let $\sigma$ be a permutation of $\{\alpha_{1},\ldots,\alpha_{r}\}$. Suppose we have sets $L' \subseteq S$ and $L'' \subseteq S$ such that $L' \cap L'' = \emptyset$, $L' \cup L'' = S$, and up to relabeling, there exists an $l \in \{1,\ldots,r-1\}$ such that  all polynomials in $L'$ are in the polynomial ring $\mathbb{Q}[\alpha_{1},\ldots,\alpha_{l}]$ and all polynomials in $L''$ are in the polynomial ring $\mathbb{Q}[\alpha_{l+1},\ldots,\alpha_{r}]$. Then $S$ is reducible if and only if both of the sets $L'$ and $L''$ are reducible.
\end{lemma}

\begin{proof}
If $S$ is reducible with respect to $\sigma$, then by Lemma \ref{subsetlemma} both $L'$ and $L''$ are reducible with respect to $\sigma$.

Now suppose $L'$ and $L''$ are reducible with respect to $\sigma_{1}$ and $\sigma_{2}$ respectively. Let $\sigma = \sigma_{1}\sigma_{2}$ represent the concatenation of $\sigma_{1}$ and $\sigma_{2}$. We claim $S$ is reducible with respect to $\sigma$. It suffices to show that $S_{[\sigma(1),\ldots, \sigma(l)]} = L''$, and that $C_{[\sigma(1),\ldots,\sigma(l)]} = K_{|L''|}$. By induction, we assume at the $k^{th}$ step, $k \leq l$, we have $S_{[\sigma(1),\ldots,\sigma(k-1)]} = \{L'_{[\sigma(1),\ldots,\sigma(k-1)]}, L''\}$ and that the compatibility graph is the compatibility graph for $L'_{[\sigma(1),\ldots,\sigma(k-1)]}$ with every vertex of $L''$ being a dominating vertex. We note that the base case calculation is the same as the inductive step, and thus we ignore the base case. Then we want to show that \[S_{[\sigma(1),\ldots,\sigma(k)]} = \bigcap_{1 \leq i \leq k+1} S_{[\sigma(1),\ldots,\hat{\sigma(i)},\ldots,\sigma(k)](\sigma(i))} = \{L'_{[\sigma(1),\ldots,\sigma(k-1)]}, L''\}.\]

So it suffices to show that for any $i$, $1 \leq i \leq k+1$, that 

\[S_{[\sigma(1),\ldots,\hat{\sigma(i)},\ldots,\sigma(k)](\sigma(i))} = \{L'_{[\sigma(1),\ldots,\hat{\sigma(i)},\ldots,\sigma(k)](\sigma(i))}, L'' \}.\]

 Fix $i \in \{1,\ldots,k\}$. Notice that $S_{[\sigma(1),\ldots,\hat{\sigma(i)},\ldots,\sigma(k)](\sigma(i))}$ exists if and only if \\ $L'_{[\sigma(1),\ldots,\hat{\sigma(i)},\ldots,\sigma(k)](\sigma(i))}$ exists, and thus we may assume $L'_{[\sigma(1),\ldots,\hat{\sigma(i)},\ldots,\sigma(k)](\sigma(i))}$ exists.

 Let $p \in L'_{[\sigma(1),\ldots,\hat{\sigma(i)},\ldots,\sigma(k)](\sigma(i))}$. Then $p$ is the irreducible factor for some $p'$ where either $p' = g_{m}$, $p'= h_{m}$ or $p' = g_{m}h_{j} - h_{m}g_{j}$ for some $f_{m}, f_{j} \in L'_{[\sigma(1),\ldots,\sigma(k)]}$. By our induction hypothesis, $f_{m}, f_{j} \in  S_{[\sigma(1),\ldots,\sigma(k)]}$, and thus $p \in S_{[\sigma(1),\ldots,\hat{\sigma(i)},\ldots,\sigma(k)](\sigma(i))}'$. Also, all polynomials from $L''$ are in $S_{[\sigma(1),\ldots,\hat{\sigma(i)},\ldots,\sigma(k)](\sigma(i))}$ since all polynomials in $L''$ are irreducible and do not contain $\sigma(i)$ as a variable.

Therefore, it suffices to show that in the compatibility graph for $S_{[\sigma(1),\ldots,\hat{\sigma(i)},\ldots,\sigma(k)](\sigma(i))}$, all polynomials in $L''$ are dominating vertices. Let $f \in L''$. By assumption $\deg(f,\alpha(i)) = 0$ so $f$ has receives a $2$-tuple containing $0$ and a $2$-tuple containing $\infty$. Consider any polynomial  $f_{z} \in S_{[\sigma(1),\ldots,\hat{\sigma(i)},\ldots,\sigma(k)](\sigma(i))}$. Then the polynomial $f_{z}f \in S^{4}$ gives $f$ a $2$-tuple containing $z$. Therefore for every possible value one of the associated $2$-tuples can have, $f$ has a $2$-tuple containing that value. Thus $f$ is adjacent to every polynomial, completing the claim.
\end{proof}

Now we will prove an important relationship between reducibility, leading coefficients and setting variables to zero, which is going to give us our desired graph minor closed property.

\begin{theorem}[\cite{Bognerpaper}]
\label{FubiniReducibility}
Let $S = \{P_{1},\ldots,P_{N}\}$ be a set of polynomials which is Fubini reducible in the order $(\alpha_{1},\ldots, \alpha_{n})$. Fix some $l \in \{1,\ldots,n\}$  and suppose that each polynomial $P_{i} \in S$ is linear in $\alpha_{l}$ for all $1 \leq i \leq N$. Then the sets $S^{l} = \{\frac{\partial P_{1}}{\partial\alpha_{l}},\frac{\partial P_{2}}{\partial\alpha_{l}},\ldots, \frac{\partial P_{n}}{\partial\alpha_{l}}\}$ and $S_{l} = \{P_{1}|_{\alpha_{l} =0}, \ldots, P_{N}|_{\alpha_{l}=0}\}$ are Fubini reducible in the order $(\alpha_{1}, \ldots, \hat{\alpha_{l}}, \ldots, \alpha_{n})$.
\end{theorem}

The proof given in \cite{Bognerpaper} for Theorem \ref{FubiniReducibility} is flawed as printed,  however it appears that their proof can be fixed with some effort, but we do not go through the details (the suggested fix pointed was out by Christian Bogner via email correspondence (\cite{BognerCommunication})). Instead, we prove essentially the same theorem for reducibility (split into two theorems as the arguments are slightly different).

\begin{theorem}
\label{reducibilityevalutation}
Let $S = \{P_{1},\ldots,P_{N}\}$ be a set of polynomials which is reducible in the order $(\alpha_{1},\ldots, \alpha_{n})$. Fix some $l \in \{1,\ldots,n\}$. Then the set $S^{l} = \{P_{1}|_{\alpha_{l} =0}, \ldots, P_{N}|_{\alpha_{l}=0}\}$ is reducible in the order $(\alpha_{1},\ldots, \alpha_{n})$.
\end{theorem}

\begin{proof}
Let $S$ and $S^{l}$ be a counterexample with $n$ minimized, as in, $S$ is reducible with order $(\alpha_{1},\ldots,\alpha_{n})$, but $S^{l}$ is not reducible with order $(\alpha_{1},\ldots,\alpha_{n})$. 

First suppose that $l= 1$. Then for all polynomials $P \in S$, we have that $\deg(P,\alpha_{l}) \leq 1$. Then notice that $S^{l} = S^{2}$ where $S^{2}$ is the set obtained by applying the reduction algorithm to $S$ for $\alpha_{l}$. Then $S^{l}_{[\alpha_{l}]} \subseteq S_{[\alpha_{1}]}$. Then by Lemma \ref{subsetlemma}, we have that $S^{l}_{[\alpha_{l}]}$ is reducible with order $(\alpha_{2},\ldots,\alpha_{n})$, which implies that $S^{l}$ is reducible with order $(\alpha_{1},\ldots,\alpha_{n})$, a contradiction.

 Therefore we assume that $l \neq 1$ and consider one step of the reduction algorithm. As $S$ is reducible with order $(\alpha_{1},\ldots,\alpha_{n})$, we have that $S_{[\alpha_{1}]}$ is reducible with order $(\alpha_{2},\ldots,\alpha_{n})$. Consider the set $S^{l'}_{[\alpha_{1}]} = \{f|_{\alpha_{l}=0} | f \in S_{[\alpha_{1}]} \}$.  Then since $S$ is a minimal counterexample with respect to $n$, we have that $S^{l'}_{[\alpha_{1}]}$ is reducible with order $(\alpha_{2},\ldots,\alpha_{n})$. Additionally, by Lemma \ref{irreduciblepolynomials}, the set of irreducible polynomials of  $S^{l'}_{[\alpha_{1}]}$  is reducible with order $(\alpha_{2},\ldots,\alpha_{n})$. Notice from the definitions we have:
\begin{align*}
S^{l'}_{[\alpha_{1}]} &= \{f|_{\alpha_{l}=0} | f \in S_{[\alpha_{1}]} \} 
\\&=
\{f|_{\alpha_{l}=0} | f \in \text{irreducible factors of } S^{4} \}.
\end{align*}

 For notational convenience, we will say $S^{jl}$ will be the set $S^{j}$ obtained by starting with $S^{l}$ for $j \in \{1,2,3,4\}$. Suppose we have a polynomial $f$, and $f = f_{1}f_{2}$ for some polynomials $f_{1}$ and $f_{2}$. Then fix any variable $\alpha$ and notice that  $f|_{\alpha =0} = f_{1}|_{\alpha = 0}f_{2}|_{\alpha =0}$. This follows since any term which contains $\alpha$ in $f$ is generated by a pair of terms in $f_{1}$ and $f_{2}$, where at least one of these terms contains $\alpha$. Now, notice that $S^{l}_{[\alpha_{1}]}$ exists and,

\begin{align*}
S^{l}_{[\alpha_{1}]} &= \text{irreducible factors of } S^{4l}
 \\ &=
\text{irreducible factors of } (S^{1l} \cup S^{2l} \cup S^{3l}) 
\\ &=
\text{irreducible factors of } (\{ \frac{\partial f}{\partial \alpha_{1}} | f \in S^{l} \} \cup    \{f|_{\alpha_{1}=0} | f \in S^{l} \} \cup \\& \qquad \{\frac{\partial f_{1}}{\partial \alpha_{1}}f_{2}|_{\alpha_{1} = 0} - \frac{\partial f_{2}}{\partial \alpha_{1}}f_{1}|_{\alpha_{1}=0} | f_{1},f_{2} \in S^{l} \}) 
\\&=
\text{irreducible factors of } (\{ \frac{\partial f|_{\alpha_{l}=0}}{\partial \alpha_{1}}  | f \in S\} \cup    \{f|_{\alpha_{l}, \alpha_{1}=0} | f \in S\} \cup \\& \qquad \{\frac{\partial f_{1}|_{\alpha_{l}=0}}{\partial \alpha_{1}}f_{2}|_{\alpha_{l},\alpha_{1} = 0} - \frac{\partial f_{2}|_{\alpha_{l}=0}}{\partial \alpha_{1}}f_{1}|_{\alpha_{l},\alpha_{1}=0} | f_{1},f_{2} \in S \}) 
\\ &=
\text{irreducible factors of } (\{ \frac{\partial f}{\partial \alpha_{1}}|_{\alpha_{l}=0}  | f \in S\} \cup    \{f|_{\alpha_{l}, \alpha_{1}=0} | f \in S\} \cup \\& \qquad \{\frac{\partial f_{1}}{\partial \alpha_{1}}|_{\alpha_{l}=0}f_{2}|_{\alpha_{l},\alpha_{1} = 0} - \frac{\partial f_{2}}{\partial \alpha_{1}}|_{\alpha_{l}=0}f_{1}|_{\alpha_{l},\alpha_{1}=0} | f_{1},f_{2} \in S \})
\\&= \text{irreducible factors of } (\{ \frac{\partial f}{\partial \alpha_{1}}|_{\alpha_{l}=0}  | f \in S\} \cup    \{f|_{\alpha_{l}, \alpha_{1}=0} | f \in S\} \cup \\& \qquad \{(\frac{\partial f_{1}}{\partial \alpha_{1}}f_{2}|_{\alpha_{1} = 0} - \frac{\partial f_{2}}{\partial \alpha_{1}}f_{1}|_{\alpha_{1}=0})|_{\alpha_{l} =0} | f_{1},f_{2} \in S \}) \\
& \subseteq  \text{irreducible factors of } \{f|_{\alpha_{l}=0} | f \in \text{irreducible factors of } S^{4} \} \\
&= \text{irreducible factors of } S^{l'}_{[\alpha_{1}]}.
\end{align*}

By previous observations, the set of irreducible factors of $S^{l'}_{[\alpha_{1}]}$ is reducible with order $(\alpha_{2},\ldots,\alpha_{n})$, so by Lemma \ref{subsetlemma}, we get that  $S^{l}_{[\alpha_{1}]}$ is reducible with order $(\alpha_{2},\ldots,\alpha_{n})$, completing the proof. 
\end{proof}

\begin{lemma}
\label{equationgrind}
Let $f_{1}$ and $f_{2}$ be polynomials linear in $\alpha_{1}$. Fix a variable $\alpha_{l}$ such that $\alpha_{l} \neq \alpha_{1}$. For a polynomial $f$, let $lc(f)$ be the leading coefficient of $f$ with respect to $\alpha_{l}$.  Then either
\begin{enumerate}
 \item{$\frac{\partial lc(f_{1})}{\partial \alpha_{1}}lc(f_{2})|_{\alpha_{1} = 0} - \frac{\partial lc(f_{2})}{\partial \alpha_{1}}lc(f_{1})|_{\alpha_{1}=0} = 0$ or}
 \item{ $\frac{\partial lc(f_{1})}{\partial \alpha_{1}}lc(f_{2})|_{\alpha_{1} = 0} - \frac{\partial lc(f_{2})}{\partial \alpha_{1}}lc(f_{1})|_{\alpha_{1}=0} = lc(\frac{\partial f_{1}}{\partial \alpha_{1}} f_{2}|_{\alpha_{1}=0} - \frac{\partial f_[2}{\partial \alpha_{1}} f_{1}|_{\alpha_{1} = 0}).$}
\end{enumerate}
\end{lemma}

\begin{proof}
Let $f_{1} = \sum_{i=0}^{k} (f_{1,i,1}\alpha_{1} + f_{1,i,2})\alpha^{i}_{l}$ and $f_{2} = \sum_{j=0}^{t} (f_{2,j,1}\alpha_{1} + f_{2,j,2})\alpha^{j}_{l}$. Then 
\begin{align*}
\frac{\partial lc(f_{1})}{\partial \alpha_{1}}lc(f_{2})|_{\alpha_{1} = 0} - \frac{\partial lc(f_{2})}{\partial \alpha_{1}}lc(f_{1})|_{\alpha_{1}=0} = f_{1,k,1}f_{2,t,2} - f_{2,t,1}f_{1,k,2},
\end{align*}
and 
\begin{align*}
lc(\frac{\partial f_{1}}{\partial \alpha_{1}} f_{2}|_{\alpha_{1}=0} - \frac{\partial f_{2}}{\partial \alpha_{1}} f_{1}|_{\alpha_{1} = 0}) =
lc((\sum_{i=0}^{k}f_{1,i,1}\alpha_{l}^{i})(\sum_{j=0}^{t}f_{2,j,2}\alpha_{l}^{j}) - (\sum_{j=0}^{t} f_{2,j,1}\alpha_{l}^{j})(\sum_{i=0}^{k}f_{1,i,2}\alpha_{l}^{i})).
\end{align*}

Now we consider various cases. If $f_{1,k,1}f_{2,t,2} - f_{2,t,1}f_{1,k,2} =0$, then the claim immediately holds. Therefore we assume that $f_{1,k,1}f_{2,t,2} - f_{2,t,1}f_{1,k,2} \neq 0$. Suppose  each of $f_{1,k,1},f_{2,t,2},f_{2,t,1}$ and $f_{1,k,2}$ are non-zero, then the largest power of $\alpha_{l}$ in $(\sum_{i=0}^{k}f_{1,i,1}\alpha_{l}^{i})(\sum_{j=0}^{t}f_{2,j,2}\alpha_{l}^{j}) - (\sum_{j=0}^{t} f_{2,j,1}\alpha_{l}^{j})(\sum_{i=0}^{k}f_{1,i,2}\alpha_{l}^{i})$ is $\alpha_{l}^{k+t}$, and thus 
\[lc((\sum_{i=0}^{k}f_{1,i,1}\alpha_{l}^{i})(\sum_{j=0}^{t}f_{2,j,2}\alpha_{l}^{j}) - (\sum_{j=0}^{t} f_{2,j,1}\alpha_{l}^{j})(\sum_{i=0}^{k}f_{1,i,2}\alpha_{l}^{i})) = f_{1,k,1}f_{2,t,2} - f_{2,t,1}f_{1,k,2}.\] 

Now suppose that $f_{1,k,1} =0$ and $f_{2,t,2}, f_{2,t,1},$ and $f_{1,k,2}$ are not zero. Then the highest power of $\alpha_{l}$ in $(\sum_{i=0}^{k}f_{1,i,1}\alpha_{l}^{i})(\sum_{j=0}^{t}f_{2,j,2}\alpha_{l}^{j}) - (\sum_{j=0}^{t} f_{2,j,1}\alpha_{l}^{j})(\sum_{i=0}^{k}f_{1,i,2}\alpha_{l}^{i})$ is $\alpha_{l}^{k+t}$ which only occurs in  $(\sum_{j=0}^{t} f_{2,j,1}\alpha_{l}^{j})(\sum_{i=0}^{k}f_{1,i,2}\alpha_{l}^{i})$ so,

\[lc((\sum_{i=0}^{k}f_{1,i,1}\alpha_{l}^{i})(\sum_{j=0}^{t}f_{2,j,2}\alpha_{l}^{j}) - (\sum_{j=0}^{t} f_{2,j,1}\alpha_{l}^{j})(\sum_{i=0}^{k}f_{1,i,2}\alpha_{l}^{i})) =  -f_{2,t,1}f_{1,k,2}.\]

Now suppose that $f_{1,k,2} = 0$ and $f_{2,t,2}, f_{2,t,1},$ and $f_{1,k,2}$ are not zero. Then the highest power of $\alpha_{l}$ in $(\sum_{i=0}^{k}f_{1,i,1}\alpha_{l}^{i})(\sum_{j=0}^{t}f_{2,j,2}\alpha_{l}^{j}) - (\sum_{j=0}^{t} f_{2,j,1}\alpha_{l}^{j})(\sum_{i=0}^{k}f_{1,i,2}\alpha_{l}^{i})$ is $\alpha_{l}^{k+t}$ so, 

\[lc((\sum_{i=0}^{k}f_{1,i,1}\alpha_{l}^{i})(\sum_{j=0}^{t}f_{2,j,2}\alpha_{l}^{j}) - (\sum_{j=0}^{t} f_{2,j,1}\alpha_{l}^{j})(\sum_{i=0}^{k}f_{1,i,2}\alpha_{l}^{i})) =  f_{1,k,1}f_{2,t,2}.\]

The case where $f_{2,t,1} =0$ and $f_{2,t,2}, f_{1,k,1},$ and $f_{1,k,2}$ are non-zero and the case where $f_{2,t,2} = 0$ and $f_{2,t,1}$, $f_{1,k,1},f_{1,k,2}$ are non-zero follow similarly. 

Now consider the case where $f_{1,k,1} = 0$, $f_{2,t,1} =0$, and $f_{1,k,2}$, $ f_{2,t,1}$ are not zero. Then $f_{1,k,1}f_{2,t,2} - f_{2,t,1}f_{1,k,2} =0$,  satisfying the claim.

Now consider the case where $f_{1,k,2} =0$,$f_{2,t,2} =0$, and $f_{1,k,1}$,$f_{2,t,1}$ are not zero. Then $f_{1,k,1}f_{2,t,2} - f_{2,t,1}f_{1,k,2} =0$, satisfying the claim. 

Now consider the case where $f_{1,k,1} =0$, $f_{2,t,2} =0$ and $f_{1,k,2}$, $f_{2,t,1}$ are non-zero. Then the highest power of $\alpha_{l}$ in $(\sum_{i=0}^{k}f_{1,i,1}\alpha_{l}^{i})(\sum_{j=0}^{t}f_{2,j,2}\alpha_{l}^{j}) - (\sum_{j=0}^{t} f_{2,j,1}\alpha_{l}^{j})(\sum_{i=0}^{k}f_{1,i,2}\alpha_{l}^{i})$ is $\alpha_{l}^{k+t}$ so, 

\[lc((\sum_{i=0}^{k}f_{1,i,1}\alpha_{l}^{i})(\sum_{j=0}^{t}f_{2,j,2}\alpha_{l}^{j}) - (\sum_{j=0}^{t} f_{2,j,1}\alpha_{l}^{j})(\sum_{i=0}^{k}f_{1,i,2}\alpha_{l}^{i})) =  -f_{2,t,1}f_{1,k,2}.\]

The case where $f_{1,k,2} = 0$, $f_{2,t,1}=0$ and $f_{1,k,1}$,$f_{2,t,1}$ are non-zero follows similarly. Notice that these are all possible cases, so the claim holds. 
\end{proof}

\begin{theorem}
\label{leadingterms}
Let $S = \{P_{1},\ldots,P_{N}\}$ be a set of polynomials which is reducible in the order $(\alpha_{1},\ldots,\alpha_{n})$. Fix $l \in \{1,\ldots,n\}$. For all $i \in \{1,\ldots,N\}$,  Let $S^{l} = \{lc(P)| P \in S\}$. Then $S^{l}$ is reducible with the order $(\alpha_{1},\ldots, \alpha_{n})$. 
\end{theorem}

\begin{proof}
We proceed by induction on $n$. If $n=1$, then all polynomials in $S$ are monomials, and the result is immediate. Therefore we assume $n > 1$. 

First we suppose that $l =1$. Then for all $P \in S$, we have $\deg(P, \alpha_{l}) \leq 1$. Then $S^{l} = S^{1}$ where $S^{1}$ is the set obtained in the reduction algorithm applied to $S$ and $\alpha_{l}$. Then $S^{l}_{[\alpha_{l}]} \subseteq S_{[\alpha_{l}]}$. Since $S_{[\alpha_{l}]}$ is reducible with order $(\alpha_{2},\ldots,\alpha_{n})$, by Lemma \ref{subsetlemma} we have that $S^{l}_{[\alpha_{l}]}$ is reducible with order $(\alpha_{2},\ldots,\alpha_{n})$, and thus $S^{l}$ is reducible with order $(\alpha_{1},\ldots,\alpha_{n})$.

Therefore we can assume that $l \neq 1$. Now consider the set $(S_{[\alpha_{1}]})^{l}$, which we define to be the set of leading coefficients of polynomials in $S_{[\alpha_{1}]}$ with respect to $\alpha_{l}$. Since $S$ is reducible with order $(\alpha_{1},\ldots,\alpha_{n})$, we have that $S_{[\alpha_{1}]}$ is reducible with order $(\alpha_{2},\ldots,\alpha_{n})$ and thus by induction, $(S_{[\alpha_{1}]})^{l}$ is reducible with order $(\alpha_{2},\ldots, \alpha_{n})$.  Notice from the definitions we have:
\begin{align*}
(S_{[\alpha_{1}]})^{l} &= \{lc(f) | f \in S_{[\alpha_{1}]} \} 
\\&=
\{lc(f) | f \in \text{irreducible factors of } S^{4} \}.
\end{align*}

 For notational convenience, we will say $S^{jl}$ will be the set $S^{j}$ obtained by starting with $S^{l}$ for $j \in \{1,2,3,4\}$. Suppose we have a polynomial $f$, and $f = f_{1}f_{2}$ for some polynomials $f_{1}$ and $f_{2}$. Now, notice that $S^{l}_{[\alpha_{1}]}$ exists and,
\allowdisplaybreaks
\begin{align*}
S^{l}_{[\alpha_{1}]} 
&= \text{irreducible factors of } S^{4l}
 \\ &=
\text{irreducible factors of } (S^{1l} \cup S^{2l} \cup S^{3l}) 
\\ &=
\text{irreducible factors of } (\{ \frac{\partial f}{\partial \alpha_{1}} | f \in S^{l} \} \cup    \{f|_{\alpha_{1}=0} | f \in S^{l} \} \cup \\& \qquad \{\frac{\partial f_{1}}{\partial \alpha_{1}}f_{2}|_{\alpha_{1} = 0} - \frac{\partial f_{2}}{\partial \alpha_{1}}f_{1}|_{\alpha_{1}=0} | f_{1},f_{2} \in S^{l} \}) 
\\&=
\text{irreducible factors of } (\{ \frac{\partial lc(f)}{\partial \alpha_{1}}  | f \in S\} \cup    \{lc(f)|_{\alpha_{1}=0} | f \in S\} \cup \\& \qquad \{\frac{\partial lc(f_{1})}{\partial \alpha_{1}}lc(f_{2})|_{\alpha_{1} = 0} - \frac{\partial lc(f_{2})}{\partial \alpha_{1}}lc(f_{1})|_{\alpha_{1}=0} | f_{1},f_{2} \in S \}) 
\\& \subseteq \text{irreducible factors of } (\{ lc(\frac{\partial f}{\partial \alpha_{1}})  | f \in S\} \cup    \{lc(f|_{\alpha_{1}=0}) | f \in S\} \cup \\& \qquad \{lc(\frac{\partial f_{1}}{\partial \alpha_{1}}f_{2}|_{\alpha_{1} = 0} - \frac{\partial f_{2}}{\partial \alpha_{1}}f_{1}|_{\alpha_{1}=0}) | f_{1},f_{2} \in S \}) \\
&=  \text{irreducible factors of } \{lc(f) | f \in \text{irreducible factors of } S^{4} \}  \\
&= \text{irreducible factors of } (S_{[\alpha_{1}]})^{l}
\end{align*}

We note that the subset relationship between lines $4$ and $5$ above holds by appealing to Lemma \ref{equationgrind} and assuming that if $\frac{\partial lc(f_{1})}{\partial \alpha_{1}}lc(f_{2})|_{\alpha_{1} = 0} - \frac{\partial lc(f_{2})}{\partial \alpha_{1}}lc(f_{1})|_{\alpha_{1}=0} = 0$, then we remove the $0$ from the set. This does not affect reducibility since $0$ is a constant (see Remark \ref{monomialsandconstants}).

Then since $(S_{[\alpha_{1}]})^{l}$ is reducible with order $(\alpha_{2},\ldots,\alpha_{n})$, the set of irreducible factors of $(S_{[\alpha_{1}]})^{l}$ is reducible with order $(\alpha_{2},\ldots,\alpha_{n})$ by Lemma \ref{irreduciblepolynomials}. Then by Lemma \ref{subsetlemma},  we have that $S^{l}_{[\alpha_{1}]}$ is reducible with order $(\alpha_{2},\ldots,\alpha_{n})$, completing the claim. 
\end{proof}

\begin{corollary}
\label{MinorClosed}
Let $r$ be a fixed positive integer. Let $\mathcal{G}$ be the set of graphs which have $r$ external momenta, no massive edges, and are reducible with respect to $\{\phi,\psi\}$. Then if $G \in \mathcal{G}$, all minors of $G$ are in $\mathcal{G}$.
\end{corollary}

\begin{proof}
Let $G \in \mathcal{G}$. Pick any edge $e \in E(G)$. We consider two cases.

 \textbf{Case 1:} Suppose $e$ is a loop. Notice for loops, deletion and contraction are the same. Then by Lemma \ref{loopedgesSymanzikPolys}, we have $\{\phi_{G}, \psi_{G}\} = \{\phi_{G \setminus e}\alpha_{e}, \psi_{G \setminus e}\alpha_{e}\}$. By Lemma \ref{irreduciblepolynomials}, $\{\phi_{G \setminus e}\alpha_{e}, \psi_{G \setminus e}\alpha_{e}\}$ is reducible if and 
only if $\{\phi_{G \setminus e}, \psi_{G \setminus e}, \alpha_{e}\}$ is reducible. Since $\alpha_{e}$ is a monomial, we may remove it and thus we get that $\{\phi_{G}, \psi_{G}\}$ is reducible if and only if $\{\phi_{G \setminus e}, \psi_{G \setminus e}\}$ is reducible. Therefore $G \setminus e \in \mathcal{G}$ and thus by induction, all minors of $G \setminus e$ are in $\mathcal{G}$. 

\textbf{Case 2:} Suppose that $e$ is not a loop edge and consider $G \setminus e$ and $G / e$. By Lemma \ref{derivativeandevaluationsetrelationship} the Symanzik polynomials for $G \setminus e$ are $\{ \frac{\partial}{\partial\alpha_{e}}\phi_{G}, \frac{\partial}{\partial\alpha_{e}}\psi_{G} \}$. Similarly by Lemma \ref{derivativeandevaluationsetrelationship} the Symanzik polynomials for $G /e$ are $\{\psi_{G|_{\alpha_{e} = 0}}, \phi_{G|_{\alpha_{e}=0}}\}$. Notice that since we have no massive edges, both of the Symanzik polynomials are linear in $\alpha_{l}$, thus $\frac{\partial}{\partial \alpha_{e}} \phi_{G} = lc(\phi_{G})$ and $\frac{\partial}{\partial \alpha_{e}} \phi_{G} = lc(\phi_{G})$. Thus both sets  are reducible by appealing to Theorem \ref{leadingterms} and Theorem \ref{reducibilityevalutation} and thus in $\mathcal{G}$. Therefore by induction, all minors of $G$ are in $\mathcal{G}$, completing the claim.   
\end{proof}

Therefore we have that reducibility with respect to the Symanzik polynomials when we have a fixed set of external momenta is graph minor closed. This fact motivates chapter three. Notice that by similar arguments, reducibility with respect to the first (or second) Symanzik polynomial is also graph minor closed. We note it is interesting that in the above proof, that leading coefficients were the key to prove reducibility is closed under contraction, whereas in the literature contraction is usually viewed as taking a derivative. Now we can give some easy connectivity reductions for reduciblity of graphs with respect to the Symanzik polynomials. 

\begin{lemma}
Let $G$ be a graph with $1$-separation $(A,B)$. Then $G$ is reducible with respect to the first Symanzik polynomial if and only if $G[A]$ and $G[B]$ are reducible with respect to the first Symanzik polynomial.
\end{lemma}

\begin{proof}
By Lemma \ref{firstsymanzikoneconnected} we have $\psi_{G} = \psi_{G[A]}\psi_{G[B]}$. So the set $\{\psi_{G}\}$ is reducible if and only if $\{\psi_{G[A]},\psi_{G[B]}\}$ is reducible by Lemma \ref{irreduciblepolynomials}. Observe that there are no shared variables between $\psi_{G[A]}$ and $\psi_{G[B]}$. Then by Lemma \ref{independantpolynomials}, $\{\psi_{G[A]},\psi_{G[B]}\}$ is reducible if and only if both $\{\psi_{G[A]}\}$ and $\{\psi_{G[B]}\}$ are reducible. By definition, that occurs exactly when both $G[A]$ and $G[B]$ are reducible with respect to the first Symanzik polynomial. 
\end{proof}

\begin{lemma}
Let $G$ be a graph with where $X = \{a,b,c,d\} \subseteq V(G)$ is the set of vertices with momenta, and assume that all external momenta is on-shell. Furthermore, suppose $G$ has no massive edges.  Suppose $G$ has a $1$-separation $(A,B)$. Suppose $a,b,c \in  A \setminus (A \cap B)$ and $d \in B \setminus (A \cap B)$. Then $G$ is reducible with respect to both Symanzik polynomials if and only if $G[A]$ is reducible with respect to both Symanzik polynomials and $G[B]$ is reducible with respect to the first Symanzik polynomial.
\end{lemma}

\begin{proof}
By Lemma \ref{firstsymanzikoneconnected} and Lemma \ref{momentaoneconnectivityreduction}, we have $\phi_{G} = \phi_{G[A]}\psi_{G[B]}$ and $\psi_{G} = \psi_{G[A]}\psi_{G[B]}$. Then $\{\phi_{G},\psi_{G}\}$ is reducible if and only if $\{\phi_{G[A]},\psi_{G[A]},\psi_{G[B]}\}$ is reducible by Lemma \ref{irreduciblepolynomials}. By Lemma \ref{independantpolynomials}, $\{\phi_{G[A]},\psi_{G[A]},\psi_{G[B]}\}$ is reducible if and only if both $\{\phi_{G[A]}, \psi_{G[A]}\}$ and $\{\psi_{G[B]}\}$ are reducible. By definition, $\{\phi_{G[A]}, \psi_{G[A]}\}$ and $\{\psi_{G[B]}\}$ are reducible exactly when $G[A]$ is reducible with respect to both the first and second Symanzik polynomials, and $G[B]$ is reducible with respect to the first Symanzik polynomial. 
\end{proof}

\begin{lemma}
Let $G$ be a graph with where $X = \{a,b,c,d\} \subseteq V(G)$ is the set of vertices with momenta, and assume that all external momenta are on-shell. Furthermore, suppose $G$ has no massive edges. Suppose $G$ has a $1$-separation $(A,B)$ and that $X \subseteq A$. Then $G$ is reducible with respect to both Symanzik polynomials if and only $G[A]$ is reducible with respect to both Symanzik polynomials and $G[B]$ is reducible with respect to the first Symanzik polynomial. 
\end{lemma}

\begin{proof}
By Lemma \ref{firstsymanzikoneconnected} and Lemma \ref{momentaoneconnectivityreduction} we have that $\phi_{G} = \phi_{G[A]}\psi_{G[B]}$ and $\psi_{G} = \psi_{G[A]}\psi_{G[B]}$.
Then by $\{\phi_{G},\psi_{G}\}$ is reducible if and only if $\{\phi_{G[A]},\psi_{G[A]},\psi_{G[B]}\}$ is reducible by Lemma \ref{irreduciblepolynomials}. By Lemma \ref{independantpolynomials}, $\{\phi_{G[A]},\psi_{G[A]},\psi_{G[B]}\}$ is reducible if and only if both $\{\phi_{G[A]}, \psi_{G[A]}\}$ and $\{\psi_{G[A]}\}$ are reducible. By definition, $\{\phi_{G[A]}, \psi_{G[A]}\}$ and $\{\psi_{G[B]}\}$ are reducible exactly when $G[A]$ is reducible with respect to both the first and second Symanzik polynomials, and $G[B]$ is reducible with respect to the first Symanzik polynomial. 
\end{proof}

One cannot extend the above lemmas to when we have a $1$-separation $(A,B)$ such that two of the vertices with external momenta are in $A \setminus (A \cap B)$ and two of the momenta are in $B \setminus (A \cap B)$. This follows since the second Symanzik polynomial does not factor nicely to allow a reduction to the $2$-connected case (see Lemma \ref{symanzikpolynomialtwomomentaoneachside}).

\section{Graphs which are not reducible}

\begin{figure}
\begin{center}
\includegraphics[scale =0.5]{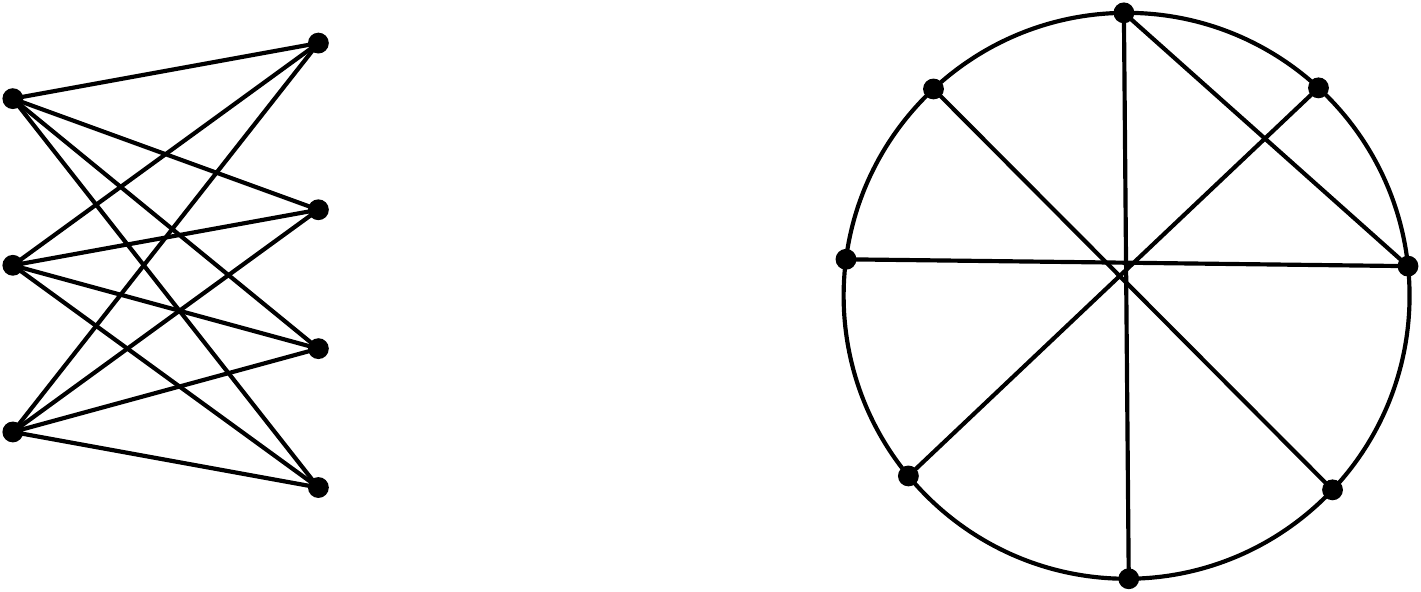}
\end{center}
\caption{Forbidden minors for reducibility with respect to $\psi$.}
\label{forbiddenminorsfirst}
\end{figure}

A celebrated result is that if a graph class is graph minor closed, then it is characterized by a finite set of forbidden minors (\cite{WQOtheorem}). Then by Corollary \ref{MinorClosed}, reducibility with a fixed number of momenta is characterized by a finite set of forbidden minors. Therefore it is of interest to try and exactly determine that set. We do not claim to have anywhere near a complete characterization, but we list some of the graphs which are known to be non-reducible here. 

Quite a bit of work has been done on reducibility with respect to the first Symanzik polynomial. In \cite{FrancisBig}, the graph $K_{3,4}$ was shown to be non-reducible with respect to $\psi$ and that, in particular, for any permutation $\sigma$, the set $S_{[\sigma(1),\ldots,\sigma(7)]}$ (if it exists) is quadratic in $\sigma(8)$. Furthermore, by direct computation, one can show that all proper minors of $K_{3,4}$ are reducible. Therefore $K_{3,4}$ is a forbidden minor for the first Symanzik polynomial (\cite{FrancisBig}). 

Also in \cite{FrancisBig}, Brown notes that the graph $V_{8} \cup \{02\}$ (see Figure \ref{forbiddenminorsfirst}, also known as $Q48$ in \cite{olivercensus}, $V_{8}$ is Wagner's graph as defined in chapter $1$) is not reducible with respect to $\psi$. Again by direct computation, it is possible to show that $V_{8} \cup \{02\}$ is a forbidden minor with respect to $\psi$ (\cite{FrancisBig}).

A census of periods of Feynman integrals has been constructed in \cite{olivercensus}. Since reducibility is tied to Brown's integration algorithm, and we know that when Brown's algorithm succeeds that the Feynman integral is a multiple zeta value, all graphs in \cite{olivercensus} whose Feynman integral is not a multiple zeta value are not reducible. There are some graphs which are known to not have periods which are multiple zeta values, but they are almost certainly not forbidden minors. However, it is quite difficult computationally to try and find the forbidden minor from these graphs. For example, we attempted to show $K_{6}$ was not reducible by running it on the Westgrid servers for seven days but it was only around halfway through the reduction, and seemingly slowing down as the reduction continued.

One graph which I suspect to be a forbidden minor for reducibility with respect to the first Symanzik polynomial is $K_{6}$. While I leave it open on whether or not $K_{6}$ is reducible, we will see in Corollary \ref{k6minors}, that if $K_{6}$ is not reducible it does not change the class of reducible graphs much.

\begin{figure}
\begin{center}
\includegraphics[scale = 0.5]{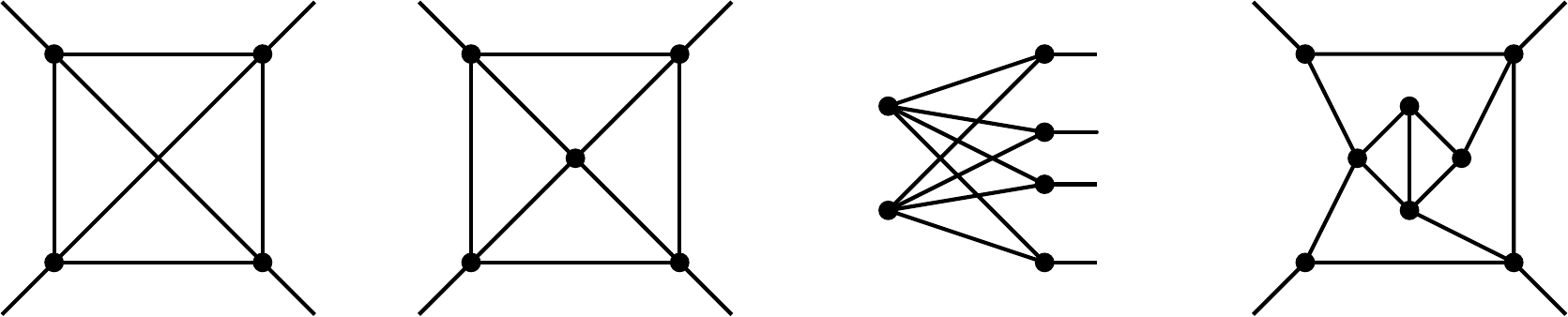}
\end{center}
\caption{Forbidden minors for reducibility with respect to both Symanzik polynomials. The vertices with the on-shell momenta are denoted by lines with adjacent to only one vertex.}
\label{forbiddenminorssecond}
\end{figure}

Now we shift focus to reducibility with respect to both Symanzik polynomials, where the underlying graph has four on-shell external momenta, and $G$ has no massive edges. From example \ref{K4notreducible} and \ref{W4notreducible} we have that $K_{4}$ and $W_{4}$ are not reducible with four on-shell external momenta. Furthermore, one can directly compute that those examples are minor minimal. Using HyperInt, one can show that $K_{2,4}$ and the other graph in Figure \ref{forbiddenminorssecond} which we will call $L$, are forbidden minors. The graph $L$ was found by doing a brute force test for reducibility on small graphs. We note that by Lemma \ref{subsetlemma}, all graphs which are not reducible with respect to the first Symanzik polynomial are also not reducible with respect to both Symanzik polynomials. 

So to give a list of known forbidden minors, $K_{3,4}$ and $V_{8} \cup \{02\}$ (see Figure \ref{forbiddenminorsfirst}) are forbidden minors for the first Symanzik polynomial. The graphs $K_{2,4}$, $W_{4}$, $K_{4}$ and the graph $L$ (see Figure \ref{forbiddenminorssecond}) are forbidden minors for both Symanzik polynomials. We note that we almost surely do not have the full list of forbidden minors. In the next chapter, we look at graphs not containing the above minors and see what structure those graphs must have.

\chapter{Graph Minors}

 Recall that a graph $G$ has a \textit{$H$-minor} for some graph $H$, if a graph isomorphic to $H$ can be obtained from $G$ through the contraction and deletion of some edges, and possibly the removal of some isolated vertices. It is easy to see that if $G$ is connected, then we never need to delete isolated vertices. Before jumping into new results on graph minors, we give a brief survey.
 
  Graph minors have been heavily studied in the literature. The most notable contribution is the graph minors project by Neil Robertson and Paul Seymour. The brunt of the project is dedicated to approximately describing the structure of graphs without a $K_{t}$-minor. The theorem which accomplishes this is usually known as the Graph Structure Theorem, which says if a graph $G$ does not have a $K_{t}$-minor, then $G$ ``almost embeds'' on a surface of low Euler genus relative to $t$ (\cite{graphstructuretheorem}). From this theorem one obtains one of the most well known and celebrated results of the graph minors project. 

\begin{theorem}[\cite{WQOtheorem}]
Graphs are well quasi ordered with respect to the minor operation.
\end{theorem}

In other words, if $\mathcal{G}$ is a minor closed class of graphs, then $\mathcal{G}$ is characterized by a finite set of forbidden minors. Of course, this theorem is meaningless without some examples of classes of graphs which are minor closed, but fortunately there are lots. Before listing some examples, a few definitions are needed. For some integer $k$, a graph $G$ has $\text{treewidth} < k$ if $G$ can be decomposed into clique sums of graphs with less than $k+1$ vertices. We note that is not the standard definition of treewidth but can easily be shown to be equivilant (see, for example, \cite{TreewidthSurvey}).  A graph $G$ is \textit{outerplanar} if $G$ is planar and all of its vertices lie on the outerface. Some examples of graph minor closed classes are:

\begin{enumerate}
\item{Planar Graphs.}

\item{Outerplanar graphs.}

\item{Graphs with treewidth less than $k$ for a fixed positive integer $k$.}

\item{Graphs with vertex width less than $k$ for a fixed positive integer $k$.}

\item{Graphs which embed in a fixed surface.} 

\item{Graphs which do not contain $k$ pairwise disjoint cycles, for a fixed positive integer $k$.}

\item{Graphs which are reducible with respect to the Symanzik polynomials for a fixed number of external momenta.}

\item{Graphs with no path of length $k$.}

\item{For a fixed set graph $H$, graphs not containing an $H$-minor.}

\end{enumerate}

Notice for the classes $1,2,3,4,5,$ and $7$, it is not obvious what the complete set of forbidden minors are. But for some of these classes the complete set of forbidden minors is known. For example, a graph is planar if and only if it does not contain a $K_{5}$ or $K_{3,3}$ as a minor (\cite{WagnersTheorem}). A graph is outerplanar if and only if it does not contain $K_{4}$ or $K_{2,3}$ as a minor (\cite{outerplanarcharacterization}). On the other hand, it is an open problem to determine the full set of forbidden minors for graphs with treewidth less than $k$ for all values of $k \geq 4$ (\cite{TreewidthSurvey}, \cite{twlessthan3minors}). Also, with the exception of the plane (\cite{WagnersTheorem}) and the projective plane (\cite{Projectiveplaneminors}), the full forbidden minor set for graphs which embed in some fixed surface is unknown (\cite{Lovazgraphsurvey}). In general the size of the list of forbidden minors can get rather large, for instance Chambers and Myrvold have found that there are at least $16,629$ forbidden minors for graphs which embed on the torus (\cite{Torusminors}). As an even more extreme example, Wye-Delta-Wye reducibility currently has at least $69$ billion known forbidden minors (\cite{DeltaWyereducibility}). So a natural question is, given a minor closed class $\mathcal{G}$, what are the forbidden minors for $\mathcal{G}$? 

However for the classes $6,8$, and $9$, that question is essentially useless since the class is being defined from the set of forbidden minors. So for those classes, it would be nice to have a theorem describing the structure of graphs in these classes in some meaningful way. These type of theorems are usually called excluded minor theorems. They generally outline some small set of graphs which do not contain an $H$-minor for some topological reason and then give a way to glue the small graphs together in a way which prevents an $H$-minor. Examples of such theorems include:

\begin{theorem}[\cite{noK33minor}]
A graph $G$ is $K_{3,3}$-minor free if and only if $G$ can be constructed from a series of $0$-sums, $1$-sums, and $2$-sums of planar graphs and $K_{5}$.
\end{theorem}

\begin{theorem}[\cite{WagnersTheorem}]
A graph $G$ is $K_{5}$-minor free if and only if $G$ can be constructed from a series of $0$-sums, $1$-sums, $2$-sums, and $3$-sums of planar graphs and $V_{8}$ (as defined in the graph theory basics section). 
\end{theorem}

\begin{theorem}[\cite{noK4minor}]
A graph $G$ is $K_{4}$-minor free if and only if $G$ can be constructed from a series of $0$-sums, $1$-sums and $2$-sums of $K_{1}$, $K_{2}$ and $K_{3}$.
\end{theorem}
 
 Similar theorems are known for the cube (\cite{Excludedminorcube}), prism (\cite{noprismminor}), octahedron (\cite{Excludedminorocta}), $V_{8}$ (\cite{NoV8minors}), and all $3$-connected graphs with less than $11$ edges (\cite{Smallminors}). One might notice that all the graphs I have listed are $3$-connected, and this is due to the usefulness of Seymour's Splitter Theorem (see Theorem \ref{splittertheorem}). That being said, there are non-trivial excluded minor theorems known for graphs which are not $3$-connected, for instance graphs not containing $K_{2,4}$-minors were completely characterized in \cite{k24excludedminors}. 

This chapter is dedicated to looking at excluded minor theorems for graphs which are forbidden minors for reducibility. First let us reformulate the notion of a minor in a way which will be easy to generalize. 

\begin{proposition} 
\label{branchsets}
Let $G$ and $H$ be graphs. The graph $G$ has an $H$-minor if and only if we can find a set $\{G_{x} : x \in V(H)\}$ of pairwise disjoint connected subgraphs of $G$ indexed by the vertices of $H$, such that if $xy \in E(H)$ then there exists a vertex $v \in V(G_{x})$ and a vertex $u \in V(G_{y})$ such that $uv \in E(G)$. 
\end{proposition}

\begin{proof}
Suppose $G$ has an $H$-minor. For each $x \in V(H)$, let $V(G_{x})$ be the set of vertices in $G$ which were contracted to obtain $x$. Then let $G_{x} = G[V(G_{x})]$. From the definition of contraction it follows that $G_{x}$ is connected. Additionally, since we started with an $H$-minor, if $xy \in E(H)$ there exists a vertex in $V(G_{x})$ which is adjacent to a vertex in $V(G_{y})$.

Conversely suppose we are given a set of disjoint subgraphs $\{G_{x} : x \in V(H)\}$. Then contract each subgraph to a distinct vertex. All other edges can be contracted arbitrarily and, if needed, other isolated vertices can be deleted. It is easy to see this results in an $H$-minor.
\end{proof}

We will refer to the set of subgraphs $\{G_{x} : x \in V(H)\}$ as a \textit{$H$-model of $G$} and the subgraph $G_{x}$ is the \textit{branch set} of $x$. Note for a connected graph $G$ and an $H$-model of $G$, without loss of generality  we can assume that $\bigcup_{x \in V(H)} V(G_{x})= V(G)$.

 Now we can naturally generalize our notion of graph minors. Let $X \subseteq V(G)$ and let $\pi : X \to V(H)$ be a one-to-one mapping. We say $G$ has a \textit{rooted $H$-minor on $X$ with respect to $\pi$} if there is an $H$-model of $G$ such that for all $v \in X$, we have $v \in G_{\pi(v)}$. If $G$ has a rooted $H$-minor on $X$ with respect to $\pi$, we will say $G$ has an $H(X)$-minor with respect to $\pi$, and we will call the vertices in $X$ the \textit{roots} or \textit{terminals} of the $H$-minor. For the application of reducibility, the terminals are taking the role of the $4$ on-shell external momenta.

 We note this definition is a little stricter than what we want from the point of view of reducibility. In general we will allow families of maps from $X$ to $H$, say $\pi_{1}, \pi_{2},\ldots,\pi_{n}$, such that if a graph $G$ has an $H(X)$-minor with respect to $\pi_{i}$, for any $i \in \{1,\ldots,n\}$, we will say that $G$ has an $H(X)$-minor with respect to $\pi_{1},\ldots,\pi_{n}$. For ease of notation, once the family of maps is defined for some graph $H$ and $X$, we will just say that $G$ has an $H(X)$-minor.

Now we can try and understand excluded minor theorems for rooted minors. So the question we are trying to solve is: if the size of $X$ is fixed, the graph $H$ is fixed, and the family of mappings $\pi_{i}: X \to V(H)$ is fixed, what is the structure of graphs without an $H(X)$-minor? We note that such problems have been studied in the literature, and even show up in the graph minors project (\cite{rootedk23theorem}), but they appear to be much less studied than standard excluded minor theorems.

We will survey some examples of excluded rooted minor theorems. Consider a $K_{3}(X)$-minor where $X = \{a,b,c\}$ and where the family of maps is the family of surjective maps from $X$ to $V(K_{3})$. Then we have the following theorem.

\begin{theorem}[\cite{rootedk3}]
For distinct vertices $a,b,c$ in a graph $G$, there is a $K_{3}(X)$-minor on $a,b,c$ unless for some vertex $v \in V(G)$, at most one of $a,b,c$ are in each component of $G - v$. 
\end{theorem}

Jung's $2$-linkage Theorem states that given four vertices $(s_{1},t_{1},s_{2},t_{2})$ in a $4$-connected graph $G$, there is a $(s_{1},s_{2})$-path which is disjoint from a $(t_{1},t_{2})$-path unless $G$ is planar and $s_{1},t_{1},s_{2},t_{2}$ lie on a face in that cyclic order (\cite{Junglinkage}). This can be viewed as a partial characterization for the rooted minor problem on a graph $K_{2} + K_{2}$ where $V(K_{2} + K_{2}) = \{s_{1},s_{2},t_{1},t_{2}\}$ and $E(K_{2} + K_{2}) = \{s_{1}s_{2}, t_{1}t_{2}\}$, $X = \{a,b,c,d\}$, and the family of maps we are considering is the family where $\{a,b\}$ gets mapped surjectively to $\{t_{1},t_{2}\}$ and $\{c,d\}$ gets mapped subjectively to $\{s_{1},s_{2}\}$. Generalizations of the $2$-linkage theorem with a rooted minor flavour have been studied in \cite{Paulwollansthesis}. Various results have been found on the following question: Given a graph $G$ with $n$ vertices, what is the maximium number of edges $G$ can have before it must have an $H(X)$-minor (\cite{extermalrootedminorleif}, \cite{Paulwollansthesis})? 

This chapter will give some characterizations to assorted $H(X)$-minor classes which are relevant to reducibility, namely graphs without $K_{4}(X)$-minors, $W_{4}(X)$-minors, $K_{2,4}(X)$-minors and $L(X)$-minors (see Figure \ref{forbiddenminorssecond} for a picture of the graphs, Figure \ref{L(X)labelling} for a picture just of the graph $L$). As mentioned before, from the point of view of reducibility, the set $X$ represents the external momenta of a graph. First we outline a characterization of $K_{4}(X)$-minors which was proved in \cite{root}. 

For the purposes of this thesis, given a graph $G$ and $X = \{a,b,c,d\} \subseteq V(G)$, $G$ has a $K_{4}(X)$ minor if and only if $G$ has a $K_{4}(X)$-minor with respect to $\pi$, where $\pi$ is any surjective map from $X$ to $V(K_{4})$. Thus $G$ has a $K_{4}(X)$-minor if $G$ has a $K_{4}$-minor where each vertex of $X$ ends up in a distinct branch set. Alternatively, $G$ has a $K_{4}(X)$ minor if $G$ has a $K_{4}$-minor where none of $a,b,c$ or $d$ get contracted together (here we are enforcing that when we contract $a,b,c$ or $d$, that the resulting vertex is called $a,b,c$ or $d$, respectively). 
Before we can state Monroy and Wood's result on $K_{4}(X)$-minors, we need some definitions.

\begin{figure}
\begin{center}
\includegraphics[scale =0.4]{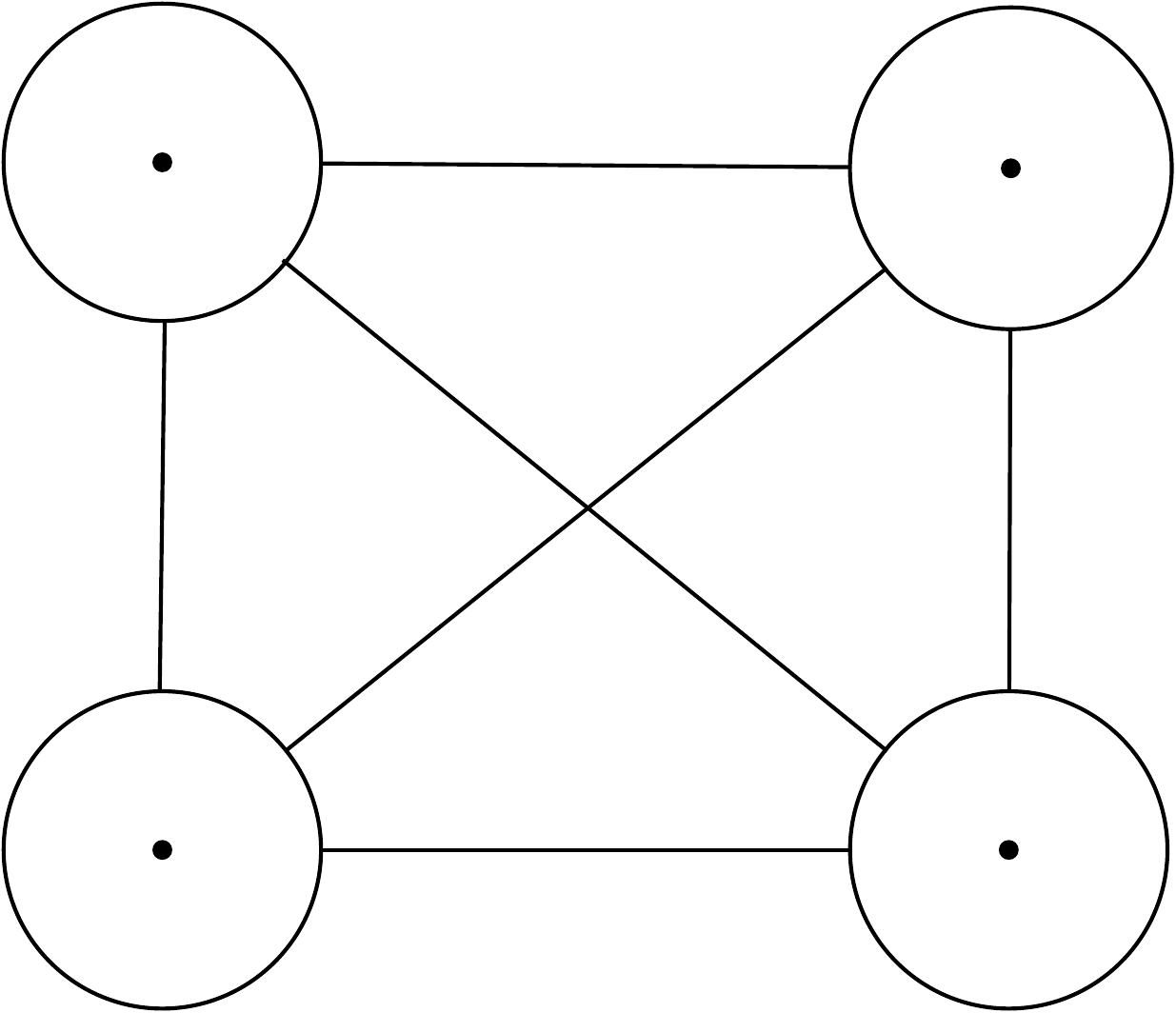}
\caption{A graph $G$ that has a model of a $K_{4}(X)$-minor. The circles represent connected subgraphs. The vertices represent vertices from $X$.}
\end{center}
\end{figure}

Let $H$ be a graph. The graph $H^{+}$ is defined in the following way: For each triangle  $T = \{x,y,z\}$  in $H$, let $F_{T}$ be a clique of arbitrary size such that all vertices in $F_{T}$ are adjacent to every vertex in $T$ and no other vertices of $H$. Let $F$ be the disjoint union of $F_{T}$ for all triangles $T$. Then we define $V(H^{+}) = V(H) \cup V(F)$, and $E(H^{+}) =  E(H) \cup E(F) \cup \{xy | x \in V(T), \text{T a triangle}, y \in F_{T}\}$. Note that given $H$ and $F$, the graph $H^{+}$ is uniquely defined so we will frequently use the notation $H^{+} = (H,F)$. Note if $H$ is a planar graph, it is said that the graph $H^{+}$ admits a \textit{flat embedding}.  

Now, consider a planar graph $H$ where the outerface is a $4$-cycle which we will call $C_{4}$, and every internal face of $H$ is a triangle, and every triangle is a face. Let $V(C_{4}) = \{x_{1},x_{2},x_{3},x_{4}\}$. Then we call the graph $H^{+}$ an \textit{$\{x_{1},x_{2},x_{3},x_{4}\}$-web}. Now we can state the excluded $K_{4}(X)$-minor Theorem.

\begin{theorem}[Monroy, Wood, \cite{root}]
\label{k4free}
Let $G$ be a graph and $X = \{a,b,c,d\} \subseteq V(G)$. Then either $G$ has a $K_{4}(X)$-minor or $G$ is the spanning subgraph of a graph belonging to one of the following six classes of graphs:

\begin{itemize}

\item{Class $\mathcal{A}$: Let $H$ be the graph with vertex set $V(H) = \{a,b,c,d,e\}$ and with edge set $E(H) = \{ae,ad,be,bd,ce,cd,de\}$.  Class $\mathcal{A}$ is the set of all graphs $H^{+}$.}

\item{Class $\mathcal{B}$: Let $H$ be the graph with vertex set $V(H) = \{a,b,c,d,e,f\}$ and with edge set $E(H) = \{ae,af,be,bf,ce,cf,de,df\}$. Class $\mathcal{B}$ is the set of all graphs $H^{+}$.}

\item{Class $\mathcal{C}$: Let $H$ be the graph such that $V(H) = \{a,b,c,d,e,f,g\}$ and with edge set $E(H) = \{ae,ag,be,bg,cf,cg,df,dg,ef,eg,fg\}$. Class $\mathcal{C}$ is the set of all graphs $H^{+}$.}

\item{Class $\mathcal{D}$: The set of all $\{a,b,c,d\}$-webs.}

\item{Class $\mathcal{E}$: Let $H'$ be a planar graph with a $4$-cycle on the outerface and every internal face is a triangle, and every triangle is a face. Let $V(C) = \{c,d,e,f\}$ and suppose they appear in that order on $C$. Let $H$ be the graph with vertex set $V(H) = V(H') \cup \{a,b\}$ and edge set $E(H) = E(H') \cup \{ae,af,be,bf\}$. Class $\mathcal{E}$ is the set of all graphs $H^{+}$.}

\item{Class $\mathcal{F}$: Let $H'$ be a planar graph with a $4$-cycle on the outerface and every internal face is a triangle, and every triangle is a face. Let $V(C) = \{e,f,g,h\}$ and suppose they appear in that order in $H'$. Let $H$ be the graph with vertex set $V(H) = V(H') \cup \{a,b,c,d\}$ and edge set $E(H) = E(H') \cup \{ae,af,be,bf,cg,ch,dg,dh\}$. Class $\mathcal{E}$ is the set of all graphs $H^{+}$.}

\end{itemize} 
\end{theorem}

Figure \ref{k4freepic} gives a pictorial representation of the graphs $H$ in the above classes. It is easily seen that the result simplifies significantly when we restrict to $3$-connected graphs.

\begin{corollary}[\cite{root}]
\label{k4(x)3conn}
Let $G$ be a $3$-connected graph and $X = \{a,b,c,d\} \subseteq V(G)$. Then either $G$ has a $K_{4}(X)$-minor or $G$ is a spanning subgraph of a Class $\mathcal{D}$ graph. In other words, $G$ is a spanning subgraph of an $\{a,b,c,d\}$-web.
\end{corollary}

When we have $3$-connected planar graphs $G$, this result can be simplified even further. 

\begin{corollary}[\cite{root}]
\label{k4(x)planar}
Let $G$ be a $3$-connected planar graph and $X = \{a,b,c,d\} \subseteq V(G)$.  Then $G$ does not have a $K_{4}(X)$-minor if and only if all the vertices of $X$ lie on the same face.
\end{corollary}

\begin{figure}
\centering
\includegraphics[scale = 0.5]{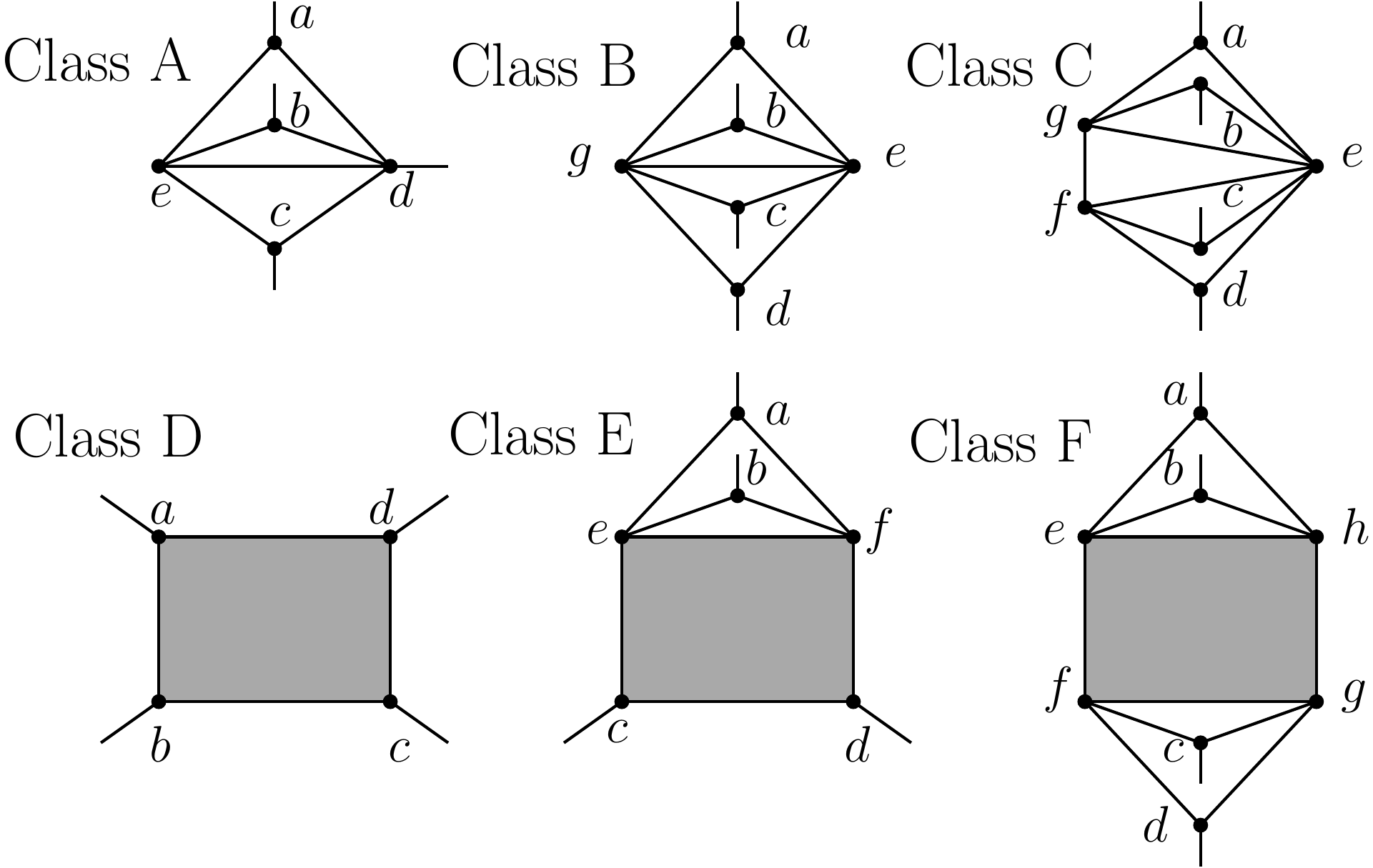}
\caption{The graphs $H$ which appear in Theorem \ref{k4free}. The cliques in the triangles are omitted. The shaded sections are webs.}
\label{k4freepic}
\end{figure}

We will use this characterization as a starting point for most of the upcoming minor characterizations. 

\section{Connectivity reductions}

In this section we will show that if I have a graph $G$, and I want to determine an $H(X)$-minor, then in some sense, I can assume that the connectivity of $G$ is the same as the connectivity of $H$. Throughout this section, if we do not say what the underlying family of maps $ \pi_{1},\pi_{2},\ldots,\pi_{n}$ is, it is assumed that we have an arbitrary family of maps. We make the following easy observations.

\begin{observation}
\label{connectedfromdisconnected}
Let $H$ be a connected graph. Let $G$ be a graph and $X \subseteq V(G)$. Then $G$ has an $H(X)$-minor if and only if $X$ is contained in a connected component of $G$, and the connected component has an $H(X)$-minor. 
\end{observation}

\begin{proof}
If $X$ is contained in a connected component of $G$ and that component has an $H(X)$-minor then immediately $G$ has an $H(X)$-minor.

Conversely, suppose $\{G_{x} | x \in V(H)\}$ is a model of an $H(X)$-minor in $G$. Let $u,v \in X$ such that $u$ and $uv$ lie in different connected components of $G$. Let $u \in G_{x}$ and $v \in G_{y}$ respectively.  Since $H$ is connected, there is a $(x,y)$-path in $H$. Then there is a sequence of branch sets $G_{x},G_{x_{1}}, G_{x_{2}}, \ldots,G_{x_{n}},G_{y}$ such that $G_{x}$ has a vertex which is adjacent to a vertex in $G_{x_{1}}$, $G_{x_{i}}$ has a vertex which is adjacent to a vertex in $G_{x_{i+1}}$ for all $i \in \{1,\ldots,n-1\}$ and $G_{x_{n}}$ has a vertex adjacent to $G_{y}$. Then since each branch set is connected, each of $G_{x}$,$G_{x_{1}} \ldots,G_{y}$ are in the same connected component. But this contradicts that we assumed $x$ and $y$ were in different components. Therefore all vertices of $X$ lie in the same component. Call this component $C$. Notice that if there is a branch set $G_{z}$ which is not contained in $C$, then it has no vertices in $C$ since branch sets are connected. By essentially the same argument as above, this implies that there is no path from $x$ to $z$ in $H$, which contradicts $H$ being connected, and so $C$ has an $H(X)$-minor.
\end{proof}

\begin{observation}
Let $H$ be a simple graph. Let $G$ be a graph and $X \subseteq V(G)$. Then $G$ has an $H(X)$-minor if and only if the graph obtained from $G$ by removing all parallel edges and loop edges has an $H(X)$-minor.  
\end{observation}

\begin{proof}
Let $\{G_{x} | x \in V(H)\}$ be a model of an $H$-minor in $G$. Since $H$ is simple, all the required adjacencies between branch sets are single edges. Therefore $\{G_{x} | x \in V(H)\}$ is a model of an $H$-minor in $G'$, where $G'$ is the graph obtained from $G$ by deleting all parallel and loop edges. Conversely, if $G'$ has an $H(X)$-model, this immediately implies that $G$ has an $H(X)$-model.
\end{proof}

Then since all of our forbidden minors we will discuss are simple and connected,  we will assume all graphs are simple and connected. Notice that in observation \ref{connectedfromdisconnected}, since $H$ was connected, we could restrict our attention to the connected components of $G$. There is a more general principle at work here. If $H$ is $k$-connected, then essentially we can reduce the problem of determining if $G$ has an $H$-minor to looking at the ``$k$-connected-components'' of $G$. The following lemma is a well known example of this.

\begin{lemma}
\label{lowordersepsminors}
Let $H$ be a $3$-connected graph. Let $G$ be a $k$-sum of $G_{1}$ and $G_{2}$ where $k \in \{0,1,2\}$. Then $G$ has an $H$-minor if and only if $G_{1}$ or $G_{2}$ has an $H$-minor. 
\end{lemma}

The next two sections give an extension of this lemma to rooted minors when we have four roots, which the author could not find written down anywhere.

\subsection{Cut vertices}
For the following sequence of lemmas suppose that we have a  simple $2$-connected graph $H$, and a connected graph $G$ where $G$ has a $1$-separation $(A,B)$ such that $A \cap B = \{v\}$. Furthermore, let $X = \{a,b,c,d\} \subseteq V(G)$, and without loss of generality suppose that $|A \cap X| \geq |B \cap X|$. Let $\mathcal{F}$ be an arbitrary family of maps from $X$ to $V(H)$. Figure \ref{Oneconnectedreductionspic} gives a picture representation for the upcoming lemmas. 

\begin{lemma}
\label{allonesidecutvertex}
If $X \subseteq A$, then $G$ has an $H(X)$-minor if and only if $G[A]$ has an $H(X)$-minor.
\end{lemma}

\begin{proof}
First suppose $G[A]$ has an $H(X)$-minor. Then $G$ has an $H(X)$-minor by extending the branch set containing $v$ to contain all of $G[B]$.

Conversely, suppose $G$ has an $H(X)$-minor and let $\{G_{x} |x \in V(H)\}$ be a model of an $H(X)$-minor in $G$. As $X \subseteq A$, and the graph $H$ is $2$-connected, and $v$ is a cut vertex, there is no branch set which is contained inside $B \setminus \{v\}$. Notice if all of the branch sets are contained inside $G[A]$, then we are done since $\{G[V(G_{x}) \cap A] | x \in V(H)\}$ would be the desired $H(X)$-model in $G[A]$. Therefore we assume at least one branch set contains vertices from $B$. Furthermore, as $v$ is a cut vertex and each branch set is a connected subgraph, we have at most one branch set with vertices in $B$. Let $G_{z}$, $z \in V(H)$, be such branch set. Notice that $G_{z} \cap G[A]$ is a connected subgraph of $G[A]$, and as $G_{z}$ was the only branch set containing vertices in $B$, $\{G_{x} \cap G[A] | x \in V(H)\}$ is an $H(X)$-model in $G[A]$. 
\end{proof}

\begin{lemma}
Suppose $a,b,c \in A \setminus \{v\}$, and $d \in B \setminus \{v\}$. Let $X_{A} = \{a,b,c,v\}$ and for each $\pi \in \mathcal{F}$, define $\pi': X_{A} \to V(H)$ such that $\pi' = \pi$ except that $\pi'(v) = \pi(d)$.  Then $G$ has an $H$-minor if and only if $G[A]$ has an $H(X_{A})$-minor.
\end{lemma}

\begin{proof}
Let $\{G_{x} | x \in V(H)\}$ be an $H(X_{A})$-model in $G[A]$. Suppose that $G_{d}$ is the branch set where $v \in V(G_{d})$. Then we obtain an $H(X)$-model of $G$ by extending $G_{d}$ to contain all of $G[B]$. By construction, the now extended $G_{d}$ is connected, and thus we have an $H(X)$-model in $G$. 

Conversely, let $\{G_{x} |x \in V(H)\}$ be a model of an $H(X)$-minor in $G$. Let $G_{d}$ be the branch set where $d \in V(G_{d})$. As $H$ is $2$-connected, and $v$ is a cut vertex, and since we may assume that the branch sets partition $V(G)$,  $B \subseteq V(G_{d})$, and in particular, $v \in G_{d}$. Therefore all other branch sets are contained inside $G[A]$, and thus all required adjacencies for the $H(X)$-minor exist in $G[A]$. Therefore $\{G_{x} \cap G[A] \ | \ x \in V(H)\}$ is an $H(X_{A})$-minor of $G[A]$.  
\end{proof}

\begin{lemma}
Suppose $v=a$. Then $G$ has an $H(X)$-minor if and only if $X \subseteq A$ and $G[A]$ has an $H(X)$-minor.
\end{lemma}

\begin{proof}
Sufficiency follows from Lemma \ref{allonesidecutvertex}.

Conversely, let $\{G_{x} |x \in V(H)\}$ be a model of an $H(X)$-minor in $G$. Towards a contradiction,  we consider the case where $d,c \in A \setminus \{v\}$ and $b \in B \setminus \{v\}$. Let $G_{v_{1}}$, $G_{v_{n}}$, $G_{z}$ be the branch sets for which $b \in V(G_{v_{1}})$, $c \in G_{v_{n}}$, and $a \in G_{z}$.  As $H$ is $2$-connected, there is a path $P = v_{1},v_{2},\ldots,v_{n}$ in $H$ such that $z \not \in V(P)$. Then there is a sequence of branch sets, $G_{v_{1}},\ldots,G_{v_{n}}$ such that $G_{v_{i}}$ has a vertex which is adjacent to a vertex in $G_{v_{i+1}}$ for all $i \in \{1,\ldots,n-1\}$. But $G_{v_{1}} \subseteq G[A-v]$, $G_{v_{n}} \subseteq G[B-v]$,  $G_{z} \neq G_{v_{i}}$ for any $i \in \{1,\ldots,n\}$ and $a \in V(G_{z})$, a contradiction. The other cases follow similarly. 
\end{proof}

\begin{lemma}
If exactly two vertices of $X$ are in $B \setminus \{v\}$ and exactly two vertices of $X$ are in $A \setminus \{v\}$, then $G$ does not have an $H(X)$-minor.
\end{lemma}

\begin{proof}

Let $\{G_{x} \ | \ x \in V(H)\}$ be an $H$-model of $G$. Let $G_{y}$ be the branch set containing $v$.  Suppose $a \in A \setminus \{v\}$ and $b \in B \setminus \{v\}$ and $a \in G_{a}$ and $b \in G_{b}$. Then if we contract each branch set to a vertex to obtain the $H$-minor, all $(a,b)$-paths in $H$ contain $y$ since $v \in G_{y}$ and $v$ is a cut vertex. But then $H$ is not $2$-connected, a contradiction. 
\end{proof}

\begin{figure}
\begin{center}
\includegraphics[clip,trim=0.5cm 4cm 0.5cm 0.5cm, scale=0.5]{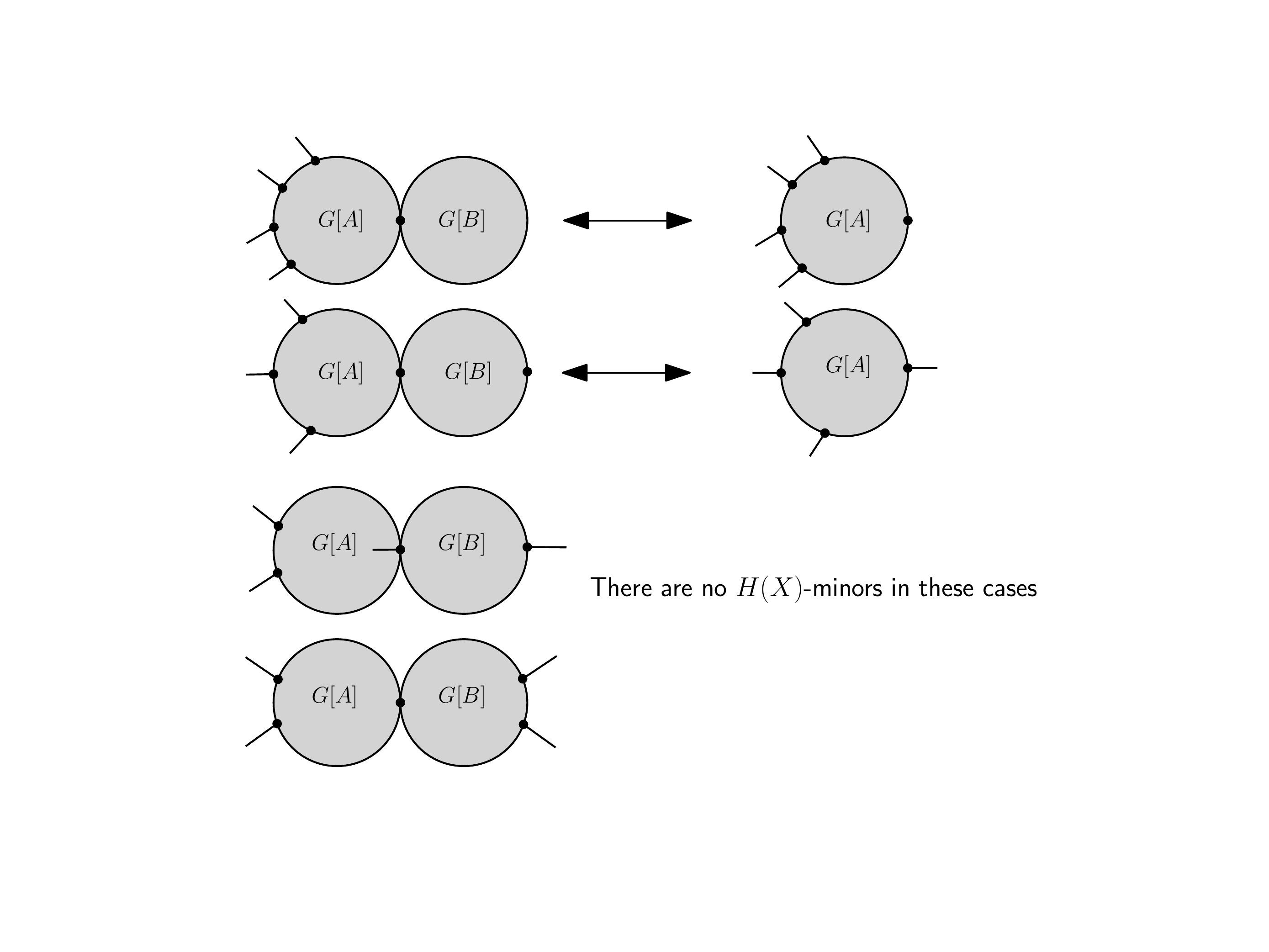}
\caption{Cut vertex reductions. Edges with only one endpoint represent vertices of $X$.  }
\label{Oneconnectedreductionspic}
\end{center}
\end{figure}

\subsection{$2$-connected reductions}

For this section suppose that $H$ is a $3$-connected simple graph and that $G$ is a $2$-connected graph. Furthermore, assume that $G$ has a $2$-separation $(A,B)$ such that $A \cap B = \{u,v\}$, and let $X = \{a,b,c,d\} \subseteq V(G)$. We define $G_{A} = G[A] \cup \{uv\}$ and $G_{B} = G[B] \cup \{uv\}$ (if this results in the graph having a parallel edge, we simply remove the parallel edge). Let $\mathcal{F}$ be a family of maps from $X$ to $V(H)$. By the discussion in the previous section, there is no loss of generality by assuming $G$ is $2$-connected.

\begin{lemma}
\label{2sepW4}
Let $L \subseteq X$ such that $|L| =2$ and suppose $L \subseteq A \setminus \{u,v\}$. Furthermore, suppose that $X \setminus L \subseteq B \setminus \{u,v\}$. Let $X_{A} = \{L,u,v\}$ and $X_{B} = \{(X \setminus L),u,v\}$. Suppose that for any $\pi_{1} \in \mathcal{F}$, there exists a $\pi_{2} \in \mathcal{F}$ such that $\pi_{1}(c) = \pi_{2}(d)$ and $\pi_{1}(d) = \pi_{2}(c)$. For each $\pi \in \mathcal{F}$, define $\pi_{A}$ and $\pi_{B}$ in the natural way so that $u,v$ replace the vertices in $L$ and $X \setminus L$ respectively. Then $G$ has an $H(X)$-minor if and only if either $G_{A}$ has an $H(X_{A})$-minor or $G_{B}$ has an $H(X_{B})$-minor. 

\end{lemma}

\begin{proof}
Suppose that $G_{A}$ contains an $H(X_{A})$-minor and suppose that $c,d \in B \setminus \{u,v\}$. By Corollary \ref{XYpaths}, there are two disjoint paths between $\{u,v\}$ and $\{c,d\}$. Since $c,d \in B \setminus \{u,v\}$, these paths are contained inside of $G[B]$. Thus we can contract $G[B]$ to $\{u,v\}$ in such a way that $c$ and $d$ do not get identified together. Since we supposed that for any $\pi_{1} \in \mathcal{F}$, there exists a $\pi_{2} \in \mathcal{F}$ such that $\pi_{1}(c) = \pi_{2}(d)$ and $\pi_{1}(d) = \pi_{2}(c)$, the graph $G$ has an $H(X)$-minor. The other cases follow similarly. 

Conversely, let $\{G_{x} | x \in V(H)\}$ be a model of an $H(X)$-minor. Suppose $a,b \in A \setminus \{u,v\}$ and $c,d \in B \setminus \{u,v\}$. Let $a \in G_{a}$, $b \in G_{b}$, $c \in G_{c}$, and $d \in G_{d}$. First, suppose there are branch sets $G_{y}$ and $G_{z}$ such that $G_{y} \subseteq G[A - \{u,v\}]$ and $G_{z} \subseteq G[B- \{u,v\}]$. Then if we contract each branch set down to a vertex, there would be at most two internally disjoint $(y,z)$-paths, contradicting that $H$ is $3$-connected.  

Therefore we can assume that either $u \in G_{a}$, $v \in G_{b}$ and $A \subseteq V(G_{a}) \cup V(G_{b})$  or $u \in G_{c}$, $v \in G_{d}$ and $B \subseteq V(G_{c}) \cup V(G_{d})$. Suppose $u \in G_{a}$ and $v \in G_{b}$ and $A \subseteq V(G_{a}) \cup V(G_{b})$. Then all other branch sets are contained in $G[B- \{u,v\}]$. Then since $uv \in E(G_{b})$,  $\{G'_{x} = G_{B}[V(G_{x}) \cap V(G_{B})] \; | x \in V(H)\}$ is an $H(X)$-model in $G_{B}$. The other case follows similarly.
\end{proof}

We remark that all rooted graph minors we will see throughout this thesis satisfy the technical condition in the above lemma.

\begin{lemma}
\label{oneterminalonesideothersother}
Suppose  $a \in A \setminus \{u,v\}$ and $b,c,d \in B \setminus \{u,v\}$.  Let $X_{1} = \{u,b,c,d\}$ and $X_{2} = \{v,b,c,d\}$. For each $\pi \in \mathcal{F}$, let $\pi_{1}$ satisfy $\pi_{1} = \pi$ on $X_{2} \setminus \{u\}$ and $\pi_{1}(u) = \pi(a)$. Let $\pi_{2} = \pi$ on $X_{1} \setminus \{v\}$ and $\pi_{2}(v) = \pi(a)$. Then $G$ has an $H(X)$-minor if and only if either $G_{B}$ has an $H(X_{1})$-minor or an $H(X_{2})$-minor. 
\end{lemma}

\begin{proof}
Suppose $G_{B}$ has an $H(X_{1})$-minor. By Menger's Theorem there exists a path from $a$ to $u$ which does not contain $v$, and therefore we can contract $G[A]$ to $\{u,v\}$ in such a way that $a$ gets contracted onto $u$. Therefore $G$ has an $H(X)$-minor. The case where $G_{B}$ has an $H(X_{2})$-minor follows similarly.

Conversely, let $\{G_{x} | x \in V(H)\}$ be a model of an $H(X)$-minor in $G$. Let $a \in G_{a}$, $b \in G_{b}, c \in G_{c}$, and $d \in G_{d}$. First, suppose for some $y \in V(H)$, $G_{y} \subseteq G[A - \{u,v\}]$. If $yb \in E(H)$, then one of $u$ or $v$ is in $V(G_{b})$. Then at least two of the following occur: $V(G_{a}) \subseteq A \setminus \{u,v\}$, $V(G_{c}) \subseteq B \setminus \{u,v\}$ or $V(G_{d}) \subseteq B \setminus \{u,v\}$. If $V(G_{a}) \subseteq A \setminus \{u,v\}$ and $V(G_{c}) \subseteq B \setminus \{u,v\}$, then if we contract all the branch sets down to a vertex, there are at most two internally disjoint $(a,c)$-paths contradicting that $H$ is $3$-connected. The other case when $V(G_{a}) \subseteq A \setminus \{u,v\}$ follows similarly, and thus we can assume that $G_{a}$ contains one of $u$ or $v$.   But then, contracting all branch sets to a vertex there are at most two internally disjoint $(c,y)$-paths contracting that $H$ is $3$-connected. 

Thus we may assume that $yb \not \in E(H)$. Then by Menger's Theorem there are three internally disjoint $(b,y)$-paths. If either $u$ or $v$ is in $V(G_{b})$, then the above argument can be applied to derive a contradiction. Therefore we assume that $u,v \not \in V(G_{b})$. But then contradicting all the branch sets down to a vertex, every $(b,y)$-path uses the vertex which was obtained by contracting the branch sets that $u$ or $v$ were in down to a single vertex.  But that implies there are at most two internally disjoint $(b,y)$-paths, a contradiction. Therefore for every $y \in V(H)$, $G_{y} \not \subseteq G[A - \{u,v\}]$. 

Then since $a \in A \setminus \{u,v\}$, at least one of $u$ or $v$ is contained in $G_{a}$. Therefore at most one other branch set contains vertices from $A$. Then since $uv \in G_{B}$, $\{G'_{x} = G_{B}[V(G_{x}) \cap V(G_{B})] \; | x \in V(H)\}$ is a model for either an $H(X_{1})$ or $H(X_{2})$-minor in $G_{B}$, depending on which of $u$ and $v$ is in $G_{a}$. 
\end{proof}

\begin{lemma}
\label{allonesidegeneralH}
Suppose that $X \subseteq A$. Then $G$ has an $H(X)$-minor if and only if $G_{A}$ has an $H(X)$-minor.
\end{lemma}

\begin{proof}
Suppose $G_{A}$ has an $H(X)$-minor. Then contracting $G[B]$ onto $\{u,v\}$ gives $G_{A}$, and thus $G$ has a $W_{4}(X)$-minor.

Let $\{G_{x} | x \in V(H)\}$ be a model of an $H(X)$-minor. Let $a \in G_{a}, b \in G_{b}, c \in G_{c}$ and $d \in G_{d}$. Suppose there is a $y \in V(H)$ such that $G_{y}$ is contained in $G[B \setminus \{u,v\}]$. Note that $y \neq a,b,c$ or $d$. Then since $\{u,v\}$ is a $2$-vertex cut, at least two of the following occur: $G_{a} \in G[A \setminus \{u,v\}]$, $G_{b} \in G[A \setminus \{u,v\}]$, $G_{c} \in G[A \setminus \{u,v\}]$ and $G_{d} \in G[A \setminus \{u,v\}]$. Without loss of generality, suppose that $G_{a} \in G[A \setminus \{u,v\}]$. But then if we contract each branch set down to a vertex, there is at most two internally disjoint $(a,y)$-paths in $H$, contradicting that $H$ is $3$-connected. Therefore there are no branch sets contained in $G[B \setminus \{u,v\}]$. Then since $\{u,v\}$ is a $2$-vertex cut, there are at most two branch sets using vertices in $B$. If there is only one branch set using vertices from $B$, then easily $\{G'_{x} = G[V(G_{x}) \cap V(G_{A})] \; | x \in V(H)\}$ is an $H(X)$-minor of $G_{A}$. If two branch sets contain vertices from $B$, then since $uv \in E(G_{A})$,  $\{G'_{x} = G_{A}[V(G_{x}) \cap V(G_{A})] \; | x \in V(H)\}$ is a model of an $H(X)$-minor of $G_{A}$. 
\end{proof}

\begin{lemma}
\label{noW4minor}
Suppose $L \subseteq X$ where $L = \{u,v\}$. Furthermore, suppose that there is a vertex of $X \setminus L$ in $A \setminus \{u,v\}$ and a vertex of $X \setminus L$ in $B \setminus \{u,v\}$. Then $G$ does not have an $H(X)$-minor.
\end{lemma}

\begin{proof}

Consider any $L \subseteq X$. Suppose for a contradiction that $\{G_{x} \ | \ x \in V(H)\}$ is a model of an $H(X)$-minor.  Consider the case where $b,c \in L$ and $a \in A \setminus \{u,v\}$, $d \in B \setminus \{u,v\}$, and $a \in G_{a}$, and $d \in G_{d}$. Then $V(G_{a}) \subseteq A \setminus \{u,v\}$ and $V(G_{d}) \subseteq B \setminus \{u,v\}$. But then contracting each branch set down to a vertex, there are at most two internally disjoint $(a,d)$-paths, contradicting that $H$ is $3$-connected. The other cases follow similarly. 
\end{proof}

\begin{figure}
\begin{center}
\includegraphics[scale =0.5]{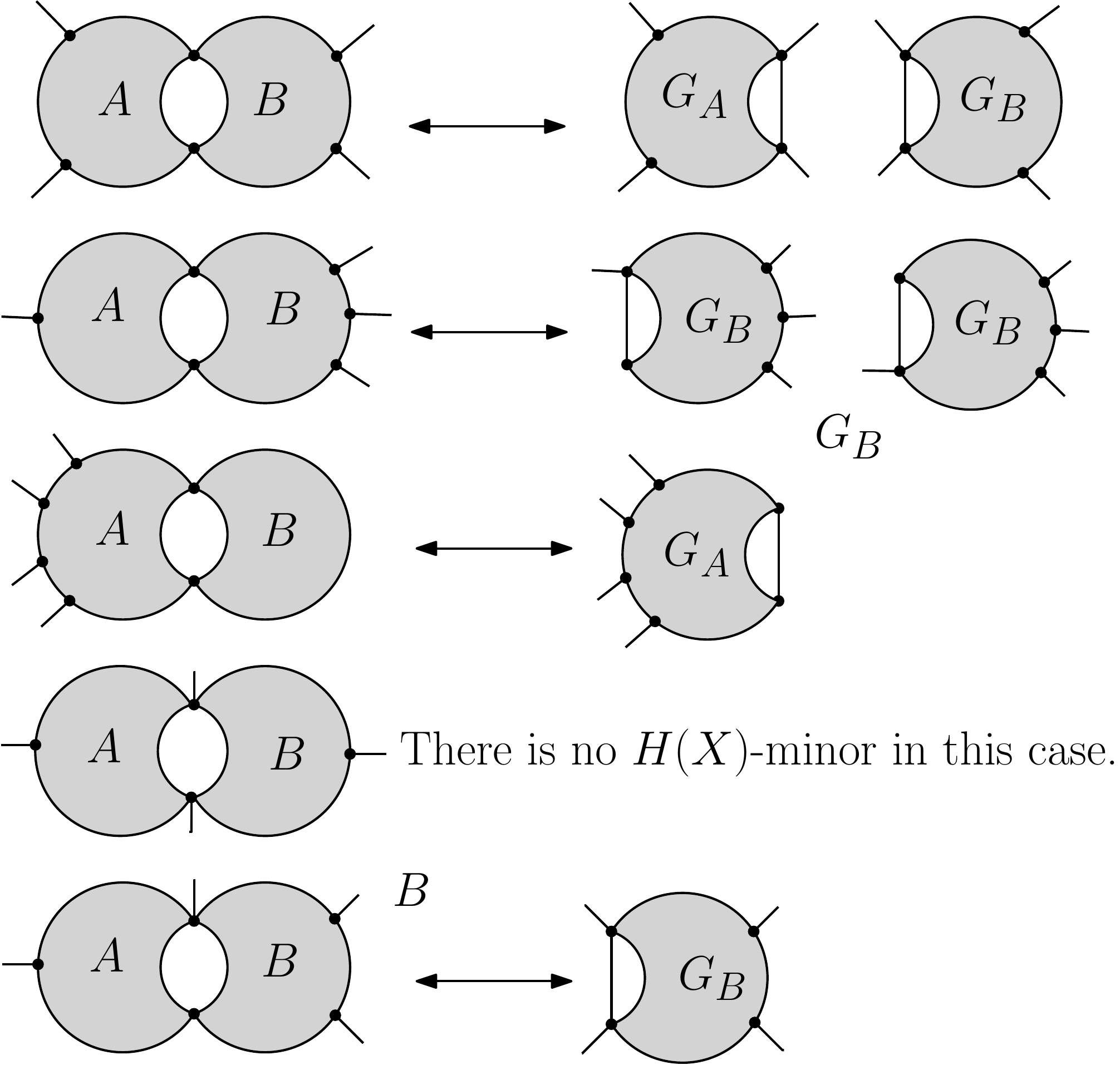}
\caption{The $2$-connected reductions. Edges with only one endpoint represent vertices from $X$.}
\label{2connectivitylemmas}
\end{center}
\end{figure}

\begin{lemma}
\label{oneterminalinthecut2conn}
Suppose $u = a$, and that there is exactly one vertex of $X$ in $A \setminus \{u,v\}$ and two vertices of $X$ in $B \setminus \{u,v\}$. Let $X_{A} = \{u,v,(X \cap B \setminus \{u,v\})\}$. For each $\pi \in \mathcal{F}$, define $\pi_{A} = \pi$ on $X_{A} \setminus \{v\}$ and let $\pi_{A}(v) = \pi(x)$ where $x \in X \setminus \{u, (X \cap A \setminus \{u,v\})\}$. Then $G$ has an $H(X)$-minor if and only if $G_{A}$ has an $H(X_{A})$-minor. 
\end{lemma}

\begin{proof}

 Suppose $G_{A}$ has an $H(X_{A})$-minor. Then by Menger's Theorem there is a path from the vertex of $X \cap A \setminus \{u,v\}$ to $v$ not containing $u$. Thus we can contract $G[A]$ to  $\{u,v\}$ such that we obtain the edge $uv$ and the vertex of $X$ in $A \setminus \{u,v\}$ is contracted to $v$. The resulting graph is isomorphic to $G_{A}$ and thus $G$ has an $H(X)$-minor.

Conversely, let $\{G_{x} | x \in V(H)\}$ be a model of an $H(X)$-minor in $G$. Let $b \in A \setminus \{u,v\}$ and suppose $a \in G_{a}$, $b \in G_{b}$, $c \in G_{c}$ and $d \in G_{d}$. First, suppose for some $y \in V(H)$,  $V(G_{y}) \subseteq A \setminus \{u,v\}$.  As $a =u$,  either $V(G_{c}) \subseteq B \setminus \{u,v\}$ or $V(G_{d}) \subseteq B \setminus \{u,v\}$. Without loss of generality, suppose that $V(G_{c}) \subseteq B \setminus \{u,v\}$.  Then if we contract all of the branch sets to a vertex, there are at most two internally disjoint $(y,c)$-paths, contradicting that $H$ is $3$-connected. Thus the only two branch sets with vertices in $A$ are $G_{a}$ and $G_{b}$. Since $uv \in E(G_{B})$, we have that $\{G'_{x} = G_{B}[V(G_{x}) \cap V(G_{B})] \; | x \in V(H)\}$ is a model of an $H(X)$-minor of $G_{B}$. 
 
\end{proof}

\section{A characterization of graphs without a $K_{4}(X)$ or $W_{4}(X)$ minors}

First we define what we mean by $W_{4}(X)$-minors. Let $G$ be a graph and $X = \{a,b,c,d\} \subseteq V(G)$. Let $\mathcal{F}$ be the family of maps from $X$ to $V(W_{4})$ such that each vertex of $X$ is mapped to a distinct rim vertex of $W_{4}$.  For the purposes of this thesis, a $W_{4}(X)$-minor refers to the $X$ and $\mathcal{F}$ given above (see Figure \ref{W4(X)modelpic}).

In this section we will prove that all planar $3$-connected graphs have either a $W_{4}(X)$-minor or a $K_{4}(X)$-minor, and then we will use that result to show that every $3$-connected graph has either a $W_{4}(X)$-minor or a $K_{4}(X)$-minor (see Theorem \ref{3connectivityW4K4}). Since the graphs $W_{4}$ and $K_{4}$ are $3$-connected, by the previous sections discussion, this completely characterizes graphs with no $W_{4}(X)$ or $K_{4}(X)$-minor.  
\begin{figure}
\begin{center}
\includegraphics[scale =0.4]{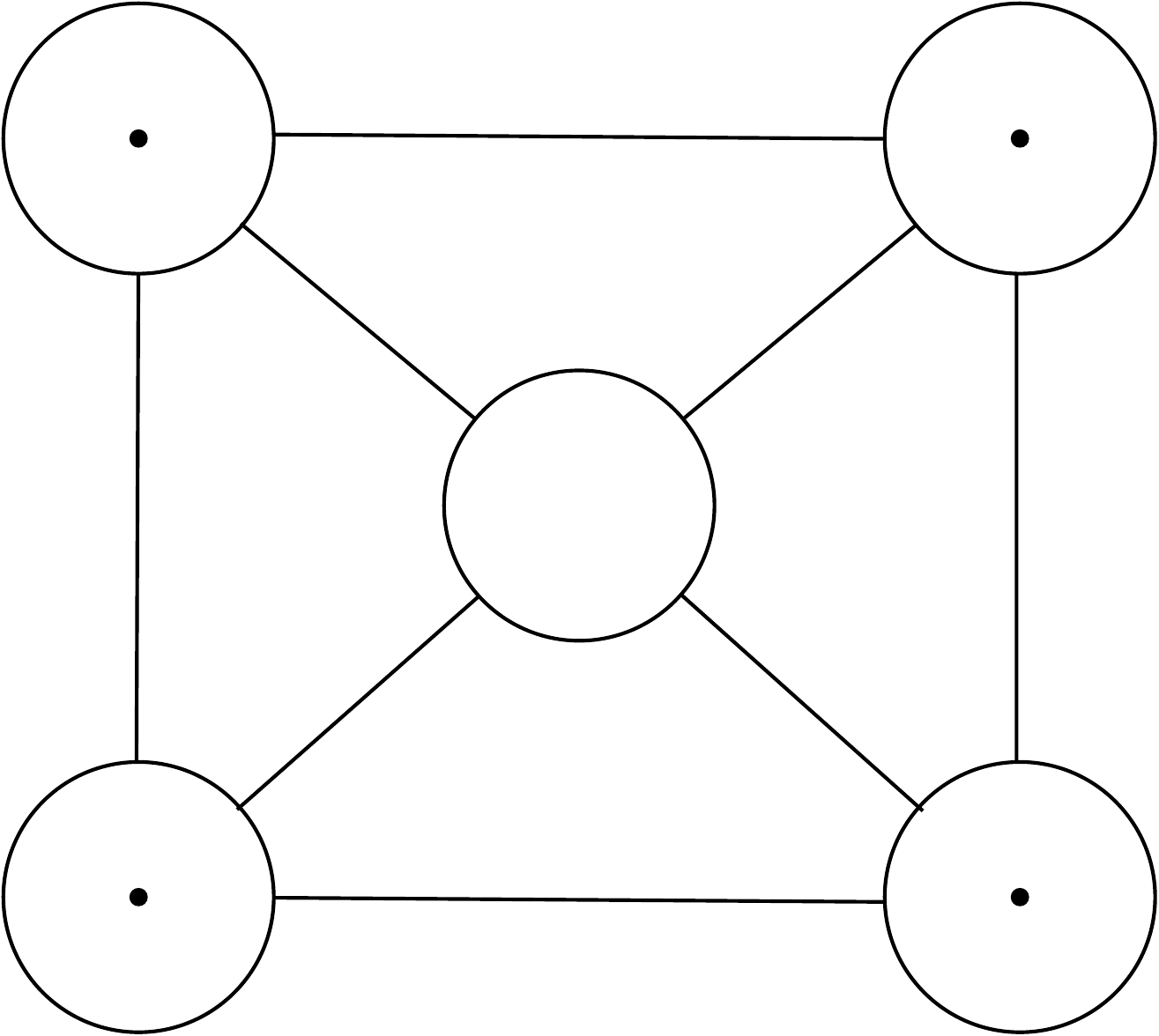}
\caption{A model of a $W_{4}(X)$-minor in some graph $G$. Each circle represents a connected subgraph of $G$. The vertices represent vertices from $X$}
\label{W4(X)modelpic}
\end{center}
\end{figure}

\begin{lemma}
\label{3connectedplanarpaths}
Let $G$ be a $3$-connected planar graph and let $C$ be a facial cycle of $G$. Suppose $v,w \in V(C)$ and $vw \not \in E(G)$. Then there is a $(v,w)$-path $P$ such that $V(P) \cap V(C) = \{v,w\}$.  
\end{lemma} 

\begin{proof}
  Let $P_{1},P_{2},P_{3}$ be internally disjoint $(v,w)$-paths, and let $F_{1},F_{2}$ be the two facial walks of $C$ from $v$ to $w$. If any two of $P_{1},P_{2},P_{3}$ are the facial walks then we are done.

Therefore we assume at most one of $P_{1},P_{2},P_{3}$ is a facial walk. First we claim that $i \in \{1,2,3\}$, either $V(P_{i}) \cap V(F_{1}) = \{v,w\}$ or $V(P_{i}) \cap V(F_{2}) = \{v,w\}$. If not, then there exists an $i \in \{1,2,3\}$ such that we have two vertices $x_{1} \in V(F_{1}) \setminus \{v,w\}$ and $x_{2} \in V(F_{2}) \setminus \{u,v\}$ where $x_{1},x_{2} \in V(P_{i})$. Then consider the $(x_{1},x_{2})$-subpath on $P_{i}$ which we denote $P_{x_{1},x_{2}}$. We may assume that this subpath has no additional vertices from $F_{1}$ or $F_{2}$ (if it did, then we change the subpath to one with fewer vertices from $F_{1}$ and $F_{2}$). Note that this subpath partitions $C$ into two cycles which separate $v$ and $w$. But then as $C$ is a facial cycle, and $G$ is planar, and $G$ is $3$-connected any path from $u$ to $v$ intersects the subpath $P_{x_{1},x_{2}}$, contradicting that $P_{1},P_{2},P_{3}$ were three internally disjoint $(v,w)$-paths.

Now suppose $\{v,w\} \subsetneq V(P_{1}) \cap V(F_{1})$. We claim that both $V(P_{2}) \cap V(F_{1}) = \{v,w\}$ and $V(P_{3}) \cap V(F_{1}) = \{v,w\}$. By the above argument, $V(P_{1}) \cap V(F_{2}) = \{v,w\}$.  Therefore every vertex of $F_{1}$ either belongs to $P_{1}$ or lies in a cycle created from some subpath of $P_{1}$ and a subpath of $F_{1}$. But then since $C$ is a facial cycle, $G$ is planar, and $P_{2}$ and $P_{3}$ are internally disjoint from $P_{1}$, $P_{2}$ and $P_{3}$ both are internally disjoint from $F_{1}$. By similar arguments, at most one of $P_{2}$, $P_{3}$ contains a vertex in $F_{2}$ which is not $v,w$. Therefore at least one of $P_{1}$,$P_{2}$, and $P_{3}$ has no vertex from $C$ except for $v$ and $w$, completing the claim.  
\end{proof}

\begin{lemma}
\label{Planar $3$-connected}
Let $G$ be a $3$-connected planar graph and let $X = \{a,b,c,d\} \subseteq V(G)$. Then $G$ has either a $K_{4}(X)$-minor or a $W_{4}(X)$-minor.
\end{lemma}

\begin{proof}
We assume $G$ does not contain a $K_{4}(X)$-minor. Then by Corollary \ref{k4(x)planar},  $a,b,c$ and $d$ lie on a facial cycle, $F$ in that order.

First, suppose that $ac \in E(G)$. Without loss of generality we may assume $ac$ lies in the interior of $F$. Notice either the edge $bd \in E(G)$ or by Menger's Theorem  there is a $(b,d)$-path $P$ where $a,c \not \in V(P)$. In either case, this would contradict $F$ being a facial cycle, and $G$ not having a $K_{4}(X)$-minor. 

Therefore we can assume that $ac \not \in E(G)$. Then by Lemma \ref{3connectedplanarpaths} there is an $(a,c)$-path, $P_{a,c}$, which is internally disjoint from $F$. Without loss of generality we may assume that $P_{a,c}$ lies in the interior of $F$. Let $F_{a,c}$ be a facial walk from $a$ to $c$. Notice $C = P_{a,c} \cup F_{a,c}$ partitions the interior of $F$ into two regions, and that $b$ and $d$ lie in distinct regions. By Lemma \ref{3connectedplanarpaths}, we also have a $(b,d)$-path, $P_{b,d}$, which is disjoint from $F$. If $P_{b,d}$ lies on the exterior of $F$, then this would contradict that $F$ is a face. Thus $P_{b,d}$ lies in the interior of $F$. As $b$ and $d$ lie on differing sides of the partition of $C$, by the Jordan Curve Theorem, $P_{a,c}$ intersects $P_{b,d}$ at some vertex $v$.
   Then, contracting $F$ down to a four cycle on $a,b,c,d$ and contracting the subpaths from $v$ to $a,b,c,d$ on $P_{a,c}$ and $P_{b,d}$  internal vertices of the paths $P_{a,c}$ and $P_{b,d}$ to a vertex give a $W_{4}(X)$-minor. 
\end{proof}

Now that we have shown that all planar $3$-connected graphs have a $W_{4}(X)$ or a $K_{4}(X)$-minor, we will prove that all $3$-connected graphs have a $W_{4}(X)$ or $K_{4}(X)$-minor by reducing to the planar case.

\begin{lemma}
\label{planarreduction}
Let $G$ be a $3$-connected graph and $X = \{a,b,c,d\} \subseteq V(G)$. Suppose $G$ is a spanning subgraph of an $\{a,b,c,d\}$-web, $H^{+} = (H,F)$. Then there is a $3$-connected planar graph $K$ such that $K$ is a minor of $G$.
\end{lemma}

\begin{proof}
For every triangle $T \in H$, consider the graph $G[V(F_{T})]$ and let $C_{1}, \ldots, C_{n}$ be the connected components of $G[V(F_{T})]$. Now for every triangle $T \in H$, contract $C_{1}$ down to a vertex, which we will call $v_{T}$, and contract all of $C_{2}, \ldots, C_{n}$ to any of the vertices of $T$, removing any parallel edges or loops created.   Let $K$ be the graph obtained from $G$ after applying the above construction. We claim $K$ is planar and $3$-connected. 

First we show $G$ is $3$-connected. Notice for each $T$ where $V(T) = \{x_{1},x_{2},x_{3}\}$, there is a vertex $v \in V(C_{1})$ such that $vx_{i} \in E(G)$ for all $i \in \{1,2,3\}$. Thus $v_{T}$ is adjacent to $x_{i}$ for all $i \in \{1,2,3\}$. Then $v_{T}$ has $3$ internally disjoint paths to any other vertex of $K$ as $G$ is $3$ connected. Now let $x,y 
\in V(H)$ and consider $3$ internally disjoint $(x,y)$-paths $P_{1},P_{2}$ and $P_{3}$ in $G$. Now notice that for every triangle $T \in H$, at most one of $P_{1},P_{2}$ or $P_{3}$ uses vertices from $V(F_{T})$, as $|V(T)| = 3$. Therefore in $K$, if necessary, we can reroute the path using vertices from $V(F_{T})$ to use $v_{T}$, and thus we still have $3$ internally disjoint $(x,y)$-paths. Therefore $K$ is $3$-connected.
 
 So it suffices to show that $K$ is planar. Notice that if given a some planar embedding of $H$, to all the faces bounded by a triangle, we can add a vertex to the interior of the face and make the vertex adjacent to every vertex in the triangle and remain planar. The graph from that construction contains $K$ as a subgraph, so $K$ is planar, completing the proof.  
\end{proof}

\begin{theorem}
\label{3connectivityW4K4}
Let $G$ be a $3$-connected graph and $X = \{a,b,c,d\} \subseteq V(G)$. Then $G$ either has a $K_{4}(X)$-minor or a $W_{4}(X)$-minor.
\end{theorem}

\begin{proof}
We may assume $G$ does not have a $K_{4}(X)$-minor. By Corollary \ref{k4(x)3conn}, $G$ is a spanning subgraph of an $\{a,b,c,d\}$-web. Then by Lemma \ref{planarreduction}, $G$ has a $3$-connected planar minor $K$. By Lemma \ref{Planar $3$-connected}, $K$ has $W_{4}(X)$-minor and thus $G$ has a $W_{4}(X)$-minor. 
\end{proof}

\subsection{An alternative characterization of $2$-connected graphs without $W_{4}(X)$ and $K_{4}(X)$-minors}

We have a characterization of graphs without $W_{4}(X)$ and $K_{4}(X)$ minors, given by Theorem \ref{3connectivityW4K4}, however it relies on reducing to the $3$-connected case, which will not be useful as $K_{2,4}$ and $L$ are both $2$-connected (and as every $3$-connected graph either has a $K_{4}(X)$ or $W_{4}(X)$-minor). Therefore it would be nice to have a characterization in terms of spanning subgraphs as in Theorem \ref{k4free}. This section gives such a characterization  (see Theorem \ref{w4cuts}).

\begin{lemma}
Let $G$ be $2$-connected and a spanning subgraph of a class $\mathcal{A},\mathcal{B}$ or $\mathcal{C}$ graph from Theorem \ref{k4free}. Then $G$ does not have a $W_{4}(X)$-minor. 
\end{lemma}

\begin{proof}
We treat each case separately.

If $G$ is the spanning subgraph of a class $\mathcal{A}$ graph, then $\{d,e\}$ is a $2$-vertex cut. Applying Lemma \ref{oneterminalinthecut2conn} and Lemma \ref{noW4minor} successively to the separation induced by $\{d,e\}$, we see $G$ does not have a $W_{4}(X)$-minor.

If $G$ is the spanning subgraph of a class $\mathcal{B}$ graph, then $\{e,f\}$ is a $2$-vertex cut. Applying  Lemma \ref{2sepW4} and Lemma \ref{noW4minor} successively to the separation induced by $\{e,f\}$, we see $G$ does not have a $W_{4}(X)$-minor.

 If $G$ is the spanning subgraph of a class $\mathcal{C}$ graph, then $\{g,f\}$ is a $2$-vertex cut. Apply Lemma \ref{2sepW4} to the separation induced by $\{g,f\}$ and let $G_{1}$ and $G_{2}$ be the graphs obtained from Lemma \ref{2sepW4}. Without loss of generality, let $G_{1}$ be the graph such that  $\{f,g\}$ induces a separation satisfying Lemma \ref{noW4minor}. Then $G_{1}$ does not have a $W_{4}(X)$-minor.  Then in $G_{2}$, to the separation induced by the $2$-vertex cut $\{g,e\}$, apply Lemma \ref{oneterminalinthecut2conn} to obtain a graph $G_{3}$. Then in $G_{3}$, notice that Lemma \ref{noW4minor} applies, thus $G_{3}$ does not have a $W_{4}(X)$-minor, and then by working back through the lemmas, we get that $G$ does not have a $W_{4}(X)$-minor. 
\end{proof}

Notice that by applying Lemma \ref{2sepW4} to $2$-connected spanning subgraphs of class $\mathcal{E}$ and $\mathcal{F}$ graphs, we see that one of these graphs has a $W_{4}(X)$-minor if and only if the corresponding web from class $\mathcal{D}$ has a $W_{4}(X)$-minor. So now we restrict ourselves to looking at $\{a,b,c,d\}$-webs. First we show that $\{a,b,c,d\}$-webs always have a cycle which contains $\{a,b,c,d\}$.

\begin{observation}
\label{planarface}
Suppose $G$ is a planar spanning subgraph of some $\{a,b,c,d\}$-web $H^{+} = (H,F)$. Then the graph $G'$ defined by $V(G') = V(G)$ and $E(G') = E(G) \cup \{ab,bc,cd,da\}$ is planar, and furthermore the cycle $C$ with edge set $ab,bc,cd,da$ is a face in $G'$. 
\end{observation}

\begin{proof}

Fix a planar embedding $\tilde{H}$ of $H$. As $G$ is planar, for each triangle $T \in H$ the graph $F_{T} \cup T$ is planar. For each triangle $T \in H$, fix a  planar embedding of $F_{T} \cup T$ where $T$ is the outerface. Then we can combine the planar embedding of $H$ with the planar embeddings of $F_{T} \cup T$ by joining $F_{T} \cup T$ to the appropriate triangle. This implies that the graph $H^{+}$ is planar, and thus $G'$ is planar. Additionally, notice in  $H^{+}$, the cycle with edge set $ab,bc,cd,da$ is a face in $H^{+}$, and thus the cycle with edge set $ab,bc,cd,da$ in $G'$ is a face in $G'$. 
\end{proof}

\begin{lemma}
\label{webcycleplanar}
Let $G$ be a $2$-connected planar graph and let $X = \{a,b,c,d\} \subseteq V(G)$. If $G$ is the spanning subgraph of an $\{a,b,c,d\}$-web, then there is a cycle, $C$, such that $X \subseteq V(C)$.
\end{lemma}

\begin{proof}
Let $G_{1},\ldots,G_{n}$ be a sequence of graphs where $G_{1} = H^{+}$, $G_{n} = G$ and $G_{i+1} = G_{i} \setminus \{e\}$ where $e$ is some edge of $G_{i}$. We proceed by induction on $i$.  When $i =1$, $G_{1} = H^{+}$ and the $4$-cycle on $a,b,c,d$ in $H$ is our desired cycle.

 Now consider $G_{i}$, $i \geq 2$ and let $e=xy \in E(G)$ be the edge such that $G_{i} = G_{i-1} \setminus \{e\}$. By induction, $G_{i-1}$ contains a cycle $C$ containing $X$. We may assume that $e \in E(C)$ as otherwise $C$ completes the claim. Let $P= C \setminus \{e\}$.   Without loss of generality, suppose that $a,b,c,d$ appear in that order in $C$, and that $x$ and $y$ lie on the $(a,d)$-path, $P_{a,d}$, in $C$ in $G_{i-1}$ which does not contain $c$ and $d$, such that $a,x,y,d$ appear in that order. Similarly define paths $P_{a,b}, P_{b,c}$ and $P_{d,a}$. Additionally define $P_{a,x}$ to be the $(a,x)$-subpath on $P_{a,d}$ and $P_{y,d}$ to be the $(y,d)$-subpath on $P_{a,d}$.
 
  By observation \ref{planarface}, in the graph $G'_{i-1}$, we have that $P_{a,d} \cup \{ad\}$, $P_{a,b} \cup \{ab\}, P_{b,c} \cup \{bc\}$ and $P_{c,d} \cup \{cd\}$ are cycles. Furthermore, we may assume that in a planar embedding of $G'_{i-1}$, no edges cross $ad,ab,bc,$ or $cd$. We define the interior of $C$ to be the component of $G'_{i-1} -C$ which does not contain any of $ab,bc,cd$ or $ab$, and the exterior is the component which is not the interior. We abuse notation and will refer also interior and exterior of $C$ in $G_{i-1}$. Notice that if we have a path whose two endpoints are on $P_{a,d}$ and whose vertices only use exterior vertices, then that path does not contain any vertices from $P_{a,b} \cup P_{b,c} \cup P_{c,d} \setminus \{a,d\}$.
  
As $G_{i}$ is $2$-connected, there is an $(x,y)$-path, $P'$ such that $P' \neq P$. If $V(P') \cap V(P) = \{x,y\}$, then $P' \cup P$ is our desired cycle. Therefore we may assume that every $(x,y)$-path intersects $P$.
   If there is any path $P''$ from a vertex $x' \in V(P_{a,x})$ to a vertex $y' \in  V(P_{y,d})$ using only vertices from the exterior, then $P_{a,x'} \cup P'' \cup P_{y',d} \cup P_{c,d} \cup P_{b,c} \cup P_{a,d}$ is a cycle, since no edges cross the edge $ad \in E(G'_{i-1})$. Here $P_{a,x'}$ is the $(a,x')$-subpath on $P_{a,x}$ and $P_{y',d}$ is the $(y',d)$-subpath on $P_{y,d}$. Therefore we assume no such path of that form exists. By essentially the same argument, we can assume no path of that form exists with vertices in the interior which does not intersect any of $P_{a,b}, P_{b,c}$ and $P_{c,d}$.
   
Since $G_{i}$ is $2$-connected, there are two internally disjoint $(x,y)$-paths, say $P'$ and $P''$. By our previous discussion, we may assume that both $P'$ and $P''$ are not $P$, and that both $P'$ and $P''$ intersect $P$. Suppose that $P'$ intersects all of $P_{b,c}$, $P_{c,d}$ and $P_{a,b}$. Notice that by planarity, these paths cannot cross, so it is well defined to say that one of $P' - \{x,y\}$ or $P'' - \{x,y\}$ lies on the interior of the cycle $P' \cup \{xy\}$ or $P'' \cup \{xy\}$.  Without loss of generality, suppose that $P' -\{x,y\}$ lies on the interior of $P'' \cup \{xy\}$. Then by our previous discussion, and planarity, the only way for $P''$ to be an internally disjoint $(x,y)$-path is for $P''$ to intersect $P_{a,b}$, go through the exterior of $C$ and intersect $P_{a,b}$ again, do this some finite number of times, then intersect $P_{b,c}$, go through the exterior of $C$ and intersect $P_{b,c}$ again, do this some finite number of times, then intersect $P_{c,d}$, go through the exterior of $C$ and intersect $P_{c,d}$ again. Then we can reroute $P$ along paths in the exterior of $C$ along $P''$ to get a new path $P'''$ such that $P'''$ contains all of $a,b,c,d$ and $V(P''') \cap V(P') \setminus \{x,y\} = \emptyset$. But then $P''' \cup P'$ a cycle satisfying the claim. We note the same strategy holds if $P'$ intersects any subset of $P_{b,c}$, $P_{c,d}$ and $P_{a,b}$. Therefore there is a cycle containing $a,b,c$, and $d$ in $G$. 
\end{proof}

\begin{figure}
\begin{center}
\includegraphics[scale = 0.5]{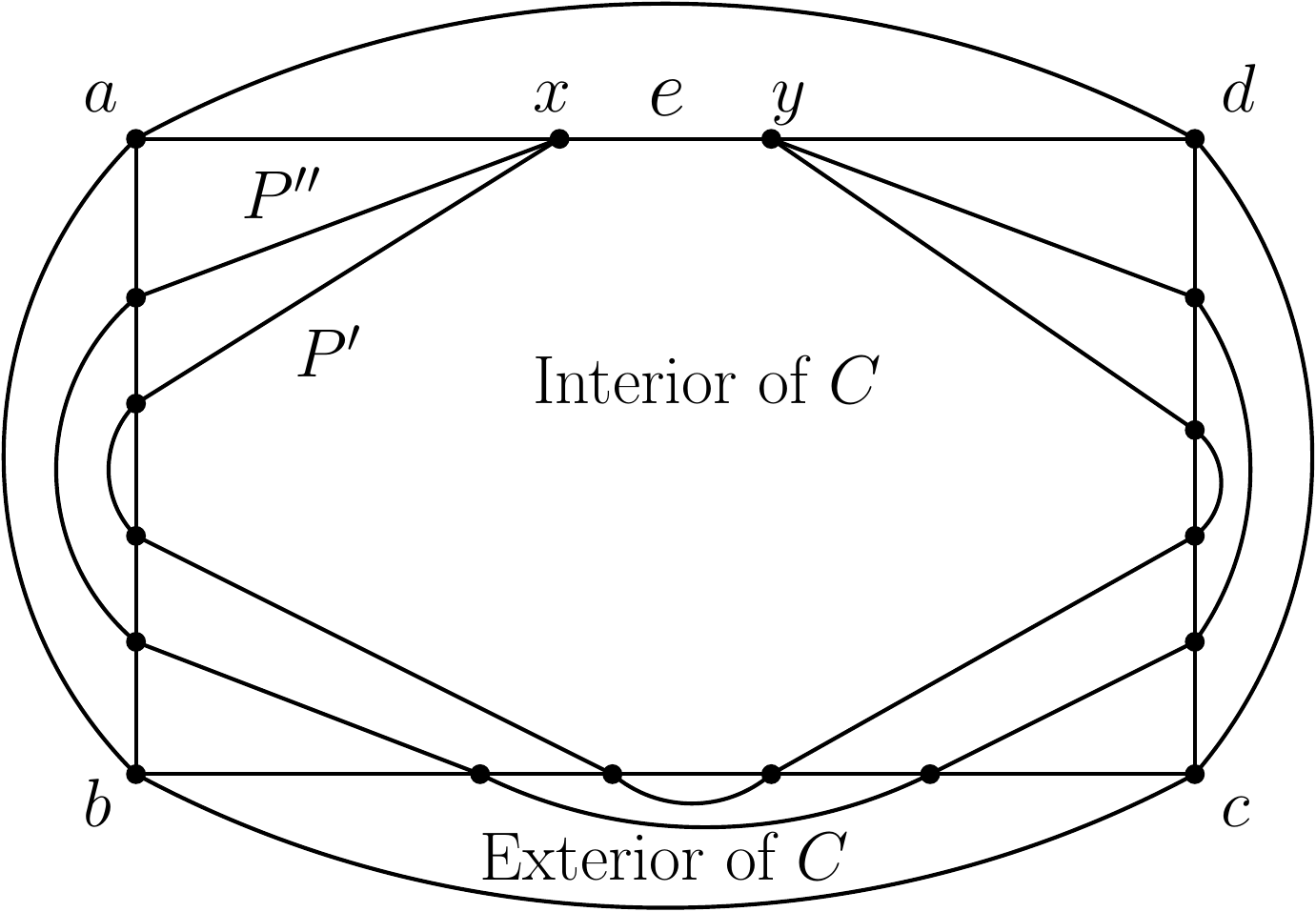}
\caption{The situation in Lemma \ref{webcycleplanar}. The edges $ab,bc,cd$ and $ab$ exist only in $G'_{i}$. The path $P'''$ is obtained by rerouting along paths in the exterior of $C$ which are subpaths of $P''$.}
\end{center}
\end{figure}

\begin{corollary}
\label{webcycle}
Let $G$ be a $2$-connected graph and let $X = \{a,b,c,d\} \subseteq V(G)$. If $G$ is the spanning subgraph of an $\{a,b,c,d\}$-web, $H^{+} = (H,F)$, then there is a cycle, $C$, such that $X \subseteq V(C)$.
\end{corollary}

\begin{proof}
By Lemma \ref{webcycleplanar}, we may assume that $G$ is non-planar. For each triangle $T \in H$, consider the graph $G[V(F_{T})]$ and let $M_{1},\ldots,M_{n}$ be the connected components of $G[V(F_{T})]$. Now for each triangle $T \in H$ let $G'$ be the graph obtained by contracting each connected component, $M_{i}$, down to a vertex, call it $v^{i}_{T}$, $i \in \{1,\ldots,n\}$.

 Now consider some triangle $T \in H$ such that $V(T) = \{x_{1},x_{2},x_{3}\}$. First suppose there exists an $i \in \{1,\ldots,n\}$ such that $v^{i}_{T}$ is adjacent to $x_{j}$ for all $j \in \{1,2,3\}$. Then for all $v^{k}_{T}$, $k \in \{1,\ldots,n\}, k \neq i$, contract $v^{k}_{T}$ to any vertex of $T$.
 
 Now suppose that there was no $i \in \{1,\ldots,n\}$ such that $v^{i}_{T}$ is adjacent to $x_{j}$ for all $j \in \{1,2,3\}$. Since $G$ is $2$-connected, that means that for all $i \in \{1,\ldots,n\}$, $v^{i}_{T}$ is adjacent to exactly two of $x_{1},x_{2}$ and $x_{3}$. Let $v^{i}_{T}, v^{j}_{T}, v^{k}_{T}$ be vertices such that $v^{i}_{T}$ is adjacent to $x_{1},x_{2}$, and $v^{j}_{T}$ is adjacent to $x_{1},x_{3}$, and $v^{k}_{T}$ is adjacent to $x_{2},x_{3}$ for $i,j,k \in \{1,\ldots,n\}$, $i \neq j \neq k$. Then for all $v^{l}_{T}$, $l \neq i,j,k$, contract $v^{l}_{T}$ to an arbitrary vertex of $T$. Let $G'$ be the resulting graph after applying the above procedure to every triangle $T \in H$ to $G$. We note that some subset of the vertices $v^{i}_{T}, v^{j}_{T}, v^{k}_{T}$ may not exist, but in this case we just do not have that subset of vertices in $G'$ . We claim that $G'$ is planar and $2$-connected.
 
First we show that $G'$ is $2$-connected. Notice that since $G$ is $2$-connected, for every $T \in H$, all of the $v^{i}_{T}$ have $2$ internally disjoint paths to every other vertex. Now consider two vertices $x,y \in V(H)$ and let $P_{1},P_{2}$ be two internally disjoint $(x,y)$-paths in $G$. Notice that for each triangle $T \in H$, at most one of $P_{1}$ or $P_{2}$ uses vertices from any connected component of $G[V(F_{T})]$, so these paths exist in $G'$ by possibly augmenting them to the appropriate vertex $v^{i}_{T},v^{j}_{T}$ or $v^{k}_{T}$. Therefore $G'$ is $2$-connected. 

So it suffices to show that $G'$ is planar. Take any planar embedding of $H$, and for every face bounded by a triangle $T$, either add a vertex adjacent to all of the vertices of $T$ to the interior of the face, or add three vertices $v_{1},v_{2},v_{3}$ to the interior of the face such that for all $i \in \{1,2,3\}$, $v_{i}$ is adjacent to two vertices of $T$, and $N(v_{i}) \neq N(v_{j})$ if $i \neq j$. Note that  the resulting graph is planar. Furthermore, using this construction, we can obtain a planar graph $K$ such that $G'$ is a subgraph of $K$, and thus $G'$ is planar. 

Now since $G'$ is $2$-connected, planar, and by construction we did not contract any of $a,b,c$ or $d$ together, we can apply Lemma \ref{webcycleplanar}. Thus $G'$ has a cycle $C'$ containing $X$. But then $G$ has a cycle $C$ containing $X$, obtained by extending $C'$ along the contracted edges, if necessary. 
\end{proof}

We note that we cannot extend the above lemma to the non-web $K_{4}(X)$-free classes, as none of them have a cycle containing $X$. 

\begin{lemma}
Let $G$ be a $2$-connected graph and $X = \{a,b,c,d\} \subseteq V(G)$. Suppose $G$ is the spanning subgraph of a class $\mathcal{A},\mathcal{B},\mathcal{C}, \mathcal{E}$ or $\mathcal{F}$ graph. Then there is no cycle $C$ in $G$ which contains $X$. 
\end{lemma}

\begin{proof}
Suppose $G$ is the spanning subgraph of a class $\mathcal{A}$ graph. Then notice that $G - \{d,e\}$ has at least $3$ components, say $A,B,$ and $C$ such that $a \in A$, $b \in B$, and $c \in D$. Since $|\{d,e\}|=2$, there is no cycle containing $a,b,c$ in $G$. Therefore there is no cycle containing $a,b,c$ and $d$ in $G$. 

Suppose $G$ is the spanning subgraph of a class $\mathcal{B}$ graph. Notice $G - \{e,f\}$ has at least four components and each vertex of $X$ lies in a distinct component. Then there is no cycle $C$ in $G$ which contains $X$.  

Suppose $G$ is the spanning subgraph of a class $\mathcal{C}$ graph. Then notice $G - \{e,f\}$ has at least $3$ components and at least $3$ components  contain vertices from $X$. Then $G$ does not have a cycle containing $X$.

Suppose $G$ is the spanning subgraph of a class $\mathcal{E}$ or $\mathcal{F}$ graph. Then notice $G - \{e,f\}$ has at least $3$ components and at least $3$ components  contain vertices from $X$. Then $G$ does not have a cycle containing $X$.
\end{proof}

Now we will give a characterization of graphs not containing $K_{4}(X)$ and $W_{4}(X)$-minors of a different flavour that from the previous sections. First we give some definitions.

A common idea which appears in the study of graph minors is the notion of a $k$-dissection, which is simply a sequence of nested $k$-separations. Formally, a sequence $((A_{1},B_{1}),\\ \ldots,(A_{n},B_{n}))$ is a \textit{$k$-dissection} if for all $i \in \{1,\ldots,n\}$,  $(A_{i},B_{i})$ is a $k$-separation, and for all $i \neq n$,  $A_{i} \subseteq A_{i+1}$, and $B_{i+1} \subseteq B_{i}$. We will use special types of $k$-dissections. 

\begin{definition}
Let $G$ be a $2$-connected graph and $X = \{a,b,c,d\} \subseteq V(G)$. Let $n$ be any positive integer. Let $((A_{1},B_{1}), \ldots,(A_{n},B_{n}))$ be a $2$-dissection. If $A_{i} \cap  B_{i} \cap  A_{i+1} \cap B_{i+1} \neq \emptyset$  for all $i \in \{1,\ldots, n-1\}$, then we will say the $2$-dissection is a $2$-chain. Let $((A_{1},B_{1}),\ldots,(A_{n},B_{n}))$ be a $2$-chain. Suppose both $A_{1} \cap B_{1}$ and $A_{n} \cap B_{n}$ contain at least one vertex of $X$, and there is exactly one vertex of $X$ in $A_{1} \setminus (A_{1} \cap B_{1})$, and exactly one vertex of $X$ in $B_{n} \setminus (A_{n} \cap B_{n})$. Then we say $((A_{1},B_{1}),\ldots, (A_{n},B_{n}))$ is a \emph{terminal separating $2$-chain}. 
\end{definition}

\begin{definition}
Let $G$ be a $2$-connected graph and $X = \{a,b,c,d\} \subseteq V(G)$. Suppose there are three distinct $2$-separations $(A_{1},B_{1}),(A_{2},B_{2}),(A_{3},B_{3})$. We say these separations form a \emph{triangle} if $A_{1} \cap B_{1} = \{x,y\}$, $A_{2} \cap B_{2} = \{x,v\}$ and $A_{3} \cap B_{3} = \{v,y\}$ for distinct vertices $x,y,v \in V(G)$. For notational convenience, we will enforce that in a triangle, $(A_{i} \setminus (A_{i} \cap B_{i})) \cap (A_{j} \setminus (A_{j} \cap B_{j})) = \emptyset$ for any $i,j \in \{1,2,3\}$, $i \neq j$.  We say a triangle $(A_{1},B_{1}),(A_{2},B_{2}),(A_{3},B_{3})$, is \emph{terminal separating} if exactly two vertices of $X$ are contained in $A_{1}$, exactly one vertex of $X$ is contained in $A_{2} \setminus (A_{2} \cap B_{2})$ and exactly one vertex of $X$ is contained in $A_{3} \setminus (A_{3} \cap B_{3})$. Two triangles  $((A^{1}_{1},B^{1}_{1}), (A^{1}_{2},B^{1}_{2}), (A^{1}_{3},B^{1}_{3}))$, $((A^{2}_{1},B^{2}_{1}),(A^{2}_{2},B^{2}_{2}),(A^{2}_{3},B^{2}_{3}))$, are distinct if there exists an $i \in \{1,2,3\}$, such that for $A^{2}_{i}$, $A^{1}_{j} \cap B^{1}_{j} \subseteq A^{2}_{i}$ for all $j \in \{1,2,3\}$ and there exists an $i \in \{1,2,3\}$ such that for  $A^{1}_{i}$, $A^{2}_{j} \cap B^{2}_{j} \subseteq A^{1}_{i}$ for all $j \in \{1,2,3\}$.  
\end{definition}

\begin{figure}
\begin{center}
\includegraphics[scale=0.5]{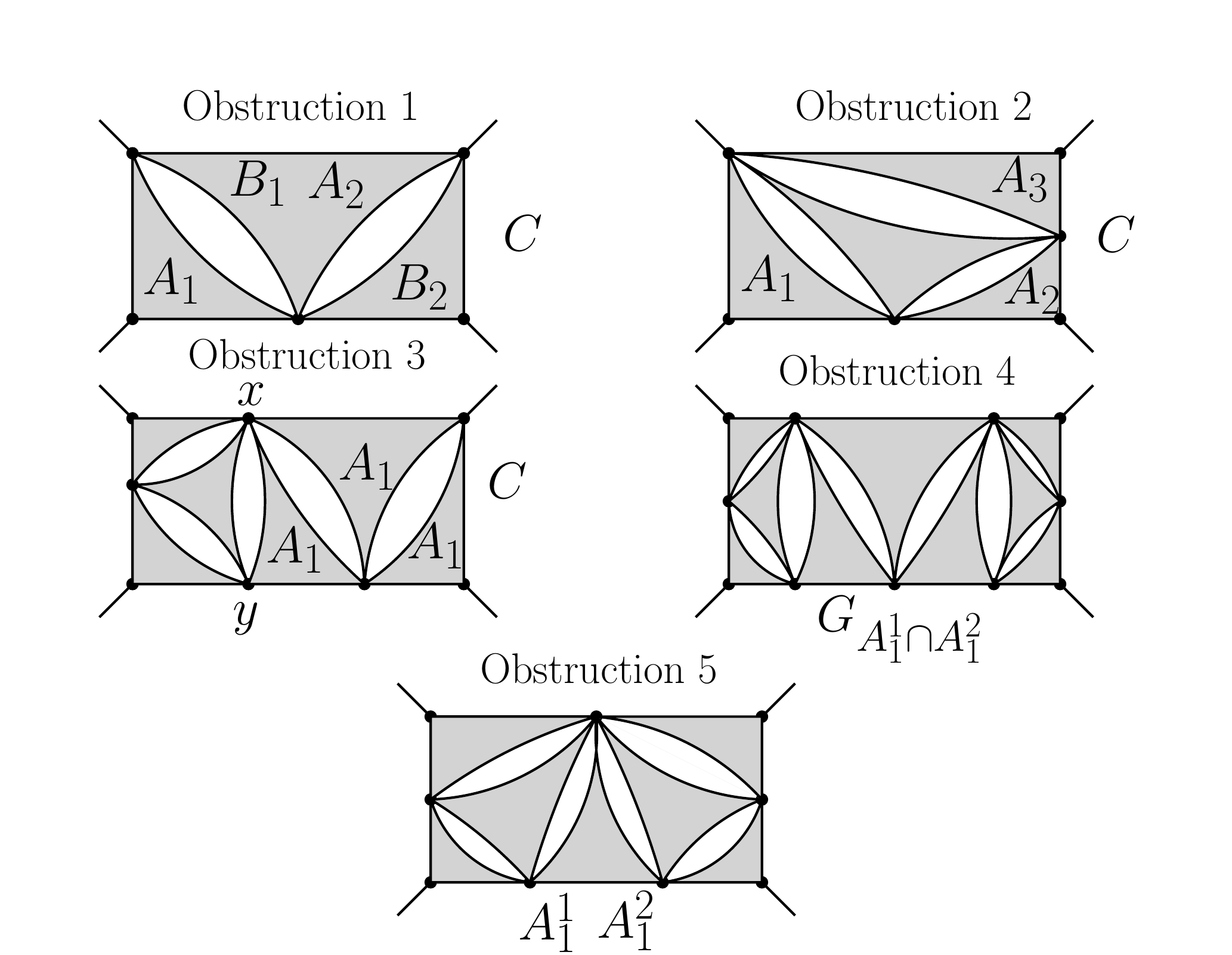}
\caption{The obstructions for Theorem \ref{w4cuts}. The vertices of $X$ are represented by vertices which have an edge not adjacent to a vertex. Curved lines represent a $2$-separation. Shaded sections are spanning subgraphs of webs.}
\label{obstructionset}
\end{center}
\end{figure}

\begin{theorem}
\label{w4cuts}
Let $G$ be a $2$-connected graph and $X = \{a,b,c,d\} \subseteq V(G)$. Suppose $G$ is a spanning subgraph of an $\{a,b,c,d\}$-web. Then $G$ is $W_{4}(X)$-minor free if and only if for every cycle $C$ where $X \subseteq V(C)$, we have one of the following obstructions (See Figure \ref{obstructionset}).  
\begin{enumerate}
\item{There is a terminal separating $2$-chain $((A_{1},B_{1}),\ldots,(A_{n},B_{n}))$ such that $A_{i} \cap B_{i} \subseteq V(C)$ for all $i \in \{1,\ldots, n\}$.}
\item{There is a terminal separating triangle $(A_{1},B_{1}),(A_{2},B_{2}),(A_{3},B_{3})$, such that $A_{1} \cap B_{1}$ contains a vertex from $X$ and $A_{i} \cap B_{i} \subseteq V(C)$, for all $i \in \{1,2,3\}$.}
\item{There is a terminal separating triangle $(A_{1},B_{1}),(A_{2},B_{2}),(A_{3},B_{3})$ such that $A_{1} \cap B_{1} = \{x,y\}$, where $x,y \not \in X$. Furthermore, the graph $G_{A_{1}} = G[A_{1}] \cup \{xy\}$ and $C_{A} = G[V(C) \cap A] \cup \{xy\}$  has a terminal separating $2$-chain $(A'_{1},B'_{1}),\ldots,(A'_{n},B'_{n})$ where we let $x$ and $y$ replace the two vertices in $X$ from $G$ not in $G_{A_{1}}$. Additionally, $A_{i} \cap B_{i} \subseteq V(C)$, for all $i \in \{1,2,3\}$, and $A'_{i} \cap B'_{i} \subseteq V(C_{A})$ for all $i \in \{1,\ldots,n\}$.}
\item{There are two distinct terminal separating triangles $((A^{1}_{1},B^{1}_{1}),(A^{1}_{2},B^{1}_{2}), (A^{1}_{3},B^{1}_{3}))$, and $((A^{2}_{1},B^{2}_{1}),(A^{2}_{2},B^{2}_{2}),(A^{2}_{3},B^{2}_{3}))$  where for all $i \in \{1,2,3\}$,  $A^{1}_{i} \cap B^{1}_{i} \subseteq A^{2}_{1}$ and $A^{2}_{i} \cap B^{2}_{i} \subseteq A^{1}_{3}$. Furthermore, consider the graph $G_{A^{2}_{1} \cap A^{1}_{1}} = G[A^{2}_{1} \cap A^{1}_{1}] \cup \{xy | x,y \in A^{i}_{1} \cap B^{i}_{1}, i \in \{1,2\}\}$ and the cycle $C' = G[V(C) \cap A^{2}_{1} \cap A^{1}_{1}] \cup \{xy | x,y \in A^{i}_{1} \cap B^{i}_{1}, i \in \{1,2\}\}$. Let $X'$ be defined to be the vertices $A^{2}_{1} \cap B^{2}_{1}$ and $A^{2}_{3} \cap B^{2}_{3}$. Then there is a terminal separating $2$-chain with respect to $X'$, $((A_{1},B_{1}),\ldots,(A_{n},B_{n}))$, in $G_{A^{2}_{1} \cap A^{1}_{1}}$ such that $A_{i} \cap B_{i} \subseteq V(C')$.}
\item{There are $2$ distinct terminal separating triangles $((A^{1}_{1},B^{1}_{1}), (A^{1}_{2},B^{1}_{2}), (A^{1}_{3},B^{1}_{3})),\\  ((A^{2}_{1},B^{2}_{1}), (A^{2}_{2},B^{2}_{2}),(A^{2}_{3},B^{2}_{3}))$ where for all $i \in \{1,2,3\}$,  $A^{1}_{i} \cap B^{1}_{i} \subseteq A^{2}_{1}$  and $A^{2}_{i} \cap B^{2}_{i} \subseteq A^{1}_{1}$, the set $A^{2}_{1} \cap B^{2}_{1} \cap A^{1}_{1} \cap B^{1}_{1}$ is not empty and $A^{j}_{i} \cap B^{j}_{i} \subseteq V(C)$ for all $i \in \{1,2,3\}$, and $j \in \{1,2\}$.}
\end{enumerate}
\end{theorem}

Before proving this, we prove some lemmas to make the proof cleaner.

\begin{lemma}
\label{obstructionsareobstructions}
Let $G$ be a $2$-connected graph, $X = \{a,b,c,d\} \subseteq V(G)$. Let $C$ be a cycle in $G$ such that $X \subseteq V(C)$. If any of the obstructions in Theorem \ref{w4cuts} occur, then $G$ does not have a $W_{4}(X)$-minor.
\end{lemma}

\begin{proof}
We deal with each case separately. In each case suppose $G$ is a minimal counterexample with respect to the number of vertices. 

\textbf{Case 1:} Suppose we have a terminal separating $2$-chain $((A_{1},B_{1}),\ldots,(A_{n},B_{n}))$. If $n=1$, then $(A_{1},B_{1})$ is a $2$-separation satisfying the conditions in Lemma \ref{noW4minor} and thus $G$ does not have a $W_{4}(X)$-minor. Therefore we assume $n \geq 2$. Then $(A_{1},B_{1})$ satisfies the conditions in Lemma \ref{oneterminalinthecut2conn}. Let $G'$ be the graph obtained after applying Lemma \ref{oneterminalinthecut2conn} to $(A_{1},B_{1})$. Then $G$ has a $W_{4}(X)$-minor if and only if $G'$ has a $W_{4}(X_{1})$-minor, where $X_{1}$ is defined from Lemma \ref{oneterminalinthecut2conn}. Notice in $G'$, $((A_{2},B_{2}),\ldots,(A_{n},B_{n}))$ is a terminal separating $2$-chain satisfying the properties of obstruction $1$ for the cycle $G'[V(C) \cap B_{1}]$ when we replace the vertex of $X$ in $A_{1} \setminus (A_{1} \cap B_{1})$ with the vertex in $(A_{1} \cap B_{1}) \setminus X$. Since $G$ is a vertex minimal counterexample, $G'$ has no $W_{4}(X)$-minor, and thus $G$ has no $W_{4}(X)$-minor.

\textbf{Case 2:} Suppose there is a terminal separating triangle $(A_{1},B_{1}),(A_{2},B_{2}),(A_{3},B_{3})$ satisfying properties in the second obstruction.  Then we apply Lemma \ref{oneterminalinthecut2conn} to $(A_{1},B_{1})$ giving a new graph $G'$ which has a $W_{4}(X)$-minor if and only if $G'$ has a $W_{4}(X_{1})$-minor, where $X_{1}$ is defined from Lemma \ref{oneterminalinthecut2conn}. In $G'$, apply Lemma \ref{oneterminalinthecut2conn} to $(A_{2},B_{2})$ giving a graph $G''$. Then $G''$ has a $W_{4}(X_{2})$-minor if and only if $G$ has a $W_{4}(X)$-minor. Observe in $G''$,  $(A_{3},B_{3})$ is a separation satisfying Lemma \ref{noW4minor} and thus $G''$ does not have a $W_{4}(X_{2})$-minor. But $G''$ has a $W_{4}(X)$-minor if and only if $G$ has a $W_{4}(X)$-minor, so therefore $G$ has no $W_{4}(X)$-minor.

\textbf{Case 3:} Suppose there is a terminal separating triangle $(A_{1},B_{1}),(A_{2},B_{2}),(A_{3},B_{3})$ and a terminal separating $2$-chain $((A'_{1},B'_{1}),\ldots (A'_{n},B'_{n}))$ in the graph $G_{A_{1}}$, as in the third obstruction. Apply Lemma \ref{oneterminalonesideothersother} to $(A_{2},B_{2})$ to obtain two new reduced graphs $G'_{1}$ and $G'_{2}$. Now in one of $G'_{1}$ and $G'_{2}$, we can apply Lemma \ref{oneterminalinthecut2conn} to $(A_{1},B_{1})$ to obtain a graph $G''$, and in one of $G'_{1}$ and $G'_{2}$, we can apply Lemma \ref{oneterminalinthecut2conn} twice to $(A_{3},B_{3})$ and $(A_{1},B_{1})$ to obtain the graph $G''$ (note that the graph $G''$ obtained from both $G'_{1}$ and $G'_{2}$ is indeed the same graph). Then $G''$ has a $W_{4}(X)$-minor if and only if $G$ has a $W_{4}(X)$-minor. Notice that in the graph $G''$, $((A'_{1},B'_{1}),
\ldots,(A'_{n},B'_{n}))$ is a terminal separating $2$-chain satisfying obstruction $1$. Then by case one, $G''$ has no $W_{4}(X)$-minor, and thus $G$ has no $W_{4}(X)$-minor.

\textbf{Case 4:} Suppose we have the fourth obstruction in Theorem \ref{w4cuts} and let $(A_{1},B_{1}), \\ (A_{2},B_{2}),(A_{3},B_{3})$ be one of the terminal separating triangles. If we apply Lemma \ref{oneterminalonesideothersother} and Lemma \ref{oneterminalinthecut2conn} to $(A_{1},B_{1}),(A_{2},B_{2}),(A_{3},B_{3})$ as we did in case $3$, we obtain a graph $G'$ which has a $W_{4}(X)$-minor if and only if $G$ does. Furthermore, in the graph $G'$, the other terminal separating triangle and terminal separating $2$-chain given by the fourth obstruction for $G$ and $C$ satisfy the properties of obstruction $3$. But then by case $3$, $G'$ has no $W_{4}(X)$-minor, and thus $G$ has no $W_{4}(X)$-minor. 

\textbf{Case 5:} Let $(A^{1}_{1},B^{1}_{1}), (A^{1}_{2},B^{1}_{2}), (A^{1}_{3},B^{1}_{3})$ and $(A^{2}_{1},B^{2}_{1}), (A^{2}_{2},B^{2}_{2}),(A^{2}_{3},B^{2}_{3})$ be the two distinct triangles satisfying the properties in obstruction five. Applying Lemma \ref{oneterminalonesideothersother} and Lemma \ref{oneterminalinthecut2conn} to $(A^{1}_{1},B^{1}_{1}), (A^{1}_{2},B^{1}_{2}), (A^{1}_{3},B^{1}_{3})$ as in case $3$, we obtain a graph $G'$ which has a $W_{4}(X)$-minor if and only if $G$ has a $W_{4}(X)$-minor. If $A^{1}_{1} \cap B^{1}_{1} = A^{2}_{1} \cap B^{2}_{1}$, then in $G'$, the separations $(A^{2}_{2}, B^{2}_{2})$ and $(A^{2}_{3},B^{2}_{3})$ are a terminal separating $2$-chain in $G'$ satisfying obstruction $1$ on the cycle $C' = G'[V(C) \cap B^{1}_{1}]$. Then by case $1$, $G'$ has no $W_{4}(X)$-minor. Therefore we assume that $A^{1}_{1} \cap B^{1}_{1} \neq A^{2}_{1} \cap B^{2}_{1}$. Then in $G'$ on the cycle $C'$, $(A^{2}_{1},B^{2}_{1}), (A^{2}_{2},B^{2}_{2}),(A^{2}_{3},B^{2}_{3})$ is a terminal separating triangle, so by case $2$, $G'$ has no $W_{4}(X)$-minor, and thus $G$ has no $W_{4}(X)$-minor.    
\end{proof}

\begin{lemma}
\label{nicesplittinglemma}
Let $G$ be a $2$-connected graph, $X = \{a,b,c,d\} \subseteq V(C)$. Let $C$ be a cycle in $G$ such that $X \subseteq V(C)$. Suppose that $a,b,c,d$ appear in that order on $C$, and suppose that $cd \in E(C)$. Then
\begin{itemize}
  \item{if there is a terminal separating triangle satisfying the properties of obstruction $2$, then  $A_{1} \cap B_{1}$ contains exactly one of vertices $c$ or $d$.}
  \item{If $C$ has a terminal separating triangle and a terminal separating $2$-chain $(A_{1},B_{1}),\ldots, \\ (A_{n},B_{n})$ as in obstruction $3$, then either $A_{1} \cap B_{1}$ contains $c$ or $d$ or $A_{n} \cap B_{n}$ contains $c$ or $d$.}
  \item{ Obstructions $4$ and $5$ do not occur on $C$.}
  \end{itemize}  
\end{lemma}

\begin{proof}

Let $P_{a,b}$ be the $(a,b)$-path on $C$ such that $c,d \not \in V(P_{a,b})$. Similarly define $P_{b,c}$, $P_{c,d}$ and $P_{d,a}$. Then by construction, $P_{a,b} \cup P_{b,c} \cup P_{c,d} \cup P_{d,a} = C$. 

We first show that if we have a terminal separating triangle $(A_{1},B_{1}),(A_{2},B_{2}),(A_{3},B_{3})$ as in obstruction $2$, that $A_{1} \cap B_{1}$ contains one of $c$ or $d$.

 If $A_{1} \cap B_{1}$ contains $c$ or $d$ we are done. Therefore we assume that $A_{1} \cap B_{1}$ contains $a$. Let $(A_{1} \cap B_{1}) \setminus \{a\} = v$. Then to satisfy the definition of a terminal separating triangle we have that $v \in V(P_{b,c})$ or $v \in V(P_{c,d})$. This follows since if $v$ belonged to either of $P_{a,b} \setminus \{a,b\}$ or $P_{d,a} \setminus \{d,a\}$, then either there would not be two exactly two vertices of $X$ contained in $A_{1}$, or we could not satisfy the condition that $A_{j} \not \subseteq A_{i}$ for all $i,j \in \{1,2,3\}$, $i \neq j$ and maintain that $A_{2} \setminus (A_{2} \cap B_{2})$ and $A_{3} \setminus (A_{3} \cap B_{3})$ both contain a vertex of $X$. 

If $v=b$, then notice that the vertices of $A_{2} \cap B_{2}$ and $A_{3} \cap B_{3}$ that are not $a$ or $b$ lie on $P_{c,d}$. Since $cd \in E(C)$, without loss of generality $A_{2} \cap B_{2} = \{b,c\}$ and $A_{3} \cap B_{3} = \{a,c\}$, but this is not a terminal separating triangle, a contradiction. 

Therefore $v \neq b$. If $v = c$ or $v = d$ we are done, and since $cd \in E(C)$ it suffices to consider the case when $v \in P_{b,c} \setminus \{b,c\}$. Suppose $A_{2} \cap B_{2} = \{v,x\}$. If $x \not \in V(P_{a,d})$, then either $A_{2} \setminus (A_{2} \cap B_{2})$ does not contain a vertex of $X$, or $A_{3} \setminus (A_{3} \cap B_{3})$ does not contain a vertex of $X$ which is a contradiction. If $x = d$, then $A_{3} \cap B_{3} = \{a,d\}$ which implies that $A_{3} \setminus (A_{3} \cap B_{3})$ does not contain a vertex from $X$, a contradiction. Therefore $A_{1} \cap B_{1}$ cannot contain $a$. Mirroring the above argument, $A_{1} \cap B_{1}$ cannot contain $b$, and thus since by definition $A_{1} \cap B_{1}$ contains a vertex of $X$ it contains $c$ or $d$. 

Now suppose we have a terminal separating triangle $(A^{1}_{1},B^{1}_{1}), (A^{1}_{2}, B^{1}_{2}), (A^{1}_{3}, B^{1}_{3})$ and a terminal separating $2$-chain $(A_{1},B_{1}),\ldots,(A_{n},B_{n})$ in $G_{A_{1}}$ satisfying the properties of obstruction $3$. We will show that either $A_{1} \cap B_{1}$ or $A_{n} \cap B_{n}$ contains $c$ or $d$. 

 Since $A^{1}_{1} \cap B^{1}_{1} = \{x,y\}$ and $x,y \not \in X$, to be a terminal separating triangle, for all $i \in \{1,2,3\}$,  $A^{1}_{i} \cap B^{1}_{i}$ does not contain a vertex from $X$. This follows from the definition of terminal separating triangle as two vertices of $X$ lie in $A^{1}_{1} \setminus (A^{1}_{1} \cap B^{1}_{1})$, so if one of $A^{1}_{i} \cap B^{1}_{i}$, for $i \in \{2,3\}$ contained a vertex of $X$, then at least one of $A^{1}_{2} \setminus (A^{1}_{2} \cap B^{1}_{2})$ or $A^{1}_{3} \setminus (A^{1}_{3} \cap B^{1}_{3})$ does not contain a vertex from $X$, a contradiction. Now notice that since $A^{1}_{1} \setminus \{x,y\}$ contains two vertices of $X$, and $cd \in E(C)$, either $c,d \in A^{1}_{1} \setminus \{x,y\}$ or $a,b \in A^{1}_{1} \setminus \{x,y\}$. If $a,b \in A^{1}_{1} \setminus \{x,y\}$ then to remain a terminal separating triangle, one of $A^{1}_{2} \cap B^{1}_{2}$ or $A^{1}_{3} \cap B^{1}_{3}$ contains a vertex from $P_{c,d}$. But since $cd \in E(C)$, and we know that for all $i \in \{1,2,3\}$,  no vertex of $X$ is contained in $A^{1}_{i} \cap B^{1}_{i}$, a contradiction. Therefore $c,d \in A^{1}_{1} \setminus \{x,y\}$. Then by definition of terminal separating $2$-chain and since $a,b,c,d$ appear in that order on $C$, the terminal separating $2$-chain in $G_{A_{1}}$ contains either $d$ or $c$ in $A_{n} \cap B_{n}$ or $A_{1} \cap B_{1}$.

Now suppose we have two distinct terminal separating triangles $(A^{1}_{1}, B^{1}_{1}), (A^{1}_{2},B^{1}_{2}), (A^{1}_{3}, B^{1}_{3})$ and $(A^{2}_{1}, B^{2}_{1}), (A^{2}_{2},B^{2}_{2}), (A^{2}_{3}, B^{2}_{3})$, and a terminal separating $2$-chain in $G[A_{1} \cap A_{2}]$ satisfying the properties of obstruction $4$. We will show that this obstruction does not exist since $cd \in E(C)$.

 Notice from the assumptions that $A^{i}_{1} \cap B^{i}_{1}$ does not contain any vertices from $X$ for $i \in \{1,2\}$. Then since $cd \in E(C)$, without loss of generality we may assume that $a,b \in A^{1}_{1}$ and $c,d \in A^{2}_{1}$. But then by the previous discussion, one of $A^{2}_{2} \cap B^{2}_{2}$ and $A^{2}_{3} \cap B^{2}_{3}$ contains a vertex from $P_{c,d}$. But then one of $A^{2}_{2} \cap B^{2}_{2}$ and $A^{2}_{3} \cap B^{2}_{3}$ contains $c$ or $d$, which implies that $(A^{2}_{1}, B^{2}_{1}), (A^{2}_{2},B^{2}_{2}), (A^{2}_{3}, B^{2}_{3})$ is not a terminal separating triangle, a contradiction. 

Finally, suppose that the fifth obstruction occurred. Then we have two distinct terminal separating triangles $((A^{1}_{1},B^{1}_{1}), (A^{1}_{2},B^{1}_{2}), (A^{1}_{3},B^{1}_{3})), ((A^{2}_{1},B^{2}_{1}),  (A^{2}_{2},B^{2}_{2}),(A^{2}_{3},B^{2}_{3}))$ satisfying the properties of obstruction $5$. Then from our assumptions and since $cd \in E(C)$, without loss of generality we may assume  $c,d \in A^{1}_{1} \setminus (A^{1}_{1} \cap B^{1}_{1})$ and $a,b \in A^{2}_{1} \setminus (A^{2}_{1} \cap B^{2}_{1}$. By the same argument as for the obstruction $4$ case, this gives a contradiction, completing the proof.    
\end{proof}

We will extensively make use of the following sub-modularity lemma for separations.

\begin{proposition}
Let $G$ be a graph with $2$-separations $(A_{1},B_{1})$ and $(A_{2},B_{2})$. Let $A_{1} \cap B_{1} = \{u,v\}$ and $A_{2} \cap B_{2} = \{x,y\}$ where $x,y,u$ and $v$ are distinct vertices. Furthermore, suppose that $x \in A_{1} \setminus (A_{1} \cap B_{1})$ and $y \in B_{1} \setminus (A_{1} \cap B_{1})$. Then let $u \in A_{2} \setminus (A_{2} \cap B_{2})$ and $v \in B_{2} \setminus (A_{2} \cap B_{2})$. Then there is a separation $(A',B')$ such that $A' \cap B' = \{u,x\}$ and $A' \subseteq A_{1}$ and $B_{1} \subseteq B'$. 
\end{proposition}

\begin{proof}
Let $A' = A_{1} \cap A_{2}$ and $B' = V(G) \setminus (A_{1} \cap A_{2})$. Let $z$ be any vertex in $A_{1} \cap A_{2}$ and consider a path $P$ from $z$ to any vertex not in $A_{1} \cap A_{2}$. Consider the first vertex in $P$ which is not in $A_{1} \cap A_{2}$. If this vertex is in $B_{1} \cap A_{2}$, then since $u \in A_{2} \setminus (A_{2} \cap B_{2})$, this vertex is $u$, and thus $u \in V(P)$. If this vertex is in $A_{1} \cap B_{2}$ then since $x \in A_{1} \setminus (A_{1} \cap B_{1})$ this vertex is $x$ and thus $x \in V(P)$. Notice that these are the only options, and thus $(A',B')$ is a $2$-separation with $A' \cap B' = \{u,x\}$. Additionally, it is immediate that $A' \subseteq A_{1}$ and $B_{1} \subseteq B'$.
\end{proof}

Now we prove the theorem. Throughout the proof we will abuse the notation of separations slightly.  Given a graph $G$ and a $2$-separation $(A,B)$, if in the graph $G_{B}$, there is a $2$-separation $(A',B')$, then we will refer to the $2$-separation $(A' \cup A,B)$ in $G$ as $(A',B')$ to avoid notational clutter. Thus essentially if a separation $(A',B')$ in $G_{B}$ induces a natural separation in $G$, then we refer to the separation in $G$ as $(A',B')$.

\begin{proof}[Proof of Theorem \ref{w4cuts}]
Lemma \ref{obstructionsareobstructions} proves one direction of the theorem.

For the other direction, consider a graph $G$ which is a minimal counterexample with respect to $|V(G)|$. That is, we consider a graph $G$ such that there exists a cycle $C$ such that $X \subseteq V(C)$ where none of the five above obstructions exist on the cycle $C$, and $G$ is $W_{4}(X)$-minor free. Note such cycle always exists by Corollary \ref{webcycle}. Without loss of generality let $a,b,c,d$ appear in that order on $C$. Let $P_{a,b}$ be the $(a,b)$-path on $C$ such that $c,d \not \in V(P_{a,b})$. Similarly define $P_{b,c},P_{c,d}$ and $P_{d,a}$.  The goal will be to show that since $G$ is a vertex minimal counterexample, $G$ is $3$-connected, but that contradicts that all $3$-connected graphs either have a $K_{4}(X)$ or $W_{4}(X)$-minor (Theorem \ref{3connectivityW4K4}). We go through all the different possibilities for where the vertices of $X$ can be in relation to a $2$-separation.

\textbf{Claim 1:} There is no $2$-separation $(A,B)$ such that $X \subseteq A$. 

Suppose we had such separation and first suppose $V(C) \subseteq A$. By applying Lemma \ref{allonesidegeneralH} to $(A,B)$, the graph $G$ has a $W_{4}(X)$-minor if and only if $G_{A}$ has a $W_{4}(X)$-minor. Since $G$ is a vertex minimal counterexample, $G_{A}$ has one of the obstructions on $C$. But then the obstruction exists in $G$, a contradiction. 

Therefore we can assume that there are vertices of $C$ in $B \setminus (A \cap B)$. Since $X \subseteq V(C)$ and $X \subseteq A$, all of the vertices in $V(C) \cap B$ lie in exactly one of $P_{a,b}$, $P_{b,c}$,  $P_{c,d}$ or $P_{d,a}$. Without loss of generality, suppose all the vertices in $V(C) \cap B$ lie on $P_{a,b}$. Furthermore, since there is a vertex in $B \setminus (A \cap B)$, this implies that $A \cap B \subseteq V(P_{a,b})$. Then after applying Lemma \ref{allonesidegeneralH} to $(A,B)$ notice that in the graph $G_{A}$, that $C_{A} = G_{A}[V(C) \cap A]$ is a cycle. As $G$ is a minimal counterexample, $C_{A}$ has one of the five obstructions. Notice that regardless of the obstruction, since all the vertices of $C$ that were in $B$ were on $P_{a,b}$, the obstruction for $C_{A}$ in $G_{A}$ exists in $G$ for $C$. But this is a contradiction. 

\textbf{Claim 2:} There is no $2$-separation $(A,B)$ such that two vertices of $X$ lie in $A \setminus (A \cap B)$ and two vertices of $X$ that lie in $B \setminus (A \cap B)$.

Suppose such separation existed and let $A \cap B = \{x,y\}$. Notice that for such a separation to exist we have that $x \in V(C)$ and $y \in V(C)$.  Consider the graphs $G_{A}$ and $G_{B}$ inherited from Lemma \ref{2sepW4} and let $C_{A}$ and $C_{B}$ be the cycles where $C_{A} = G_{A}[V(C) \cap A]$ and $C_{B} = G_{B}[V(C) \cap A]$. By minimality, both $C_{A}$ and $C_{B}$ have one of the five obstructions. We consider the various cases.
 
 \textbf{Case 1:} The cycle $C_{B}$ has a terminal separating $2$-chain as in obstruction $1$, say $((A^{2}_{1},B^{2}_{1}),\ldots,(A^{2}_{n},B^{2}_{n}))$.
 
 \textbf{Subcase 1:} The cycle $C_{A}$ has a terminal separating $2$-chain as in obstruction $1$, $((A^{1}_{1},B^{1}_{1}),\ldots,(A^{1}_{n},B^{1}_{n}))$.

  Then if $A^{1}_{n} \cap B^{1}_{n} \cap A^{2}_{1} \cap B^{2}_{1} \neq \emptyset$ we may concatenate the two terminal separating $2$-chains giving a terminal separating $2$-chain of $G$ on $C$ satisfying obstruction $1$. Therefore we can assume $A^{1}_{n} \cap B^{1}_{n}$ does not share a vertex with $A^{2}_{1} \cap B^{2}_{1}$. But then the terminal separating $2$-chain $((A^{1}_{1},B^{1}_{1}),\ldots,(A^{1}_{n},B^{1}_{n}),(A,B), (A^{2}_{1},B^{2}_{1}),\ldots,(A^{2}_{n},B^{2}_{n}))$  satisfies obstruction $1$ on $C$, a contradiction. 
 
 \textbf{Subcase 2:} The cycle $C_{A}$ has a terminal separating triangle satisfying obstruction two, $(A^{1}_{1},B^{1}_{1}), (A^{1}_{2},B^{1}_{2}),(A^{1}_{3},B^{1}_{3})$.
 
  Since $xy \in E(C_{A})$, the terminal vertex in $A^{1}_{1} \cap B^{1}_{1}$ is either $x$ or $y$, by Lemma \ref{nicesplittinglemma}. But then $(A^{2}_{1},B^{2}_{1}), (A^{2}_{2},B^{2}_{2}), \\ (A^{2}_{3},B^{2}_{3})$ combines with either the terminal separating $2$-chain in $G_{B}$ or the terminal separating $2$ chain plus the separation $(A,B)$ to form the third obstruction in $G$, a contradiction. 

\textbf{Subcase 3:}  The cycle $C_{A}$ has a terminal separating triangle $(A_{1},B_{1}),(A_{2},B_{2}),(A_{3},B_{3})$ and a terminal separating $2$-chain $((A^{1}_{1},B^{1}_{1}),\ldots,(A^{1}_{n},B^{1}_{n}))$ as in obstruction $3$. 

By Lemma \ref{nicesplittinglemma}, we can assume that either $x$ or $y$ is in $A^{1}_{1} \cap B^{1}_{1}$. But then either we can concatenate the two terminal separating $2$-chains such that there resulting terminal separating $2$-chain and the terminal separating triangle satisfy obstruction $3$, or we can add in the separation $(A,B)$ to the terminal separating $2$-chains such that the terminal separating $2$-chains plus $(A,B)$ plus the terminal separating triangle satisfy obstruction $3$. 

\textbf{Subcase 4:} Obstruction $4$ or $5$ occurs on $C_{A}$.

By Lemma \ref{nicesplittinglemma} the fourth and fifth obstructions do not occur in $G_{A}$, since $xy \in E(C_{A})$. Now we consider cases where $G_{B}$ does not have a terminal separating $2$-chain.

\textbf{Case 2:} The cycle $C_{B}$ has a terminal separating triangle $(A^{2}_{1},B^{2}_{1}),(A^{2}_{2},B^{2}_{2}),(A^{2}_{3},B^{2}_{3})$ satisfying obstruction $2$.
 
 Notice by Lemma \ref{nicesplittinglemma} the terminal vertex contained in $A^{2}_{1} \cap B^{2}_{1}$ is $x$ or $y$. 
 
 \textbf{Subcase 1:} The cycle $C_{A}$ has a terminal separating triangle $(A^{1}_{1},B^{1}_{1}),(A^{1}_{2},B^{1}_{2}),(A^{1}_{3},B^{1}_{3})$ satisfying obstruction $2$.
 
  Again, by Lemma \ref{nicesplittinglemma} the terminal vertex contained in $A^{1}_{1} \cap B^{1}_{1}$ is $x$ or $y$. If the terminal vertex in $A^{i}_{1} \cap B^{i}_{1}$ is $x$ for both $i =1,2$, then notice in $G$ the two triangles form obstruction five. A similar statement holds if they are both $y$. Then we must have that one of the triangles contains $x$ as the terminal vertex in $A^{i}_{1} \cap B^{i}_{1}$, $i \in \{1,2\}$ and the other contains $y$. But then the two triangles plus the separation $(A,B)$ form obstruction $4$ in $G$, a contradiction. 

\textbf{Subcase 2:} The cycle  $C_{A}$ has a terminal separating triangle and a terminal separating $2$-chain as in obstruction $3$. 

By Lemma \ref{nicesplittinglemma} the terminal separating $2$-chain contains one of the vertices $x$ or $y$. Then by possibly adding in the separation $(A,B)$ to the existing terminal separating $2$-chains, we can extend this to a triangle plus a terminal separating $2$-chain and another terminal separating triangle, as in obstruction $4$, a contradiction. 

\textbf{Case 3:} Both of the cycles $C_{A}$ and $C_{B}$ have obstructions $3$.

But then by Lemma \ref{nicesplittinglemma} the terminal separating $2$-chains in both $C_{A}$ and $C_{B}$ both contain $x$ or $y$, and thus after possibly adding in $(A,B)$, we get in $G$, the fourth obstruction exists on $C$, a contradiction.  

We note these are all of the cases, and therefore there is no $2$-separation $(A,B)$ such that two vertices of $X$ lie in $A \setminus (A \cap B)$ and two vertices of $X$ that lie in $B \setminus (A \cap B)$, and $A \cap B \subseteq V(C)$. 
 
\textbf{Claim 3:}  There is no $2$-separation $(A,B)$ such that $b \in A \cap B$, $a \in A \setminus (A \cap B)$, and $c,d \in B \setminus (B \cap A)$. 

Suppose there is such a separation, $(A,B)$, and let $A \cap B = \{v,b\}$. First notice that since $C$ is a cycle containing $X$, for such separation to exist $v \in V(C)$. Applying Lemma \ref{oneterminalinthecut2conn} to $(A,B)$ we get that $G$ has a $W_{4}(X)$-minor if and only if $G_{B}$ has a $W_{4}(X)$-minor. Since we picked $G$ to be a minimal counterexample, $G_{B}$ has one of the five obstructions occuring on the cycle $C_{B} = G_{B}[V(C) \cap B]$. We consider the various cases.

\textbf{Case 1:} The cycle $C_{B}$ has a terminal separating $2$-chain $((A_{1},B_{1}),\ldots,(A_{n},B_{n}))$ satisfying obstruction $1$.

If $A_{1} \cap B_{1}$ or $A_{n} \cap B_{n}$ contains $b$ then it is a terminal separating $2$-chain in $G$ for $C$. Otherwise without loss of generality we have that $A_{1} \cap B_{1}$ contains $v$. But then $((A,B),(A_{1},B_{1}),\ldots,(A_{n},B_{n}))$ is a terminal separating $2$-chain in $G$ for $C$, contradicting that $G$ is a minimal counterexample. 
 
\textbf{Case 2:} The cycle $C_{B}$ has a terminal separating triangle $(A_{1},B_{1}), (A_{2},B_{2}),(A_{3},B_{3})$ satisfying obstruction $2$.

 First suppose that  $b \in A_{1} \cap B_{1}$. Then the terminal separating triangle exists in $G$, contradicting that we have a minimal counterexample. Then $v \in A_{1} \cap B_{1}$. But then in $G$, the separation $(A,B)$ plus $(A_{1},B_{1}), (A_{2},B_{2}),(A_{3},B_{3})$ satisfies obstruction $3$, contradicting that $G$ is a minimal counterexample. 

\textbf{Case 3:} The cycle $C_{B}$ has a terminal separating triangle plus a terminal separating $2$-chain as in obstruction $3$.

Then by Lemma \ref{nicesplittinglemma}, then the terminal separating $2$-chain contains either $v$ or $b$. In the case where the terminal separating $2$-chain contains $v$, then by adding in the separation $(A,B)$, we get a terminal separating $2$-chain plus terminal separating triangle in $G$. In the case where the terminal separating $2$-chain contains $b$, the terminal separating $2$-chain and terminal separating triangle were already an obstruction in $G$, a contradiction.

\textbf{Case 4:} Obstructions $4$ or $5$ occur on $C_{B}$. 

By Lemma \ref{nicesplittinglemma}, these obstructions do not occur, a contradiction. We have considered all possible cases, so therefore we can assume there is no $2$-separation $(A,B)$ such that $b \in A \cap B$, $a \in A \setminus (A \cap B)$, and $c,d \in B \setminus (B \cap A)$. 
 
 \textbf{Claim 4:} There are no $2$-separations of the form $(A,B)$ such that $A \cap B = \{u,v\}$, $a \in A \setminus \{u,v\}$, and $b,c,d \in B \setminus \{u,v\}$.

 Without loss of generality, we may assume that $u \in V(P_{a,b})$ and $v \in V(P_{a,d})$. By applying Lemma \ref{oneterminalonesideothersother} we get that $G$ has a $W_{4}(X)$-minor if and only if $G_{B}$ has either a $W_{4}(X_{1})$-minor or a $W_{4}(X_{2})$-minor, where $X_{1} = \{b,c,d,u\}$ and $X_{2} = \{b,c,d,v\}$. Since we have a minimal counterexample with respect to the number of vertices, we get that $C_{B} = G_{B}[V(C) \cap B]$ has one of the obstructions when we consider $X_{1}$ and when we consider $X_{2}$. First notice that if under either $X_{1}$ or $X_{2}$, we get obstruction $4$ or $5$, then the obstruction exists in $G$, a contradiction. We consider the various other cases.
 
 \textbf{Case 1:} Under $X_{2}$, we get a terminal separating $2$-chain  $((A_{1},B_{1}),\ldots,(A_{n},B_{n}))$.
 
 Notice that we may assume that $v \in A_{1} \cap B_{1}$ or $v \in A_{n} \cap B_{n}$, as otherwise $((A_{1},B_{1}),\ldots, \\ (A_{n},B_{n}))$ exists in $G$ for $C$, a contradiction. Therefore without loss of generality, we suppose that $v \in A_{1} \cap B_{1}$. 
 
 \textbf{Subcase 1:} Under $X_{1}$, we get a terminal separating $2$-chain $((A'_{1},B'_{1}),\ldots,(A'_{n'},B'_{n'}))$. 
 
 By similar reasoning as above, we may assume that $u \in A'_{1} \cap B'_{1}$. Notice that the vertex in $A'_{1} \cap B'_{1} \setminus \{u\}$ lies in $P_{b,c} \setminus \{b\}$ or in $P_{d,c} \setminus \{d\}$ and and the vertex in $A_{1} \cap B_{1} \setminus \{v\}$ lies in $P_{b,c} \setminus \{b\}$ or $P_{d,c} \setminus \{d\}$. Consider the case where both the vertex in $A'_{1} \cap B'_{1} \setminus \{u\}$ and the vertex in $A_{1} \cap B_{1} \setminus \{v\}$ lie on $P_{b,c} \setminus \{b\}$.

We consider various subcases. Suppose the vertex in $A'_{1} \cap B'_{1} \setminus \{u\}$ is the same vertex as $A_{1} \cap B_{1} \setminus \{v\}$. If this vertex is $c$, then $(A,B), (B_{1},A_{1}), (A'_{1},B'_{1})$ is a terminal separating triangle in $G$ satisfying obstruction $2$, a contradiction. Thus we assume the vertex in $A'_{1} \cap B'_{1} \setminus \{u\}$ is not $c$. But then $(A,B), (A_{1},B_{1}), (A'_{1},B'_{1})$ plus $(A_{2},B_{2}),\ldots,(A_{n},B_{n})$ is a terminal separating triangle and terminal separating $2$-chain satisfying obstruction $3$, a contradiction. 
 
 Therefore we can assume that $A'_{1} \cap B'_{1} \setminus \{u\} \neq A_{1} \cap B_{1} \setminus \{v\}$. First suppose that the vertex in $A'_{1} \cap B'_{1} \setminus \{u\}$ lies in $A_{1}$. Notice that one of the vertices in $A'_{2} \cap B'_{2}$ lies on $P_{u,d}$, and call this vertex $x$  (if $x \not \in V(P_{u,d})$, then either $((A'_{1},B'_{1}),\ldots,(A'_{n},B'_{n}))$ is not a dissection, or not a terminal separating chain).  If  $x =v$, then $(A,B),(A'_{1},B'_{1}),(A'_{2},B'_{2})$ plus $((A_{1},B_{1}),\ldots (A_{n},B_{n}))$ satisfies obstruction $3$ in in $G$, a contradiction. Therefore $x \neq v$. But notice that $uv \in E(C_{B})$, which implies that $x$ lies on the $(v,d)$-subpath of $P_{u,d}$. But then the vertex in $A'_{1} \cap B'_{1} \setminus \{u\}$ and $v$ are a $2$-vertex cut. Let $(A',B')$ be the $2$-separation such that $A' \cap B' = \{u, (A'_{1} \cap B'_{1} \setminus \{u\})\}$. Then $(A,B),(A',B'),(A'_{1},B'_{1})$ plus $((A_{1},B_{1}),\ldots, (A_{n},B_{n}))$ satisfy obstruction $3$ in $G$, a contradiction. 
 
 Therefore we can assume that the vertex in $A'_{1} \cap B'_{1} \setminus \{u\}$ lies in $B_{1}$. Let $x$ be the vertex in $A_{1} \cap B_{1} \setminus \{v\}$. Notice that $x \neq c$ as the vertex in $A'_{1} \cap B'_{1} \setminus \{u\}$ lies in $B_{1}$. Since the vertex in $A'_{1} \cap B'_{1} \setminus \{u\}$ lies in $B_{1}$, we have that $\{x,u\}$ is a $2$-vertex cut. Let $(A',B')$ be this $2$-separation. Then $(A,B),(A',B'),(A_{1},B_{1})$ plus $((A_{2},B_{2}),\ldots (A_{n},B_{n}))$ is a terminal separating triangle plus terminal separating $2$-chain satisfying obstruction $3$ a contradiction. 
 
 Therefore we can assume that the vertex in  $A'_{1} \cap B'_{1} \setminus \{u\}$ and the vertex in $A_{1} \cap B_{1} \setminus \{v\}$ both do not lie on $P_{b,c} \setminus \{b\}$. By essentially the same argument, we can assume that the vertex in $A'_{1} \cap B'_{1} \setminus \{u\}$ and the vertex in $A_{1} \cap B_{1} \setminus \{v\}$ both do not lie on $P_{c,d} \setminus \{d\}$. 
 
 Now consider the case where the vertex in $A_{1} \cap B_{1} \setminus \{v\}$ lies in $P_{b,c} \setminus \{b\}$ and the vertex in $A'_{1} \cap B'_{1} \setminus \{u\}$ lies in $P_{c,d} \setminus \{d\}$. Notice that at $c$ is not both in $A_{1} \cap B_{1} \setminus \{v\}$ and $A'_{1} \cap B'_{1} \setminus \{u\}$.  Let $x$ be the vertex in $A_{1} \cap B_{1} \setminus \{v\}$ and suppose $x \neq c$. Then $(A_{2},B_{2})$ exists and $A_{2} \cap B_{2} \setminus \{x\}$ lies on $P_{v,d}$. Then notice that $\{u,x\}$ is a $2$-vertex cut in $G$. Let $(A',B')$ be the separation such that $A' \cap B' = \{u,x\}$. 
Then $(A,B), (A_{1}, B_{1}), (A',B')$ plus $((A_{2},B_{2}),\ldots,(A_{n},B_{n}))$ is a terminal separating triangle and terminal separating chain as in obstruction $3$. The case where $c$ is not in $A'_{1} \cap B'_{1} \setminus \{u\}$ follows similarly.  
 
 Now consider the case where the vertex in $A_{1} \cap B_{1} \setminus \{v\}$ lies in $P_{c,d} \setminus \{d\}$ and the vertex in $A'_{1} \cap B'_{1} \setminus \{u\}$ lies in $P_{b,c} \setminus \{b\}$. Notice that at least one of the vertices in $A'_{1} \cap B'_{1} \setminus \{u\}$ and $A_{1} \cap B_{1} \setminus \{v\}$ is not $c$. Let $x$ be the vertex in $A_{1} \cap B_{1} \setminus \{v\}$ and suppose that $x \neq c$. Then since $x \neq c$, we have that $(A_{2},B_{2})$ exists and lies on $P_{v,b}$. Then since $uv \in E(C_{B})$, notice that $\{u,x\}$ is a $2$-vertex cut. Let $(A',B')$ be the separation such that $A' \cap B' = \{u,x\}$. Then $(A,B),(A',B'), (A_{1},B_{1})$ plus $((A_{2},B_{2}),\ldots,(A_{n},B_{n}))$ is a terminal separating $2$-chain plus terminal separating triangle satisfying obstruction $3$, a contradiction. The case where $c$ is not in $A'_{1} \cap B'_{1} \setminus \{u\}$ follows similarly.

 \textbf{Subcase 2:} Suppose that under $X_{1}$, we get a terminal separating triangle $(A'_{1},B'_{1}),\\ (A'_{2},B'_{2}),(A'_{3},B'_{3})$ satisfying obstruction $2$.
 
 %First notice that if $u \not \in A'_{1} \cap B'_{1}$, then the obstruction exists in $G$, a contradiction. First, if all of $A'_{1} \cap B'_{1}$, $A'_{2} \cap B'_{2}$, and $A'_{3} \cap B'_{3}$ lie in $A_{1}$. Then to be a terminal separating triangle satisfying obstruction $2$,  notice that the vertex in $A'_{1} \cap B'_{1} \setminus \{u\}$ lies in $P_{b,c} \setminus \{b\}$. But then that implies that $v \in A'_{2} \cap B'_{2}$, which implies that $A'_{3} \cap B'_{3} = \{u,v\}$. But that implies that $(A'_{1},B'_{1}), (A'_{2},B'_{2}), (A'_{3},B'_{3})$ is not a terminal separating triangle satisfying obstruction $2$, a contradiction. Therefore we assume that not all of $A'_{1} \cap B'_{1}$, $A'_{2} \cap B'_{2}$, and $A'_{3} \cap B'_{3}$ lie in $A_{1}$. 
 
Suppose that the vertex in $A'_{1} \cap B'_{1} \setminus u$ lies in $A_{1}$.  Call this vertex $x$. First suppose $x \in A_{1} \cap B_{1} \setminus \{v\}$. Notice that $x \neq c$, as if $x =c$ then $(A'_{1},B'_{1}),(A'_{2},B'_{2}),(A'_{3},B'_{3})$ would not satisfy the definition of a terminal separating triangle, a contradiction. Therefore $x \neq c$, but then $(A,B), (A'_{1},B'_{1}),(A_{1},B_{1})$ and $((A_{2},B_{2}),\ldots,(A_{n},B_{n}))$ is a terminal separating $2$-chain satisfying obstruction $3$, a contradiction. 

Now consider when $x \in A_{1} \setminus (A_{1} \cap B_{1})$ and suppose that $x$ lies in $P_{b,c} \setminus \{b\}$. Then the vertex in $A'_{2} \cap B'_{2} \setminus \{x\}$ lies in either $P_{u,d} \setminus \{u,v\}$ or $P_{c,d}$. In either case, notice that $\{x,v\}$ is a $2$-vertex cut. Let $(A',B')$ be the separation such that $A' \cap B' = \{x,v\}$. But then $(A,B),(A',B'),(A'_{1},B'_{1})$ plus $((A_{1},B_{1}),\ldots,(A_{n},B_{n}))$ satisfies obstruction $3$ in $G$, a contradiction. Now suppose that $x$ lies in $P_{c,d} \setminus \{d\}$. Then since $x \in A_{1} \setminus (A_{1} \cap B_{1})$, we have that $u$ and the vertex in $A_{1} \cap B_{1} \setminus \{v\}$ form a $2$-vertex cut. Let $(A',B')$ be the separation induced by this $2$-vertex cut. Then if $n=1$, we have $(A,B), (A',B'), (A_{1},B_{1})$ form a terminal separating triangle in $G$, a contradiction. Otherwise $n >1$ and $(A,B), (A',B'), (A_{1},B_{1})$ plus $((A_{2},B_{2}), \ldots, (A_{n},B_{n}))$ form a terminal separating triangle plus a terminal separating $2$-chain in $G$, a contradiction. 

Therefore we can assume $x$ lies in $B_{1} \setminus (A_{1} \cap B_{1})$. Let $y$ be the vertex in $A_{1} \cap B_{1} \setminus \{v\}$. Then since $x \in B_{1} \setminus (A_{1} \cap B_{1})$,  notice that $u,y$ induces a $2$-separation. Let $(A',B')$ be the $2$-separation such that $A' \cap B' = \{u,y\}$. If $n=1$, then $(A,B), (A',B'), (A_{1},B_{1})$ is a terminal separating triangle satisfying obstruction $2$ in $G$, a contradiction. Otherwise, $(A,B),(A',B'),(A_{1},B_{1})$ plus $((A_{2},B_{2}),\ldots, (A_{n},B_{n}))$ satisfy obstruction $3$ in $G$, a contradiction. 

\textbf{Subcase 3:} Suppose that under $X_{1}$, we get a terminal separating triangle $(A'_{1},B'_{1}),\\ (A'_{2},B'_{2}),(A'_{3},B'_{3})$ and a terminal separating $2$-chain $((A''_{1},B''_{1}), \ldots, (A''_{n},B''_{n})$ satisfying obstruction $3$.

Notice that if $u \not \in A''_{1} \cap B''_{1}$ or $u \not \in A''_{n} \cap B''_{n}$, then the obstruction exists already in $G$, a contradiction. Therefore without loss of generality, $u \in A''_{1} \cap B''_{1}$. Now by the same arguments as in subcase $1$, we can look at the position of the vertex in $A''_{1} \cap B''_{1} \setminus \{u\}$ and show that some obstruction $3$ or $2$ always exists in $G$.

\textbf{Case 2:} Under $X_{2}$, we get a terminal separating triangle $(A_{1},B_{1}),(A_{2},B_{2}),(A_{3},B_{3})$ satisfying obstruction $2$.

Notice that $A_{1} \cap B_{1}$ contains $v$ as otherwise the obstruction is an obstruction of $G$, a contradiction. Then since $v \in A_{1} \cap B_{1}$, we may assume that no other vertex of $X_{2}$ is in $A_{1} \cap B_{1}, A_{2} \cap B_{2},$ or $A_{3} \cap B_{3}$.  Also note that $u$ is not in any of $A_{1} \cap B_{1}$, $A_{2} \cap B_{2}$, $A_{3} \cap B_{3}$ as otherwise $(A_{1},B_{1}),(A_{2},B_{2}),(A_{3},B_{3})$ would not be a terminal separating triangle. Then we may assume that $u \in A_{1} \setminus (A_{1} \cap B_{1})$. Also note that we may assume the vertex in $A_{1} \cap B_{1} \setminus \{v\}$ lies in $P_{b,c}$. Notice by symmetry and the above cases, we do not need to consider the case where under $X_{1}$ we get a terminal separating $2$-chain.

\textbf{Subcase 1:} Under $X_{1}$, we get a terminal separating triangle satisfying obstruction $2$, $(A'_{1},B'_{1}),(A'_{2},B'_{2}),(A'_{3},B'_{3})$.

By a similar argument as above, we may assume that $u \in A'_{1} \cap B'_{1}$ as otherwise the obstructions exists in $G$. Without loss of generality we may assume that the vertex in $A'_{1} \cap B'_{1} \setminus \{u\}$ lies in $P_{b,c} \setminus \{c\}$. Let $x$ be the vertex in $A'_{1} \cap B'_{1} \setminus \{u\}$. If $x \in A_{1} \cap B_{1} \setminus \{v\}$, then $(A,B), (A'_{1},B'_{1}),(A_{1},(B_{1})$ and $(A_{1},B_{1}),(A_{2},B_{2}),(A_{3},B_{3})$ satisfy obstruction $5$ in $G$, a contradiction.

Now suppose that $x$ is in $A_{1} \setminus (A_{1} \cap B_{1})$. Notice that $A'_{2} \cap B'_{2} \setminus \{x\}$ lies on $P_{u,d}$. But then since $uv \in E(C_{B})$, $\{x,v\}$ is a $2$-vertex cut. Let $(A',B')$ be the separation where $A' \cap B' = \{x,v\}$. But then $(A,B), (A',B'), (A'_{1},B'_{1})$ and $(A_{1},B_{1}),(A_{2},B_{2}),(A_{3},B_{3})$ satisfy obstruction $5$ in $G$, a contradiction. 

Now suppose that $x$ lies in $B_{1} \setminus (A_{1} \cap B_{1})$. Let $y \in A_{1} \cap B_{1} \setminus \{v\}$. Then $\{u,y\}$ is a $2$ vertex cut. Let $(A',B')$ be the separation such that $A' \cap B' = \{u,y\}$. Then $(A,B),(A',B'), (A_{1},B_{1})$ and $(A_{1},B_{1}),(A_{2},B_{2}),(A_{3},B_{3})$ are terminal separating triangles satisfying obstruction $5$ in $G$, a contradiction. 

Now suppose that $x$ is in $A_{1} \cap B_{1} \setminus \{v\}$. Then easily $(A,B),(A_{1},B_{1}),(A'_{1},B'_{1})$ and $(A'_{1},B'_{1}),(A'_{2},B'_{2}),(A'_{3},B'_{3})$ satisfies obstruction $5$ in $G$ a contradiction. 
 
\textbf{Subcase 2:} Suppose under $X_{1}$, we get a terminal separating triangle $(A'_{1},B'_{1}),(A'_{2},B'_{2}), \\ (A'_{3},B'_{3})$ and terminal separating $2$-chain $((A''_{1},B''_{1}) \ldots (A''_{n'},B''_{n'}))$ as in obstruction $3$.

 Notice that if $u \not \in A''_{1} \cap B''_{1}$ and $u \not \in A''_{n} \cap B''_{n}$ then the obstruction exists in $G$, a contradiction. Thus without loss of generality, we may assume that $u \in A''_{1} \cap B''_{1}$. Then notice that the vertex in $A''_{1} \cap B''_{1} \setminus \{u\}$ lies on $P_{b,c} \setminus \{b\}$ or $P_{c,d} \setminus \{d\}$. Let $x$ be the vertex in $A''_{1} \cap B''_{1}$.

Now suppose that $x$ is in $A_{1} \setminus (A_{1} \cap B_{1})$. Since we assumed that the other vertex in $A_{1} \cap B_{1} \setminus \{v\}$ was in $P_{b,c}$, this implies that $x \in P_{b,c}$. Then notice that since $c \not \in A_{1} \cap B_{1} \setminus \{v\}$, then either $(A''_{2},B''_{2})$ exists and the vertex in $A''_{2} \cap B''_{2} \setminus \{x\}$ lies on $P_{u,d}$ or $(A''_{2},B''_{2})$ does not exist and the vertex in $A'_{1} \cap B'_{1} \setminus \{x\}$ lies in $P_{u,d}$. First suppose the vertex in $A''_{2} \cap B''_{2} \setminus \{x\}$ lies on $P_{u,d}$. But then since $uv \in E(C_{B})$, $\{x,v\}$ is a $2$-vertex cut. Let $(A',B')$ be the separation where $A' \cap B' = \{x,v\}$. But then $(A,B), (A',B'), (A'_{1},B'_{1})$ and $(A_{1},B_{1}),(A_{2},B_{2}),(A_{3},B_{3})$ satisfy obstruction $5$ in $G$, a contradiction. A similar analysis holds for the case when $(A''_{2},B''_{2})$ does not exist and the vertex in $A'_{1} \cap B'_{1} \setminus \{x\}$ lies in $P_{u,d}$. 

%Now suppose that $x$ is in $A_{1} \setminus (A_{1} \cap B_{1})$ and $x \in P_{d,c}$. Then let $y \in A_{1} \cap B_{1} \setminus \{x\}$. Then as $x$ is in $A_{1} \setminus (A_{1} \cap B_{1})$, we have $\{u,y\}$ is a $2$-vertex cut. Let $(A',B')$ be the separation induced by this $2$-vertex cut. Then $(A,B), (A',B'), (A_{1},B_{1})$ 

Now suppose that $x$ lies in $B_{1} \setminus (A_{1} \cap B_{1})$. Let $y \in A_{1} \cap B_{1} \setminus \{v\}$. Then $\{u,y\}$ is a $2$ vertex cut. Let $(A',B')$ be the separation such that $A' \cap B' = \{u,y\}$. Then $(A,B),(A',B'), (A_{1},B_{1})$ and $(A_{1},B_{1}),(A_{2},B_{2}),(A_{3},B_{3})$ are terminal separating triangles satisfying obstruction $5$, a contradiction.

 \textbf{Case 3:} Under $X_{2}$, we have a terminal separating $2$-chain $((A'_{1},B'_{1}),\ldots,(A'_{n},B'_{n}))$ and a terminal separating triangle $(A_{1}, B_{1}),(A_{2},B_{2}),(A_{3},B_{3})$. 
 
 Notice that if $v \not \in A'_{1} \cap B'_{1}$ and $v \not \in A'_{n} \cap B'_{n}$ then the obstruction exists in $G$, a contradiction. Without loss of generality, assume that $v \in A'_{1} \cap B'_{1}$. Notice that the vertex in $v \in A'_{1} \cap B'_{1} \setminus \{v\}$ lies in either $P_{b,c}$ or $P_{c,d}$. By symmetry and the previous cases, we only need to consider the case when under $X_{1}$ we get obstruction $3$.

\textbf{Subcase 1:} Under $X_{1}$ we get a terminal separating $2$ chain $((A''_{1},B''_{1}), \ldots (A''_{n},B''_{n}))$ and a terminal separating triangle $(A'''_{1},B'''_{1}),(A'''_{2},B'''_{2}),(A'''_{3},B'''_{3})$ satisfying obstruction $3$. 

By the same arguments as above, without loss of generality we may assume that $u \in A''_{1} \cap B''_{1}$. Then let $x$ be the vertex in $A''_{1} \cap B''_{1} \setminus \{u\}$. Notice that $x$ lies in $P_{b,c} \setminus \{b\}$ or $P_{c,d} \setminus \{d\}$. Consider the case where $x$ lies on $P_{b,c} \setminus \{b\}$ and the vertex in $A'_{1} \cap B'_{1} \setminus \{v\}$ lies on $P_{b,c} \setminus \{b\}$.

First suppose that $x \in A_{1} \cap B_{1} \setminus \{v\}$. Then if $n =1$, $(A,B),(A''_{1},B''_{1}),(A'_{1},B'_{1})$ and $(A_{1}, B_{1}),(A_{2},B_{2}),(A_{3},B_{3}))$ satisfies obstruction $5$. Otherwise $n \geq 1$ and $(A,B),(A''_{1},B''_{1}), \\ (A'_{1},B'_{1})$ plus $((A'_{2},B'_{2}) \ldots (A'_{n},B'_{n}))$ plus $(A_{1}, B_{1}),(A_{2},B_{2}),(A_{3},B_{3})$ satisfy obstruction $4$ in $G$, a contradiction. 

Now suppose that $x \in A'_{1} \setminus (A'_{1} \cap B'_{1})$. Then either the vertex in  $A''_{2} \cap B''_{2} \setminus \{x\}$ lies on $P_{u,d}$ or $(A''_{2},B''_{2})$ does not exist and the vertex in $A'''_{1} \cap B'''_{1} \setminus \{x\}$ lies on $P_{u,d}$. Consider the case where $A''_{2} \cap B''_{2} \setminus \{x\}$ lies on $P_{u,d}$. Then since $uv \in E(C_{B})$, we have that $\{x,v\}$ is a $2$-vertex cut. Let $(A',B')$ be the separation where $A' \cap B' = \{x,v\}$. But then $(A,B), (A',B'), (A''_{1},B''_{1})$ plus $((A_{1},B_{1}),\ldots,(A_{n},B_{n}))$ plus $(A'_{1}, B'_{1}),(A'_{2},B'_{2}),(A'_{3},B'_{3})$ satisfies obstruction $4$ in $G$, a contradiction. 

Now suppose that $x \in B'_{1} \setminus (A'_{1} \cap B'_{1})$. Let $y \in A'_{1} \cap B'_{1} \setminus \{v\}$. Then notice that $\{y,u\}$ is a $2$-vertex cut in $G_{B}$. Let $(A',B')$ be the separation such that $A' \cap B' = \{y,u\}$. Then if $n =1$,  $(A,B), (A',B'),(A'_{1},B'_{1})$ plus $(A_{1},B_{1}),(A_{2},B_{2}),(A_{3},B_{3})$ satisfies obstruction $5$ in $G$, a contradiction. Otherwise, $n \geq 2$, and $(A,B), (A',B'),(A'_{1},B'_{1})$ plus $((A'_{2},B'_{2}),\ldots,(A'_{n},B'_{n}))$ plus $(A_{1},B_{1}),(A_{2},B_{2}),(A_{3},B_{3})$ satisfies obstruction $4$ in $G$, a contradiction. 

Notice the case where $x$ lies on $P_{c,d} \setminus \{d\}$ and the vertex in $A'_{1} \cap B'_{1} \setminus \{v\}$ lies on $P_{b,d}$ follows in a similar fashion.

 Now suppose that $x$ lies on $P_{b,c} \setminus \{b\}$ and the vertex in $A'_{1} \cap B'_{1} \setminus \{v\}$ lies on $P_{c,d} \setminus \{d\}$. Let $y$ be the vertex in $A'_{1} \cap B'_{1} \setminus \{v\}$. Then $\{u,y\}$ is a $2$-vertex cut. Let $(A',B')$ be the separation induced by this $2$-vertex cut. Then if $n=1$, we have $(A,B),(A',B'), (A'_{1},B'_{1})$ plus $(A_{1},B_{1}),(A_{2},B_{2}),(A_{3},B_{3})$ form obstruction $5$, a contradiction. Otherwise, $n >1$ and $(A,B),(A',B'), (A'_{1},B'_{1})$ plus $((A'_{2},B'_{2}) \ldots (A'_{n},B'_{n}))$ and $(A_{1},B_{1}),(A_{2},B_{2}),(A_{3},B_{3})$ satisfies obstruction $4$, a contradiction. 
 
 Now suppose that $x$ lies on $P_{c,d} \setminus \{d\}$ and the vertex in $A'_{1} \cap B'_{1} \setminus \{v\}$ lies in $P_{b,c} \setminus \{c\}$. We note that a similar argument to when both $x$ and the vertex in $A'_{1} \cap B'_{1} \setminus \{v\}$ were in $P_{b,c}$ and $x \in B'_{1} \setminus (A'_{1} \cap B'_{1})$ works in this case.

 Therefore, there are no $2$-separations of the form $(A,B)$ such that $A \cap B = \{u,v\}$, $a \in A \setminus \{u,v\}$, and $b,c,d \in B \setminus \{u,v\}$. 

\textbf{Claim 5:} There are no $2$-separations $(A,B)$ such that $A \cap B$ contain two vertices from $X$, and $A \setminus (A \cap B)$ contains a vertex from $X$ and $B \setminus (A \cap B)$ contains a vertex from $X$. 

If such separation existed, it would be a terminal separating $2$-chain, a contradiction. 

Therefore, our graph $G$ has no $2$-separations. Therefore $G$ is $3$-connected. But every $3$-connected graph which does not have a $K_{4}(X)$-minor has a $W_{4}(X)$-minor by Theorem \ref{3connectivityW4K4}. But this contradicts our choice of $G$, completing the claim. 
\end{proof}

\section{A characterization of graphs without a $K_{4}(X)$, $W_{4}(X)$, or a $K_{2,4}(X)$-minor}

First we define what we mean by $K_{2,4}(X)$-minor. Let $V(K_{2,4}) = \{t_{1},t_{2},t_{3},t_{4},s_{1},s_{2}\}$ where $E(K_{2,4}) = \{t_{i}s_{j} \ | \ \forall i \in \{1,2,3,4\}, j \in \{1,2\}\}$.  Let $G$ be a graph and $X = \{a,b,c,d\} \subseteq V(G)$. Let $\mathcal{F}$ be the family of maps from $X$ to $V(K_{2,4})$ such that each vertex of $X$ goes to a distinct vertex in $\{t_{1},t_{2},t_{3},t_{4}\}$. For the purposes of this thesis, a $K_{2,4}(X)$ minor refers to the $X$ and family of maps given above. 

\begin{figure}
\begin{center}
\includegraphics[scale =0.25]{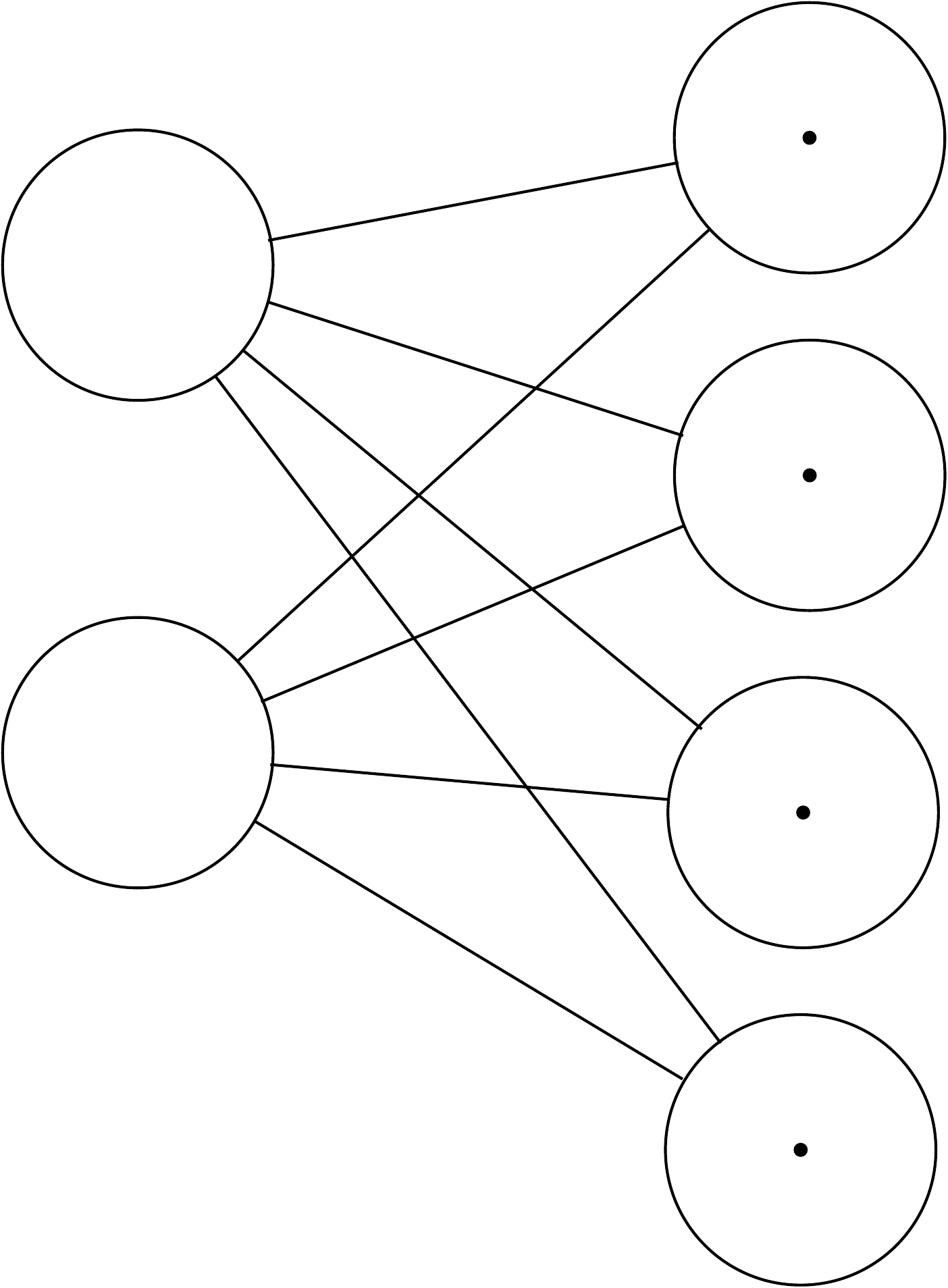}
\caption{A graph $G$ with a model of a $K_{2,4}(X)$-minor. The circles represent connected subgraphs, and the vertices in the circles represent vertices from $X$.}
\end{center}
\end{figure}

Graphs not containing $K_{2,4}(X)$-minors was the subject of Demasi's PhD thesis (\cite{linodemasithesis}). He managed to give a complete chracterization of planar graphs without a $K_{2,4}(X)$-minor. We do not require the full strength of the planar $K_{2,4}(X)$ characterization, but in obtaining a full characterization of graphs without a $K_{4}(X)$, $W_{4}(X)$ or $K_{2,4}(X)$-minor we will make use of the reduction lemmas given in \cite{linodemasithesis} to throw out most of the non-web classes of $K_{4}(X)$ and $W_{4}(X)$ free graphs. Then by appealing to Theorem \ref{w4cuts}, we will show that if an $\{a,b,c,d\}$-web has a $K_{2,4}(X)$-minor, then it also has a $W_{4}(X)$-minor (see Theorem \ref{k4w4k24characterization}). 

Of particular interest will be $K_{2,2}(X)$-minors due to a useful reduction lemma. Let $V(K_{2,2}) = \{t_{1},t_{2},s_{1},s_{2}\}$ where $E(K_{2,2}) = \{t_{i}s_{j} \ | \ \forall i,j \in \{1,2\}\}$. Let $G$ be a graph and $X = \{a,b,c,d\} \subseteq V(G)$. Define $\mathcal{F}$ to be the family of maps where $a$ and $b$ are mapped to $\{s_{1},s_{2}\}$ and $c$ and $d$ are mapped to $\{t_{1},t_{2}\}$. For the purpose of this thesis, a $K_{2,2}(X)$-minor refers to the $X$ and $\mathcal{F}$ above.

\begin{theorem}[\cite{linodemasithesis}]
\label{k22characterization}
Let $a,b,c,d$ be distinct vertices in a graph $G$. Then $G$ contains a 
$K_{2,2}(X)$-minor if and only if there exists an $(a,c)$-path $P_{a,c}$ and a $(b,d)$-path $P_{b,d}$ such that $P_{a,c} \cap P_{b,d} = \emptyset$ and there exists an $(a,d)$-path $P_{a,d}$ and a $(b,c)$-path $P_{b,c}$ such that $P_{b,c} \cap P_{a,d} = \emptyset$.
\end{theorem}

Since $K_{2,4}$ is $2$-connected, by previous discussion, we may assume we are dealing with $2$-connected graphs (We note that \cite{linodemasithesis} has the same cut vertex reductions specialized to $K_{2,4}(X)$-minors, though we note the proofs are in essence the same as those given in the cut vertex section).

Now we record some of the  $2$-connected reductions from \cite{linodemasithesis}. For the next three lemmas, suppose $G$ is a $2$-connected graph, $X = \{a,b,c,d\} \subseteq V(G)$ and $(A,B)$ a $2$-separation such that $A \cap B = \{u,v\}$. Also let $G_{A} = G[A] \cup \{uv\}$ and $G_{B} = G[B] \cup \{uv\}$, as before. 

\begin{lemma}
\label{alloneside}
Suppose that $X \subseteq B$. Then $G$ has a $K_{2,4}(X)$ minor if and only if $G_{B}$ has a $K_{2,4}(X)$-minor.
\end{lemma}

\begin{lemma}
\label{splitvertices}
If $u, v \in X$, then $G$ has a $K_{2,4}(X)$-minor if and only if $X \subseteq A$ and $G_{A}$ has a $K_{2,4}(X)$-minor or $X \subseteq B$ and $G_{B}$ has a $K_{2,4}(X)$-minor. Suppose $u \in X$ and $v \not \in X$. Furthermore, suppose $a \in A \setminus (A \cap B)$ and $b,c,d \in B$. For each $\pi \in \mathcal{F}$, let $\pi_{B}: \{u,b,c,d\} \to V(K_{2,4})$ be such that $\pi_{B}(u) = \pi(a)$ and on $\{b,c,d\}$, $\pi_{B} = \pi$.  Then $G$ has a $K_{2,4}(X)$-minor if and only if $G_{B}$ has a $K_{2,4}(X)$-minor. 
\end{lemma}

\begin{lemma}
\label{thereductionwithk22init}
Suppose $a,b \in A \setminus (A \cap B)$ and $c,d \in B \setminus (A \cap B)$. For each $\pi \in \mathcal{F}$, let $\pi_{A}: \{a,b,u,v\} \to V(K_{2,4})$ be such that $\pi_{A} = \pi$ on $\{a,b\}$ and $\pi_{A}(u) =\pi(c) $ and $\pi_{A}(v) = \pi(d)$. 
Similarly for each $\pi \in \mathcal{F}$, define $\pi_{B} : \{c,d,u,v\} \to V(K_{2,4})$. Additionally for each $\pi \in \mathcal{F}$, let $\pi'_{A}: \{a,b,u,v\} \to V(K_{2,2})$ be such that $\pi'_{A}= \pi$ on $\{a,b\}$ and $\pi'_{A}(u) = \pi(c)$ and $\pi'_{A}(v) = \pi(d)$. 
Also, for each $\pi \in \mathcal{F}$, let $\pi'_{B}: \{c,d,u,v\} \to V(K_{2,2})$ be such that $\pi'_{B} = \pi$ on $\{c,d\}$ and, $\pi'_{A}(u) = \pi(a)$ and $\pi'_{A}(v) = \pi(b)$.  Then $G$ has a $K_{2,4}(X)$-minor if and only if either $G_{A}$ or $G_{B}$ has a $K_{2,4}(X)$-minor, or both of $G_{A}$ and $G_{B}$ have a $K_{2,2}(X)$ minor. 
\end{lemma}

That completes the $2$-connected reductions from \cite{linodemasithesis} that will be needed. There is one $3$-connected reduction from \cite{linodemasithesis} which is useful. 

\begin{lemma}
\label{3connk24}
Let $G$ be a graph and let $(A,B)$ be a tight $3$-separation where $A \cap B = \{v_{1},v_{2},v_{3}\}$. Suppose that $X \subseteq A$. Let $G' = G[A] \cup \{v_{1}v_{2},v_{1}v_{3},v_{2}v_{3}\}$. Then $G$ has a $K_{2,4}(X)$-minor if and only if $G'$ has a $K_{2,4}(X)$-minor. 
\end{lemma}

Additional reduction lemmas are proven in \cite{linodemasithesis}, but these suffice for our needs.
To avoid repeating the same statements in the next lemmas, we make the following observation.

\begin{figure}
\begin{center}
\includegraphics[scale=0.5]{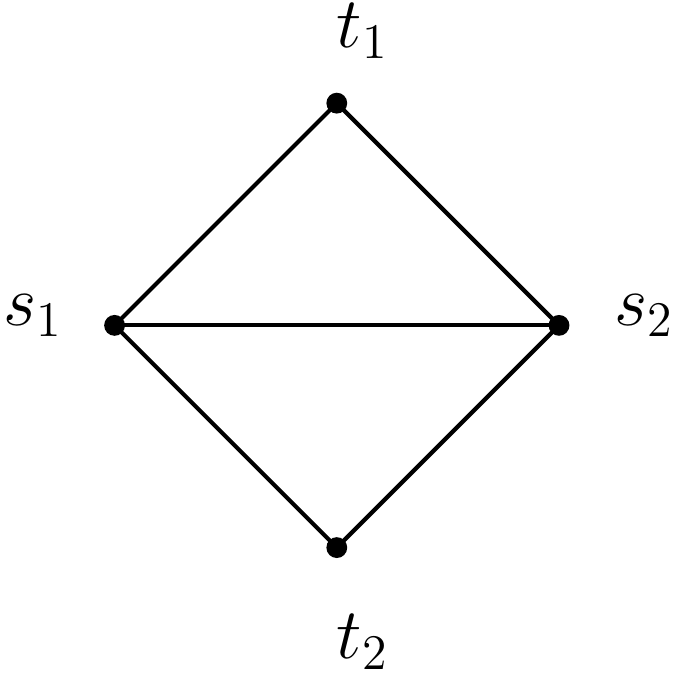}
\caption{The small graph $H$, a diamond from Lemma \ref{smallgraphk22}.}
\label{ThegraphH}
\end{center}
\end{figure}

\begin{lemma}
\label{smallgraphk22}
Let $H$ be the graph where $V(H) = V(K_{2,2})$ and $E(H) = E(K_{2,2}) \cup \{s_{1}s_{2}\}$. Let $G$ be a $2$-connected spanning subgraph of $H^{+}$. From our definition of $K_{2,2}(X)$-minors, let $t_{1},t_{2}$ take the place of $a,b$ and $s_{1}, s_{2}$ take the place of $c$ and $d$. Then $G$ has a $K_{2,2}(X)$-minor, and $G$ does not have a $K_{2,4}(X)$-minor. 
\end{lemma}

\begin{proof}
Notice that $\{s_{1},s_{2}\}$ are the vertex boundary for a $2$-separation $(A,B)$  in $H^{+}$ such that $t_{1} \in A \setminus \{s_{1},s_{2}\}$ and $t_{2} \in B \setminus \{s_{1},s_{2}\}$. Then by Lemma \ref{splitvertices},  $H^{+}$ has no $K_{2,4}(X)$-minor and thus $G$ has no $K_{2,4}(X)$ minor.

 As $G$ is $2$-connected, we can find a $(s_{1},t_{1})$-path in $G[A]$ which does not contain $s_{2}$ and a $(s_{2},t_{2})$-path in $G[B]$ which does not use $s_{1}$. Similarly there exists a $(s_{1},t_{2})$-path in $G[B]$ which does not use $s_{2}$ and a $(s_{2},t_{1})$-path in $G[A]$ which does not use $G[B]$. Therefore by Theorem \ref{k22characterization}, $G$ has a $K_{2,2}(X)$-minor.  
\end{proof}

\begin{corollary}
Let $G$ be a $2$-connected spanning subgraph of a class $\mathcal{A},\mathcal{B},$ or $\mathcal{C}$. If $G$ is the spanning subgraph of a class $\mathcal{A}$ graph, $G$ does not have a $K_{2,4}(X)$-minor. If $G$ is the spanning subgraph of a class $\mathcal{B},$ or $\mathcal{C}$, then $G$ has a $K_{2,4}(X)$-minor.
\end{corollary} 

\begin{proof}
If $G$ is the spanning subgraph of a class $\mathcal{A}$ graph, then $\{d,e\}$ is the vertex boundary of a $2$-separation $(A,B)$ where two vertices of $X$ lie in $A \setminus \{d,e\}$ and one lies in $B \setminus \{d,e\}$. By applying Lemma \ref{splitvertices} we see that $G$ has a $K_{2,4}(X)$-minor if and only if the graph in Lemma \ref{smallgraphk22} has a $K_{2,4}(X)$-minor. By Lemma \ref{smallgraphk22}, it does not. Therefore $G$ is $K_{2,4}(X)$-minor free.

If $G$ is the spanning subgraph of a class $\mathcal{B}$ or $\mathcal{C}$ graph, then $\{e,f\}$ is the vertex boundary of a $2$-separation $(A,B)$ where two vertices of $X$ lie in $A \setminus \{e,f\}$ and two vertices of $X$ lies in $B \setminus \{e,f\}$. Then applying Lemma \ref{thereductionwithk22init}, we see that $G$ has a $K_{2,4}(X)$-minor if the graph in Lemma \ref{smallgraphk22} has a $K_{2,2}(X)$-minor. By Lemma \ref{smallgraphk22}, the graph in question has a $K_{2,2}(X)$-minor and therefore $G$ has a $K_{2,4}(X)$-minor. 
\end{proof}

Now we deal with class $\mathcal{E}$ and $\mathcal{F}$ graphs.

\begin{lemma}
\label{planarreductionk22}
Let $G$ be a $2$-connected spanning subgraph of a $\{t_{1},t_{2},s_{1},s_{2}\}$-web, $H^{+} = (H,F)$. Then there is a planar $2$-connected graph $G'$ such that $G$ has a $K_{2,2}(X)$-minor if and only if $G'$ has a $K_{2,2}(X)$-minor. 
\end{lemma} 

\begin{proof}
We construct $G'$ in the same way that we construct $G'$ in Corollary \ref{webcycle}. We refer the reader to the proof of Corollary \ref{webcycle} for the verification that $G'$ is $2$-connected and planar. Notice that if $G'$ has a $K_{2,2}(X)$-minor, then immediately $G$ has a $K_{2,2}(X)$-minor. 

Therefore we assume that $G$ has a $K_{2,2}(X)$-minor. Then by Theorem \ref{k22characterization}, there is an $(t_{1},s_{1})$-path $P_{t_{1},s_{1}}$ and a $(t_{2},s_{2})$-path, $P_{t_{2},s_{2}}$ such that $P_{t_{1},s_{1}} \cap P_{t_{2},s_{2}} \neq \emptyset$ and there is an $(t_{1},s_{2})$-path, $P_{t_{1},s_{2}}$ and a $(t_{2},s_{1})$-path $P_{t_{2},s_{1}}$ such that $P_{t_{1},s_{2}} \cap P_{t_{2},s_{1}} \neq \emptyset$. 

Suppose $P_{t_{1},s_{1}}$ contains a vertex from $F_{T}$ for some $T$. Then at least two vertices from $T$ are in $P_{t_{1},s_{1}}$. Then since $|V(T)| = 3$, $P_{t_{2},s_{2}}$ does not contain any vertex from $F_{T}$. Therefore if we contract $F_{T}$ down to a vertex, after contracting appropriately      $P_{t_{1},s_{1}}$ is still a $(t_{1},s_{1})$-path, and $P_{t_{1},s_{1}} \cap P_{t_{2},s_{2}} = \emptyset$. A similar statement holds for $P_{t_{2},s_{1}}$ and $P_{t_{1},s_{2}}$. Applying that argument to each triangle $T$ in $H$ and appealing to Theorem \ref{k22characterization}, $G'$ has a $K_{2,2}(X)$-minor. 
\end{proof}

\begin{lemma}
\label{k22minorswebs}
Let $G$ be a $2$-connected spanning subgraph of an $\{t_{1},t_{2},s_{1},s_{2}\}$-web, $H^{+} = (H,F)$. Assume that $t_{1},t_{2},s_{1}$ and $s_{2}$ appear in that order in the outerface of $H$. Then $G$ does not contain a $K_{2,2}(X)$-minor. 
\end{lemma}

\begin{proof}
We note it suffices to show that $H^{+}$ does not have a $K_{2,2}(X)$-minor. By Lemma \ref{planarreductionk22}, we may assume that $H^{+}$ is planar. Then by observation \ref{planarface} the cycle with edge set $t_{1}t_{2}, t_{2}s_{1}, s_{1}s_{2}, s_{2}t_{1}$ is a face. Consider any  $(t_{1},s_{1})$-path $P_{t_{1},s_{1}}$ and any $(t_{2},s_{2})$-path $P_{t_{2},s_{2}}$. We claim that $P_{t_{1},s_{1}} \cap P_{t_{2},s_{2}} \neq \emptyset$. If $t_{2},s_{2} \in P_{t_{1},s_{1}}$ or $t_{1},s_{1} \in P_{t_{2},s_{2}}$ then we are done. Therefore we assume that $t_{2},s_{2} \not \in P_{t_{1},s_{1}}$ and $t_{1},s_{1} \not \in P_{t_{2},s_{2}}$. But then by the Jordan Curve Theorem, $P_{t_{1},s_{1}} \cap P_{t_{2},s_{2}}$ is non-empty, completing the claim. 
\end{proof}

\begin{lemma}
Let $G$ be a $2$-connected spanning subgraph of a class $\mathcal{E}$ or a $\mathcal{F}$ graph. If $G$ is a spanning subgraph of a class $\mathcal{E}$ graph, then $G$ has a $K_{2,4}(X)$-minor if and only if the $\{e,f,c,d\}$-web has a $K_{2,4}(X)$-minor. If $G$ is a spanning subgraph of a class $\mathcal{F}$ graph, then $G$ has a $K_{2,4}(X)$-minor if and only if the $\{e,f,g,h\}$-web has a $K_{2,4}(X)$-minor. 
\end{lemma}

\begin{proof}
First suppose that $G$ is a $2$-connected spanning subgraph of a class $\mathcal{E}$ graph. If the $\{e,f,c,d\}$-web has a $K_{2,4}(X)$-minor, then immediately $G$ has a $K_{2,4}(X)$-minor. 

Therefore assume that $G$ has a $K_{2,4}(X)$-minor. Apply Lemma \ref{thereductionwithk22init} to the two separation with vertex boundary $\{e,f\}$. Then by appealing to Lemma \ref{k22minorswebs} and Lemma \ref{smallgraphk22}, we get that $G$ having a $K_{2,4}(X)$-minor implies the $\{e,f,c,d\}$-web has a $K_{2,4}(X)$-minor, completing the claim. 
Essentially the same argument gives the claim for the class $\mathcal{F}$ graphs. 
\end{proof}

Now it suffices to deal with webs to complete the characterization. Note we can reduce the problem of finding $K_{2,4}(X)$-minors down to the planar case. 

\begin{lemma}
\label{k24planarreduction}
Let $H^{+} = (H,F)$ be an $\{a,b,c,d\}$-web. Let $G$ be $2$-connected spanning subgraph of $H^{+}$. Then there is a planar graph $K$ such that $G$ has a $K_{2,4}(X)$-minor if and only if $K$ has a $K_{2,4}(X)$-minor. 
\end{lemma}

\begin{proof}
For each triangle $T$ in $H$, and every two element subset of $V(T)$ which induces a $2$-separation $(A,B)$ such that $B \setminus (A \cap B) = V(F_{T})$, apply Lemma \ref{allonesidegeneralH}. After doing this to every triangle, notice that for every triangle $T \in H$, $T$ induces a tight $3$-separation $(A,B)$ such that $B = V(F_{T}) \cup V(T)$. Then we may apply Lemma \ref{3connk24} to $(A,B)$. Call the resulting graph $K$. By construction, $K$ has a $K_{2,4}(X)$-minor if and only if $G$ has a $K_{2,4}(X)$-minor. Additionally, notice that $K$ is a subgraph of $H$, and $H$ is planar, so thus $K$ is planar.
\end{proof}

Since Demasi characterized exactly when planar graphs have $K_{2,4}(X)$ minors (\cite{linodemasithesis}), a full characterization of graphs without $K_{2,4}(X)$ and $K_{4}(X)$ is known. It is exactly when the the characterizations given in \cite{root} and \cite{linodemasithesis} agree. We do not explore this to try and give a cleaner characterization, we just note that for webs, $K_{2,4}(X)$ minors only appear when $W_{4}(X)$-minors appear for webs.

\begin{theorem}
\label{k4w4k24characterization}
Let $G$ be a $2$-connected graph which is the spanning subgraph of an $\{a,b,c,d\}$-web. Let $X = \{a,b,c,d\} \subseteq V(G)$.  If $G$ does not have a $W_{4}(X)$-minor, then $G$ does not have a $K_{2,4}(X)$-minor. 
\end{theorem}

\begin{proof}

Since $G$ does not have a $W_{4}(X)$-minor, by Theorem \ref{w4cuts} for every cycle $C$ such that $X \subseteq V(C)$, we have one of five obstructions. By Corollary \ref{webcycle} we know at least one such cycle exists. Suppose for sake of contradiction, that $G$ is a minimal counterexample with respect to vertices. We proceed by checking all the cases. 

\textbf{Case 1:} Suppose that there is a terminal separating $2$-chain $((A_{1},B_{1}),(A_{2},B_{2}),\ldots, \\ (A_{n},B_{n}))$ such that $A_{i} \cap B_{i} \subseteq V(C)$ for all $i \in \{1,\ldots,n\}$. If $A_{1} \cap B_{1}$ contains two vertices of $X$, then by Lemma \ref{splitvertices} there is no $K_{2,4}(X)$-minor. Therefore we can assume that $A_{1} \cap B_{1}$ contains exactly one vertex from $X$. Then by Lemma \ref{splitvertices}, the graph $G$ has $K_{2,4}(X)$-minor if and only if the graph $G_{B_{1}}$ has a $K_{2,4}(X_{1})$-minor, where $X_{1}$ is defined from Lemma \ref{splitvertices}. Since $G$ did not have a $W_{4}(X)$-minor, $G_{B_{1}}$ does not have a $W_{4}(X_{1})$-minor, and thus by minimality $G_{B_{1}}$ does not have a $K_{2,4}(X_{1})$-minor.

\textbf{Case 2:} Suppose that there is a terminal separating triangle $(A_{1},B_{1}),(A_{2},B_{2}),(A_{3},B_{3})$ such that $A_{i} \cap B_{i} \subseteq V(C)$ and $A_{1} \cap B_{1}$ contains a vertex of $X$. Then applying Lemma \ref{splitvertices} to the separation $(A_{1},B_{1})$ we that $G$ has a $K_{2,4}(X)$-minor if and only if the graph $G_{B_{1}}$ has a $K_{2,4}(X_{1})$-minor, where $X_{1}$ is defined from Lemma \ref{splitvertices}. Since $G$ does not have a $W_{4}(X)$-minor, $G_{B_{1}}$ does not have a $W_{4}(X_{1})$-minor, and thus by minimality, $G_{B_{1}}$ does not have a $K_{2,4}(X_{1})$-minor, and therefore $G$ does not have a $K_{2,4}(X)$-minor. 

\textbf{Case 3:} Suppose that there is a terminal separating triangle $(A^{1}_{1},B^{1}_{1}), (A^{1}_{2},B^{1}_{2}),(A^{1}_{3},B^{1}_{3})$ and a terminal separating $2$-chain $((A_{1},B_{1}),\ldots (A_{n},B_{n}))$ in $G_{A^{1}_{1}}$. Then suppose that $A_{n} \cap B_{n}$ contains a vertex of $X$. Then applying Lemma \ref{splitvertices} to the separation $(A_{n},B_{n})$, we have that $G$ has a $K_{2,4}(X)$-minor if and only if $G_{A_{n}}$ has a $K_{2,4}(X_{1})$-minor, where $X_{1}$ is defined from Lemma \ref{splitvertices}. Since $G$ does not have a $W_{4}(X)$-minor, $G_{A_{n}}$ does not have a $W_{4}(X_{1})$-minor, and thus by minimality, $G_{A_{n}}$ does not have a $K_{2,4}(X_{1})$-minor so $G$ does not have a $K_{2,4}(X)$-minor. 

\textbf{Case 4:} Suppose we have terminal separating triangles $((A^{1}_{1},B^{1}_{1}), (A^{1}_{2},B^{1}_{2}), (A^{1}_{3},B^{1}_{3}))$, $((A^{2}_{1},B^{2}_{1}),(A^{2}_{2},B^{2}_{2}),(A^{2}_{3},B^{2}_{3}))$  where for all $i \in \{1,2,3\}$,  $A^{1}_{i} \cap B^{1}_{i} \subseteq A^{2}_{1}$  and $A^{2}_{i} \cap B^{2}_{i} \subseteq A^{1}_{3}$, and if we consider the graph $G[A^{2}_{1} \cap A^{1}_{1}]$, and the cycle $C' = G[V(C) \cap A^{2}_{1} \cap A^{1}_{1}] \cup \{xy | x,y \in A^{i}_{1} \cap B^{i}_{1}, i \in \{1,2\}\}$ and we let $X'$ be defined to be the vertices $A^{2}_{1} \cap B^{2}_{1}$ and $A^{2}_{3} \cap B^{2}_{3}$, then there is a terminal separating $2$-chain in $G[A^{2}_{3} \cap A^{1}_{1}]$ with respect to $X'$. 

Consider the two separation $(A^{1}_{1},B^{1}_{1})$. From previous analysis, there are two vertices of $X \in A^{1}_{1} \setminus (A^{1}_{1} \cap B^{1}_{1})$ and two vertices of $X$ are in $B^{1}_{1} \setminus (A^{1}_{1} \cap B^{1}_{1})$. Without loss of generality, let $a,b$ be the two vertices in $X \in A^{1}_{1} \setminus (A^{1}_{1} \cap B^{1}_{1})$. Therefore we can apply Lemma \ref{thereductionwithk22init} to $(A^{1}_{1},B^{1}_{1})$ and get a two graphs $G_{A^{1}_{1}}$, $G_{B^{1}_{1}}$ such that $G$ has a $K_{2,4}(X)$-minor if and only if either $G_{A^{1}_{1}}$ has a $K_{2,4}(X_{1})$ or both have a $K_{2,2}(X_{1})$-minor, where $X_{1}$ is defined from Lemma \ref{thereductionwithk22init}. Since $G$ did not have a $W_{4}(X)$-minor both $G_{A^{1}_{1}}$ and $G_{B^{1}_{1}}$ do not have a $W_{4}(X_{1})$-minor, and thus by minimality, both $G_{A^{1}_{1}}$ and $G_{B^{1}_{1}}$ do not have $K_{2,4}(X_{1})$-minors. 

Therefore it suffices to show that $G_{B^{1}_{1}}$ does not have a $K_{2,2}(X_{1})$-minor (see Figure \ref{ThegraphGA} for a picture). Let $A^{1}_{1} \cap B^{1}_{1} = \{t_{1},t_{2}\}$. Without loss of generality, let $A^{1}_{2} \cap B^{1}_{2} = \{v,t_{1}\}$ and $A^{1}_{3} \cap B^{1}_{3} = \{v, t_{2}\}$. Now since we had a terminal separating triangle, we have that $v \neq a,b$. So without loss of generality, we may assume $a \in A^{1}_{2} \setminus (A^{1}_{2} \cap B^{1}_{2})$ and $b \in A^{1}_{3} \setminus (A^{1}_{3} \cap B^{1}_{3})$. Now by Theorem \ref{k22characterization}, it suffices to show that for every  $(a,t_{2})$-path, $P_{a,t_{2}}$ and every $(b,t_{1})$-path, $P_{b,t_{1}}$ in $G_{B^{1}_{1}}$ we have $P_{a,t_{2}} \cap P_{b,t_{1}} \neq \emptyset$. Let $P_{a,t_{2}}$ be any $(a,t_{2})$-path. Since $a \in A^{1}_{2} \setminus \{v,t_{1}\}$ and $t_{2} \in B^{1}_{2} \setminus \{v,t_{1}\}$, either $v$ or $t_{1} \in V(P_{a,t_{2}})$. If $t_{1} \in V(P_{a,t_{2}})$, then any $(b,t_{2})$-path $P_{b,t_{2}}$ contains $t_{1}$ by definition, so $P_{a,t_{2}} \cap P_{b,t_{1}} \neq \emptyset$. Therefore we only have to consider when $v \in P_{a,t_{2}}$.
Now since $b \in A^{1}_{3} \setminus \{v,t_{2}\}$ and $t_{1} \in B^{1}_{3} \setminus \{v,t_{2}\}$ every $(b,t_{1})$-path $P_{b,t_{1}}$ contains either $v$ or $t_{2}$. By similar reasoning as above, we may assume that $t_{2} \not \in P_{b,t_{1}}$. Therefore $v \in P_{b,t_{1}}$. But then $P_{b,t_{1}} \cap P_{a,t_{2}} \neq \emptyset$, which implies $G_{B^{1}_{1}}$ does not have a $K_{2,2}(X_{1})$-minor. Combining this with what we already showed, this implies that $G$ does not have a $K_{2,4}(X)$-minor. 

\textbf{Case 5:} Now suppose there are $2$ distinct terminal separating triangles $((A^{1}_{1},B^{1}_{1}), \\ (A^{1}_{2},B^{1}_{2}), (A^{1}_{3},B^{1}_{3})), ((A^{2}_{1},B^{2}_{1}), (A^{2}_{2},B^{2}_{2}),(A^{2}_{3},B^{2}_{3}))$ where for all $i \in \{1,2,3\}$,  $A^{1}_{i} \cap B^{1}_{i} \subseteq A^{2}_{1}$  and $A^{2}_{i} \cap B^{2}_{i} \subseteq A^{1}_{1}$, $A^{2}_{1} \cap B^{2}_{1} \cap A^{1}_{1} \cap B^{1}_{1}$ is not empty and $A^{j}_{i} \cap B^{j}_{i} \subseteq V(C)$ for all $i \in \{1,2,3\}$, $j \in \{1,2\}$. Now notice if we apply Lemma \ref{thereductionwithk22init} to $(A^{1}_{1},B^{1}_{1})$, we get two graphs $G_{A^{1}_{1}}$ and $G_{B^{1}_{1}}$ such that $G$ has a $K_{2,4}(X)$-minor if and only if either one of $G_{A^{1}_{1}}$ or $G_{B^{1}_{1}}$ has a $K_{2,4}(X_{1})$-minor or both have $K_{2,2}(X_{1})$-minors, where $X_{1}$ is defined from Lemma \ref{thereductionwithk22init}. Now since $G$ does not have a $W_{4}(X)$-minor both of $G_{B^{1}_{1}}$ and $G_{A^{1}_{1}}$  do not have $W_{4}(X_{1})$-minors and thus by minimality both do not have $K_{2,4}(X_{1})$-minors. Observe that we can apply the same argument as case four to $G_{B^{1}_{1}}$ to obtain that $G_{B^{1}_{1}}$ does not have a $K_{2,2}(X_{1})$-minor. Therefore $G$ does not have a $K_{2,4}(X)$-minor, completing the claim. 
\end{proof}

\begin{figure}
\begin{center}
\includegraphics[scale =0.5]{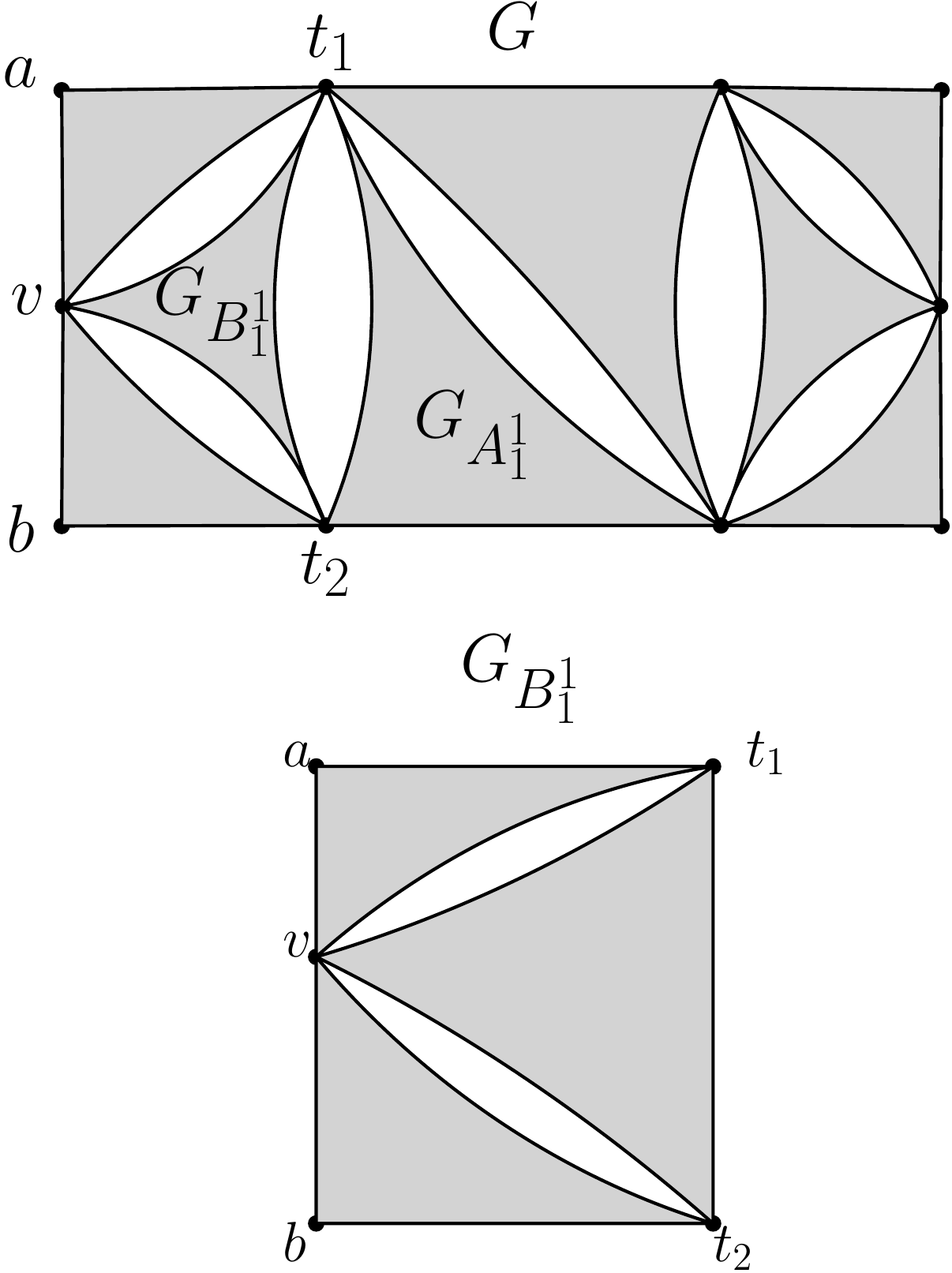}
\caption{The graph $G_{B^{1}_{1}}$ in case four of Theorem \ref{k4w4k24characterization}. The shaded sections are spanning subgraphs of webs.}
\label{ThegraphGA}
\end{center}
\end{figure}

As a recap of what we have so far.

\begin{corollary}
Let $G$ be a $2$-connected graph and $X = \{a,b,c,d\} \subseteq V(G)$. The graph $G$ has no $K_{4}(X)$, $K_{2,4}(X)$ or $W_{4}(X)$-minor if and only if $G$ belongs to class $\mathcal{A}$ (see Theorem \ref{k4free}) or $G$ is the spanning subgraph of a class $\mathcal{D}$, $\mathcal{E}$ and $\mathcal{F}$ graph and the corresponding web does not have a $W_{4}(X)$-minor (see Theorem \ref{w4cuts}). 
\end{corollary}

\section{A characterization of graphs without a $K_{4}(X)$, $W_{4}(X)$, $K_{2,4}(X)$ or an $L(X)$-minor}
We define the graph $L$ to have vertex set $V(L) = \{v_{1},\ldots,v_{8}\}$ and $E(L) = \{v_{1}v_{2},v_{1}v_{5},v_{2}v_{7}, \\ v_{2}v_{8},v_{2}v_{3},v_{3}v_{4},v_{4}v_{5}, v_{4}v_{7},v_{5}v_{6},v_{6}v_{7},v_{6}v_{8},v_{7}v_{8}\}$ (see Figure \ref{L(X)labelling}). Let $G$ be a graph and $X = \{a,b,c,d\} \subseteq V(G)$. Let $\mathcal{F}$ be the family of maps from $X$ to $V(L)$ where each vertex of $X$ goes to a distinct vertex in $\{v_{1},v_{3},v_{4},v_{5}\}$. For the purposes of this thesis, an $L(X)$-minor refers to the $X$ and family of maps defined above. It is easy to see that the graph $L$ is $2$-connected, so the cut vertex section applies. Therefore we may assume that all graphs are at least $2$-connected. 

We let $L'$ denote the graph induced by $\{v_{2},v_{8},v_{7},v_{6},v_{5},v_{4}\}$ in $L$. Let $G$ be a graph and $X = \{a,b,c\} \subseteq V(G)$. Let $\mathcal{F}$ be the family of surjective maps from $X$ to $\{v_{2},v_{4},v_{5}\}$. An $L'(X)$-minor will refer to the $\mathcal{F}$ and $X$ above. It is easy to see that $L'$ is $2$-connected and thus we may assume all graphs are $2$-connected. 

\begin{figure}
\begin{center}
\includegraphics[scale=0.5]{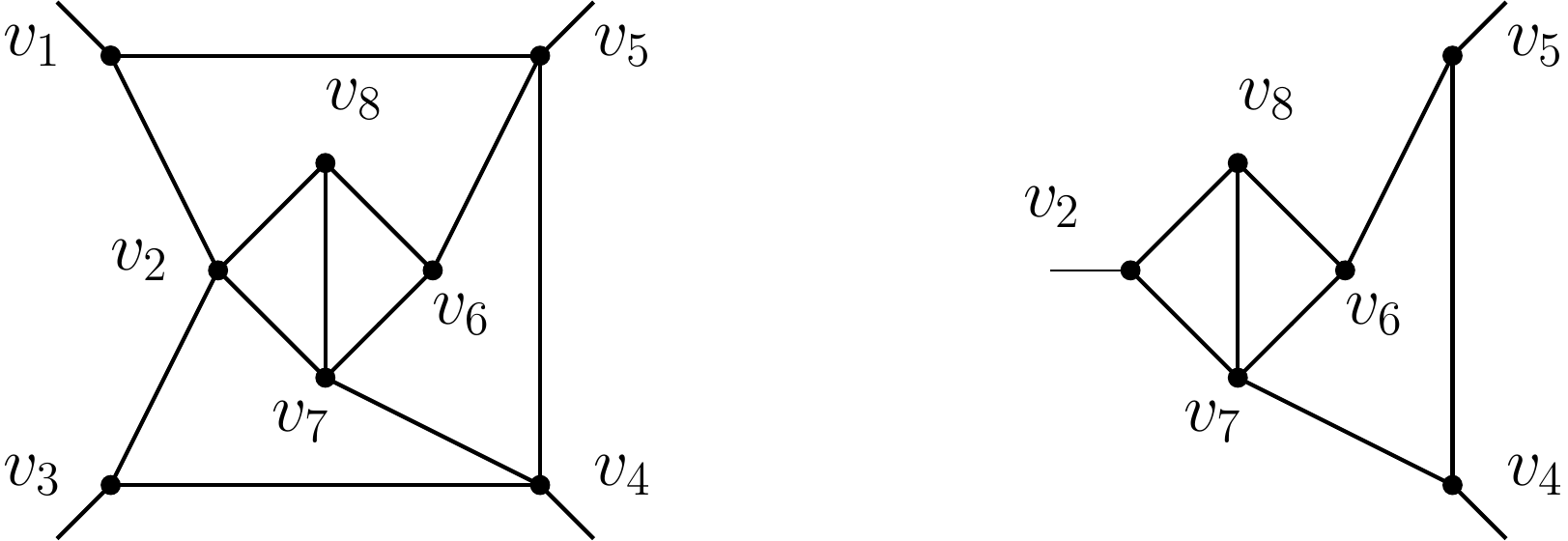}
\caption{The graph $L$ and the graph $L'$. Vertices with lines with only one endpoint indicate the vertices where the roots are being mapped to.}
\label{L(X)labelling}
\end{center}
\end{figure}

\begin{lemma}
\label{L'(X')minors}
Let $G$ be a $2$-connected graph and let $X = \{a,b,c\}$. Then $G$ has an $L'(X)$-minor if and only if there are three distinct cycles $C_{1}$, $C_{2}$, $C_{3}$ and three distinct paths $P_{1},P_{2},P_{3}$ satisfying the following properties: 
\begin{enumerate}
\item{$|V(C_{1}) \cap V(C_{2})| \geq 2$, $|V(C_{2}) \cap V(C_{3})| \geq  2$,  $|C_{3}| \geq 4$, there is at least one edge in $E(C_{2})$ which is not contained in either of $E(C_{1})$ and $E(C_{2})$, and there exists a vertex $v_{1} \in (V(C_{2}) \cap V(C_{3})) \setminus V(C_{1})$ and a vertex $v_{2} \in (V(C_{1}) \cap V(C_{2})) \setminus V(C_{3})$.}

\item{The vertices $a,b$ and $c$ are endpoints of $P_{1},P_{2}$ and $P_{3}$ respectively. Additionally, the other endpoint of $P_{1}$ is in $V(C_{1})$, and the other endpoint of point $P_{2}$ and $P_{3}$ is in  $V(C_{3})$. Furthermore, $P_{1} \cap C_{j} = \emptyset$ for all $j \in \{2,3\}$, and $P_{i} \cap C_{j} = \emptyset$ for all $i \in \{2,3\}, j \in \{1,2\}$.}
\end{enumerate}
\end{lemma}

\begin{proof}

Let $\{G_{x} | x \in V(L')\}$ be a model of an $L'(X)$-minor in $G$. Let $x_{1}$ be a vertex in $G_{v_{2}}$ which is adjacent to a vertex $x_{2} \in G_{v_{7}}$. Let $x_{3}$ be a vertex in $G_{v_{7}}$ which is adjacent to a vertex $x_{4}$ in $G_{v_{8}}$. Let $x_{5}$ be a vertex in $G_{v_{8}}$ which is adjacent to a vertex $x_{6}$ in $G_{v_{2}}$. Then since $G_{v_{2}}$ is connected, there is a $(x_{1},x_{6})$-path, $P_{x_{1},x_{6}}$, contained in $G_{v_{2}}$. Similarly, there is a $(x_{2},x_{3})$-path, $P_{x_{2},x_{3}}$, contained in $G_{v_{7}}$ and a $(x_{4},x_{5})$-path, $P_{x_{4},x_{5}}$, contained in $G_{v_{8}}$. Then $C_{1} = P_{x_{1},x_{6}} \cup P_{x_{2},x_{3}} \cup P_{x_{4},x_{5}} \cup \{x_{1}x_{2}, x_{3}x_{4},x_{5}x_{6}\}$ is a cycle. 

Now there is a vertex $v_{1} \in G_{v_{7}}$ which is adjacent to a vertex $v_{2} \in G_{v_{6}}$. Additionally there is a vertex $v_{3} \in G_{v_{6}}$ which is adjacent to a vertex, $v_{4} \in G_{v_{8}}$. As $G_{v_{6}}$ is connected, there is a $(v_{2},v_{3})$-path, $P_{v_{2},v_{3}}$, contained in $G_{v_{6}}$. As $v_{4}$ and $x_{4}$ are in $G_{v_{8}}$, there is a $(v_{4},x_{4})$-path, $P_{v_{4},x_{4}}$, contained in $G_{v_{8}}$. Similarly, as $v_{1}$ and $x_{3}$ are in $G_{v_{7}}$, there is a $(v_{1},x_{3})$-path, $P_{v_{1},x_{3}}$, contained in $G_{v_{7}}$. Then $C_{2} = P_{v_{2},v_{3}} \cup P_{v_{1},x_{3}} \cup P_{v_{4},x_{4}} \cup \{v_{3}v_{4}, x_{3}x_{4},v_{1}v_{2}\}$ is a cycle.

By definition of an $L(X)$-model, there is a vertex $y_{1} \in G_{v_{5}}$ which is adjacent to a vertex $y_{2} \in G_{v_{4}}$. There is a vertex $y_{3} \in G_{v_{4}}$ which is adjacent to a vertex $y_{4} \in G_{v_{7}}$. There is a vertex $y_{5}$ in $G_{v_{5}}$ which is adjacent to a vertex $y_{6}$ in $G_{v_{6}}$. As $G_{v_{7}}$ is connected, this is a $(v_{1},y_{4})$-path, $P_{v_{1},y_{4}}$, contained in $G_{v_{7}}$. As $G_{v_{6}}$ is connected, there is a $(v_{2},y_{6})$-path, $P_{v_{2},y_{6}}$, which is contained in $G_{v_{6}}$. As $G_{v_{5}}$ is connected, there is a $(y_{1},y_{5})$-path, $P_{y_{1},y_{5}}$ contained in $G_{v_{5}}$. As $G_{v_{4}}$ is connected there is a $(y_{3},y_{2})$-path, $P_{y_{3},y_{2}}$ which is contained inside $G_{v_{4}}$. Then let $C_{3} = P_{y_{3},y_{2}} \cup P_{v_{1},y_{4}} \cup P_{y_{3},y_{2}} \cup P_{y_{1},y_{5}} \cup \{y_{3}y_{4}, v_{1}v_{2},y_{6}y_{5},y_{1}y_{2}\}$ is a cycle. Notice that $|C_{3}| \geq 4$ and $|C_{3} \cap C_{2}| \geq 2$. Now without loss of generality let $a \in G_{v_{2}}$. There is a $(a, x_{1})$-path $P_{a,x_{1}}$ contained in $G_{v_{2}}$ as $G_{v_{2}}$ is connected. Without loss of generality let $b \in G_{v_{5}}$. There is a $(b,y_{1})$-path, $P_{b,y_{1}}$, contained in $G_{v_{5}}$ since $G_{v_{5}}$ is connected. Then $c \in G_{v_{4}}$ and there is a $(c,y_{2})$-path $P_{c,y_{2}}$ contained in $G_{v_{4}}$. Then notice it is easy to see that that $P_{1},P_{2},P_{3}, C_{1},C_{2},C_{3}$ satisfy the claim.

Conversely, if given $P_{1},P_{2},P_{3},C_{1},C_{2}$ and $C_{3}$ satisfying the lemma statement, we simply contract $P_{1}$, $P_{2}$,$P_{3}$ down to a single vertex, contract $C_{1}$ and $C_{2}$ to a diamond, and then $C_{3}$ to a $4$-cycle. 
\end{proof}

As with the other rooted minor sections, we start off with some lemmas about how the $L(X)$-minor behaves across $2$-separations.

\begin{lemma}
\label{allononesideL(X)}
Let $G$ be a graph and $(A,B)$ be a $2$-separation with vertex boundary $\{x,y\}$. If $X = \{a,b,c,d\} \subseteq A$, then $G$ has an $L(X)$-minor if and only if $G_{A} = G[A] \cup \{xy\}$ has an $L(X)$-minor. 
\end{lemma}

\begin{proof}
If $G_{A}$ has an $L(X)$-minor then since $G_{A}$ is a minor of $G$ by contracting all of $G[B]$ onto $\{x,y\}$, we get that $G$ has an $L(X)$-minor.

Conversely, let $\{G_{x} | x \in V(L)\}$ be a model for an $L(X)$-minor. We claim that $\{G_{A}[V(G_{x}) \cap A] | x \in V(L)\}$ is a model of an $L(X)$-minor in $G_{A}$. If there is only one branch set containing vertices of $B$ then the result trivially holds. Therefore we may assume there are at least two distinct branch sets containing vertices of $B$. If the branch sets containing vertices of $B$ are two of $G_{v_{1}}, G_{v_{3}},G_{v_{4}},G_{v_{5}}$, then since $X \subseteq A$, all other branch sets are contained inside $G[A]$. Then since $xy \in E(G_{A})$, $\{G_{A}[V(G_{x}) \cap A] | x \in V(L)\}$ is a model of an $L(X)$-minor.

 Now suppose that exactly one of $G_{v_{1}},G_{v_{3}},G_{v_{4}},$ and $G_{v_{5}}$ contains a vertex from $B$. Suppose that $G_{v_{1}}$ is the branch set which contains the vertex. Then suppose that $G_{v_{2}}$ is the other branch set which contains a vertex from $B$. Since $\deg(v_{2}) >2$, we may assume that $x \in G_{v_{1}}$ and $y \in G_{v_{2}}$ and all other branch sets are contained in $G[A \setminus \{x,y\}]$. Then since $xy \in E(G_{A})$, $\{G_{A}[V(G_{x}) \cap A] \ | \ x \in V(L)\}$ is a model for an $L(X)$-minor. A similar analysis works for the other cases.

Therefore we can assume that two of $G_{v_{2}}$, $G_{v_{8}}$, $G_{v_{7}}$, and $G_{v_{6}}$ contain vertices from $B$. Suppose $G_{v_{2}}$ and $G_{v_{8}}$ are two branch sets containing vertices from $B$. Then since $\deg(v_{2}) >2$ and $\deg(v_{8}) >2$ in $L$, we may assume that $u \in G_{v_{2}}$ and $x \in G_{v_{8}}$ and that all other branch sets are contained in $A \setminus \{x,y\}$. But then since $xy \in E(G_{A})$, $\{G_{A}[V(G_{x}) \cap A] | x \in V(L)\}$ is a model for an $L(X)$-minor. The other cases follow by essentially the same argument. 
\end{proof}

\begin{lemma}
\label{terminalseparatingL(X)}
Let $G$ be a graph and $(A,B)$ a $2$-separation with vertex boundary $\{x,y\}$. Let $X = \{a,b,c,d\} \subseteq V(G)$. If $x,y \in X$ and one vertex of $X$ lies in $A \setminus \{x,y\}$, and the other lies in $B \setminus \{x,y\}$, then $G$ does not have an $L(X)$-minor. 
\end{lemma}

\begin{proof}
Suppose towards a contradiction that $\{G_{x} | x \in V(L)\}$ was a model of an $L(X)$-minor in $G$. If $x = v_{1}$ and $y = v_{3}$ then without loss of generality we have that $G_{v_{4}} \subseteq G[A \setminus \{x,y\}]$ and $G_{v_{5}} \subseteq G[B \setminus \{x,y\}]$. But this contradicts that there is a vertex in $G_{v_{4}}$ which is adjacent to a vertex in $G_{v_{5}}$. The same argument works if $x= v_{1}$, $y =v_{5}$ or if $x = v_{3}$ or $y = v_{4}$.  

Now assume that $x = v_{1}$ and $y = v_{4}$. Then without loss of generality we may assume that $G_{v_{3}} \subseteq G[A \setminus \{x,y\}]$ and $G_{v_{5}} \subseteq G[B \setminus \{x,y\}]$. Then $G_{v_{2}} \subseteq G[A \setminus \{x,y\}]$ as $G_{v_{2}}$ has a vertex which is adjacent to a vertex in $G_{v_{3}}$. By similar reasoning, $G_{v_{8}}, G_{v_{7}}$ and $G_{v_{6}}$ are all contained in $G[A \setminus \{x,y\}]$. But then there is no vertex in $G_{v_{6}}$ which is adjacent to a vertex in $G_{v_{5}}$, a contradiction. The case where $x = v_{3}$ and $y = v_{5}$ follows similarly. 

Lastly, assume that $x= v_{4}$ and $y = v_{5}$. Without loss of generality, we may assume that $G_{v_{3}}$ is contained in $G[B \setminus \{x,y\}]$ and $G_{v_{1}}$ is contained in $G[A \setminus \{x,y\}]$. Then $G_{v_{2}}$ is contained in either $G[A \setminus \{x,y\}]$ or $G[B \setminus \{x,y\}]$. Suppose that $G_{v_{2}}$ is contained in $G[A \setminus \{x,y\}]$. But then there is no vertex in $G_{v_{3}}$ which is adjacent to a vertex in $G_{v_{2}}$. Then $G_{v_{2}}$ is contained in $G[B \setminus \{x,y\}]$. But then there is no vertex in $G_{v_{1}}$ which is adjacent to a vertex in $G_{v_{2}}$, a contradiction. 
\end{proof}

\begin{lemma}
\label{subdivisionlabellingL(X)}
Let $G$ be a graph and $(A,B)$ be a $2$-separation with vertex boundary $\{u,v\}$. Let $X = \{a,b,c,d\} \subseteq V(G)$.  Suppose that there is exactly one vertex, $z$, such that $z \in X \cap (A \setminus \{u,v\})$, and exactly two vertices from $X$ in $B \setminus \{x,y\}$ and $u \in X$. Let $X_{1} = X \setminus \{z\} \cup \{v\}$. For each $\pi \in \mathcal{F}$, define $\pi': X_{1} \to V(L)$ such that $\pi' = \pi$ on $X \setminus \{z\}$ and $\pi'(v) = \pi(z)$. If there is a model of an $L(X)$-minor, $\{G_{x} | x \in L(X)\}$, then either $\{G_{B}[V(G_{x}) \cap B)]| x \in V(L)\}$ is a model of an $L(X_{1})$-minor in $G_{B}$, or the vertex from $X$ in $ A \setminus \{u,v\}$ is not in branch sets $G_{v_{4}}$ or $G_{v_{5}}$. 
\end{lemma}

\begin{proof}
Suppose $\{G_{x} | x \in L(X)\}$ is a model of an $L(X)$-minor, and suppose that $\{G_{B}[V(G_{x}) \cap B)] \ | \  x \in V(L)\}$ is not a model of an $L(X_{1})$-minor. Furthermore, suppose the vertex from $X$ in $A \setminus \{u,v\}$ is in $G_{v_{4}}$.

 First consider the case when $v \in G_{v_{4}}$. Suppose any of $G_{v_{2}}$, $G_{v_{7}}$, $G_{v_{6}}$, or $G_{v_{8}}$ 
is contained in $G[A \setminus \{u,v\}]$. Since $v_{2},v_{7},v_{6}$ and $v_{8}$ induce a diamond in $L$, and $v \in G_{v_{4}}$, and $u \in X$, each of $G_{v_{2}}$, $G_{v_{7}}$, $G_{v_{6}}$ and $G_{v_{8}}$ are contained in $G[A \setminus \{u,v\}]$. But then at least two of $G_{v_{1}}, G_{v_{3}}$, and $G_{v_{5}}$ are contained in $G[B \setminus \{u,v\}]$.  But this is a contradiction, since in $L$, all of $v_{1},v_{5}$ and $v_{4}$ are adjacent to at least one of $v_{2},v_{7}$ and $v_{6}$. Therefore we can assume that none of $G_{v_{2}}$, $G_{v_{7}}$, $G_{v_{6}}$, or $G_{v_{8}}$ are in $G[A \setminus \{u,v\}]$. But then since two vertices of $X$ lie in $B \setminus \{u,v\}$, there are at most two branch sets containing vertices from $A$. But then $\{G_{B}[V(G_{x}) \cap B)]| x \in V(L)\}$ is a model of an $L(X_{1})$-minor in $G_{B}$, a contradiction. 

Therefore we can assume that $v \not \in V(G_{v_{4}})$, and thus $G_{v_{4}} \subseteq G[A \setminus \{u,v\}]$. Now at least one of $G_{v_{3}}$ and $G_{v_{5}}$ contains a vertex from $B$ which is not $u$, and thus either $v \in G_{v_{3}}$ or $v \in G_{v_{5}}$. In either case, this implies that $G_{v_{7}}$ is contained in $G[A \setminus \{u,v\}]$. By the same reasoning as before, this implies that all of $G_{v_{2}}, G_{v_{8}}$ and $G_{v_{6}}$ are contained in $G[A \setminus \{u,v\}]$. But then since at least one of $G_{v_{1}}, G_{v_{3}}$ and $G_{v_{5}}$ are contained in $G[B \setminus \{u,v\}]$, which contradicts that $\{G_{x} | x \in L(X)\}$ is an $L(X)$-model. The case where the vertex from $X$ in $ A \setminus \{u,v\}$ is in $G_{v_{5}}$ follows similarly. 
 \end{proof}

\begin{lemma}
\label{twooneachsideL(X)}
Let $G$ be a $2$-connected graph and $(A,B)$ be a $2$-separation with vertex boundary $\{x,y\}$. Let $X = \{a,b,c,d\} \subseteq V(G)$.  Suppose  $a, b \in A \setminus (A \cap B)$ and $c,d \in B \setminus (B \cap A)$. Let $X_{1}= (X \cap A) \cup \{x,y\}$. For each $\pi \in \mathcal{F}$, define $\pi_{1}: X_{1} \to V(L)$ such that $\pi_{1} = \pi$ on $a,b$ and $\pi_{1}(x) = \pi(c)$ and $\pi_{1}(d) = \pi$. Let $X_{2} = (X \cap B) \cup \{x,y\}$. For each $\pi \in \mathcal{F}$, define $\pi_{2}: X \to V(L)$ such that $\pi_{2} = \pi$ on $\{c,d\}$ and $\pi_{2}(x) = \pi(a)$ and $\pi_{2}(y) = \pi(b)$.  Then $G$ has an $L(X)$-minor if and only if $G_{A} = G[A] \cup \{xy\}$ has an $L(X_{1})$-minor or $G_{B} = G[B] \cup \{xy\}$ has an $L(X_{2})$-minor. 
\end{lemma}

\begin{proof}
If $G_{A}$ has an $L(X_{1})$-minor, then we can contract $B$ onto $\{x,y\}$ such that the vertices of $X$ do not get contracted together. This is possible since $G$ is $2$-connected. Then $G$ has an $L(X)$-minor. A similar argument works when $G_{B}$ has an $L(X_{2})$-minor. Now assume that $\{G_{x} | x \in V(L)\}$ is a model of an $L(X)$-minor in $G$.

Suppose $x$ is in one of $G_{v_{2}}$, $G_{v_{7}}, G_{v_{6}},$ or $G_{v_{8}}$.  We consider cases based on which branch sets contain the vertices of $X$.

First, suppose that the vertices from $X$ in $A \setminus \{x,y\}$ are in branch sets $G_{v_{1}}$ and $G_{v_{3}}$. Then the vertices in $X$ in $B \setminus \{x,y\}$ are in branch sets $G_{v_{5}}$ and $G_{v_{4}}$. But then since at most one branch set contains $y$, either there is no vertex in $G_{v_{3}}$ adjacent to a vertex in $G_{v_{4}}$ or there is no vertex in $G_{v_{1}}$ which is adjacent to a vertex in $G_{v_{5}}$.

Now suppose  the vertices from $X$ in $A \setminus \{x,y\}$ are in branch sets $G_{v_{1}}$ and $G_{v_{4}}$. Then the vertices in $X$ in $B \setminus \{x,y\}$ are in branch sets $G_{v_{5}}$ and $G_{v_{3}}$. But then since at most one branch set contains $y$, either there is no vertex in $G_{v_{3}}$ adjacent to a vertex in $G_{v_{4}}$ or there is no vertex in $G_{v_{1}}$ which is adjacent to a vertex in $G_{v_{5}}$. In either case, this is a contradiction. 

Now suppose the vertices from $X$ in $A \setminus \{x,y\}$ are in branch sets $G_{v_{1}}$ and $G_{v_{5}}$. Then the vertices from $X$ in $B \setminus \{x,y\}$ are in branch sets $G_{v_{3}}$ and $G_{v_{4}}$. First suppose that $x \not \in G_{v_{2}}$. Then since $v_{1}$ and $v_{3}$ are adjacent to $v_{2}$ in $L$, we have that $y \in G_{v_{2}}$. But then there is no vertex in $G_{v_{4}}$ which is adjacent to a vertex in $G_{v_{5}}$, a contradiction. Therefore $x \in G_{v_{2}}$. Then $y \in G_{v_{4}}$ or $G_{v_{5}}$. Then $G_{v_{8}},G_{v_{7}}$ and $G_{v_{6}}$ all have vertex sets in either $B \setminus \{x,y\}$ or $A \setminus \{x,y\}$. But then either $G_{v_{6}}$ does not have a vertex adjacent to a vertex in $G_{v_{5}}$ or $G_{v_{4}}$ does not have a vertex adjacent to a vertex in $G_{v_{7}}$. In either case, this is a contradiction.

Therefore $x$ is not in $V(G_{v_{2}})$, $V(G_{v_{6}}), V(G_{v_{7}}),$ or $V(G_{v_{8}})$. Similarly, $y$ is not in $V(G_{v_{2}})$, $V(G_{v_{6}}), V(G_{v_{7}}),$ or $V(G_{v_{8}})$.  Then $x$ and $y$ belong to two of $G_{v_{1}},G_{v_{3}},G_{v_{4}}$ and $G_{v_{5}}$. Now suppose that $G_{v_{1}}$ is contained in $G[A \setminus \{x,y\}]$ and $G_{v_{3}}$ is contained in $G[B \setminus \{x,y\}]$. Then since $x,y$ belong to two of $G_{v_{1}},G_{v_{3}},G_{v_{4}}$ and $G_{v_{5}}$, we have that $G_{v_{2}}$ is contained in one of $G[B \setminus \{x,y\}]$ or $G[A \setminus \{x,y\}]$. But then without loss of generality there is no vertex in $G_{v_{2}}$ which is adjacent to a vertex in $G_{v_{1}}$, a contradiction. A similar analysis shows that for any two of $G_{v_{1}},G_{v_{3}},G_{v_{4}}$ and $G_{v_{5}}$, if one of the branch sets is contained in $G[B \setminus \{x,y\}]$ and the other in $G[A \setminus \{x,y\}]$ we get a contradiction. Therefore the two branch sets from $G_{v_{1}},G_{v_{3}},G_{v_{4}}$ and $G_{v_{5}}$ which do not contain $x,y$ are contained on the same side of the $2$-separation. Suppose the two branch sets from $G_{v_{1}},G_{v_{3}},G_{v_{4}}$ and $G_{v_{5}}$ which do not contain $x,y$ are contained in $G[A \setminus \{x,y\}]$. Then by the same reasoning as above, all of $G_{v_{2}},G_{v_{8}},G_{v_{6}},$ and $G_{v_{7}}$ are contained in $G[A \setminus \{x,y\}]$. But then $\{G_{A}[V(G_{x}) \cap A] | x \in V(L)\}$ is a model of an $L(X_{1})$-minor in $G_{A}$. In the other case, by the same argument we get an $L(X_{2})$-minor in $G_{B}$.
\end{proof}

Now we show that class $\mathcal{A}$ graphs do not have an $L(X)$-minors.

\begin{lemma}
\label{noclassAgraphsL(X)}
Let $G$ be $2$-connected spanning subgraph of a class $\mathcal{A}$ graph. Then $G$ does not have an $L(X)$-minor. 
\end{lemma}

\begin{proof}
Suppose towards a contradiction that we have a model of an $L(X)$-minor in $G$, $\{G_{x} | x \in V(L)\}$. Then observe that $\{e,d\}$ is the vertex boundary of a separation $(A,B)$ such that $a \in A \setminus \{d,e\}$ and $b,c \in B \setminus \{d,e\}$ so the hypotheses of Lemma \ref{subdivisionlabellingL(X)} are satisfied. Then we consider two cases. If $\{G_{B}[V(G_{x} \cap B)]| x \in V(L)\}$ is a model of an $L(X_{1})$-minor, then notice that this means that a spanning subgraph of a graph $H^{+}$, where $H$ is the graph defined in Lemma \ref{smallgraphk22} has an $L(X_{1})$-minor. But in this graph $H^{+}$, $\{d,e\}$ is the vertex boundary of a $2$-separation satisfying Lemma \ref{terminalseparatingL(X)}, and thus $H^{+}$ does not have an $L(X_{1})$-minor, a contradiction. 

Therefore we must be in the case where $a \in G_{v_{1}}$ or $a \in G_{v_{3}}$. Notice that $\{e,d\}$  induces a $2$-separation in $G$, $(A',B')$ such that $c \in A' \setminus \{e,d\}$ and $a,b \in B' \setminus \{e,d\}$. Additionally, $\{e,d\}$ is the vertex boundary of a $2$-separation $(A'',B'')$ such that $b \in A'' \setminus \{e,d\}$, and $a,c \in B'' \setminus \{e,d\}$. Therefore by applying the exact same analysis as the the $(A,B)$ separation, we get that all of $a,b,$ and $c$ are in $G_{v_{1}}$ or $G_{v_{3}}$, which is a contradiction.  
\end{proof}

Now we reduce class $\mathcal{E}$ and $\mathcal{F}$ down to looking at webs. 

\begin{lemma}
Let $G$ be a $2$-connected spanning subgraph of a class $\mathcal{E}$ or a $\mathcal{F}$ graph. If $G$ is a spanning subgraph of a class $\mathcal{E}$ graph, then $G$ has an $L(X)$-minor if and only if the $\{e,f,c,d\}$-web has an $L(X)$-minor. If $G$ is a spanning subgraph of a class $\mathcal{F}$ graph, then $G$ has an $L(X)$-minor if and only if the $\{e,f,g,h\}$-web has an $L(X)$-minor.
\end{lemma}

\begin{proof}
First suppose that $G$ is a $2$-connected spanning subgraph of a class $\mathcal{E}$ graph. Then apply Lemma \ref{twooneachsideL(X)} to the separation whose vertex boundary is $\{e,f\}$. Notice that one of the graphs we get from Lemma \ref{twooneachsideL(X)} is the graph $H^{+}$ where $H$ is the graph from Lemma \ref{smallgraphk22}. Notice that $H^{+}$ does not have an $L(X)$-minor since $\{e,f\}$ forms a separation satisfying Lemma \ref{terminalseparatingL(X)}. Notice that the other graph we obtain from Lemma \ref{twooneachsideL(X)} is the $\{e,f,c,d\}$-web, which completes the claim. The argument for the class $\mathcal{F}$ graphs is essentially the same.
\end{proof}

Therefore it suffices to look at graphs which are spanning subgraphs of webs satisfying the obstructions given in Theorem \ref{w4cuts}. We will look at each case separately, but it turns out that essentially all the obstructions reduce to instances of terminal separating $2$-chains having $L(X)$-minors. First we look at the terminal separating $2$-chain obstruction. 

\begin{figure}
\begin{center}
\includegraphics[scale =0.5]{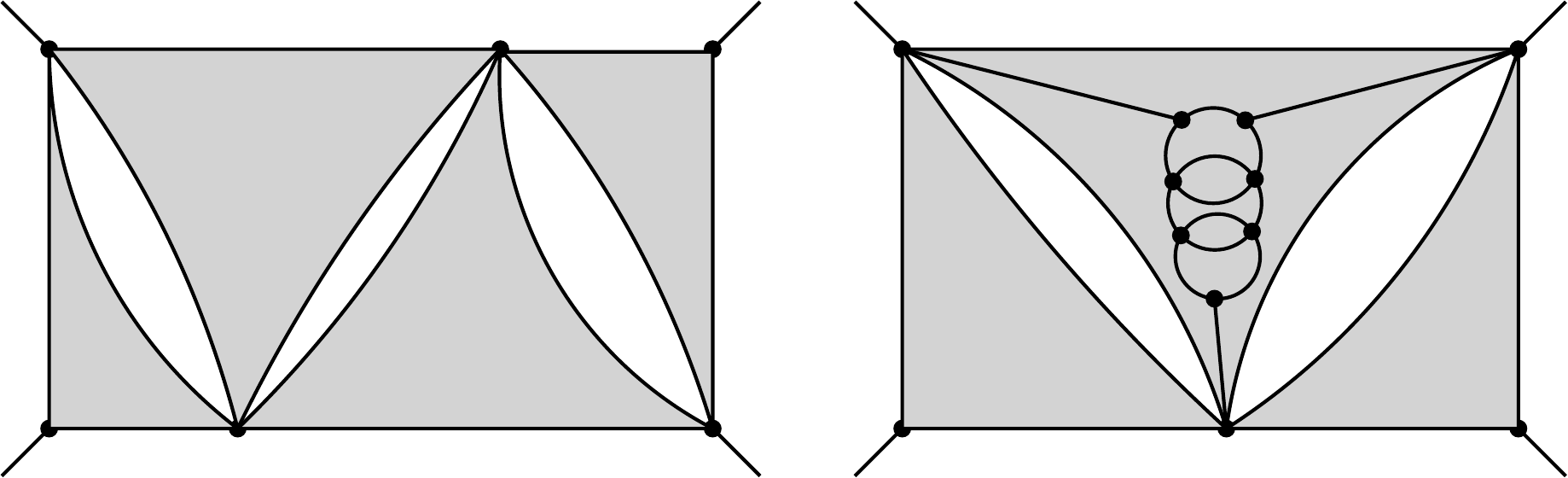}
\caption{Graph on the left has no $L(X)$-minor by Lemma \ref{terminalseparting2chainsL(X)} and the graph on the right has an $L(X)$-minor and satisfies the conditions of Lemma \ref{terminalseparting2chainsL(X)}. Curved lines indicate $2$-separations. The shaded sections are spanning subgraphs of webs.}
\end{center}
\end{figure}

\begin{lemma}
\label{terminalseparting2chainsL(X)}
Let $G$ be a $2$-connected spanning subgraph of an $\{a,b,c,d\}$-web. Let $X =\{a,b,c,d\} \subseteq V(G)$ and suppose that $G$ does not have a $W_{4}(X)$-minor. Consider any cycle $C$ for which $X \subseteq V(C)$. Suppose $a,b,c,d$ appears in that order on $C$. Suppose there is a terminal separating $2$-chain $((A_{1},B_{1}),(A_{2},B_{2}),\ldots, (A_{n},B_{n}))$ satisfying obstruction $1$ from Theorem \ref{w4cuts}. Furthermore, suppose that $A_{1} \cap B_{1} \cap X = \{a\}$, and $b = A_{1} \cap X \setminus (A_{1} \cap B_{1})$. If $n \equiv 1 \pmod 2$, then $G$ has no $L(X)$-minor. Now consider when $n \equiv 0 \pmod 2$. Then $c \in A_{n} \cap X \setminus (A_{n} \cap B_{n})$. Then $G$ has an $L(X)$-minor if and only if the following occurs. There exists an $i \in \{1,\ldots,n-1\}$, $i \equiv 1 \pmod 2$, such that the graph $G_{B_{i} \cap A_{i+1}} =  G[B_{i} \cap A_{i+1}] \cup \{xy | x,y \in A_{i} \cap B_{i}\} \cup \{xy | x,y \in A_{i+1} \cap B_{i+1}\}$ has an $L'(X')$-minor, where $X' = \{(A_{1} \cap B_{1}), (A_{2} \cap B_{2})\}$. Furthermore, the vertex from $X'$ which is contained in the branch set $G_{v_{2}}$ lies on the $(b,c)$-path $P_{b,c}$ where $a,d \not \in V(P_{b,c})$ and $V(P_{b,c}) \subseteq V(C)$.
\end{lemma}

\begin{proof}
First suppose that $n \equiv 1 \pmod 2$. If $n=1$, then $(A_{1},B_{1})$ satisfies Lemma \ref{terminalseparatingL(X)} and thus $G$ does not have an $L(X)$-minor. Now suppose $n \geq 3$ and $G$ is a vertex minimal counterexample to the claim. Then apply Lemma \ref{twooneachsideL(X)} to $(A_{2},B_{2})$ gives two graphs $G_{A_{2}}$ and $G_{B_{2}}$ such that $G$ has an $L(X)$-minor if and only if $G_{A_{2}}$ has an $L(X_{1})$-minor or $G_{B_{2}}$ has an $L(X_{2})$-minor.   Now both $G_{A_{2}}$ and $G_{B_{2}}$ have terminal separating $2$-chains of odd length. Namely, $(A_{1},B_{1})$ is a terminal separating $2$-chain in $G_{A_{2}}$ and $(A_{3},B_{3}),\ldots,(A_{n},B_{n})$ is a terminal separating $2$-chain in $G_{B_{2}}$ which has length $n-2$, which is odd.  Then since $G$ is a vertex minimal counterexample, both $G_{A_{2}}$ and $G_{B_{2}}$ do not have $L(X)$-minors and therefore $G$ does not have an $L(X)$-minor. 

Now suppose that $n \equiv 0 \pmod 2$. First suppose that $n \geq 4$ and suppose that $G$ is a vertex minimal counterexample. Then apply Lemma \ref{twooneachsideL(X)} to $(A_{2},B_{2})$. Then we get two graphs $G_{A_{2}}$ and $G_{B_{2}}$ such that $G$ has an $L(X)$-minor if and only if either $G_{A_{2}}$ has an $L(X_{1})$-minor or $G_{B_{2}}$ has an $L(X_{2})$-minor. We note that $G_{A_{2}}$ does not have an $L(X_{1})$-minor since $(A_{1},B_{1})$ satisfies Lemma \ref{terminalseparatingL(X)}. Therefore since $G$ has an $L(X)$-minor, $G_{B_{2}}$ has an $L(X_{2})$-minor. Furthermore, $((A_{3},B_{3}),\ldots,(A_{n},B_{n}))$ is an even length terminal separating $2$-chain in $G_{B_{2}}$. Then as $G$ is a vertex minimal counterexample, $G_{B_{2}}$ satisfies the claim. But then notice that the vertex in $X_{2}$ which takes the place of $b$ in the claim lies on $P_{b,c}$ since $b \in A_{1} \setminus (A_{1} \cap B_{1})$ implies that the vertex in $(A_{1} \cap B_{1}) \setminus \{a\}$ lies on $P_{b,c}$. Therefore $G$ satisfies the claim, a contradiction.

Therefore we can assume that $n =2$, and suppose $\{G_{x} \ | \ x \in V(L)\}$ is a model of an $L(X)$-minor in $G$. Apply Lemma \ref{subdivisionlabellingL(X)} to $(A_{1},B_{1})$. Then we have two cases. First suppose we have an $L(X_{1})$-minor in $G_{B_{1}}$ where $X_{1} = X \setminus \{b\} \cup (A_{1} \cap B_{1})$. Then in $G_{B_{1}}$ under $X_{1}$, the $2$-separation $(A_{2},B_{2})$ satisfies Lemma \ref{terminalseparatingL(X)}, implying that $G_{B_{1}}$ does not have an $L(X_{1})$-minor, a contradiction. Therefore $b \in G_{v_{1}}$ or $b \in G_{v_{3}}$. First suppose that $b \in G_{v_{1}}$. Notice that $(A_{2},B_{2})$ satisfies Lemma \ref{subdivisionlabellingL(X)}, and that by symmetry we may assume that $c \in G_{v_{1}}$ or $c \in G_{v_{3}}$. As we assumed $b \in G_{v_{1}}$, we assume that $c \in G_{v_{3}}$. Suppose that $G_{v_{1}}$ is not contained in $G[A_{1} \setminus (A_{1} \cap B_{1})]$. Then the vertex in $A_{1} \cap B_{1} \setminus \{a\}$ is in $G_{v_{1}}$. Then $G_{v_{3}}$ is contained in $G[B_{2}]$. Then since $G_{v_{3}}$ has a vertex adjacent to a vertex in $G_{v_{2}}$, we have that $G_{v_{2}}$ is contained in $G[B_{1}]$. Then this implies that $G_{v_{8}}, G_{v_{6}},$ and $G_{v_{7}}$ are contained in $G[B_{2} \setminus (A_{2} \cap B_{2})]$. But $a \in G_{v_{5}}$ or $a \in G_{v_{4}}$, so either $G_{v_{4}}$ is contained in $G[A_{2} \setminus (A_{2} \cap B_{2})]$ or $G_{v_{5}}$ is contained in $G[A_{2} \setminus (A_{2} \cap B_{2})]$. But that contradicts that $\{G_{x} \ | \ x \in L(X)\}$ is an $L(X)$-minor. Therefore $G_{v_{1}}$ is contained in $G[A_{1} \setminus (A_{1} \cap B_{1})]$. By symmetry, $G_{v_{3}}$ is contained in $G[B_{2} \setminus (A_{2} \cap B_{2})]$. That implies that $G_{v_{2}}$ contains the vertex in $A_{1} \cap B_{1} \setminus \{a\}$. Suppose that $G_{v_{8}}$ is contained in $G[A_{1} \setminus (A_{1} \cap B_{1})]$. Then since $G_{v_{2}}$ contains the vertex in $A_{1} \cap B_{1} \setminus \{a\}$, the branch sets $G_{v_{7}}$ and $G_{v_{6}}$ are contained in $G[A_{1} \setminus (A_{1} \cap B_{1})]$. But then since $d \in G_{v_{5}}$ or $d \in G_{v_{4}}$, this implies that either $G_{v_{5}}$ is contained in $G[B_{1} \setminus (A_{1} \cap B_{1})]$ or $G_{v_{4}}$ is contained in $G[B_{1} \setminus (A_{1} \cap B_{1})]$. But this contradicts that $\{G_{x} \ | \ x \in V(L)\}$ is a model of an $L(X)$-minor. By essentially the same argument, none of $G_{v_{8}}, G_{v_{7}}$ or $G_{v_{6}}$ are contained in $G[B_{2} \setminus (A_{2} \cap B_{2})]$. Therefore all of $G_{v_{8}}, G_{v_{7}}$ and $G_{v_{6}}$ are contained in $G[B_{1} \cap A_{2}]$. But then the set $\{G_{B_{1} \cap A_{2}}[G_{x} \cap B_{1} \cap A_{2}] \ | \ x \in V(L')\}$ is a model of $L'(X')$ minor satisfying the properties of the lemma.  

Now we prove the converse.  Suppose there exists an $i \in \{1,\ldots,n-1\}$ and $i \equiv 1 \pmod 2$  such that the graph $G_{B_{i} \cap A_{i+1}}$ has an $L'(X')$-minor satisfying the properties in the lemma. Let $\{G_{x} | x \in V(L')\}$ be a model of an $L'(X')$-minor satisfying the above properties. Let $v$ be the vertex in $X'$ which is also in $G_{v_{2}}$ in the $L'(X)$ model. Observe that since $i \equiv 1 \pmod 2$, there is exactly one vertex in $X'$ which lies on $P_{b,c}$, so therefore $v \in P_{b,c}$ and both other vertices of $X'$ lie on $P_{a,d}$, the $(a,d)$-path such that $V(P_{a,d}) \subseteq V(C)$ and $b,c \not \in V(P_{a,d})$. Then to get an $L(X)$-minor in $G$, first let $x$ be the vertex in $X'$ such that the $(a,x)$-subpath on $P_{a,d}$ does not contain the other vertex from $X'$ on $P_{a,d}$, and without loss of generality, let $x \in G_{v_{5}}$. Then extend $G_{v_{5}}$ to include the subpath from $(a,x)$ on $P_{a,d}$. Now let $y$ be the vertex in $X$ such that the $(y,d)$-subpath on $P_{a,d}$ does not contain $x$. Since $x \in G_{v_{5}}$, we have $y \in G_{v_{4}}$. Extend $G_{v_{4}}$ along the $(y,d)$-subpath on $P_{a,d}$. Now we create $G_{v_{1}}$  by letting it be the $(a,v)$-path, $P_{a,v}$ such that $V(P_{a,v}) \subseteq V(C)$ and $d,c \not \in P_{a,v}$ and do not include either $a$ or $v$. Note that $b \in V(P_{a,v})$. Similarly, let $G_{v_{3}}$ be the $(d,v)$-path, $P_{d,v}$ such that $V(P_{d,v}) \subseteq V(C)$ and $a,b \not \in P_{d,v}$, not including either $d$ or $v$.
 we $G_{v_{1}}$ contain the path from $v$ to $b$ on $P_{b,c}$, not including $v$, we let $G_{v_{3}}$ contain the path from $v$ to $c$ on $P_{b,c}$ not including $v$. Then by construction and since we already had an $L'(X')$-minor, we have an $L(X)$-minor in $G$.
\end{proof}

Now we look at when our graph has a terminal separating triangle as in obstruction $2$.

\begin{lemma}
\label{trianglereductionL(X)}
Let $G$ be a $2$-connected spanning subgraph of an $\{a,b,c,d\}$-web. Let $X =\{a,b,c,d\} \subseteq V(G)$ and suppose that $G$ does not have a $W_{4}(X)$-minor. Consider any cycle $C$ for which $X \subseteq V(C)$, and suppose that there is a terminal separating triangle $(A_{1},B_{1}),(A_{2},B_{2}),(A_{3},B_{3})$ satisfying obstruction $2$ of Theorem \ref{w4cuts}. Then $G$ has an $L(X)$-minor if and only if either the graph $G_{B_{1}}$ has an $L(X_{1})$-minor where $X_{1} = (X \cap B_{1}) \cup (A_{1} \cap B_{1})$ or the graph $G_{B_{3}}$ has an $L(X_{2})$-minor  where $X_{2} = (X \cap B_{3}) \cup (A_{3} \cap B_{3})$. 
\end{lemma}

\begin{figure}
\begin{center}
\includegraphics[scale =0.5]{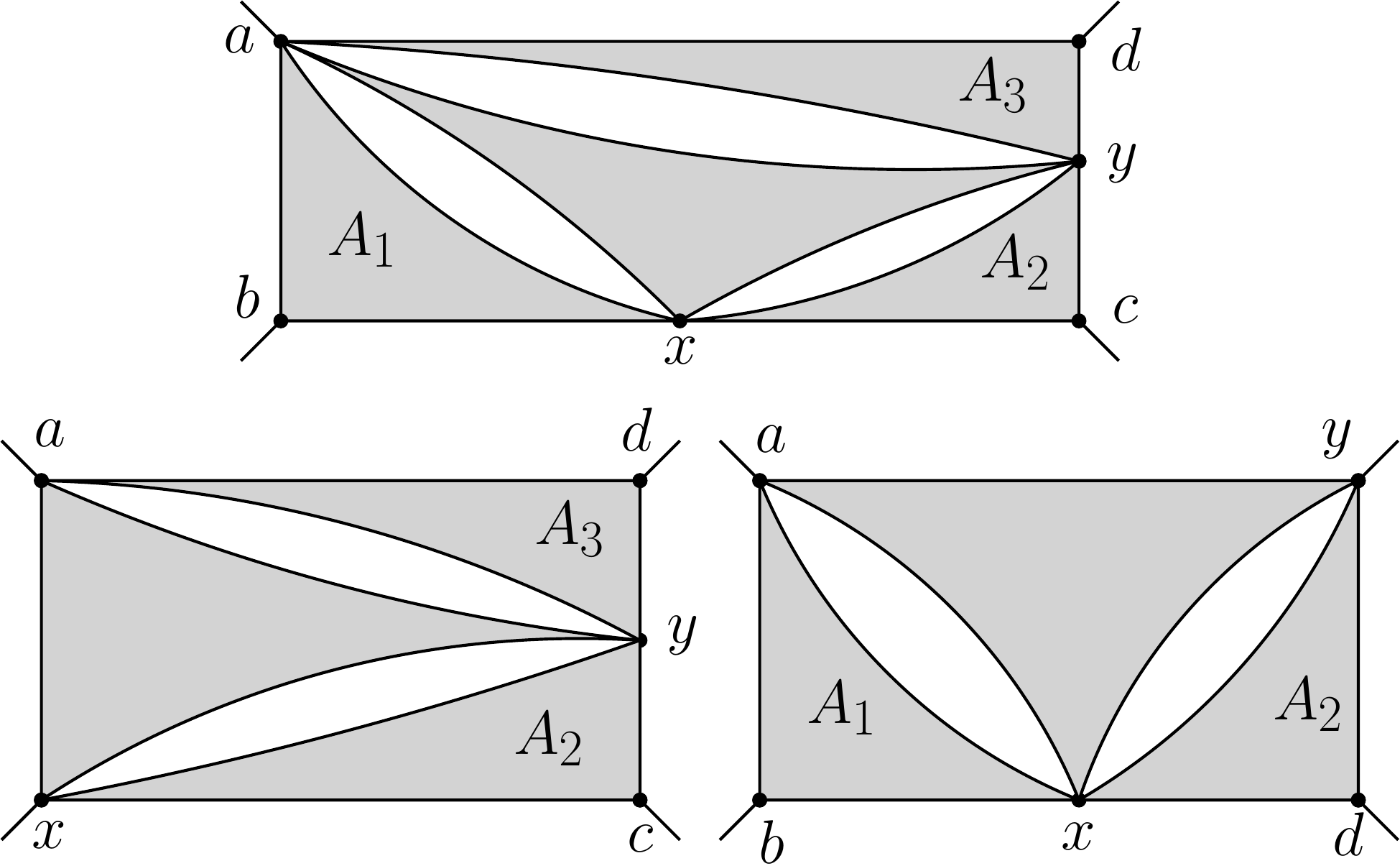}
\caption{A terminal separating triangle and the graphs $G_{B_{1}}$ and $G_{B_{3}}$ as in Lemma \ref{trianglereductionL(X)}. Note both $G_{B_{1}}$ and $G_{B_{3}}$ have terminal separating $2$-chain obstructions. Shaded sections are spanning subgraphs of webs.}
\end{center}
\end{figure}

\begin{proof}
If $G_{B_{1}}$ has an $L(X_{1})$-minor then since $G$ is $2$-connected we can contract the vertex in $X$ in $A_{1}$ to $A_{1} \cap B_{1}$ such that it does not get contracted together with another vertex of $X$. But then $G$ has an $L(X)$-minor. A similar argument holds for $G_{B_{3}}$.

Now suppose that $\{G_{x} | x \in V(L)\}$ is a model of an $L(X)$-minor in $G$, and furthermore suppose we have a vertex minimal counterexample to the claim. Without loss of generality let $a,b,c,d$ appear in that order on $C$, and without loss of generality suppose that $a$ is the vertex from $X$ in $A_{1} \cap B_{1}$, and that $b \in  A_{1} \setminus (A_{1} \cap B_{1})$. Similarly, without loss of generality we may assume that $c \in A_{2} \setminus (A_{2} \cap B_{2})$ and $d \in A_{3} \setminus (A_{3} \cap B_{3})$. Then notice that $(A_{1},B_{1})$ satisfies Lemma \ref{subdivisionlabellingL(X)}. We consider two cases. 

First consider the case where $b \in G_{v_{1}}$ or $b \in G_{v_{3}}$. We assume that $b \in G_{v_{1}}$. Now notice that since we assumed $d \in A_{3}$, $(A_{3},B_{3})$ is a $2$-separation satisfying Lemma \ref{subdivisionlabellingL(X)}. First suppose that $d \in G_{v_{3}}$ (note this is without loss of generality, since $b \in G_{v_{1}}$). First consider the case where $G_{v_{1}}$ is contained in $G[A_{1} \setminus (A_{1} \cap B_{1})]$. Then $G_{v_{2}}$ contains the vertex in $(A_{1} \cap B_{1})$ which is not $a$. Similarly, if $G_{v_{3}}$ is contained in $G[A_{3} \setminus (A_{3} \cap B_{3})]$, then $G_{v_{2}}$ contains the vertex in $A_{3} \cap B_{3} \setminus \{a\}$. But then $a \in B_{2} \setminus (A_{2} \cap B_{2})$ and $c \in A_{2} \setminus (A_{2} \cap B_{2})$, which implies that there is no vertex in $G_{v_{5}}$ which is adjacent to a vertex in $G_{v_{4}}$, a contradiction. Therefore $G_{v_{3}}$ is not contained in $G[A_{3} \setminus (A_{3} \cap B_{3})]$ and the vertex in $A_{3} \cap B_{3} \setminus \{a\}$ is in $G_{v_{3}}$. But then as before, $a \in B_{2} \setminus (A_{2} \cap B_{2})$ and $c \in A_{2} \setminus (A_{2} \cap B_{2})$, contradicting that there is a vertex in $G_{v_{5}}$ which is adjacent to a vertex in $G_{v_{4}}$. Therefore we can assume that $d \not \in G_{v_{3}}$. But then by Lemma \ref{subdivisionlabellingL(X)}, $G_{B_{3}}$ has an $L(X_{2})$-minor, which satisfies the claim. 

Therefore we can assume that $G_{v_{1}}$ is not contained in $G[A_{1} \setminus (A_{1} \cap B_{1})]$. Since the cases are symmetric, the only situation left to consider is when  $d \in G_{v_{3}}$ and not contained in $G[A_{3} \setminus (A_{1} \cap B_{1})]$. Then the vertex in $(A_{1} \cap B_{1}) \setminus \{a\}$ and the vertex in $(A_{3} \cap B_{3}) \setminus \{a\}$ are in $G_{v_{1}}$ and $G_{v_{3}}$ respectively. But then there is no vertex in $G_{v_{4}}$ which is adjacent to $G_{v_{5}}$, a contradiction. 
The same situation happens when $b \in G_{v_{3}}$. Therefore we assume that $b \not \in G_{v_{3}}$ and $b \not \in G_{v_{1}}$. But then by Lemma \ref{subdivisionlabellingL(X)}, $G_{B_{1}}$ has an $L(X_{1})$-minor, which completes the claim.
\end{proof}

Observe that both the graphs $G_{B_{1}}$ and $G_{B_{2}}$ have terminal separating $2$-chains as obstructions, and thus when we have a terminal separating triangle as an obstruction, the problem reduces to appealing to Lemma \ref{terminalseparting2chainsL(X)}. Now we look at what happens when the graph has a terminal separating $2$-chain and a terminal separating triangle as in obstruction $3$. 

\begin{lemma}
\label{terminalseparatingchainplustriangleL(X)}
Let $G$ be a $2$-connected spanning subgraph of an $\{a,b,c,d\}$-web. Let $X = \{a,b,c,d\} \subseteq V(G)$ and suppose that $G$ does not have a $W_{4}(X)$-minor. Consider any cycle $C$ for which $X \subseteq V(C)$, and suppose there is a terminal separating triangle $(A_{1},B_{1}),(A_{2},B_{2}), \\ (A_{3},B_{3})$ such that $A_{1} \cap B_{1} = \{x,y\}$, where $x,y \not \in X$. Furthermore, the graph $G_{A_{1}} = G[A_{1}] \cup \{xy\}$ and $C_{A} = G[V(C) \cap A] \cup \{x,y\}$  has a terminal separating $2$-chain $(A'_{1},B'_{1}),\ldots,(A'_{n},B'_{n})$ where we let $x$ and $y$ replace the two vertices in $X$ from $G$ not in $G_{A_{1}}$ and $A_{i} \cap B_{i} \subseteq V(C)$, $i \in \{1,2,3\}$, and $A'_{i} \cap B'_{i} \subseteq V(C_{A})$ for all $i \in \{1,\ldots,n\}$. Then $G$ has an $L(X)$-minor if and only if the graph $G_{B_{1}} = G[B_{1}] \cup \{xy | x,y \in A_{1} \cap B_{1}\}$ has an $L(X_{1})$-minor where $L(X_{1}) = X \cap B_{1} \cup (A_{1} \cap B_{1})$ or $G_{A_{1}} =  G[B_{1}] \cup \{xy | x,y \in A_{1} \cap B_{1}\}$ has an $L(X_{2})$-minor where $X_{2} = X \cap A_{1} \cup (A_{1} \cap B_{1})$.
\end{lemma}

\begin{proof}
This follows immediately from applying Lemma \ref{twooneachsideL(X)} to $(A_{1},B_{1})$.
\end{proof}

The point of the above observation is that the graphs  $G_{A_{1}}$ and $G_{B_{1}}$ have terminal separating $2$-chain obstructions, and thus the problem of finding $L(X)$-minors in that case reduces to the problem of finding $L(X)$-minors in terminal separating  $2$-chains case, which is done in Lemma \ref{terminalseparatingL(X)}. Now we look at what happens when we have two terminal separting triangles and a terminal separating $2$-chain as in obstruction $4$.

\begin{lemma}
Let $G$ be a $2$-connected spanning subgraph of an $\{a,b,c,d\}$-web. Let $X = \{a,b,c,d\} \subseteq V(G)$ and suppose that $G$ does not have a $W_{4}(X)$-minor. Consider any cycle $C$ for which $X \subseteq V(C)$, and suppose there are two distinct terminal separating triangles $((A^{1}_{1},B^{1}_{1}), (A^{1}_{2},B^{1}_{2}), (A^{1}_{3},B^{1}_{3}))$, $((A^{2}_{1},B^{2}_{1}),(A^{2}_{2},B^{2}_{2}),(A^{2}_{3},B^{2}_{3}))$  where for all $i \in \{1,2,3\}$,  $A^{1}_{i} \cap B^{1}_{i} \subseteq A^{2}_{1}$  and $A^{2}_{i} \cap B^{2}_{i} \subseteq A^{1}_{3}$. Furthermore, if we consider the graph $G[A^{2}_{1} \cap A^{1}_{1}]$ and the cycle $C' = G[V(C) \cap A^{2}_{1} \cap A^{1}_{1}] \cup \{xy | x,y \in A^{i}_{1} \cap B^{i}_{1}, i \in \{1,2\}\}$ and we let $X'$ be defined to be the vertices $A^{2}_{1} \cap B^{2}_{1}$ and $A^{2}_{3} \cap B^{2}_{3}$, then there is a terminal separating $2$-chain in $G[A^{2}_{3} \cap A^{1}_{1}]$ with respect to $X'$. Then $G$ has an $L(X)$-minor if and only if either $G_{B^{1}_{1}}$ has an $L(X_{1})$-minor where $X_{1} = X \cap B^{1}_{1} \cup (A^{1}_{1} \cap B^{1}_{1})$ or $G_{B^{2}_{1}}$ has an $L(X_{2})$-minor where $X_{2} = X \cap B^{2}_{1} \cup (A^{1}_{1} \cap B^{1}_{1})$ or $G[A^{2}_{3} \cap A^{1}_{1}]$ has an $L(X')$-minor. 
\end{lemma}

\begin{proof}
Apply Lemma \ref{twooneachsideL(X)} to $(A^{1}_{1},B^{1}_{1})$. Then $G$ has an $L(X)$-minor if and only if $G_{B^{1}_{1}}$ has an $L(X_{1})$ minor where $X_{1} = X \cap B^{1}_{1} \cup (A^{1}_{1} \cap B^{1}_{1})$ or $G_{A^{1}_{1}}$ has an $L(X'_{1})$-minor where $X'_{1} = X \cap A^{1}_{1} \cup (A^{1}_{1} \cap B^{1}_{1})$. Notice that in $G_{A^{1}_{1}}$, we have a terminal separating $2$-chain and terminal separating triangle as in obstruction $3$. But then we can apply Lemma \ref{terminalseparatingchainplustriangleL(X)} which gives us exactly the claim.  
\end{proof}

\begin{lemma}
\label{2disjointtrianglesL(X)}
Let $G$ be a $2$-connected spanning subgraph of an $\{a,b,c,d\}$-web. Let $X = \{a,b,c,d\} \subseteq V(G)$ and suppose that $G$ does not have a $W_{4}(X)$-minor. Consider any cycle $C$ for which $X \subseteq V(C)$, and suppose that there two distinct terminal separating triangles $((A^{1}_{1},B^{1}_{1}), (A^{1}_{2},B^{1}_{2}), (A^{1}_{3},B^{1}_{3})), ((A^{2}_{1},B^{2}_{1}), (A^{2}_{2},B^{2}_{2}),(A^{2}_{3},B^{2}_{3}))$ which satisfy obstruction $5$ in Theorem \ref{w4cuts}. Then $G$ has an $L(X)$-minor if and only if either the graph $G_{B^{1}_{1}} = G[B^{1}_{1}] \cup \{xy | x,y \in A^{1}_{1} \cap B^{1}_{1}\}$ has an $L(X_{1})$-minor where $X_{1} = (X \cap B^{1}_{1}) \cup (A^{1}_{1} \cap B^{1}_{1})$ or the graph $G_{B^{2}_{1}}= G[B^{2}_{1}] \cup \{xy | x,y \in A^{2}_{1} \cap B^{2}_{1}\}$ has an $L(X_{2})$-minor where $X_{2} = (X \cap B^{2}_{1}) \cup (A^{2}_{1} \cap B^{2}_{1})$.  
\end{lemma}

\begin{proof}
If $G_{B^{1}_{1}}$ has an $L(X_{1})$-minor, then we simply contract all of $A^{1}_{1}$ onto $B^{1}_{1} \cap A^{1}_{1}$ such that we do not contract two vertices of $X$ together. This is possible as $G$ is $2$-connected. The same strategy applies to $G_{B^{2}_{1}}$.

Now suppose that $G$ has an $L(X)$-minor. Then $(A^{1}_{1},B^{1}_{1})$ satisfies Lemma \ref{twooneachsideL(X)}. Therefore $G$ has an $L(X)$-minor if and only if $G_{B^{1}_{1}}$ has an $L(X_{1})$-minor where $X_{1} = X \cap B^{1}_{1} \cup (B^{1}_{1} \cap A^{1}_{1})$  or $G_{A^{1}_{1}}$ has an $L(X'_{1})$-minor where $X'_{1} = A^{1}_{1} \cap X \cup (A^{1}_{1} \cap B^{1}_{1})$. If $(A^{1}_{1},B^{1}_{1}) = (A^{2}_{1},B^{2}_{1})$, then the claim follows immediately. Thus we assume that $(A^{1}_{1},B^{1}_{1}) \neq (A^{2}_{1},B^{2}_{1})$. Notice that in $G_{A^{1}_{1}}$, the triangle  $(A^{2}_{1},B^{2}_{1}), (A^{2}_{2},B^{2}_{2}),(A^{2}_{3},B^{2}_{3})$ is a terminal separating triangle satisfying obstruction $2$. Then by Lemma \ref{trianglereductionL(X)}, $G_{A^{1}_{1}}$ has an $L(X'_{1})$-minor if and only if the graph $G_{B^{2}_{1}}$ has an $L(X_{2})$-minor, which completes the claim.
\end{proof}

With that, one can determine exactly which graphs do not have an $K_{4}(X)$-minor, $W_{4}(X)$-minor, $K_{2,4}(X)$-minor and an $L(X)$-minor. To sum up what we have, a graph does not have one of the four above minors if and only if $G$ is a class $\mathcal{A}$ graph, or it is the spanning subgraph of a class $\mathcal{D}$, $\mathcal{E}$, or $\mathcal{F}$ graph where the corresponding web has one of the obstructions from Theorem \ref{w4cuts}, and then after the reductions given above, we do not end up with an $L'(X')$-minor (Lemma \ref{L'(X')minors}) satisfying the properties in Lemma \ref{terminalseparatingL(X)}. We note it would be nice to obtain a cleaner structure theorem for when a graph has an $L'(X')$-minor. 

\section{$K_{3,4}$ and $K_{6}$-minors}

As we noted before, $K_{3,4}$ is a forbidden minor for reduciblity of the first Symanzik polynomial. Therefore it is interesting to look at the structure of graphs not containing a $K_{3,4}$-minor. A full characterization of graphs without $K_{3,4}$-minors is not known, although partial progress has been made. Namely, when $G$ is embeddable in the projective plane, two different but equivalent characterizations of graphs not containing $K_{3,4}$-minors are known (\cite{K34minorsprojpart1}, \cite{K34minorsprojpart2}). We can use the fact that $K_{3,4}$ is not reducible with respect to the first Symanzik polynomial to gain some insight on the structure of reducible graphs.

It will be useful to consider some rooted minor classes. First, let $V(K_{2,3}) = \{s_{1},s_{2},t_{1},t_{2},t_{3}\}$, and $E(K_{2,3}) = \{s_{i}t_{j} \ | i \in \{1,2\}, j \in \{1,2,3\}\}$. Let $G$ be a graph and $X = \{a,b,c\} \subseteq V(G)$. Let $\mathcal{F}$ be the family of maps such that each vertex of $X$ gets mapped to a distinct vertex of $\{t_{1},t_{2},t_{3}\}$. For the purposes of this thesis, a $K_{2,3}(X)$-minor will refer to the $X$ and $\mathcal{F}$ above.

Let $V(K_{3,3}) = \{s_{1},s_{2},s_{3},t_{1},t_{2},t_{3}\}$ and $E(K_{3,3}) = \{s_{i}t_{j} \ | i,j \in \{1,2,3\} \}$. Let $G$ be a graph and $X = \{a,b,c\} \subseteq V(G)$. Let $\mathcal{F}$ be the family of maps from $X$ to $\{t_{1},t_{2},t_{3}\}$. For the purposes of this thesis, a $K_{3,3}(X)$-minor will refer to the $X$ and $\mathcal{F}$ above.

Let $V(K_{3,1}) = \{s_{1},t_{1},t_{2},t_{3}\}$ and $E(K_{3,1}) = \{s_{1}t_{i} | i \in \{1,2,3\}\}$. Let $G$ be a graph and $X = \{a,b,c\} \subseteq V(G)$. Let $\mathcal{F}$ be the family of maps from $X$ to $\{t_{1},t_{2},t_{3}\}$. For the purposes of this thesis, $K_{3,1}(X)$-minor will refer to the $X$ and $\mathcal{F}$ above. 

Suppose we know a graph $G$ does not have a $K_{4}(X)$-minor. Then we know from Theorem \ref{k4free} that $G$ belongs to one of six classes of graphs. We can ask the question, when do one of these graphs have a $K_{3,4}$-minor? Clearly, if one of these graphs has a $K_{3,4}$-minor, then the graph must be non-planar. By how each of the classes of graphs are constructed, that implies that one of the $F_{T}$ sections must be causing the non-planarity. Thus there is a section of the graph contained inside a $3$-separation which is causing the non-planarity. It is reasonable to suspect if there are a few distinct triangles where the graph induced by $F_{T}$ is non-planar, and there is enough connectivity between the triangles, that you could obtain a $K_{3,4}$-minor. We show this now, using the following theorem from Robertson and Seymour.

\begin{theorem}[\cite{rootedk23theorem}]
\label{K23(X)minors}
Let $G$ be a $3$-connected graph and $X = \{a,b,c\} \subseteq V(G)$. Then either $G$ has a $K_{2,3}(X)$-minor on $a,b,c$ or $G$ is planar and $a,b,c$ lie on a face.
\end{theorem}

\begin{corollary}
Let $G$ be a graph. Let $T_{1} = \{x_{1},x_{2},x_{3}\}$ be the vertex boundary of a $3$-separation $(A_{1},B_{1})$. Let $T_{2} = \{y_{1},y_{2},y_{3}\}$ be the vertex boundary of a $3$-separation $(A_{2},B_{2})$. Suppose that $A_{1} \cap A_{2} = \emptyset$, $G[A_{i}]$ is non-planar for $i \in \{1,2\}$,  and that in $G[B_{1} \cap B_{2}]$ there are $3$ disjoint $(T_{1},T_{2})$-paths. Then $G$ has a $K_{3,4}$-minor. 
\end{corollary} 

\begin{proof}
Let $X_{1} = T_{1}$ and $X_{2} = T_{2}$. By Theorem \ref{K23(X)minors}, $G[A_{i}]$ has a $K_{2,3}(X_{i})$ minor for $i \in \{1,2\}$. Then since $A_{1} \cap A_{2} = \emptyset$ and there are three disjoint $(T_{1},T_{2})$-paths, we can contract $G[A_{i}]$ to a $K_{2,3}(X_{i})$-minor and then contract along the three disjoint $(T_{1},T_{2})$-paths to obtain a $K_{3,4}$-minor. 
\end{proof}

Recall that for graphs without a $K_{4}(X)$-minor, they fall into one of six classes that are all formed based off of taking a graph $H$ and then constructing the graph $H^{+} = (H,F)$. Then for that scenario, the above corollary says to not have a $K_{3,4}$-minor, for every pair of triangles $T_{1}$ and $T_{2}$ which are not separated by a $2$-separation, at most one of $T_{1} \cup F_{T_{1}}$, and $T_{2} \cup F_{T_{2}}$ is non-planar. 

We could play the exact same game with $K_{3,3}(X)$-minors and rooted $K_{3,1}(X)$-minors. We have the following immediate proposition. 

\begin{proposition}
Let $G$ be a graph and let $(A,B)$ be a $3$-separation where $A \cap B = X$. If $G[A]$ has a $K_{3,3}(X)$-minor and $G[B]$ has a $K_{3,1}(X)$ minor, then $G$ has a $K_{3,4}$-minor.
\end{proposition}

\begin{proof}
Contract $G[B]$ down to the $K_{3,1}(X)$-minor and $G[A]$ down to the $K_{3,3}(X)$-minor. Since $A \cap B = X$, there is a vertex from the $K_{3,1}(X)$-minor which is adjacent to all of $X$, and there are $3$ vertices from the $K_{3,3}(X)$-minor adjacent to all of $X$, and all of these vertices are distinct. Therefore $G$ has a $K_{3,4}$-minor.  
\end{proof}

Deciding if a graph has a $K_{3,1}(X)$-minor is easy. 

\begin{observation}
Let $G$ be a graph and $X = \{a,b,c\} \subseteq V(G)$. Then $G$ has a $K_{3,1}(X)$-minor if and only if there is no $X' \subset X$ such that $X'$ is the vertex boundary of a separation $(A,B)$ such that $B = X \setminus X'$.
\end{observation}

\begin{proof}
It is immediate to see if we have such a separation $X'$ that there is no $K_{3,1}(X)$-minor. For the converse if $\{G_{x} | x \in V(K_{3,1})\}$ is a model of a $K_{3,1}(X)$-minor in $G$. Then we can delete the branch sets containing any strict subset of $X$ and remain connected by definition of a $K_{3,1}(X)$-minor. 
\end{proof}

However it appears to be of great difficulty to give a nice characterization of $K_{3,3}(X)$-minors. The best result known appears to be the following bound.

\begin{theorem}[\cite{extermalrootedminorleif}]
\label{rootedk33minors}
Let $G$ be a $3$-connected graph and let $X$ be a set of three vertices. If $|E(G)| \geq 4|V(G)| -9$, then $G$ has a $K_{3,3}(X)$-minor. 
\end{theorem}

It is left as an open problem in \cite{extermalrootedminorleif} if this bound is tight. While the problem of determining when a graph has $K_{3,4}$-minor appears to be quite difficult,  Bohme et al. have a result on $K_{3,k}$-minors, $k \in \mathbb{N}$, for large highly connected graphs. 

\begin{theorem}[\cite{large7connectedk34}]
For any positive integer $k$, there exists a constant $N(k)$ such that every $7$-connected graph $G$ on at least $N(k)$ vertices contains $K_{3,k}$ as a minor.
\end{theorem}

While this is of no use when looking at reduciblity of both Symanzik polynomials and $4$-external momenta, it does apply to reducibility with respect to the first Symanzik polynomial. 

Shifting gears, we left as an open problem whether or not $K_{6}$ was reducible with respect to the first Symanzik polynomial. We show here that even if $K_{6}$ is not reducible, we do not gain that many additional graphs which we know are not reducible. To do this we need to introduce some definitions.

 Let $G$ be a graph. If $u$ and $v$ are non-adjacent vertices in $G$, then adding an edge is simply creating the graph $G \cup \{uv\}$.  Now let $v \in V(G)$ such that $\deg(v) \geq 4$. To split $v$ is to first delete $v$ from $G$, and then add two vertices to $G$, $v_{1}$ and $v_{2}$ such that $v_{1}v_{2} \in E(G)$, each neighbour of $v$ in $G$ is a neighbour of exactly one of $v_{1}$ or $v_{2}$, and $\deg(v_{i}) \geq 3$ for $i \in \{1,2\}$. Now we give a well-known result of Tutte.
 
\begin{theorem}[\cite{wheelsandwhirls}]
A graph is $3$-connected if and only if it is obtained from a wheel by repeatedly adding edges and splitting vertices.
\end{theorem}

A partial extension is given by Seymour's splitter theorem.

\begin{theorem}[\cite{splittertheorem}]
\label{splittertheorem}
Suppose a $3$-connected graph $H \neq W_{3}$ is a proper minor of a $3$-connected graph $G \neq W_n$. Then $G$ has a minor $J$, which is obtained from $H$ by either adding an edge or splitting a vertex.
\end{theorem}

Now we have the following corollary:

\begin{corollary}
\label{k6minors}
Let $G$ be a $3$-connected graph. Suppose $G$ has either  $K_{3,4}$-minor or a $K_{6}$-minor. Then either $G$ has a $K_{3,4}$-minor, or $G$ is isomorphic to $K_{6}$.
\end{corollary}

\begin{proof}
It suffices to consider the case where we have a graph $G$ which has a $K_{6}$-minor and no $K_{3,4}$-minor. We will show that $G$ is isomorphic to $K_{6}$. Suppose the graph $G$ is not isomorphic $K_{6}$. Observe that $G$ non-planar and thus is not a wheel. Then since $G$ has a $K_{6}$-minor, by the splitter theorem, $G$ has a minor $J$ which is obtained by either adding an edge or splitting a vertex of $K_{6}$. As $K_{6}$ is complete, we can only split a vertex, and up to isomorphism splitting any vertex results in the same graph. It is easy to check that after splitting, the resulting graph has $K_{3,4}$ as a subgraph (see Figure \ref{splittingvertex}). But this contradicts that $G$ is not $K_{6}$. 
\end{proof}

\begin{figure}
\begin{center}
\includegraphics[scale=0.5]{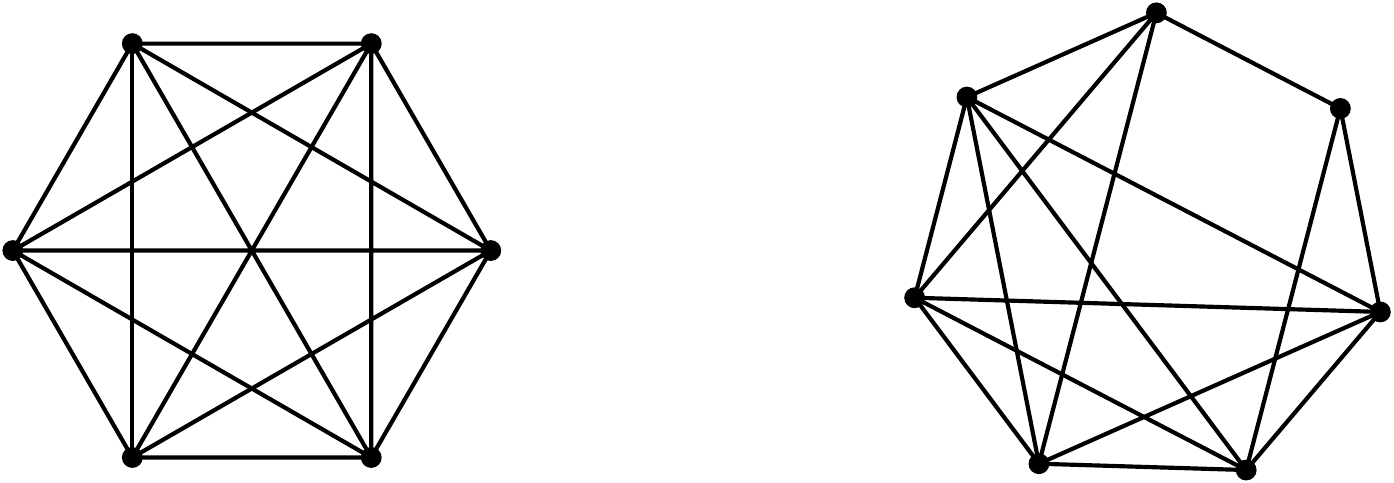}
\caption{The graph $K_{6}$ and $K_{6}$ after splitting any vertex.}
\label{splittingvertex}
\end{center}
\end{figure}

Since both $K_{3,4}$ and $K_{6}$ are $3$-connected, Lemma \ref{lowordersepsminors} applies so the $3$-connected assumption in Corollary \ref{k6minors} does not really matter. Therefore from the point of view of understanding reducibility with respect to the first Symanzik polynomial, if  $K_{6}$ is not reducible then we do not add that many more non-reducible graphs. Then if $K_{6}$ is not reducible, in some sense one can consider $K_6$ an exception since almost all graphs are not reducible for different reasons, in the same sense that $K_{5}$ is an exception for planarity.

\section{$V_8$ plus an edge minors}

One of the other graphs which we know is not reducible for the first Symanzik polynomial is $V_{8} \cup \{02\}$.  Therefore it is of interest to understand graphs which do not contain $V_{8} \cup \{02\}$ as a minor. We do not make any progress towards this, we just outline some known results and a possible approach.

One approach is to start with Robertson's characterization of $V_{8}$-topological minor free graphs. Before stating the theorem we have some definitions.

Let $G$ be a $3$-connected graph and let $(G_{1},G_{2})$ be a $3$-separation and suppose $|V(G)| \geq 5$. Suppose for every partition $(A,B)$ of $E(G)$ either $|A| \leq 3$, or $|B| \leq 3$, or at least four vertices of $G$ are incident with an edge in $A$ and an edge in $B$. Any graph satisfying these conditions is called \textit{internally $4$-connected}. Now we can state Robertson's Theorem.

\begin{theorem}[\cite{NoV8minors}]
\label{graphswithnoV8minors}
 If G is an internally-$4$-connected graph then either $G$ contains a subgraph $S$ which is isomorphic to a subdivision of $V_{8}$ or one of the following structural conditions hold:
\begin{enumerate}
\item{$G$ is planar,}
\item{ $G - \{x,y \}$ is a cycle, for some adjacent $x,y \in V(G)$,}
\item{$G - \{w,x,y,z\}$ is edgeless, for some $w,x,y,z \in V(G)$,}
\item{ $G$ is isomorphic to $L(K_{3,3})$}
\item{ or $|V(G)| \leq 7$.}
\end{enumerate}
\end{theorem}

While the above characterization is for topological minors, as $V_{8}$ has maximum degree $3$, topological minors and minors coincide. In general we have the following well known fact (see, for example, \cite{BondyAndMurty})

\begin{observation}
Let $H$ be a graph with maximum degree $3$. Then a graph $G$ has an $H$-minor if and only if $G$ has an $H$-topological minor. 
\end{observation}

Therefore the above characterization applies to minors. Unfortunately $V_{8} \cup \{02\}$ is not internally $4$-connected, and thus we cannot reduce the problem of determining if a graph has a $V_{8} \cup \{02\}$-minor to internally $4$-connected graphs. However, $V_{8}$ is internally $4$-connected so we can remove the internally $4$-connected condition via standard reduction lemmas. 

\begin{lemma}[\cite{NoV8minors}]
Let $G$ be a $3$-connected graph and let $(G_{1},G_{2})$ be a $3$-separation. Suppose  $|E(G_{1})| \geq  4$ and $4 \leq |E(G_{2})|$. Define $G_{1}^{+}$, $G_{2}^{+}$  to be $G_{1}$ and $G_2$, with new vertices $x_1$ and $x_2$ respectively. Make $x_{1}$ and $x_{2}$ adjacent to each of the three vertices of $G_{1} \cap G_{2}$. Then $G$ has $G^{+}_{1}, G^{+}_{2}$ as topological minors, and $G$ has $V_{8}$ as a topological minor if and only if one of $G^{+}_{1}$ or $G^{+}_{2}$ have $V_{8}$ as a topological minor. 
\end{lemma}

Then from the splitter theorem, to characterize graphs with no $V_{8} \cup \{02\}$-minor, it suffices to show exactly what sequences of splits and adding edges you can do when you start with $V_{8}$. 

We have looked at this, however it is not obvious what the exactly class ends up being. Even if we exclude both $V_{8} \cup \{02\}$ and $K_{3,4}$-minors it is not clear what the class should be. From experimenting we appear to get at infinite class of graphs with a $V_{8}$-minor but no $V_{8} \cup \{02\}$ or $K_{3,4}$-minors. We leave it as an open problem to determine what the class of graphs with no $K_{3,4}$ or $V_{8} \cup \{02\}$-minor is. 

\chapter{Conclusion}

The main problem of this thesis was to give a forbidden minor characterization for reducibility when the graph has no massive edges and four on-shell momenta. We did not manage to do this, and instead looked at giving a structural characterization of what graphs which might be reducible look like. Since $K_{4}$, $W_{4}$, $K_{2,4}$ and $L$ are all forbidden minors when we have $4$ on-shell momenta, we characterized the graphs not containing any of those minors. In the end, what we obtained was that if a graph is reducible with respect to both Symanzik polynomials, then it has to satisfy a few properties. It is either a class $\mathcal{A}$ graph, or it is the spanning subgraph of a class $\mathcal{D}$, $\mathcal{E}$, or $\mathcal{F}$ graph, such that on the web in each of these classes, for every cycle containing $a,b,c,d$ one of the five obstructions from Theorem \ref{w4cuts} occurs. Then along certain terminal separating $2$-chains, we must avoid having an $L'(X')$-minor, which occur according to Lemma \ref{L'(X')minors}. Additionally, from looking at $K_{3,4}$-minors in $K_{4}(X)$-minor free graphs, we learned that in ``$3$-connected pieces'' of a reducible graph, there is at most one triangle $T$ for which $F_{T}$ is non-planar.  So we gained some understanding of what it means for a graph to be reducible, but we do not have the full picture as of yet. We leave the following as a question:

\begin{question}
What is the full forbidden minor set for reducibility with respect to both Symanzik polynomials when we have four on-shell momenta and no edges are massive?
\end{question}

In a similar vein we have the likely much harder question is still open:

\begin{question}
What is the full forbidden minor set for reducibility with respect to the first Symanzik polynomial? 
\end{question}

We leave the following as conjectures and questions as possible starting points for these questions.

\begin{Conjecture}
The graph $K_6$ is not reducible with respect to the first Symanzik polynomial.
\end{Conjecture}

\begin{question}
What is an excluded minor theorem for graphs not containing $K_{3,4}$? What is the excluded minor theorem for graphs not containing $K_{3,4}$ and $V_{8} \cup \{02\}$? Or just $V_{8} \cup \{02\}$?
\end{question}

The next conjecture is due to J{\o}rgensen and Kawarabayshi.

\begin{Conjecture}[\cite{extermalrootedminorleif}]
The bound in Theorem \ref{rootedk33minors} is tight. 
\end{Conjecture}

Additionally, all the structural results in this thesis apply to $2$-connected graphs. However reducibility with $4$-external momenta does not reduce to $2$-connected graphs in all cases. We ask the following question:

\begin{question}
Let $G$ be a graph with a $1$-separation $(A,B)$ such that $A \cap B = \{v\}$. Suppose $G$ has four external momenta and no massive edges. Furthermore suppose that at least two of the momenta lie in $A$ and two of the momenta lie in $B$. When is $G$ reducible with respect to both Symanzik polynomials?
\end{question}

An additional problem which we did not tackle but is interesting is the algorithmic complexity of determining when a graph has a $K_{4}(X)$-minor, $W_{4}(X)$-minor, $K_{2,4}(X)$-minor or an $L(X)$-minor. Monroy and Wood (\cite{root}) conjecture that there is a linear time algorithm in the number of vertices for determining if a graph has a $K_{4}(X)$-minor. This is plausible due to the heavy relationship between $K_{4}(X)$-minors and the two disjoint path problem, which has a practical and efficient algorithm. Given that if we know a graph is $K_{4}(X)$-free, that to determine if the graph is $W_{4}(X)$-free can be reduced to looking a two separations,  we conjecture the following:

\begin{Conjecture}
There is a linear time algorithm for determining if a graph $G$ has a $K_{4}(X)$-minor or $W_{4}(X)$-minor. 
\end{Conjecture} 

Additionally we pose the following question,

\begin{question}
Is there a linear time algorithm for determining if a graph $G$ has a $K_{4}(X)$-minor, $W_{4}(X)$-minor, $K_{2,4}(X)$-minor, or a $L(X)$-minor?
\end{question}

We note there is a cubic time algorithm by appealing to the algorithm of Robertson and Seymour in \cite{Minorsalgorithm}. However one should note that this algorithm is very impractical in general.

Lastly, in general one wants to use reducibility to actually compute integrals, and for that one needs to know a reducible order. Forbidden minors give a method for computing if a graph is reducible, but does not find a reducible order, therefore a natural question is

\begin{question}
Suppose $G$ is reducible. Is there a polynomial time algorithm to find a permutation $\sigma$ such that $G$ is reducible with respect to $\sigma$? 
\end{question}

%   BACK MATTER  %%%%%%%%%%%%%%%%%%%%%%%%%%%%%%%%%%%%%%%%%%%%%%%%%%%%%%%%%%%%%%
%
%   References and appendices. Appendices come after the bibliography and
%   should be in the order that they are referred to in the text.
%
%   If you include figures, etc. in an appendix, be sure to use
%
%       \caption[]{...}
%
%   to make sure they are not listed in the List of Figures.
%

\backmatter%
	\addtoToC{Bibliography}
	\bibliographystyle{plain}
	\bibliography{ThesisBib}

%\begin{appendices} % optional
%	\chapter{Code}
%\end{appendices}
\end{document}